\documentclass[12pt]{amsart}
\usepackage{amsmath,amsthm,amssymb,amscd,cite,array}
\usepackage[dvips]{graphicx}

\newtheorem{pr}{Proposition}
\newtheorem{lem}[pr]{Lemma}
\newtheorem{thm}[pr]{Theorem}
\newtheorem{s}[pr]{Corollary}
\theoremstyle{remark}
\newtheorem{zam}{Remark}
\newtheorem*{obozn}{{\rm\bf Denotation}}
\newtheorem*{obozns}{{\rm\bf Denotations}}

\renewcommand\div{\text{ }\vdots\text{ }}
\newcommand\ndiv{\not\vdots\text{ }}
\newcommand{\myChar}{\mathrm{char\,}K}
\newcommand{\myNod}{\text{{\rm gcd}}}
\newcommand{\Hom}{\mathrm{Hom}}
\renewcommand{\Im}{\mathrm{Im}}
\newcommand{\Ker}{\mathrm{Ker}}
\newcommand{\HH}{\mathrm{HH}}
\newcommand{\N}{\mathbb{N}}
\newcommand{\Z}{\mathbb{Z}}
\newcommand{\cl}{\mathrm{cl}}

\def\a{\alpha}

\def\g{\gamma}
\def\le{\leqslant}
\def\ge{\geqslant}
\def\ra{\rightarrow}

\begin{document}

\title{Hochschild cohomology ring for self-injective algebras of tree class $E_6$. II.}
\author{Mariya Kachalova}
\email{mashakachalova@mail.ru}

\begin{abstract}
We describe the Hochschild cohomology ring for a family of self-injective
algebras of tree class $E_6$ in terms of generators and relations. 
Together with the results of the previous paper, 
this gives a complete description of the Hochschild cohomology ring for a self-injective
algebras of tree class $E_6$.
\end{abstract}
\maketitle

\tableofcontents

\section{Introduction}

Consider a self-injective basic algebra of finite representation
type over an algebraically closed field. According to Riedtmann's
classification, the stable $AR$-quiver of such an algebra can be
described with the help of an associated tree, which must be
congruent with one of the Dynkin diagrams $A_n, D_n, E_6, E_7$, or
$E_8$ (see \cite{Riedt}). The complete description of the Hochschild
cohomology ring was obtained for an algebras of the types $A_n$ and
$D_n$, see \cite{Erd,Gen&Ka,Ka,Pu} (type $A_n$) and
\cite{Volkov1,Volkov2,Volkov3,Volkov4,Volkov5,Volkov6} (type $D_n$).
For one of the two self-injective algebras of tree class $E_6$ Hochschild 
cohomology ring was obtained in \cite{Pu2}. 

In this paper we consider 
the second part of algebras of tree class $E_6$. Any algebra of the class
$E_6$ is derived equivalent to the path algebra for some quiver with
relations. Namely, let $\mathcal Q_s$ ($s\in\N$) is the following
quiver:
\newpage

\begin{figure}[h]
\includegraphics[width=12cm, scale=1]{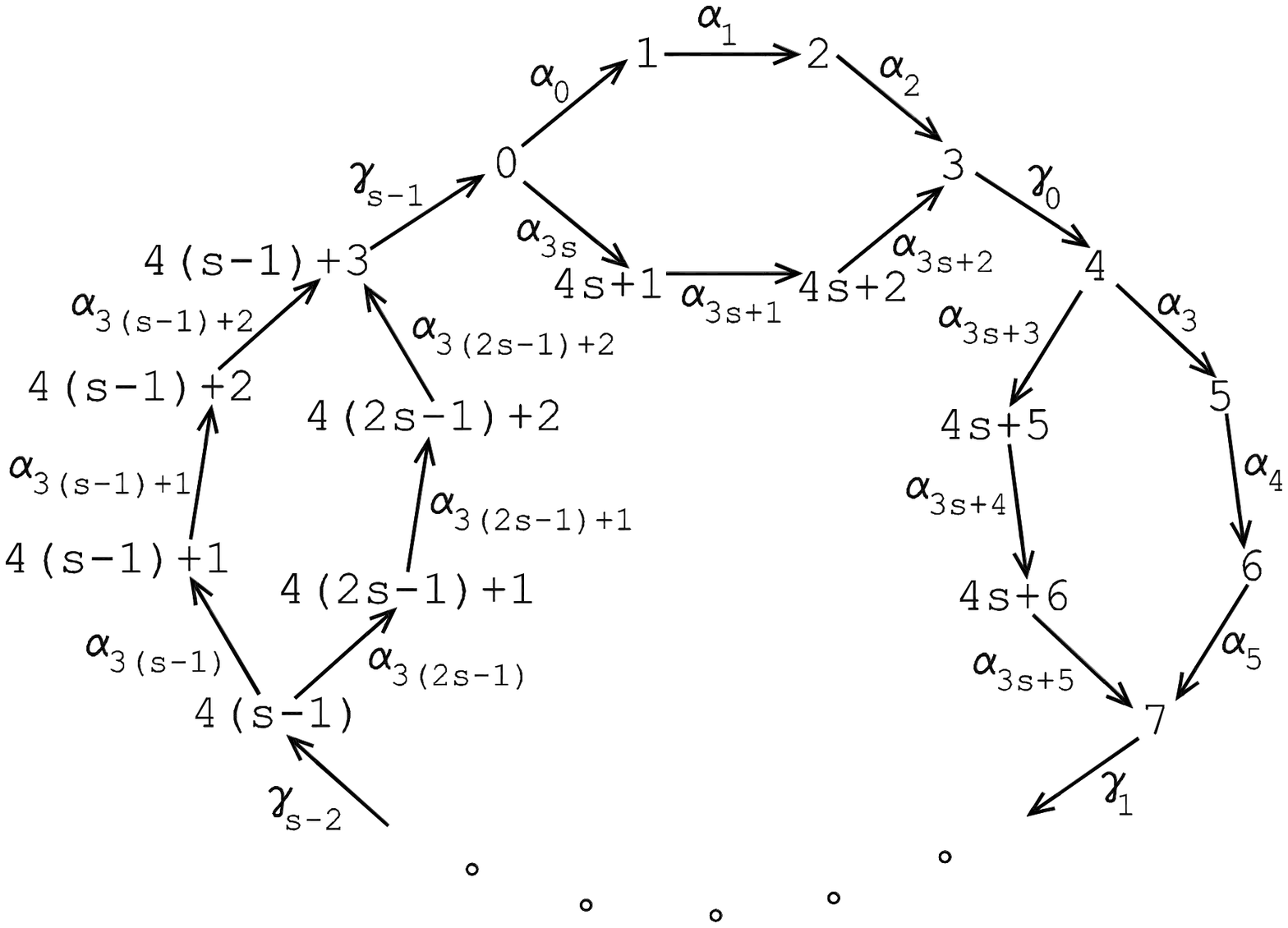}
\end{figure}

Then any algebra of the class $E_6$ is derived equivalent to one of the two following algebras:

1) $R_s=K\left[\mathcal Q_s\right]/I$, where $K$ is a field, and $I$ is the ideal in the path
algebra $K\left[\mathcal Q_s\right]$ of the quiver $\mathcal Q_s$, generated by

a) all the paths of length $5$;

b) the expressions of the form: 
$$\a_{3t+2}\a_{3t+1}\a_{3t}-\a_{3(t+s)+2}\a_{3(t+s)+1}\a_{3(t+s)},$$ 
$$\a_{3t}\g_{t-1}\a_{3(t+s)-1},\text{ }t\in [0,2s-1] \setminus\{0,s\},$$
$$\a_0\g_{s-1}\a_{3(2s-1)+2},\text{ }\a_{3s}\g_{s-1}\a_{3(s-1)+2}.$$

2) $R_s^\prime=K\left[\mathcal Q_s\right]/I^\prime$, where $K$ is a field, and $I^\prime$ is the
ideal in the path algebra $K\left[\mathcal Q_s\right]$ of the quiver $\mathcal Q$, generated by

a) all the paths of length $5$;

b) the expressions of the form: $$\a_{3t+2}\a_{3t+1}\a_{3t}-\a_{3(t+s)+2}\a_{3(t+s)+1}\a_{3(t+s)},$$ 
$$\a_{3t}\g_{t-1}\a_{3(t+s)-1}.$$

Henceforth we will often omit indexes in arrows $\a_i$ and $\g_i$ as long as subscripts are
clear from the context.

The present paper is dedicated to the study of Hochschild cohomology ring structure for algebra
$R_s^\prime$. We obtain the ring structure in terms
of generators and relations. In studies of the structure of cohomology ring we will construct the
bimodule resolution of $R_s^\prime$, which could be seen as a whole result.

\section{Statement of the main results}

In what follows, we assume $n=6$.

Let $\HH^t(R)$ is the $t$th group of the Hochschild cohomology ring of $R$ with coefficients in
$R$. Let $\ell$ be the aliquot, and $r$ be the residue of division of $t$ by $11$, $m$ be the
aliquot of division of $r$ by $2$.

Consider the case of $s>1$. To describe Hochschild cohomology ring of algebra $R_s^\prime$ we must
introduce the following conditions on an arbitrary degree $t$:

$($1$)$ $r=0$, $\ell(n+s)+m\equiv 0(2s)$, $\ell\div 2$ or $\myChar=2$;\label{degs}

$($2$)$ $r=0$, $\ell(n+s)+m\equiv s+1(2s)$, $\ell\div 2$, $\myChar=3$;

$($3$)$ $r=1$, $\ell(n+s)+m\equiv 0(2s)$, $\ell\div 2$ or $\myChar=2$;

$($4$)$ $r=1$, $\ell(n+s)+m\equiv s(2s)$, $\ell\ndiv 2$ or $\myChar=2$;

$($5$)$ $r=2$, $\ell(n+s)+m\equiv s+1(2s)$, $\ell\ndiv 2$ or $\myChar=2$;

$($6$)$ $r=3$, $\ell(n+s)+m\equiv 0(2s)$;

$($7$)$ $r=3$, $\ell(n+s)+m\equiv s(2s)$, $\myChar=2$;

$($8$)$ $r=4$, $\ell(n+s)+m\equiv s+1(2s)$, $\myChar=2$;

$($9$)$ $r=4$, $\ell(n+s)+m\equiv s(2s)$, $\ell\ndiv 2$, $\myChar=3$;

$($10$)$ $r=4$, $\ell(n+s)+m\equiv 1(2s)$;

$($11$)$ $r=5$, $\ell(n+s)+m\equiv 0(2s)$, $\ell\div 2$, $\myChar=3$;

$($12$)$ $r=5$, $\ell(n+s)+m\equiv s(2s)$, $\ell\ndiv 2$, $\myChar=3$;

$($13$)$ $r=6$, $\ell(n+s)+m\equiv 0(2s)$, $\myChar=2$;

$($14$)$ $r=6$, $\ell(n+s)+m\equiv 1(2s)$, $\ell\div 2$, $\myChar=3$;

$($15$)$ $r=6$, $\ell(n+s)+m\equiv s(2s)$;

$($16$)$ $r=7$, $\ell(n+s)+m\equiv 0(2s)$, $\myChar=2$;

$($17$)$ $r=7$, $\ell(n+s)+m\equiv s(2s)$;

$($18$)$ $r=8$, $\ell(n+s)+m\equiv 0(2s)$, $\ell\div 2$ or $\myChar=2$;

$($19$)$ $r=9$, $\ell(n+s)+m\equiv 0(2s)$, $\ell\div 2$ or $\myChar=2$;

$($20$)$ $r=9$, $\ell(n+s)+m\equiv s(2s)$, $\ell\ndiv 2$ or $\myChar=2$;

$($21$)$ $r=10$, $\ell(n+s)+m\equiv s+1(2s)$, $\ell\ndiv 2$ or $\myChar=2$;

$($22$)$ $r=10$, $\ell(n+s)+m\equiv 0(2s)$, $\ell\ndiv 2$, $\myChar=3$.

Let $$M_0=\frac{2s}{\myNod(n+s,2s)},\quad M=\begin{cases}11M_0,\quad\myChar=2\text{ or }M_0\div 4;\\22M_0\quad\text{otherwise.}\end{cases}$$

\begin{zam}
We will prove in paragraph \ref{sect_res} that the minimal period of bimodule resolution of $R_s^\prime$
is $M$.
\end{zam}

Let $\{t_{1, i},\dots,t_{\alpha_i, i} \}$ be a set of all degrees $t$, that satisfy the conditions
of item $i$ from the above list, and such that $0\le t_{j, i}<M$ $(j=1,\dots,\alpha_i)$. Consider
the set $$\mathcal
X=\bigcup_{i=1}^{22}\left\{X^{(i)}_{t_{j,i}}\right\}_{j=1}^{\alpha_i}\cup\{T\},$$ and define a
graduation of polynomial ring $K[\mathcal X]$ such that
\begin{align*}\label{degs2}
&\deg X^{(i)}_{t_{j,i}}=t_{j, i} \:\text{for all} \: i=1,\dots,22 \:\text{and}\: j=1,\dots,\alpha_i;\tag{$\circ$}\\
&\deg T=M.
\end{align*}

\begin{zam}\label{brief_notation}
Hereafter we shall use simplified denotation $X^{(i)}$ instead of $X^{(i)}_{t_{j,i}}$, since lower
indexes are clear from context.
\end{zam}

\begin{obozn}
$$\widetilde X^{(i)}=\begin{cases}X^{(i)},\quad\deg\widetilde
X^{(i)}<\deg T;\\TX^{(i)},\quad\text{otherwise.}\end{cases}$$
\end{obozn}

Define a graduate $K$-algebra $\mathcal A=K[\mathcal X]/I$, where $I$ is the ideal generated by
homogeneous elements corresponding to the following relations.
\begin{align*}
&X^{(3)}X^{(2)}=X^{(3)}X^{(3)}=X^{(3)}X^{(5)}=X^{(3)}X^{(7)}=X^{(3)}X^{(8)}=0;\\
&X^{(3)}X^{(9)}=X^{(3)}X^{(10)}=X^{(3)}X^{(11)}=X^{(3)}X^{(12)}=X^{(3)}X^{(13)}=0;\\
&X^{(3)}X^{(14)}=X^{(3)}X^{(16)}=X^{(3)}X^{(17)}=X^{(3)}X^{(19)}=X^{(3)}X^{(21)}=X^{(3)}X^{(21)}=0;\\
&X^{(3)}X^{(1)}=\widetilde X^{(3)},\quad X^{(3)}X^{(4)}=\widetilde X^{(5)},\quad X^{(3)}X^{(6)}=\widetilde X^{(10)};\\
&X^{(3)}X^{(15)}=\widetilde X^{(17)},\quad X^{(3)}X^{(18)}=\widetilde X^{(19)},\quad X^{(3)}X^{(20)}=\widetilde X^{(21)};\\
\end{align*}
\begin{align*}
X^{(4)}X^{(6)}&=\begin{cases}\widetilde X^{(8)},\quad\myChar
=2,\\0,\quad\text{otherwise};\end{cases}&\text{(r1)}\\
X^{(6)}X^{(6)}&=\begin{cases}-s\widetilde X^{(14)},\quad\myChar
=3,\\0,\quad\text{otherwise};\end{cases}&\text{(r2)}\\
X^{(4)}X^{(15)}&=\begin{cases}\widetilde X^{(16)},\quad\myChar
=2,\\0,\quad\text{otherwise};\end{cases}&\text{(r3)}\\
X^{(6)}X^{(18)}&=\begin{cases}-s\widetilde X^{(2)},\quad\myChar
=3,\\0,\quad\text{otherwise};\end{cases}&\text{(r4)}\\
X^{(15)}X^{(20)}&=\begin{cases}\widetilde X^{(8)},\quad\myChar
=2,\\0,\quad\text{otherwise};\end{cases}&\text{(r5)}\\
X^{(18)}X^{(18)}&=\begin{cases}s\widetilde X^{(12)},\quad\myChar
=3,\\0,\quad\text{otherwise};\end{cases}&\text{(r6)}\\
X^{(18)}X^{(20)}&=\begin{cases}s\widetilde X^{(14)},\quad\myChar
=3,\\0,\quad\text{otherwise}.\end{cases}&\text{(r7)}
\end{align*}

Describe the rest relations as a tables (numbers (r1)--(r7) in tables cells are the number of
relation that defines a multiplication of the following elements).

\setlength{\extrarowheight}{1mm}
\begin{tabular}{c|c|c|c|c|c|c|c|c}
&$X^{(1)}$&$X^{(2)}$&$X^{(4)}$&$X^{(6)}$&$X^{(7)}$&$X^{(8)}$&$X^{(9)}$&$X^{(11)}$\\
\hline
$X^{(1)}$&$X^{(1)}$&$X^{(2)}$&$X^{(4)}$&$X^{(6)}$&$X^{(7)}$&$X^{(8)}$&$X^{(9)}$&$X^{(11)}$ \\
\hline
$X^{(2)}$& &0&0&0&0&0&$-X^{(10)}$&0 \\
\hline
$X^{(4)}$& & & 0&(1)&$-X^{(10)}$&0&$X^{(11)}$&0 \\
\hline
$X^{(6)}$& & & &(2)&0&0&$sX^{(17)}$&0\\
\hline
$X^{(7)}$& & & & &0&0&$0$&0 \\
\hline
$X^{(8)}$& & & & & &0&0&0\\
\hline
$X^{(9)}$& & & & & & &0&0\\
\hline
$X^{(11)}$& & & & & &  & &0 \\
\end{tabular}

$\quad$

$\quad$

\begin{tabular}{c|c|c|c|c|c|c|c|c}
&$X^{(12)}$&$X^{(13)}$&$X^{(14)}$&$X^{(15)}$&$X^{(16)}$&$X^{(18)}$&$X^{(20)}$&$X^{(22)}$\\
\hline
$X^{(1)}$&$X^{(12)}$&$X^{(13)}$&$X^{(14)}$&$X^{(15)}$&$X^{(16)}$&$X^{(18)}$&$X^{(20)}$&$X^{(22)}$ \\
\hline
$X^{(2)}$&0&0&0&$X^{(14)}$&0&0&0&$-X^{(21)}$ \\
\hline
$X^{(4)}$&$X^{(14)}$&$X^{(17)}$&0&(3)&0&$X^{(20)}$&0&0 \\
\hline
$X^{(6)}$&0&$X^{(19)}$&0&$-X^{(20)}$&0&(4)&0&$sX^{(5)}$\\
\hline
$X^{(7)}$&0&0&0&$X^{(19)}$&$X^{(21)}$&0&0&0 \\
\hline
$X^{(8)}$&0&$X^{(21)}$&0&0&0&0&0&0\\
\hline
$X^{(9)}$&$X^{(19)}$&0&$X^{(21)}$&$-X^{(22)}$&0&$sX^{(3)}$&$sX^{(5)}$&0\\
\hline
$X^{(11)}$&$X^{(21)}$&0&0&0&0&$sX^{(5)}$&0&0 \\
\end{tabular}

$\quad$

$\quad$

\begin{tabular}{c|c|c|c|c|c|c|c|c}
&$X^{(12)}$&$X^{(13)}$&$X^{(14)}$&$X^{(15)}$&$X^{(16)}$&$X^{(18)}$&$X^{(20)}$&$X^{(22)}$\\
\hline
$X^{(12)}$&0&0&0&$-X^{(2)}$&0&0&0&$-X^{(10)}$ \\
\hline
$X^{(13)}$& &0&0&$X^{(3)}$&$X^{(5)}$&$X^{(7)}$&$X^{(10)}$&0 \\
\hline
$X^{(14)}$& & &0&0&0&0&0&0\\
\hline
$X^{(15)}$& & & &$-X^{(4)}$&0&$X^{(6)}$&(5)&$X^{(11)}$ \\
\hline
$X^{(16)}$& & & & &0&$X^{(8)}$&0&0 \\
\hline
$X^{(18)}$& & & & & &(6)&(7)&$-sX^{(17)}$\\
\hline
$X^{(20)}$& & & & & & &0&0\\
\hline
$X^{(22)}$& & & & & &  & &0 \\
\end{tabular}

\begin{thm}\label{main_thm}
Let $s>1$, $R=R_s^\prime$ is algebra of the type $E_6$. Then the Hochschild cohomology ring $\HH^*(R)$ is
isomorphic to $\mathcal A$ as a graded $K$-algebra.
\end{thm}

Consider the case of $s=1$.

Let us introduce the set $$\mathcal X^\prime=\begin{cases}\mathcal X\cup \left\{X^{(23)}_0,
X^{(24)}_0\right\},\quad\myChar\ne 3;\\
\mathcal X\cup \left\{X^{(24)}_0\right\}, \quad
\myChar=3;\end{cases}$$ and define a graduation of polynomial ring $K[\mathcal X^\prime]$ such that
\begin{align*}
&\deg X^{(i)}_{t_{j,i}}=t_{j, i} \:\text{for all} \: i=1,\dots,22 \:\text{and}\: j=1,\dots,\alpha_i;\\
&\deg T=M \text{ (similar to (\ref{degs2}))};\\&\deg X^{(23)}_0=\deg X^{(24)}_0=0.
\end{align*}

Define a graduate $K$-algebra $\mathcal A^\prime=K[\mathcal X^\prime]/I^\prime$, where $I^\prime$
is the ideal generated by homogeneous elements corresponding to the relations described in the case
of $s>1$, and by the following relations:
\begin{align*}
X^{(1)}X^{(23)}=&\begin{cases}\widetilde X^{(23)},\quad
t_1=0;\\0,\quad\text{otherwise};\end{cases}\\
X^{(1)}X^{(24)}=&\begin{cases}\widetilde X^{(24)},\quad
t_1=0;\\\widetilde X^{(2)},\quad t_1>0\text{ and
}\myChar=3;\\0,\quad\text{otherwise};\end{cases}\\
X^{(9)}X^{(24)}=&-\widetilde X^{(10)};\\
X^{(15)}X^{(24)}=&\begin{cases}\widetilde X^{(14)},\quad
\myChar=3;\\0,\quad\text{otherwise};\end{cases}\\
X^{(22)}X^{(24)}=&-\widetilde X^{(21)};\\
X^{(j)}X^{(i)}=&0,\quad j\in[2, 24]\setminus\{9,15,22\},\quad
i\in\{23, 24\},
\end{align*}
where $t_1$ denotes a degree of the element $X^{(1)}$.

\begin{thm}\label{main_thm2}
Let $s=1$, $R=R_1^\prime$ is algebra of the type $E_6$. Then the Hochschild cohomology ring $\HH^*(R)$ is
isomorphic to $\mathcal A^\prime$ as a graded $K$-algebra.
\end{thm}

\begin{zam}
From the descriptions of rings $\HH^*(R)$ given in theorems
\ref{main_thm} and \ref{main_thm2} it implies, in particular, that
they are commutative.
\end{zam}

\section{Bimodule resolution}\label{sect_res}

We will construct the minimal projective bimodule resolution of the $R$ in the following form: $$
\dots\longrightarrow Q_3\stackrel{d_2}\longrightarrow Q_2\stackrel{d_1}\longrightarrow
Q_1\stackrel{d_0}\longrightarrow Q_0\stackrel\varepsilon\longrightarrow R\longrightarrow 0
$$

Let $\Lambda$ be an enveloping algebra of algebra $R$. Then $R$--$R$-bimodules can be considered as
left $\Lambda$-modules.

\begin{obozns}$\quad$

(1) Let $e_i,\text{ }i\in \Z_{(n+2)s}=\{0, 1,\dots, (n+2)s-1\},$ be the idempotents of the algebra
$K\left[\mathcal Q_s\right]$, that correspond to the vertices of the quiver $\mathcal Q_s$.

(2) Denote by $P_{i,j}=R(e_i\otimes e_j)R=\Lambda(e_i\otimes e_j)$, $i,j\in \Z_{(n+2)s}$. Note that the
modules $P_{i,j}$, forms the full set of the (pairwise non-isomorphic
 by) indecomposable projective $\Lambda$-modules.

(3) For $a\in\Z$, $t\in\N$ we denote the smallest nonnegative deduction of $a$ modulo $t$ with
$(a)_t$ (in particular, $0\le(a)_t\le t-1$).

\end{obozns}

Let $R=R_s^\prime$. We introduce an automorphism $\sigma\text{: }R\rightarrow R$, which is mapping as
follows:
$$\sigma(e_i)=e_{4(n+s)+i},$$
 $$\sigma(\g_i)=\begin{cases}\g_{i+n},\quad (i)_s=s-1;\\-\g_{i+n},\quad (i)_s<s-1,\end{cases}\quad\sigma(\a_i)=\begin{cases}
 -\a_{3(n+s)+i},\quad (i)_3=0,\text{ }(i)_{6s}<3s;\\
 \a_{3(n+s)+i},\quad (i)_3=0,\text{ }(i)_{6s}\ge 3s;\\
 -\a_{3(n+s)+i},\quad (i)_3=1;\\
 -\a_{3(n+s)+i},\quad (i)_3=2,\text{ }(i)_{6s}\ge 3s;\\
 \a_{3(n+s)+i},\quad (i)_3=2,\text{ }(i)_{6s}<3s.\\
 \end{cases}$$

Define the helper functions $f\text{: }\Z\times\Z\rightarrow\Z$ and $h\text{:
}\Z\times\Z\rightarrow\Z$, which act in the following way:
$$f(x,y)=\begin{cases}1,\quad x=y;\\0,\quad x\ne y,\end{cases}\quad
h(x,y)=\begin{cases}
1,\quad x\div 2,\text{ }x<y;\\0,\quad x\ndiv 2,\text{ }x<y;\\
1,\quad x\ndiv 2,\text{ }x\ge y;\\0,\quad x\div 2,\text{ }x\ge y.
\end{cases}$$

Introduce $Q_r\text{ }(r\le 10)$. Let $m$ be the aliquot of division of $r$ by $2$ for considered
degree $r$. We have
\begin{align*}
Q_{2m}&=\bigoplus_{r=0}^{s-1} Q_{2m,r}^\prime,\quad 0\le m\le n-1,\\
Q_{2m+1}&=\bigoplus_{r=0}^{s-1} Q_{2m+1,r}^\prime,\quad 0\le m\le
n-2,
\end{align*}
where

\begin{multline*}
Q_{2m,r}^\prime=\left(\bigoplus_{i=0}^{f(m,2)}P_{4(r+m)-1+h(m,2)+i,4r}\right)\\
\oplus\bigoplus_{i=0}^{f(m,3)}\bigoplus_{j=0}^1P_{4(r+m+js+f(m,2)s+f(m,5)s)+2-h(m,3)+i(4s+1),4(r+js)+1}\\
\oplus\bigoplus_{i=0}^{f(m,2)}\bigoplus_{j=0}^1P_{(r+m+js+f(m,1)s+f(m,4)s+f(m,5)s)+1+h(m,2)+i(4s+1),4(r+js)+2}\\
\oplus\left(\bigoplus_{i=0}^{f(m,3)}P_{4(r+m+1)-h(m,3)+i,4r+3}\right);
\end{multline*}

\begin{multline*}
Q_{2m+1,r}^\prime=\left(\bigoplus_{i=0}^{1-f(m,4)}P_{4(r+m)+1+h(m,0)+2f(m,4)+4si,4r}\right)\\
\oplus\bigoplus_{j=0}^1P_{4(r+m+1+js)-h(m,0)-2f(m,0),4(r+js)+1}\\
\oplus\bigoplus_{j=0}^1P_{4(r+m+1+js+f(m,4)s)-h(m,5)+2f(m,4),4(r+js)+2}\\
\oplus\left(\bigoplus_{i=0}^{1-f(m,0)}P_{4(r+m+1)+1+h(m,5)-2f(m,0)+4si,4r+3}\right).
\end{multline*}

Now we shall describe differentials $d_r$ for $r\le 10$. Since $Q_i$ are direct sums, their
elements can be concerned as column vectors, hence differentials can be described as matrixes
(which are being multiplied by column vectors from the right). Now let us describe the matrixes of
differentials componentwisely.

\begin{zam}
\textup{Numeration of lines and columns always starts with zero.}
\end{zam}

\begin{obozns}$\quad$

(1) Denote by $w_{i\ra j}$ the way that starts in $i$th vertex and ends in $j$th.

(2) Fot $j$th column of differential matrix let $j_2$ be the aliquot and $i_2$ be the residue of division of $j$ by $s$.

\end{obozns}

Define the helper functions $f_0\text{:
}\Z\times\Z\rightarrow\Z$, $f_1\text{:
}\Z\times\Z\rightarrow\Z$ and $f_2\text{: }\Z\times\Z\rightarrow\Z$,
which act in the following way:
$$
f_0(x,y)=\begin{cases}1,\quad x<y;\\0\quad x\ge y,\end{cases}\quad f_1(x,y)=\begin{cases}1,\quad
x<y;\\-1\quad x\ge y,\end{cases}\quad f_2(x,y)=\begin{cases}1,\quad
x=y;\\-1\quad x\ne y.\end{cases}
$$
\centerline{\bf Description of the $d_0$}
\centerline{$d_0:Q_1\rightarrow Q_0\text{ -- is an } (7s\times 6s)\text{ matrix}.$}
If $0\le j<2s$, then $$(d_0)_{ij}=\begin{cases}
w_{4(j+m)\rightarrow 4(j+m)+1}\otimes e_{4j},\quad i=(j)_s;\\
-e_{4(j+m)+1}\otimes w_{4j\rightarrow 4j+1},\quad i=j+s;\\
0\quad\text{otherwise.}\end{cases}$$ If $2s\le j<4s$, then
$$(d_0)_{ij}=\begin{cases}
w_{4(j+m)+1\rightarrow 4(j+m)+2}\otimes e_{4j+1},\quad i=j-s;\\
-e_{4(j+m)+2}\otimes w_{4j+1\rightarrow 4j+2},\quad i=j+s;\\
0\quad\text{otherwise.}\end{cases}$$ If $4s\le j<6s$, then
$$(d_0)_{ij}=\begin{cases}
w_{4(j+m)+2\rightarrow 4(j+m)+3}\otimes e_{4j+2},\quad i=j-s;\\
-e_{4(j+m)+3}\otimes w_{4j+2\rightarrow 4j+3},\quad i=(j)_s+5s;\\
0\quad\text{otherwise.}\end{cases}$$ If $6s\le j<7s$, then
$$(d_0)_{ij}=\begin{cases}
-e_{4(j+m+1)}\otimes w_{4j+3\rightarrow 4(j+1)},\quad i=(j+1)_s;\\
w_{4(j+m)+3\rightarrow 4(j+m+1)}\otimes e_{4j+3},\quad i=j-s;\\
0\quad\text{otherwise.}\end{cases}$$

\centerline{\bf Description of the $d_1$}
\centerline{$d_1:Q_2\rightarrow Q_1\text{ -- is an } (6s\times 7s)\text{ matrix}.$}
If $0\le j<s$, then $$(d_1)_{ij}=\begin{cases}
w_{4(j+m)+1+j_1\rightarrow 4(j+m)+3}\otimes w_{4j\rightarrow 4j+j_1},\quad i=j+2j_1s,\text{ }0\le j_1\le 2;\\
-w_{4(j+m+s)+1+j_1\rightarrow 4(j+m)+3}\otimes w_{4j\rightarrow 4(j+s)+j_1},\quad i=j+(2j_1+1)s,\text{ }0\le j_1\le 2;\\
0\quad\text{otherwise.}\end{cases}$$ If $s\le j<3s$, then
$$(d_1)_{ij}=\begin{cases}
w_{4(j+m+s+1)+1\rightarrow 4(j+m+s+1)+2}\otimes w_{4(j+s)+1\rightarrow 4(j+1)},\quad i=j-s+1,\text{ }j<2s;\\
w_{4(j+m+s+1)+1\rightarrow 4(j+m+s+1)+2}\otimes w_{4(j+s)+1\rightarrow 4(j+1)},\quad i=(j+s+1)_{2s},\text{ }j\ge 2s;\\
w_{4(j+m+s)+2\rightarrow 4(j+m+s+1)+2}\otimes e_{4(j+s)+1},\quad i=j+s;\\
e_{4(j+m+s+1)+2}\otimes w_{4(j+s)+1\rightarrow 4(j+s+1)+1},\quad i=j+s+1,\text{ }j<2s;\\
e_{4(j+m+s+1)+2}\otimes w_{4(j+s)+1\rightarrow 4(j+s+1)+1},\quad i=(j+s+1)_{2s}+2s,\text{ }j\ge 2s;\\
w_{4(j+m)+3\rightarrow 4(j+m+s+1)+2}\otimes w_{4(j+s)+1\rightarrow 4(j+s)+2},\quad i=j+3s;\\
w_{4(j+m+1)\rightarrow 4(j+m+s+1)+2}\otimes w_{4(j+s)+1\rightarrow 4j+3},\quad i=(j)_s+6s;\\
0\quad\text{otherwise.}\end{cases}$$ If $3s\le j<5s$, then
$$(d_1)_{ij}=\begin{cases}
e_{4(j+m+1)+1}\otimes w_{4(j+s)+2\rightarrow 4(j+1)},\quad i=(j+1)_{2s},\text{ }j<4s;\\
e_{4(j+m+1)+1}\otimes w_{4(j+s)+2\rightarrow 4(j+1)},\quad i=j-4s+1,\text{ }j\ge 4s;\\
w_{4(j+m)+3\rightarrow 4(j+m+1)+1}\otimes e_{4(j+s)+2},\quad i=j+s;\\
w_{4(j+m+1)\rightarrow 4(j+m+1)+1}\otimes w_{4(j+s)+2\rightarrow 4j+3},\quad i=(j)_s+6s;\\
0\quad\text{otherwise.}\end{cases}$$ If $5s\le j<6s$, then
$$(d_1)_{ij}=\begin{cases}
w_{4(j+m+(1-f(i_2, s-1))s+1)+1\rightarrow 4(j+m+2)}\otimes w_{4j+3\rightarrow 4(j+1)},\quad i=(j+1)_s;\\
w_{4(j+m+(1-f(i_2, s-1))s+1)+2\rightarrow 4(j+m+2)}\otimes w_{4j+3\rightarrow 4(j+(1-f(i_2, s-1))s+1)+1},\quad i=(j+1)_s+2s;\\
w_{4(j+m+1)+3\rightarrow 4(j+m+2)}\otimes w_{4j+3\rightarrow 4(j+(1-f(i_2, s-1))s+1)+2},\quad i=(j+1)_s+4s;\\
w_{4(j+m+1)\rightarrow 4(j+m+2)}\otimes e_{4j+3},\quad i=j+s;\\
e_{4(j+m+2)}\otimes w_{4j+3\rightarrow 4(j+1)+3},\quad i=(j+1)_s+6s;\\
0\quad\text{otherwise.}\end{cases}$$

\centerline{\bf Description of the $d_2$}
\centerline{$d_2:Q_3\rightarrow Q_2\text{ -- is an } (8s\times 6s)\text{ matrix}.$}
If $0\le j<2s$, then $$(d_2)_{ij}=\begin{cases}
w_{4(j+m-1)+3\rightarrow 4(j+m)+2}\otimes e_{4j},\quad i=(j)_s;\\
-f_1(j, s)e_{4(j+m)+2}\otimes w_{4j\rightarrow 4j+1},\quad i=j+s;\\
f_1(j, s)w_{4(j+m)+1\rightarrow 4(j+m)+2}\otimes w_{4j\rightarrow 4(j+s)+2},\quad i=(j+s)_{2s}+3s;\\
0\quad\text{otherwise.}\end{cases}$$ If $2s\le j<4s$, then
$$(d_2)_{ij}=\begin{cases}
-f_1(j,  3s)f_1(i_2, s-1)e_{4(j+m)+3}\otimes w_{4j+1\rightarrow 4(j+1)},\quad i=(j+1)_s;\\
w_{4(j+m)+2\rightarrow 4(j+m)+3}\otimes e_{4j+1},\quad i=j-s;\\
-w_{4(j+m+s)+1\rightarrow 4(j+m)+3}\otimes w_{4j+1\rightarrow 4j+2},\quad i=j+s;\\
0\quad\text{otherwise.}\end{cases}$$ If $4s\le j<6s$, then
$$(d_2)_{ij}=\begin{cases}
w_{4(j+m)+3\rightarrow 4(j+m+1)}\otimes w_{4j+2\rightarrow 4(j+1)},\quad i=(j+1)_s,\text{ }j<5s-1\text{ or }j=6s-1;\\
w_{4(j+m+s)+1\rightarrow 4(j+m+1)}\otimes e_{4j+2},\quad i=j-s;\\
-e_{4(j+m+1)}\otimes w_{4j+2\rightarrow 4j+3},\quad i=(j)_s+5s;\\
0\quad\text{otherwise.}\end{cases}$$ If $6s\le j<8s$, then
$$(d_2)_{ij}=\begin{cases}
-w_{4(j+m)+3\rightarrow 4(j+m+1)+1}\otimes w_{4j+3\rightarrow 4(j+1)},\quad i=(j+1)_s,\text{ }j<7s-1\text{ or }j=8s-1;\\
-e_{4(j+m+1)+1}\otimes w_{4j+3\rightarrow 4(j+s+1)+2},\quad i=(j+s+1)_{2s}+3s;\\
w_{4(j+m+1)\rightarrow 4(j+m+1)+1}\otimes e_{4j+3},\quad i=(j)_s+5s;\\
0\quad\text{otherwise.}\end{cases}$$

\newpage\centerline{\bf Description of the $d_3$}
\centerline{$d_3:Q_4\rightarrow Q_3\text{ -- is an } (9s\times 8s)\text{ matrix}.$}
If $0\le j<2s$, then $$(d_3)_{ij}=\begin{cases}
w_{4(j+m+j_2s)+2\rightarrow 4(j+m)+3+j_2}\otimes e_{4j},\quad i=(j)_s;\\
-w_{4(j+m+s)+2+2j_2\rightarrow 4(j+m)+3+j_2}\otimes w_{4j\rightarrow 4j+2j_2},\quad i=j+s+3j_2s;\\
e_{4(j+m)+3+j_2}\otimes w_{4j\rightarrow 4(j+j_2s)+1+j_2},\quad i=j+(2+j_2)s;\\
w_{4(j+m)+3\rightarrow 4(j+m)+3+j_2}\otimes w_{4j\rightarrow 4(j+s)+1},\quad i=(j+s)_{2s}+2s;\\
0\quad\text{otherwise.}\end{cases}$$ If $2s\le j<4s$, then
$$(d_3)_{ij}=\begin{cases}
w_{4(j+m)+3\rightarrow 4(j+m+s+1)+1}\otimes e_{4j+1},\quad i=j;\\
w_{4(j+m+1)\rightarrow 4(j+m+s+1)+1}\otimes w_{4j+1\rightarrow 4j+2},\quad i=j+2s;\\
e_{4(j+m+s+1)+1}\otimes w_{4j+1\rightarrow 4j+3},\quad i=(j+s)_{2s}+6s;\\
0\quad\text{otherwise.}\end{cases}$$ If $4s\le j<6s$, then
$$(d_3)_{ij}=\begin{cases}
-f_1(j, 6s-1)e_{4(j+m+1)+2}\otimes w_{4(j+s)+2\rightarrow 4(j+1)},\quad i=(j+1)_{2s},\text{ }j\ge 5s;\\
w_{4(j+m+1)\rightarrow 4(j+m+1)+j_2-3}\otimes e_{4(j+(j_2-4)s)+2},\quad i=(j)_s+4s;\\
w_{4(j+m+1)+1\rightarrow 4(j+m+1)+j_2-3}\otimes w_{4(j+(j_2-4)s)+2\rightarrow 4j+3},\quad i=j+2s;\\
0\quad\text{otherwise.}\end{cases}$$ If $6s\le j<8s$, then
$$(d_3)_{ij}=\begin{cases}
f_1(j, 8s-1)e_{4(j+m+s+1)+2}\otimes w_{4j+2\rightarrow 4(j+1)},\quad i=j-7s+1,\text{ }j\ge 7s;\\
w_{4(j+m+1)\rightarrow 4(j+m+s+1)+1+j_2-6}\otimes e_{4(j+(7-j_2)s)+2},\quad i=(j)_s+5s;\\
w_{4(j+m+s+1)+1\rightarrow 4(j+m+s+1)+1+j_2-6}\otimes w_{4(j+(7-j_2)s)+2\rightarrow 4j+3},\quad i=(j+s)_{2s}+6s;\\
0\quad\text{otherwise.}\end{cases}$$ If $8s\le j<9s$, then
$$(d_3)_{ij}=\begin{cases}
f_1(j, 9s-1)w_{4(j+m+f(j, 9s-1)s+1)+2\rightarrow 4(j+m+1)+3}\otimes w_{4j+3\rightarrow 4(j+1)},\quad i=(j+1)_s;\\
f_1(j, 9s-1)e_{4(j+m+1)+3}\otimes w_{4j+3\rightarrow 4(j+f(j, 9s-1)s+1)+1},\quad i=(j+1)_s+2s;\\
w_{4(j+m+1)+1\rightarrow 4(j+m+1)+3}\otimes e_{4j+3},\quad i=j-2s;\\
-w_{4(j+m+s+1)+1\rightarrow 4(j+m+1)+3}\otimes e_{4j+3},\quad i=j-s;\\
0\quad\text{otherwise.}\end{cases}$$

\newpage\centerline{\bf Description of the $d_4$}
\centerline{$d_4:Q_5\rightarrow Q_4\text{ -- is an } (8s\times 9s)\text{ matrix}.$}
If $0\le j<2s$, then $$(d_4)_{ij}=\begin{cases}
w_{4(j+m-1)+3\rightarrow 4(j+m)+1}\otimes e_{4j},\quad i=j,\text{ }j<s;\\
-f_1(j, s)w_{4(j+m)\rightarrow 4(j+m)+1}\otimes e_{4j},\quad i=(j)_s+s;\\
-e_{4(j+m)+1}\otimes w_{4j\rightarrow 4(j+s)+1},\quad i=(j+s)_{2s}+2s;\\
e_{4(j+m)+1}\otimes w_{4j\rightarrow 4j+2},\quad i=j+(5-f_0(j, s))s;\\
0\quad\text{otherwise.}\end{cases}$$ If $2s\le j<4s$, then
$$(d_4)_{ij}=\begin{cases}
w_{4(j+m)+3\rightarrow 4(j+m+1)}\otimes w_{4j+1\rightarrow 4(j+1)},\quad i=(j+1)_s,\text{ }j<3s-1\text{ or }j=4s-1;\\
-f_1(i_2, s-1)f_1(j, 3s)e_{4(j+m+1)}\otimes w_{4j+1\rightarrow 4(j+1)},\quad i=(j+1)_s+s;\\
w_{4(j+m+s)+1\rightarrow 4(j+m+1)}\otimes e_{4j+1},\quad i=j;\\
-w_{4(j+m+s)+2\rightarrow 4(j+m+1)}\otimes w_{4j+1\rightarrow 4j+2},\quad i=j+(4-f_0(j, 3s))s;\\
0\quad\text{otherwise.}\end{cases}$$ If $4s\le j<6s$, then
$$(d_4)_{ij}=\begin{cases}
e_{4(j+m)+3}\otimes w_{4j+2\rightarrow 4(j+1)},\quad i=(j+1)_s,\text{ }j<5s-1\text{ or }j=6s-1;\\
w_{4(j+m)+1\rightarrow 4(j+m)+3}\otimes e_{4j+2},\quad i=j+(1-f_0(j, 5s))s;\\
-w_{4(j+m+s)+2\rightarrow 4(j+m)+3}\otimes e_{4j+2},\quad i=j+(2-f_0(j, 5s))s;\\
-f_1(j, 5s)e_{4(j+m)+3}\otimes w_{4j+2\rightarrow 4j+3},\quad i=(j)_s+8s;\\
0\quad\text{otherwise.}\end{cases}$$ If $6s\le j<8s$, then
$$(d_4)_{ij}=\begin{cases}
-f_1(j, 7s)w_{4(j+m)+3\rightarrow 4(j+m+1)+2}\otimes w_{4j+3\rightarrow 4(j+1)},\\
\quad\quad\quad i=(j+1)_s,\text{ }j<7s-1\text{ or }j=8s-1;\\
f_1(j, 7s)w_{4(j+m+1)+1\rightarrow 4(j+m+1)+2}\otimes w_{4j+3\rightarrow 4(j+s+1)+1},\quad i=(j+s+1)_{2s}+2s;\\
-f_1(j, 7s)e_{4(j+m+1)+2}\otimes w_{4j+3\rightarrow 4(j+s+1)+2},\\
\quad\quad\quad i=(j+1)_s+7s,\text{ }j<7s-1\text{ or }j=8s-1;\\
-f_1(j, 7s)e_{4(j+m+1)+2}\otimes w_{4j+3\rightarrow 4(j+s+1)+2},\quad i=(j+1)_s+5s,\text{ }7s-1\le j<8s-1;\\
w_{4(j+m)+3\rightarrow 4(j+m+1)+2}\otimes e_{4j+3},\quad i=(j)_s+8s;\\
0\quad\text{otherwise.}\end{cases}$$

\newpage\centerline{\bf Description of the $d_5$}
\centerline{$d_5:Q_6\rightarrow Q_5\text{ -- is an } (9s\times 8s)\text{ matrix}.$}
If $0\le j<s$, then $$(d_5)_{ij}=\begin{cases}
f_1(j_1, 2)w_{4(j+m)+1+j_1+2f(j_1, 1)\rightarrow 4(j+m+1)}\otimes w_{4j\rightarrow 4j+j_1},\quad i=j+2j_1s,\text{ }0\le j_1\le 2;\\
f_1(j_1, 2)w_{4(j+m+s)+1+j_1+2f(j_1, 1)\rightarrow 4(j+m+1)}\otimes w_{4j\rightarrow 4(j+s)+j_1},\\\quad\quad\quad i=j+(2j_1+1)s,\text{ }0\le j_1\le 2;\\
0\quad\text{otherwise.}\end{cases}$$
If $s\le j<3s$, then
$$(d_5)_{ij}=\begin{cases}
-f_1(j, 2s)w_{4(j+m+s+1)+1\rightarrow 4(j+m+s+1)+j_2}\otimes w_{4(j+s(2-j_2))+1\rightarrow 4(j+1)},\quad i=(j+s+1)_{2s};\\
w_{4(j+m+1)\rightarrow 4(j+m+s+1)+j_2}\otimes e_{4(j+s(2-j_2))+1},\quad i=(j)_s+2s;\\
-w_{4(j+m)+3\rightarrow 4(j+m+s+1)+2}\otimes w_{4j+1\rightarrow 4j+2},\quad i=j+2s,\text{ }j\ge 2s;\\
-e_{4(j+m+s+1)+2}\otimes w_{4j+1\rightarrow 4j+3},\quad i=j+5s,\text{ }j\ge 2s;\\
0\quad\text{otherwise.}\end{cases}$$
If $3s\le j<5s$, then
$$(d_5)_{ij}=\begin{cases}
-f_1(j, 4s)w_{4(j+m+1)+1\rightarrow 4(j+m+1)+j_2-2}\otimes w_{4(j+s(j_2-3))+1\rightarrow 4(j+1)},\quad i=(j+1)_{2s};\\
w_{4(j+m+1)\rightarrow 4(j+m+1)+j_2-2}\otimes e_{4(j+s(j_2-3))+1},\quad i=(j)_s+3s;\\
-w_{4(j+m)+3\rightarrow 4(j+m+1)+2}\otimes w_{4(j+s)+1\rightarrow 4(j+s)+2},\quad i=j+s,\text{ }j\ge 4s;\\
e_{4(j+m+1)+2}\otimes w_{4(j+s)+1\rightarrow 4j+3},\quad i=j+2s,\text{ }j\ge 4s;\\
0\quad\text{otherwise.}\end{cases}$$
If $5s\le j<7s$, then
$$(d_5)_{ij}=\begin{cases}
w_{4(j+m)+3\rightarrow 4(j+m+s+1)+2}\otimes e_{4(j+s)+2},\quad i=j-s;\\
f_1(j, 6s)e_{4(j+m+s+1)+2}\otimes w_{4(j+s)+2\rightarrow 4j+3},\quad i=j+s;\\
0\quad\text{otherwise.}\end{cases}$$
If $7s\le j<8s$, then
$$(d_5)_{ij}=\begin{cases}
w_{4(j+m+s+1)+1\rightarrow 4(j+m+1)+3}\otimes w_{4j+3\rightarrow 4(j+1)},\quad i=j-7s+1;\\
w_{4(j+m+1)+1\rightarrow 4(j+m+1)+3}\otimes w_{4j+3\rightarrow 4(j+1)},\quad i=(j+1)_{2s};\\
-e_{4(j+m+1)+3}\otimes w_{4j+3\rightarrow 4(j+s+1)+2},\quad i=j-3s+1;\\
-e_{4(j+m+1)+3}\otimes w_{4j+3\rightarrow 4(j+1)+2},\quad i=(j+1)_{2s}+4s;\\
w_{4(j+m+s+1)+2\rightarrow 4(j+m+1)+3}\otimes e_{4j+3},\quad i=j-s;\\
-w_{4(j+m+1)+2\rightarrow 4(j+m+1)+3}\otimes e_{4j+3},\quad i=j;\\
0\quad\text{otherwise.}\end{cases}$$ If $8s\le j<9s$, then
$$(d_5)_{ij}=\begin{cases}
-e_{4(j+m+2)}\otimes w_{4j+3\rightarrow 4(j+s+1)+1},\quad i=(j+s+1)_{2s}+2s;\\
w_{4(j+m+1)+2\rightarrow 4(j+m+2)}\otimes e_{4j+3},\quad i=j-2s;\\
0\quad\text{otherwise.}\end{cases}$$

\centerline{\bf Description of the $d_6$}
\centerline{$d_6:Q_7\rightarrow Q_6\text{ -- is an } (8s\times 9s)\text{ matrix}.$}
If $0\le j<2s$, then $$(d_6)_{ij}=\begin{cases}
w_{4(j+m)\rightarrow 4(j+m)+2}\otimes e_{4j},\quad i=(j)_s;\\
-w_{4(j+m)+1\rightarrow 4(j+m)+2}\otimes w_{4j\rightarrow 4j+1},\quad i=j+(j_2+1)s;\\
-e_{4(j+m)+2}\otimes w_{4j\rightarrow 4(j+s)+1},\quad i=j+(4-3j_2)s;\\
e_{4(j+m)+2}\otimes w_{4j\rightarrow 4j+2},\quad i=j+5s;\\
0\quad\text{otherwise.}\end{cases}$$ If $2s\le j<4s$, then
$$(d_6)_{ij}=\begin{cases}
w_{4(j+m)+1\rightarrow 4(j+m)+3}\otimes e_{4j+1},\quad i=j+s(j_2-3);\\
-w_{4(j+m+s)+2\rightarrow 4(j+m)+3}\otimes e_{4j+1},\quad i=j+s(j_2-2);\\
-w_{4(j+m)+2\rightarrow 4(j+m)+3}\otimes w_{4j+1\rightarrow 4j+2},\quad i=j+3s;\\
e_{4(j+m)+3}\otimes w_{4j+1\rightarrow 4j+3},\quad i=(j)_s+7s;\\
0\quad\text{otherwise.}\end{cases}$$ If $4s\le j<6s$, then
$$(d_6)_{ij}=\begin{cases}
e_{4(j+m+1)}\otimes w_{4j+2\rightarrow 4(j+1)},\quad i=(j+1)_s,\text{ }j\ge 5s;\\
w_{4(j+m)+2\rightarrow 4(j+m+1)}\otimes e_{4j+2},\quad i=j+s;\\
-w_{4(j+m)+3\rightarrow 4(j+m+1)}\otimes w_{4j+2\rightarrow 4j+3},\quad i=j+2s,\text{ }j\ge 5s;\\
-f_1(j, 5s)e_{4(j+m+1)}\otimes w_{4j+2\rightarrow 4j+3},\quad i=(j)_s+8s;\\
0\quad\text{otherwise.}\end{cases}$$ If $6s\le j<8s$, then
$$(d_6)_{ij}=\begin{cases}
-w_{4(j+m+1)\rightarrow 4(j+m+1)+1}\otimes w_{4j+3\rightarrow 4(j+1)},\quad i=(j+1)_s,\text{ }j<7s;\\
e_{4(j+m+1)+1}\otimes w_{4j+3\rightarrow 4(j+1)+1},\quad i=(j+1)_{2s}+s,\text{ }j<7s-1\text{ or }j=8s-1;\\
e_{4(j+m+1)+1}\otimes w_{4j+3\rightarrow 4(j+1)+1},\quad i=(j+1)_{2s}+2s,\text{ }7s-1\le j<8s-1;\\
w_{4(j+m)+3\rightarrow 4(j+m+1)+1}\otimes e_{4j+3},\quad i=j+s,\text{ }j<7s;\\
-f_1(j, 7s)w_{4(j+m+1)\rightarrow 4(j+m+1)+1}\otimes e_{4j+3},\quad i=(j)_s+8s;\\
0\quad\text{otherwise.}\end{cases}$$

\newpage\centerline{\bf Description of the $d_7$}
\centerline{$d_7:Q_8\rightarrow Q_7\text{ -- is an } (6s\times 8s)\text{ matrix}.$}
If $0\le j<s$, then $$(d_7)_{ij}=\begin{cases}
w_{4(j+m)+2\rightarrow 4(j+m)+3}\otimes e_{4j},\quad i=j;\\
-w_{4(j+m+s)+2\rightarrow 4(j+m)+3}\otimes e_{4j},\quad i=j+s;\\
e_{4(j+m)+3}\otimes w_{4j\rightarrow 4j+1},\quad i=j+2s;\\
-e_{4(j+m)+3}\otimes w_{4j\rightarrow 4(j+s)+1},\quad i=j+3s;\\
0\quad\text{otherwise.}\end{cases}$$ If $s\le j<3s$, then
$$(d_7)_{ij}=\begin{cases}
-e_{4(j+m+s+1)+2}\otimes w_{4(j+s)+1\rightarrow 4(j+1)},\quad i=(j+s+1)_{2s};\\
w_{4(j+m)+3\rightarrow 4(j+m+s+1)+2}\otimes e_{4(j+s)+1},\quad i=j+s;\\
w_{4(j+m+1)\rightarrow 4(j+m+s+1)+2}\otimes w_{4(j+s)+1\rightarrow 4(j+s)+2},\quad i=j+3s;\\
-w_{4(j+m+s+1)+1\rightarrow 4(j+m+s+1)+2}\otimes w_{4(j+s)+1\rightarrow 4j+3},\quad i=j+5s;\\
0\quad\text{otherwise.}\end{cases}$$ If $3s\le j<5s$, then
$$(d_7)_{ij}=\begin{cases}
w_{4(j+m+1)\rightarrow 4(j+m+1)+1}\otimes e_{4(j+s)+2},\quad i=j+s;\\
e_{4(j+m+1)+1}\otimes w_{4(j+s)+2\rightarrow 4j+3},\quad i=(j)_{2s}+6s;\\
0\quad\text{otherwise.}\end{cases}$$ If $5s\le j<6s$, then
$$(d_7)_{ij}=\begin{cases}
w_{4(j+m+1+sf(j, 6s-1))+2\rightarrow 4(j+m+2)}\otimes w_{4j+3\rightarrow 4(j+1)},\quad i=(j+1)_s+s;\\
-w_{4(j+m+1)+3\rightarrow 4(j+m+2)}\otimes w_{4j+3\rightarrow 4(j+1+(1+f(j, 6s-1))s)+1},\quad i=(j+1)_s+2s;\\
-e_{4(j+m+2)}\otimes w_{4j+3\rightarrow 4(j+1+(1+f(j, 6s-1))s)+2},\quad i=(j+1)_s+4s;\\
-e_{4(j+m+2)}\otimes w_{4j+3\rightarrow 4(j+1+sf(j, 6s-1))+2},\quad i=(j+1)_s+5s;\\
w_{4(j+m+s+1)+1\rightarrow 4(j+m+2)}\otimes e_{4j+3},\quad i=j+s;\\
w_{4(j+m+1)+1\rightarrow 4(j+m+2)}\otimes e_{4j+3},\quad i=j+2s;\\
0\quad\text{otherwise.}\end{cases}$$

\newpage\centerline{\bf Description of the $d_8$}
\centerline{$d_8:Q_9\rightarrow Q_8\text{ -- is an } (7s\times 6s)\text{ matrix}.$}
If $0\le j<s$, then $$(d_8)_{ij}=\begin{cases}
w_{4(j+m-1)+3\rightarrow 4(j+m)+3}\otimes e_{4j},\quad i=j;\\
-f_1(j, s-1)e_{4(j+m)+3}\otimes w_{4j\rightarrow 4(j+1)},\quad i=(j+1)_s;\\
-w_{4(j+m)+2\rightarrow 4(j+m)+3}\otimes w_{4j\rightarrow 4j+1},\quad i=j+s;\\
w_{4(j+m+s)+2\rightarrow 4(j+m)+3}\otimes w_{4j\rightarrow 4(j+s)+1},\quad i=j+2s;\\
w_{4(j+m+s)+1\rightarrow 4(j+m)+3}\otimes w_{4j\rightarrow 4j+2},\quad i=j+3s;\\
-w_{4(j+m)+1\rightarrow 4(j+m)+3}\otimes w_{4j\rightarrow 4(j+s)+2},\quad i=j+4s;\\
0\quad\text{otherwise.}\end{cases}$$ If $s\le j<3s$, then
$$(d_8)_{ij}=\begin{cases}
w_{4(j+m)+3\rightarrow 4(j+m+1)}\otimes w_{4(j+s)+1\rightarrow 4(j+1)},\quad i=(j+1)_s,\text{ }j<2s-1\text{ or }j=3s-1;\\
w_{4(j+m+s)+2\rightarrow 4(j+m+1)}\otimes e_{4(j+s)+1},\quad i=j;\\
-w_{4(j+m)+1\rightarrow 4(j+m+1)}\otimes w_{4(j+s)+1\rightarrow 4(j+s)+2},\quad i=j+2s;\\
e_{4(j+m+1)}\otimes w_{4(j+s)+1\rightarrow 4j+3},\quad i=(j)_s+5s;\\
0\quad\text{otherwise.}\end{cases}$$ If $3s\le j<5s$, then
$$(d_8)_{ij}=\begin{cases}
-w_{4(j+m)+3\rightarrow 4(j+m+1)+1}\otimes w_{4(j+s)+2\rightarrow 4(j+1)},\quad i=(j+1)_s,\text{ }j<4s-1\text{ or }j=5s-1;\\
w_{4(j+m)+1\rightarrow 4(j+m+1)+1}\otimes e_{4(j+s)+2},\quad i=j;\\
-e_{4(j+m+1)+1}\otimes w_{4(j+s)+2\rightarrow 4(j+s+1)+2},\quad i=(j+s+1)_{2s}+3s;\\
-w_{4(j+m+1)\rightarrow 4(j+m+1)+1}\otimes w_{4(j+s)+2\rightarrow 4j+3},\quad i=(j)_s+5s;\\
0\quad\text{otherwise.}\end{cases}$$ If $5s\le j<7s$, then
$$(d_8)_{ij}=\begin{cases}
w_{4(j+m)+3\rightarrow 4(j+m+s+1)+2}\otimes w_{4j+3\rightarrow 4(j+1)},\quad i=(j+1)_s,\text{ }6s-1\le j<7s-1;\\
e_{4(j+m+s+1)+2}\otimes w_{4j+3\rightarrow 4(j+s+1)+1},\quad i=(j+s+1)_{2s}+s;\\
w_{4(j+m+s+1)+1\rightarrow 4(j+m+s+1)+2}\otimes w_{4j+3\rightarrow 4(j+1)+2},\quad i=(j+1)_{2s}+3s;\\
w_{4(j+m+1)\rightarrow 4(j+m+s+1)+2}\otimes e_{4j+3},\quad i=(j)_s+5s;\\
0\quad\text{otherwise.}\end{cases}$$

\newpage\centerline{\bf Description of the $d_9$}
\centerline{$d_9:Q_{10}\rightarrow Q_9\text{ -- is an } (6s\times 7s)\text{ matrix}.$}
If $0\le j<s$, then $$(d_9)_{ij}=\begin{cases}
w_{4(j+m)+3\rightarrow 4(j+m+1)}\otimes e_{4j},\quad i=j;\\
e_{4(j+m+1)}\otimes w_{4j\rightarrow 4j+1},\quad i=j+s;\\
-e_{4(j+m+1)}\otimes w_{4j\rightarrow 4(j+s)+1},\quad i=j+2s;\\
0\quad\text{otherwise.}\end{cases}$$ If $s\le j<3s$, then
$$(d_9)_{ij}=\begin{cases}
w_{4(j+m+1)\rightarrow 4(j+m+1)+1}\otimes e_{4(j+s)+1},\quad i=j;\\
e_{4(j+m+1)+1}\otimes w_{4(j+s)+1\rightarrow 4(j+s)+2},\quad i=j+2s;\\
0\quad\text{otherwise.}\end{cases}$$ If $3s\le j<5s$, then
$$(d_9)_{ij}=\begin{cases}
w_{4(j+m+1)+1\rightarrow 4(j+m+1)+2}\otimes e_{4(j+s)+2},\quad i=j;\\
e_{4(j+m+1)+2}\otimes w_{4(j+s)+2\rightarrow 4j+3},\quad i=(j)_{2s}+5s;\\
0\quad\text{otherwise.}\end{cases}$$ If $5s\le j<6s$, then
$$(d_9)_{ij}=\begin{cases}
f_1(j, 6s-1)e_{4(j+m+1)+3}\otimes w_{4j+3\rightarrow 4(j+1)},\quad i=(j+1)_s;\\
w_{4(j+m+s+1)+2\rightarrow 4(j+m+1)+3}\otimes e_{4j+3},\quad i=j;\\
-w_{4(j+m+1)+2\rightarrow 4(j+m+1)+3}\otimes e_{4j+3},\quad i=j+s;\\
0\quad\text{otherwise.}\end{cases}$$

\centerline{\bf Description of the $d_{10}$}
\centerline{$d_{10}:Q_{11}\rightarrow Q_{10}\text{ -- is an } (6s\times 6s)\text{ matrix}.$}
If $0\le j<s$, then $$(d_{10})_{ij}=\begin{cases}
-f_1(j, s-1)e_{4(j+m+1)}\otimes w_{4j\rightarrow 4(j+1)},\quad i=(j+1)_s;\\
-f_1(j_1,  1)w_{4(j+m+s)+1+j_1\rightarrow 4(j+m+1)}\otimes w_{4j\rightarrow 4j+1+j_1},\quad i=j+(2j_1+1)s,\text{ }0\le j_1\le 2;\\
f_1(j_1,  2)w_{4(j+m)+j_1\rightarrow 4(j+m+1)}\otimes w_{4j\rightarrow 4(j+s)+j_1},\quad i=j+2j_1s,\text{ }0\le j_1\le 2;\\
0\quad\text{otherwise.}\end{cases}$$ If $s\le j<3s$, then
$$(d_{10})_{ij}=\begin{cases}
f_1(i_2,  s-1)f_1(j, 2s)w_{4(j+m+1)\rightarrow 4(j+m+1)+1}\otimes w_{4(j+s)+1\rightarrow 4(j+1)},\quad i=(j+1)_s;\\
w_{4(j+m)+1\rightarrow 4(j+m+1)+1}\otimes e_{4(j+s)+1},\quad i=j;\\
-e_{4(j+m+1)+1}\otimes w_{4(j+s)+1\rightarrow 4(j+s+1)+1},\quad i=(j+s+1)_{2s}+s;\\
-w_{4(j+m)+2\rightarrow 4(j+m+1)+1}\otimes w_{4(j+s)+1\rightarrow 4(j+s)+2},\quad i=j+2s;\\
-f_1(j, 2s)w_{4(j+m)+3\rightarrow 4(j+m+1)+1}\otimes w_{4(j+s)+1\rightarrow 4j+3},\quad i=(j)_s+5s;\\
0\quad\text{otherwise.}\end{cases}$$ If $3s\le j<5s$, then
$$(d_{10})_{ij}=\begin{cases}
-f_1(i_2, s-1)f_1(j, 4s)w_{4(j+m+1)\rightarrow 4(j+m+1)+2}\otimes w_{4(j+s)+2\rightarrow 4(j+1)},\quad i=(j+1)_s;\\
w_{4(j+m+1)+1\rightarrow 4(j+m+1)+2}\otimes w_{4(j+s)+2\rightarrow 4(j+s+1)+1},\quad i=(j+s+1)_{2s}+s;\\
w_{4(j+m)+2\rightarrow 4(j+m+1)+2}\otimes e_{4(j+s)+2},\quad i=j;\\
-e_{4(j+m+1)+2}\otimes w_{4(j+s)+2\rightarrow 4(j+s+1)+2},\quad i=(j+s+1)_{2s}+3s;\\
f_1(j, 4s)w_{4(j+m)+3\rightarrow 4(j+m+1)+2}\otimes w_{4(j+s)+2\rightarrow 4j+3},\quad i=(j)_s+5s;\\
0\quad\text{otherwise.}\end{cases}$$ If $5s\le j<6s$, then
$$(d_{10})_{ij}=\begin{cases}
-f_1(j, 6s-1)w_{4(j+m+1)\rightarrow 4(j+m+1)+3}\otimes w_{4j+3\rightarrow 4(j+1)},\quad i=(j+1)_s;\\
w_{4(j+m+1)+1\rightarrow 4(j+m+1)+3}\otimes w_{4j+3\rightarrow 4(j+s+1)+1},\quad i=j-4s+1;\\
-w_{4(j+m+s+1)+1\rightarrow 4(j+m+1)+3}\otimes w_{4j+3\rightarrow 4(j+1)+1},\quad i=(j+1)_{2s}+s;\\
-w_{4(j+m+1)+2\rightarrow 4(j+m+1)+3}\otimes w_{4j+3\rightarrow 4(j+s+1)+2},\quad i=j-2s+1;\\
w_{4(j+m+s+1)+2\rightarrow 4(j+m+1)+3}\otimes w_{4j+3\rightarrow 4(j+1)+2},\quad i=(j+1)_{2s}+3s;\\
w_{4(j+m)+3\rightarrow 4(j+m+1)+3}\otimes e_{4j+3},\quad i=j;\\
-f_1(j, 6s-1)e_{4(j+m+1)+3}\otimes w_{4j+3\rightarrow 4(j+1)+3},\quad i=(j+1)_s+5s;\\
0\quad\text{otherwise.}\end{cases}$$

\begin{thm}\label{resol_thm}
Let $R=R_s^\prime$ is algebra of the type $E_6^\prime$. Then the minimal projective resolution of the
$\Lambda$-module $R$ is of the form:
\begin{equation}\label{resolv}\tag{$+$} \dots\longrightarrow
Q_3\stackrel{d_2}\longrightarrow Q_2\stackrel{d_1}\longrightarrow
Q_1\stackrel{d_0}\longrightarrow
Q_0\stackrel\varepsilon\longrightarrow R\longrightarrow
0,
\end{equation}
where $\varepsilon$ is the multiplication map $(\varepsilon(a\otimes b)=ab)$; $Q_r\text{ }(r\le
10)$ and $d_r\text{ }(r\le 10)$ were described before; further $Q_{11\ell+r}$, where $\ell\in \N$
and $0\le r\le 10$, is obtained from $Q_r$ by replacing every direct summand $P_{i,j}$ to
$P_{\sigma^\ell(i),j}$ correspondingly $($here $\sigma(i)=j$, if $\sigma(e_i)=e_j)$, and the
differential $d_{11\ell+r}$ is obtained from $d_r$ by act of $\sigma^\ell$ by all left tensor
components of the corresponding matrix.
\end{thm}

To prove that the terms $Q_i$ are of this form we introduce $P_i=Re_i$ is the projective cover of
the simple $R$-modules $S_i$, corresponding to the vertices of the quiver $\mathcal Q_s$. We will
find projective resolutions of the simple $R$-modules $S_i$.

\begin{obozn}
For $R$-module $M$ its $m$th syzygy is denoted by $\Omega^m(M)$.
\end{obozn}

\begin{zam}\label{note_brev}
From here we denote the multiplication homomorphism from the right by an element $w$ by $w$.
\end{zam}

\begin{lem}\label{lem_s0}
The begin of the minimal projective resolution of $S_{4r}$ is of the form

\begin{multline*}
\dots\longrightarrow P_{4(r+3)+3} \stackrel{\binom{\a}{-\a}}\longrightarrow P_{4(r+3)+2}\oplus
P_{4(r+s+3)+2}\stackrel{(\a^2\text{ }\a^2)}\longrightarrow\\
\longrightarrow P_{4(r+3)} \stackrel{\binom{\g\a^2}{\g\a^2}}\longrightarrow P_{4(r+2)+1}\oplus
P_{4(r+s+2)+1} \stackrel{\binom{\a\g\phantom{-}0}{-\a\phantom{-}\a}}
\longrightarrow\\\longrightarrow P_{4(r+1)+3}\oplus P_{4(r+2)}
\stackrel{\binom{\phantom{-}\a\phantom{-}\g\a}{-\a\phantom{-}0}}\longrightarrow P_{4(r+1)+2}\oplus
P_{4(r+s+1)+2} \stackrel{(\a^2\g\text{ }\a^2\g)}\longrightarrow\\\longrightarrow
P_{4r+3}\stackrel{\binom{\a^2}{-\a^2}}\longrightarrow P_{4r+1}\oplus P_{4(r+s)+1}
\stackrel{(\a\text{ }\a)}\longrightarrow P_{4r}\longrightarrow S_{4r}\longrightarrow 0.
\end{multline*}
At that $\Omega^{9}(S_{4r})\simeq S_{4(r+4)+3}$.
\end{lem}

\begin{lem}
The begin of the minimal projective resolution of $S_{4r+1}$ is of the form $$\dots\longrightarrow
P_{4r+2} \stackrel{\a}\longrightarrow P_{4r+1}\longrightarrow S_{4r+1}\longrightarrow 0.$$ At that
$\Omega^{2}(S_{4r+1})\simeq S_{4(r+1)+2}$.
\end{lem}

\begin{lem}
The begin of the minimal projective resolution of $S_{4r+2}$ is of the form
\begin{multline*}
\dots\longrightarrow P_{4(r+s+4)+1} \stackrel{\a}\longrightarrow
P_{4(r+4)}\stackrel{\g\a}\longrightarrow P_{4(r+3)+2}\stackrel{\a^2\g}\longrightarrow
P_{4(r+2)+3}\stackrel{\binom{\a^2}{-\a}}\longrightarrow\\\longrightarrow P_{4(r+2)+1}\oplus
P_{4(r+s+2)+2}\stackrel{(\a\text{ }\a^2)}\longrightarrow
P_{4(r+2)}\stackrel{\g\a^2}\longrightarrow\\\longrightarrow
P_{4(r+s+1)+1}\stackrel{\a\g}\longrightarrow P_{4r+3}\stackrel{\a}\longrightarrow
P_{4r+2}\longrightarrow S_{4r+2}\longrightarrow 0.
\end{multline*}
At that $\Omega^{9}(S_{4r+2})\simeq S_{4(r+s+5)+1}$.
\end{lem}

\begin{lem}\label{lem_s3}
The begin of the minimal projective resolution of $S_{4r+3}$ is of the form $$\dots\longrightarrow
P_{4(r+1)} \stackrel{\g}\longrightarrow P_{4r+3}\longrightarrow S_{4r+3}\longrightarrow 0.$$ At that
$\Omega^{2}(S_{4r+3})\simeq S_{4(r+2)}$.
\end{lem}

\begin{proof}
Proofs of the lemmas consist of direct check that given sequences are exact, and it is immediate.
\end{proof}
We shall need the Happel's lemma (see \cite{Ha}), as revised in \cite{Gen&Ka}:

\begin{lem}[Happel]\label{lem_Ha}
Let
$$\dots\rightarrow Q_m\rightarrow Q_{m-1}\rightarrow\dots\rightarrow
Q_1\rightarrow Q_0\rightarrow R\rightarrow 0$$ be the minimal projective resolution of $R$. Then
$$Q_m\cong\bigoplus_{i,j}P_{i,j}^{\dim\mathrm{Ext}^m_R(S_j,S_i)}.$$
\end{lem}

\begin{proof}[Proof of the theorem \ref{resol_thm}]
Descriptions for $Q_i$  immediately follows from lemmas \ref{lem_s0} -- \ref{lem_s3} and Happel's
lemma.

As proved in \cite{VGI}, to prove that sequence \eqref{resolv} is exact in $Q_m$ ($m\le 11$)
it will be sufficient to show that $d_md_{m+1}=0$. It is easy to verify this relation by
a straightforward calculation of matrixes products.

Since the sequence is exact in $Q_{11}$, it follows that
$\Omega^{11}({}_\Lambda R)\simeq {}_1R_{\sigma}$, where
$\Omega^{11}({}_\Lambda R)=\Im d_{10}$ is the 11th syzygy of the
module $R$, and ${}_1R_{\sigma}$ is a twisted bimodule. Hence, an
exactness in $Q_t$ ($t>11$) holds.

\end{proof}

We recall that for $R$-bimodule $M$ the {\it twisted bimodule} is a linear
space $M$, on which left act right acts of the algebra $R$ (denoted by
asterisk) are assigned by the following way:
$$r*m*s = \lambda(r)\cdot m\cdot\mu(s) \text{ for } r,s\in R
\text{ and } m\in M,$$ where $\lambda,\mu$ are some automorphisms of algebra $R$.
Such twisted bimodule we shall denote by ${}_\lambda M_\mu$.

\begin{s}
We have isomorphism $\Omega^{11}({}_\Lambda R)\simeq {}_1R_{\sigma}$.
\end{s}

\begin{pr}
Automorphism $\sigma$ has a finite order, and

$(1)$ if $\myChar=2$, then order of $\sigma$ is equal to $\frac{2s}{\myNod(n+s,2s)}$;

$(2)$ if $\myChar\ne 2$, then order of $\sigma$ is equal to $\frac{2s}{\myNod(n+s,2s)}$, if
$\frac{2s}{\myNod(n+s,2s)}$ is divisible by $4$, and to $\frac{4s}{\myNod(n+s,2s)}$ otherwise.
\end{pr}

\begin{pr}
The minimal period of bimodule resolution of $R$ is $11\deg\sigma$.
\end{pr}

\section{The additive structure of $\HH^*(R)$}

\begin{pr}[Dimensions of homomorphism groups, $s>1$]\label{dim_hom}
Let $s>1$ and $R=R_s^\prime$ is algebra of the type $E_6$. Next, $t\in\N\cup\{0\}$,
$\ell$ be the aliquot, and $r$ be the residue of division of $t$ by $11$.

$(1)$ If $r=0$, then $$\dim_K\Hom_\Lambda(Q_{deg}, R)=\begin{cases}
 6s,\quad \ell (n+s)+m\equiv 0(2s)\text{ or }\ell (n+s)+m\equiv 1(2s);\\
 2s,\quad \ell (n+s)+m\equiv s(2s)\text{ or }\ell (n+s)+m\equiv s+1(2s);\\
 0,\quad\text{otherwise.}
 \end{cases}$$

$(2)$ If $r=1$, then $$\dim_K\Hom_\Lambda(Q_{deg}, R)=\begin{cases}
 7s,\quad \ell (n+s)+m\equiv 0(2s);\\
 5s,\quad \ell (n+s)+m\equiv s(2s);\\
 0,\quad\text{otherwise.}
 \end{cases}$$

$(3)$ If $r\in\{2,8\}$, then $$\dim_K\Hom_\Lambda(Q_{deg}, R)=\begin{cases}
 3s,\quad \ell (n+s)+m\equiv 0(2s)\text{ or }\ell (n+s)+m\equiv s+1(2s);\\
 s,\quad \ell (n+s)+m\equiv s(2s)\text{ or }\ell (n+s)+m\equiv 1(2s);\\
 0,\quad\text{otherwise.}
 \end{cases}$$

$(4)$ If $r\in\{3,5,7\}$, then $$\dim_K\Hom_\Lambda(Q_{deg}, R)=\begin{cases}
 8s,\quad \ell n+m\equiv 0(s);\\
 0,\quad\text{otherwise.}
 \end{cases}$$

$(5)$ If $r=4$, then $$\dim_K\Hom_\Lambda(Q_{deg}, R)=\begin{cases}
 2s,\quad \ell (n+s)+m\equiv 0(2s);\\
 6s,\quad \ell (n+s)+m\equiv s(2s);\\
 5s,\quad \ell (n+s)+m\equiv 1(2s);\\
 7s,\quad \ell (n+s)+m\equiv s+1(2s);\\
 0,\quad\text{otherwise.}
 \end{cases}$$

$(6)$ If $r=6$, then $$\dim_K\Hom_\Lambda(Q_{deg}, R)=\begin{cases}
 7s,\quad \ell (n+s)+m\equiv 0(2s);\\
 5s,\quad \ell (n+s)+m\equiv s(2s);\\
 6s,\quad \ell (n+s)+m\equiv 1(2s);\\
 2s,\quad \ell (n+s)+m\equiv s+1(2s);\\
 0,\quad\text{otherwise.}
 \end{cases}$$

$(7)$ If $r=9$, then $$\dim_K\Hom_\Lambda(Q_{deg}, R)=\begin{cases}
 5s,\quad \ell (n+s)+m\equiv 0(2s);\\
 7s,\quad \ell (n+s)+m\equiv s(2s);\\
 0,\quad\text{otherwise.}
 \end{cases}$$

$(8)$ If $r=10$, then $$\dim_K\Hom_\Lambda(Q_{deg}, R)=\begin{cases}
 2s,\quad \ell (n+s)+m\equiv 0(2s)\text{ or }\ell (n+s)+m\equiv 1(2s);\\
 6s,\quad \ell (n+s)+m\equiv s(2s)\text{ or }\ell (n+s)+m\equiv s+1(2s);\\
 0,\quad\text{otherwise.}
 \end{cases}$$

\end{pr}

\begin{proof}
The dimension $\dim_K\Hom_\Lambda(P_{i,j}, R)$ is equal to the number of linear independent nonzero
paths of the quiver $\mathcal{Q}_s$, leading from $j$th vertex to $i$th, and the proof is to
consider cases $r=0$, $r=1$ etc.
\end{proof}

\begin{pr}[Dimensions of homomorphism groups, $s=1$]

Let $R=R_1^\prime$ is algebra of the type $E_6$. Next, $t\in\N\cup\{0\}$,
$\ell$ be the aliquot, and $r$ be the residue of division of $t$ by $11$.

$(1)$ If $r\in\{0,3,5,7,10\}$, then $\dim_K\Hom_\Lambda(Q_{deg}, R)=8$.

$(2)$ If $r=1$, then $$\dim_K\Hom_\Lambda(Q_{deg}, R)=\begin{cases}
 7,\quad\ell\div 2;\\5\quad\ell\ndiv 2.\end{cases}$$

$(3)$ If $r=2$, then $$\dim_K\Hom_\Lambda(Q_{deg}, R)=\begin{cases}
 2,\quad\ell\div 2;\\6\quad\ell\ndiv 2.\end{cases}$$

$(4)$ If $r=4$, then $$\dim_K\Hom_\Lambda(Q_{deg}, R)=\begin{cases}
 9,\quad\ell\div 2;\\11\quad\ell\ndiv 2.\end{cases}$$

$(5)$ If $r=6$, then $$\dim_K\Hom_\Lambda(Q_{deg}, R)=\begin{cases}
 11,\quad\ell\div 2;\\9\quad\ell\ndiv 2.\end{cases}$$

$(6)$ If $r=8$, then $$\dim_K\Hom_\Lambda(Q_{deg}, R)=\begin{cases}
 6,\quad\ell\div 2;\\2\quad\ell\ndiv 2.\end{cases}$$

$(7)$ If $r=9$, then $$\dim_K\Hom_\Lambda(Q_{deg}, R)=\begin{cases}
 5,\quad\ell\div 2;\\7\quad\ell\ndiv 2.\end{cases}$$

\end{pr}

\begin{proof}
The proof is basically the same as proof of proposition \ref{dim_hom}.
\end{proof}

\begin{pr}[Dimensions of coboundaries groups]\label{dim_im}
Let $R=R_s^\prime$ is algebra of the type $E_6$, and let
\begin{equation}\tag{$\times$}\label{ind_resolv} 0\longrightarrow
\Hom_\Lambda(Q_0, R)\stackrel{\delta^0}\longrightarrow \Hom_\Lambda(Q_1,
R)\stackrel{\delta^1}\longrightarrow \Hom_\Lambda(Q_2, R)\stackrel{\delta^2}\longrightarrow\dots
\end{equation} be a complex, obtained from minimal projective
resolution \eqref{resolv} of algebra $R$, by applying functor
$\Hom_\Lambda(-,R)$.

Consider coboundaries groups $\Im\delta^s$ of the complex
\eqref{ind_resolv}. Let $\ell$ be the aliquot, and $r$ be the
residue of division of $t$ by $11$, $m$ be the aliquot of
division of $r$ by $2$. Then$:$

$(1)$ If $r=0$, то $$\dim_K\Im\delta^s=\begin{cases}
 6s-1,\quad\ell(n+s)+m\equiv 0(2s),\text{ }\ell\div 2\text{ or }\myChar=2;\\
 6s,\quad\ell(n+s)+m\equiv 0(2s),\text{ }\ell\ndiv 2,\text{ }\myChar\ne 2;\\
 2s,\quad\ell(n+s)+m\equiv s(2s);\\
 0,\quad\text{otherwise.}
 \end{cases}$$

$(2)$ If $r=1$, то $$\dim_K\Im\delta^s=\begin{cases}
 s,\quad\ell(n+s)+m\equiv 0(2s);\\
 3s-1,\quad\ell(n+s)+m\equiv s(2s),\text{ }\ell\ndiv 2\text{ or }\myChar=2;\\
 3s,\quad\ell(n+s)+m\equiv s(2s),\text{ }\ell\div 2,\text{ }\myChar\ne 2;\\
 0,\quad\text{otherwise.}
 \end{cases}$$

$(3)$ If $r=2$, то $$\dim_K\Im\delta^s=\begin{cases}
 3s,\quad\ell(n+s)+m\equiv 0(2s);\\
 s,\quad\ell(n+s)+m\equiv s(2s);\\
 0,\quad\text{otherwise.}
 \end{cases}$$

$(4)$ If $r=3$, то $$\dim_K\Im\delta^s=\begin{cases}
 5s-1,\quad\ell(n+s)+m\equiv 0(2s);\\
 7s-1,\quad\ell(n+s)+m\equiv s(2s),\text{ }\myChar=2;\\
 7s,\quad\ell(n+s)+m\equiv s(2s),\text{ }\myChar\ne 2;\\
 0,\quad\text{otherwise.}
 \end{cases}$$

$(5)$ If $r=4$, то $$\dim_K\Im\delta^s=\begin{cases}
 2s,\quad\ell(n+s)+m\equiv 0(2s);\\
 6s-1,\quad\ell(n+s)+m\equiv s(2s),\text{ }\ell\ndiv 2,\text{ }\myChar=3;\\
 6s,\quad\ell(n+s)+m\equiv s(2s),\text{ }\ell\div 2\text{ or }\myChar\ne 3;\\
 0,\quad\text{otherwise.}
 \end{cases}$$

$(6)$ If $r=5$, то $$\dim_K\Im\delta^s=\begin{cases}
 6s-1,\quad\ell(n+s)+m\equiv 0(2s),\text{ }\ell\div 2,\text{ }\myChar=3;\\
 6s,\quad\ell(n+s)+m\equiv 0(2s),\text{ }\ell\ndiv 2\text{ or }\myChar\ne 3;\\
 2s,\quad\ell(n+s)+m\equiv s(2s);\\
  0,\quad\text{otherwise.}
 \end{cases}$$

$(7)$ If $r=6$, то $$\dim_K\Im\delta^s=\begin{cases}
 7s-1,\quad\ell(n+s)+m\equiv 0(2s),\text{ }\myChar=2;\\
 7s,\quad\ell(n+s)+m\equiv 0(2s),\text{ }\myChar\ne 2;\\
 5s-1,\quad\ell(n+s)+m\equiv s(2s);\\
 0,\quad\text{otherwise.}
 \end{cases}$$

$(8)$ If $r=7$, то $$\dim_K\Im\delta^s=\begin{cases}
 s,\quad\ell(n+s)+m\equiv 0(2s);\\
 3s,\quad\ell(n+s)+m\equiv s(2s);\\
 0,\quad\text{otherwise.}
 \end{cases}$$

$(9)$ If $r=8$, то $$\dim_K\Im\delta^s=\begin{cases}
 3s-1,\quad\ell(n+s)+m\equiv 0(2s),\text{ }\ell\div 2\text{ or }\myChar=2;\\
 3s,\quad\ell(n+s)+m\equiv 0(2s),\text{ }\ell\ndiv 2,\text{ }\myChar\ne 2;\\
 s,\quad\ell(n+s)+m\equiv s(2s);\\
 0,\quad\text{otherwise.}
 \end{cases}$$

$(10)$ If $r=9$, то $$\dim_K\Im\delta^s=\begin{cases}
 2s,\quad\ell(n+s)+m\equiv 0(2s);\\
 6s-1,\quad\ell(n+s)+m\equiv s(2s),\text{ }\ell\ndiv 2\text{ or }\myChar=2;\\
 6s,\quad\ell(n+s)+m\equiv s(2s),\text{ }\ell\div 2,\text{ }\myChar\ne 2;\\
 0,\quad\text{otherwise.}
 \end{cases}$$

$(11)$ If $r=10$, то $$\dim_K\Im\delta^s=\begin{cases}
 2s-1,\quad\ell(n+s)+m\equiv 0(2s),\text{ }\ell\ndiv 2,\text{ }\myChar=3;\\
 2s,\quad\ell(n+s)+m\equiv 0(2s),\text{ }\ell\div 2\text{ or }\myChar\ne 3;\\
 6s,\quad\ell(n+s)+m\equiv s(2s);\\
 0,\quad\text{otherwise.}
 \end{cases}$$

\end{pr}

\begin{proof}
The proof is technical and consists in constructing the image matrixes from the description of
differential matrixes and the subsequent computations of the ranks of image matrixes.
\end{proof}

\begin{thm}[Additive structure, $s>1$]
Let $s>1$ and $R=R_s^\prime$ is algebra of the type $E_6$. Next, $t\in\N\cup\{0\}$,
$\ell$ be the aliquot, and $r$ be the residue of division of $t$ by $11$,
$m$ be the aliquot of division of $r$ by $2$. Then
$\dim_K\HH^t(R)=1$, if one of the following conditions takes place$:$

$(1)$ $r\in \{0,1,8,9\}$, $\ell(n+s)+m\equiv 0(2s),\text{ }\ell\div 2\text{ or }\myChar=2$;

$(2)$ $r=0$, $\ell(n+s)+m\equiv s+1(2s),\text{ }\ell\div 2,\text{ }\myChar=3$;

$(3)$ $r\in \{1,9\}$, $\ell(n+s)+m\equiv s(2s),\text{ }\ell\ndiv 2\text{ or }\myChar=2$;

$(4)$ $r\in \{2,10\}$, $\ell(n+s)+m\equiv s+1(2s),\text{ }\ell\ndiv 2\text{ or }\myChar=2$;

$(5)$ $r=3$, $\ell(n+s)+m\equiv 0(2s)$;

$(6)$ $r=3$, $\ell(n+s)+m\equiv s(2s),\text{ }\myChar=2$;

$(7)$ $r\in \{4,5\}$, $\ell(n+s)+m\equiv s(2s),\text{ }\ell\ndiv 2,\text{ }\myChar=3$;

$(8)$ $r=4$, $\ell(n+s)+m\equiv 1(2s)$;

$(9)$ $r=4$, $\ell(n+s)+m\equiv s+1(2s),\text{ }\myChar=2$;

$(10)$ $r=5$, $\ell(n+s)+m\equiv 0(2s),\text{ }\ell\div 2,\text{ }\myChar=3$;

$(11)$ $r\in \{6,7\}$, $\ell(n+s)+m\equiv 0(2s),\text{ }\myChar=2$;

$(12)$ $r\in \{6,7\}$, $\ell(n+s)+m\equiv s(2s)$;

$(13)$ $r=6$, $\ell(n+s)+m\equiv 1(2s),\text{ }\ell\div 2,\text{ }\myChar=3$;

$(14)$ $r=10$, $\ell(n+s)+m\equiv 0(2s),\text{ }\ell\ndiv 2,\text{ }\myChar=3$.

In other cases $\dim_K\HH^t(R)=0$.
\end{thm}

\begin{proof}
As $\dim_K\HH^t(R)=\dim_K\Ker\delta^t-\dim_K\Im\delta^{t-1}$, and
$\dim_K\Ker\delta^t=\dim_K\Hom_\Lambda(Q_t,R)-\dim_K\Im\delta^t$,
the assertions of theorem easily follows from propositions
\ref{dim_hom} -- \ref{dim_im}.
\end{proof}

\begin{thm}[Additive structure, $s=1$]
Let $R=R_1^\prime$ is algebra of the type $E_6$. Next, $t\in\N\cup\{0\}$,
$\ell$ be the aliquot, and $r$ be the residue of division of $t$ by $11$.\\
$($a$)$ $\dim_K\HH^t(R)=3$, if $t=0$.\\
$($b$)$ $\dim_K\HH^t(R)=2$, if one of the following conditions takes place$:$

$(1)$ $r\in \{0,10\}$, $t>0$, $\ell+m\div 2$, $\myChar=3$;

$(2)$ $r\in \{4,6\}$, $\ell+m\ndiv 2$, $\myChar=3$.\\
$($c$)$ $\dim_K\HH^t(R)=1$, if one of the following conditions takes
place$:$

$(1)$ $r\in \{0,10\}$, $t>0$, $\ell+m\div 2$, $\myChar\ne 3$;

$(2)$ $r\in \{1,9\}$;

$(3)$ $r\in \{2,8\}$, $\ell+m\div 2$;

$(4)$ $r=3$, $\ell+m\div 2$ or $\myChar=2$;

$(5)$ $r\in \{4,6\}$, $\ell+m\div 2$, $\myChar=2$;

$(6)$ $r\in \{4,6\}$, $\ell+m\ndiv 2$, $\myChar\ne 3$;

$(7)$ $r=5$, $\myChar=3$;

$(8)$ $r=7$, $\ell+m\ndiv 2$ or $\myChar=2$.\\ $($d$)$ In other cases
$\dim_K\HH^t(R)=0$.

\end{thm}

\section{Generators of $\HH^*(R)$}

For $s>1$ introduce the set of generators $Y^{(1)}_t$, $Y^{(2)}_t$,
\dots $Y^{(22)}_t$, such that $\deg Y_t^{(i)}=t$, $0\le t <
11\deg\sigma$ and $t$ satisfies conditions of (i)th item from the
list on page \pageref{degs}. For $s=1$ introduce the set of
generators $Y^{(1)}_t$, $Y^{(2)}_t$, \dots $Y^{(24)}_t$, such that
$\deg Y_t^{(i)}=t$, $0\le t < 11\deg\sigma$ and $t$ satisfies
conditions of (i)th item from the list on page \pageref{degs} for
$i\le 22$ and $t=0$ if $i>22$. Now let us describe the matrixes of
$Y^{(i)}_t$ componentwisely.

\begin{obozn}
Let us represent the degree $t$ of the generator element in the form
$t=11\ell+r$ ($0\le r\le 10$). Define the function $\kappa\text{: }\{w\in K\left[\mathcal Q_s\right]\}\rightarrow\Z$ 
which returns a coefficient of $\sigma(w)$; $\kappa^\ell(w)$ returns a coefficient of $\sigma^\ell(w)$.
\end{obozn}

(1) $Y^{(1)}_t$ is an $(6s\times 6s)$-matrix, whose elements $y_{ij}$ have the following form:

If $0\le j<s$, then $$y_{ij}=
\begin{cases}
\kappa^\ell(\a_{3(j+m)})e_{4j}\otimes e_{4j},\quad i=j;\\
0,\quad\text{otherwise.}\end{cases}$$

If $s\le j<3s$, then $$y_{ij}=
\begin{cases}
e_{4(j+s)+1}\otimes e_{4(j+s)+1},\quad i=j;\\
0,\quad\text{otherwise.}\end{cases}$$

If $3s\le j<5s$, then $$y_{ij}=
\begin{cases}
e_{4(j+s)+2}\otimes e_{4(j+s)+2},\quad i=j;\\
0,\quad\text{otherwise.}\end{cases}$$

If $5s\le j<6s$, then $$y_{ij}=
\begin{cases}
\kappa^\ell(\a_{3(j+m)})e_{4j+3}\otimes e_{4j+3},\quad i=j;\\
0,\quad\text{otherwise.}\end{cases}$$

(2) $Y^{(2)}_t$ is an $(6s\times 6s)$-matrix with a single nonzero
element:
$$y_{0,0}=\kappa^\ell(\a_0)w_{0\ra 4}\otimes e_{0}.$$

(3) $Y^{(3)}_t$ is an $(7s\times 6s)$-matrix with two nonzero
elements:
$$y_{0,0}=w_{0\ra 1}\otimes e_{0}\text{ and } y_{0,s}=w_{0\ra 4s+1}\otimes e_{0}.$$

(4) $Y^{(4)}_t$ is an $(7s\times 6s)$-matrix, whose elements $y_{ij}$ have the following form:

If $0\le j<s$, then $$y_{ij}=
\begin{cases}w_{4j\ra 4(j+s)+1}\otimes e_{4j},\quad i=j;\\
0,\quad\text{otherwise.}\end{cases}$$

If $s\le j<5s-1$, then $y_{ij}=0.$

If $5s-1\le j<6s-1$, then $$y_{ij}=
\begin{cases}-\kappa^\ell(\a_{3(j+m)})w_{4j+2\ra 4j+3}\otimes e_{4j+2},\quad i=j-s;\\
0,\quad\text{otherwise.}\end{cases}$$

If $6s-1\le j<7s$, then $y_{ij}=0.$

(5) $Y^{(5)}_t$ is an $(6s\times 6s)$-matrix with a single nonzero element:
$$y_{0,0}=\kappa^\ell(\a_0)w_{0\ra 3}\otimes e_{0}.$$

(6) $Y^{(6)}_t$ is an $(8s\times 6s)$-matrix, whose elements $y_{ij}$ have the following form:

If $0\le j<s$, then $$y_{ij}=
\begin{cases}w_{4j\ra 4j+2}\otimes e_{4j},\quad i=j;\\
0,\quad\text{otherwise.}\end{cases}$$

If $s\le j<3s$, then $y_{ij}=0.$

If $3s\le j<4s$, then $$y_{ij}=
\begin{cases}-\kappa^\ell(\a_{3(j+m)})w_{4j+1\ra 4j+3}\otimes e_{4j+1},\quad i=j-s;\\
0,\quad\text{otherwise.}\end{cases}$$

If $4s\le j<6s$, then $$y_{ij}=
\begin{cases}-\kappa^\ell(\a_{3(j+m+1)})w_{4j+2\ra 4(j+1)}\otimes e_{4j+2},\quad i=j-s,\text{ }j<5s-1\text{ or }j=6s-1;\\
0,\quad\text{otherwise.}\end{cases}$$

If $6s\le j<8s$, then $$y_{ij}=
\begin{cases}w_{4j+3\ra 4(j+1)+1}\otimes e_{4j+3},\quad i=5s+(j)_s,\text{ }j<7s-1\text{ or }j=8s-1;\\
0,\quad\text{otherwise.}\end{cases}$$

(7) $Y^{(7)}_t$ is an $(8s\times 6s)$-matrix, whose elements $y_{ij}$ have the following form:

If $0\le j<2s$, then $$y_{ij}=
\begin{cases}w_{4j\ra 4(j+s)+2}\otimes e_{4j},\quad i=(j)_s;\\
0,\quad\text{otherwise.}\end{cases}$$

If $2s\le j<4s$, then $$y_{ij}=
\begin{cases}w_{4j+1\ra 4j+3}\otimes e_{4j+1},\quad i=j-s;\\
0,\quad\text{otherwise.}\end{cases}$$

If $4s\le j<6s$, then $$y_{ij}=
\begin{cases}w_{4j+2\ra 4(j+1)}\otimes e_{4j+2},\quad i=j-s;\\
0,\quad\text{otherwise.}\end{cases}$$

If $6s\le j<8s$, then $y_{ij}=0.$

(8) $Y^{(8)}_t$ is an $(9s\times 6s)$-matrix with a single nonzero element:
$$y_{0,0}=w_{0\ra 3}\otimes e_{0}.$$

(9) $Y^{(9)}_t$ is an $(9s\times 6s)$-matrix, whose elements $y_{ij}$ have the following form:

If $0\le j<s$, then $y_{ij}=0.$

If $s\le j<2s$, then $$y_{ij}=
\begin{cases}\kappa^\ell(\a_{3(j+m)})e_{4j}\otimes e_{4j},\quad i=j-s;\\
0,\quad\text{otherwise.}\end{cases}$$

If $2s\le j<4s$, then $$y_{ij}=
\begin{cases}e_{4j+1}\otimes e_{4j+1},\quad i=j-s;\\
0,\quad\text{otherwise.}\end{cases}$$

If $4s\le j<5s$, then $y_{ij}=0.$

If $5s\le j<6s$, then $$y_{ij}=
\begin{cases}e_{4(j+s)+2}\otimes e_{4(j+s)+2},\quad i=j-2s;\\
0,\quad\text{otherwise.}\end{cases}$$

If $6s\le j<7s$, then $y_{ij}=0.$

If $7s\le j<8s$, then $$y_{ij}=
\begin{cases}e_{4j+2}\otimes e_{4j+2},\quad i=j-3s;\\
0,\quad\text{otherwise.}\end{cases}$$

If $8s\le j<9s$, then $$y_{ij}=
\begin{cases}-\kappa^\ell(\a_{3(j+m)})e_{4j+3}\otimes e_{4j+3},\quad i=j-3s;\\
0,\quad\text{otherwise.}\end{cases}$$

(10) $Y^{(10)}_t$ is an $(9s\times 6s)$-matrix with a single nonzero element:
$$y_{0,s}=\kappa^\ell(\a_6)w_{0\ra 4}\otimes e_{0}.$$

(11) $Y^{(11)}_t$ is an $(8s\times 6s)$-matrix, whose elements $y_{ij}$ have the following form:

If $0\le j<2s$, then $$y_{ij}=
\begin{cases}w_{4j\ra 4j+1}\otimes e_{4j},\quad i=(j)_s;\\
0,\quad\text{otherwise.}\end{cases}$$

If $2s\le j<4s$, then $$y_{ij}=
\begin{cases}\kappa^\ell(\a_{3(j+m+1)})w_{4j+1\ra 4(j+1)}\otimes e_{4j+1},\quad i=j-s;\\
0,\quad\text{otherwise.}\end{cases}$$

If $4s\le j<6s$, then $$y_{ij}=
\begin{cases}-\kappa^\ell(\a_{3(j+m)})w_{4j+2\ra 4j+3}\otimes e_{4j+2},\quad i=j-s;\\
0,\quad\text{otherwise.}\end{cases}$$

If $6s\le j<8s$, then $$y_{ij}=
\begin{cases}f_1(j,7s)w_{4j+3\ra 4(j+1)+2}\otimes e_{4j+3},\quad i=5s+(j)_s;\\
0,\quad\text{otherwise.}\end{cases}$$

(12) $Y^{(12)}_t$ is an $(8s\times 6s)$-matrix with two nonzero elements:
$$y_{3s,4s}=\kappa^\ell(\a_{3(j+m)}) w_{4j+2\ra 4j+3}\otimes e_{4j+2}\text{ and }
 y_{4s,5s}=\kappa^\ell(\a_{3(j+m)}) w_{4j+2\ra 4j+3}\otimes e_{4j+2}.$$

(13) $Y^{(13)}_t$ is an $(9s\times 6s)$-matrix, whose elements $y_{ij}$ have the following form:

If $0\le j<s$, then $y_{ij}=0.$

If $s\le j<2s$, then $$y_{ij}=
\begin{cases}e_{4(j+s)+1}\otimes e_{4(j+s)+1},\quad i=j;\\
0,\quad\text{otherwise.}\end{cases}$$

If $2s\le j<3s$, then $y_{ij}=0.$

If $3s\le j<4s$, then $$y_{ij}=
\begin{cases}e_{4j+1}\otimes e_{4j+1},\quad i=j-s;\\
0,\quad\text{otherwise.}\end{cases}$$

If $4s\le j<5s$, then $y_{ij}=0.$

If $5s\le j<7s$, then $$y_{ij}=
\begin{cases}e_{4(j+s)+2}\otimes e_{4(j+s)+2},\quad i=j-2s;\\
0,\quad\text{otherwise.}\end{cases}$$

If $7s\le j<8s$, then $y_{ij}=0.$

If $8s\le j<9s$, then $$y_{ij}=
\begin{cases}w_{4j+3\ra 4(j+1)}\otimes e_{4j+3},\quad i=j-3s;\\
0,\quad\text{otherwise.}\end{cases}$$

(14) $Y^{(14)}_t$ is an $(9s\times 6s)$-matrix with a single nonzero element:
$$y_{5s,7s}=\kappa^\ell(\a_{3(j+m)})w_{4j+3\ra 4(j+1)+3}\otimes e_{4j+3}.$$

(15) $Y^{(15)}_t$ is an $(9s\times 6s)$-matrix, whose elements $y_{ij}$ have the following form:

If $0\le j<s$, then $$y_{ij}=
\begin{cases}\kappa^\ell(\a_{3(j+m)})e_{4j}\otimes e_{4j},\quad i=j;\\
0,\quad\text{otherwise.}\end{cases}$$

If $s\le j<2s$, then $y_{ij}=0.$

If $2s\le j<3s$, then $$y_{ij}=
\begin{cases}w_{4j+1\ra 4j+2}\otimes e_{4j+1},\quad i=j-s;\\
0,\quad\text{otherwise.}\end{cases}$$

If $3s\le j<4s$, then $y_{ij}=0.$

If $4s\le j<5s$, then $$y_{ij}=
\begin{cases}w_{4(j+s)+1\ra 4(j+s)+2}\otimes e_{4(j+s)+1},\quad i=j-2s;\\
0,\quad\text{otherwise.}\end{cases}$$

If $5s\le j<7s$, then $y_{ij}=0.$

If $7s\le j<8s$, then $$y_{ij}=
\begin{cases}\kappa^\ell(\a_{3(j+m)})e_{4j+3}\otimes e_{4j+3},\quad i=j-2s;\\
0,\quad\text{otherwise.}\end{cases}$$

(16) $Y^{(16)}_t$ is an $(8s\times 6s)$-matrix with two nonzero elements:
$$y_{0,0}=w_{0\ra 2}\otimes e_0\text{ and } y_{s,2s}=w_{1\ra 3}\otimes e_1.$$

(17) $Y^{(17)}_t$ is an $(8s\times 6s)$-matrix with two nonzero elements:
$$y_{0,0}=w_{0\ra 4s+2}\otimes e_0\text{ and } y_{0,s}=w_{0\ra 2}\otimes e_0.$$

(18) $Y^{(18)}_t$ is an $(6s\times 6s)$-matrix, whose elements $y_{ij}$ have the following form:

If $0\le j<s$, then $y_{ij}=0.$

If $s\le j<3s$, then $$y_{ij}=
\begin{cases}w_{4(j+s)+1\ra 4(j+s)+2}\otimes e_{4(j+s)+1},\quad i=j;\\
0,\quad\text{otherwise.}\end{cases}$$

If $3s\le j<5s$, then $y_{ij}=0.$

If $5s\le j<6s$, then $$y_{ij}=
\begin{cases}-\kappa^\ell(\a_{3(j+m+1)})w_{4j+3\ra 4(j+1)}\otimes e_{4j+3},\quad i=j;\\
0,\quad\text{otherwise.}\end{cases}$$

(19) $Y^{(19)}_t$ is an $(7s\times 6s)$-matrix with two nonzero elements:
$$y_{s,s}=\kappa^\ell(\a_{3(j+m+1)})w_{4(j+s)+1\ra 4(j+1)}\otimes
 e_{4(j+s)+1}\text{ and } y_{2s,2s}=\kappa^\ell(\a_{3(j+m+1)})w_{4(j+s)+1\ra 4(j+1)}\otimes e_{4(j+s)+1}.$$

(20) $Y^{(20)}_t$ is an $(7s\times 6s)$-matrix, whose elements $y_{ij}$ have the following form:

If $0\le j<s$, then $$y_{ij}=
\begin{cases}\kappa^\ell(\a_{3(j+m)})w_{4j\ra 4j+3}\otimes e_{4j},\quad i=j;\\
0,\quad\text{otherwise.}\end{cases}$$

If $s\le j<2s$, then $$y_{ij}=
\begin{cases}-\kappa^\ell(\a_{3(j+m+1)})w_{4(j+s)+1\ra 4(j+1)}\otimes e_{4(j+s)+1},\quad i=j;\\
0,\quad\text{otherwise.}\end{cases}$$

If $2s\le j<3s$, then $y_{ij}=0.$

If $3s\le j<4s$, then $$y_{ij}=
\begin{cases}w_{4(j+s)+2\ra 4(j+s+1)+1}\otimes e_{4(j+s)+2},\quad i=j;\\
0,\quad\text{otherwise.}\end{cases}$$

If $4s\le j<6s$, then $y_{ij}=0.$

If $6s\le j<7s$, then $$y_{ij}=
\begin{cases}w_{4j+3\ra 4(j+1)+2}\otimes e_{4j+3},\quad i=j-s;\\
0,\quad\text{otherwise.}\end{cases}$$

(21) $Y^{(21)}_t$ is an $(6s\times 6s)$-matrix with a single nonzero element:
$$y_{0,0}=-\kappa^\ell(\a_{15})w_{0\ra 4}\otimes e_0.$$

(22) $Y^{(22)}_t$ is an $(6s\times 6s)$-matrix, whose elements $y_{ij}$ have the following form:

If $0\le j<s$, then $$y_{ij}=
\begin{cases}\kappa^\ell(\a_{3(j+m)})e_{4j}\otimes e_{4j},\quad i=j;\\
0,\quad\text{otherwise.}\end{cases}$$

If $s\le j<5s$, then $y_{ij}=0.$

If $5s\le j<6s$, then $$y_{ij}=
\begin{cases}-\kappa^\ell(\a_{3(j+m)})e_{4j+3}\otimes e_{4j+3},\quad i=j;\\
0,\quad\text{otherwise.}\end{cases}$$

(23) $Y^{(23)}_t$ is an $(6s\times 6s)$-matrix with a single nonzero element:
$$y_{0,0}=w_{0\ra 4}\otimes e_{0}.$$

(24) $Y^{(24)}_t$ is an $(6s\times 6s)$-matrix with a single nonzero element:
$$y_{5,5}=w_{3\ra 7}\otimes e_{3}.$$

\section{$\Omega$-shifts of generators of the algebra $\HH^*(R)$}

Let $Q_\bullet\rightarrow R$ be the minimal projective bimodule
resolution of the algebra $R$, constructed in paragraph
\ref{sect_res}. Any $t$-cocycle $f\in\Ker\delta^t$ is lifted
(uniquely up to homotopy) to a chain map of complexes $\{\varphi_i:
Q_{t+i}\rightarrow Q_i\}_{i\ge 0}$. The homomorphism $\varphi_i$ is
called the {\it $i$th translate} of the cocycle $f$ and will be
denoted by $\Omega^i(f)$. For cocycles $f_1\in\Ker\delta^{t_1}$ and
$f_2\in\Ker\delta^{t_2}$ we have
\begin{equation}\tag{$*$}\label{mult_formula}
\cl f_2\cdot \cl f_1=\cl(\Omega^0(f_2)\Omega^{t_2}(f_1)).
\end{equation}

We shall now describe $\Omega$-translates for generators of the
algebra $\HH^*(R)$ and then find multiplications of the generators
using the formula \eqref{mult_formula}.

\begin{obozns}$\quad$

(1) For generator degree $t$ represent it in the form $t=11\ell+r$
($0\le r\le 10$).

(2) For translate $\Omega^{t_0}$ represent $t_0$ in the form
$t_0=11\ell_0+r_0$ ($0\le r_0\le 10$).
\end{obozns}

\begin{pr}[Translates for the case 1]
$({\rm I})$ Let $r_0\in\N$, $r_0<11$. $r_0$-translates of the
elements $Y^{(1)}_t$ are described by the following way.

$(1)$ If $r_0=0$, then $\Omega^{0}(Y_t^{(1)})$ is described with
$(6s\times 6s)$-matrix with the following elements $b_{ij}${\rm:}

If $0\le j<s$, then $$b_{ij}=
\begin{cases}
\kappa^\ell(\a_{3(j+m)})e_{4j}\otimes e_{4j},\quad i=j;\\
0,\quad\text{otherwise.}
\end{cases}$$

If $s\le j<3s$, then $$b_{ij}=
\begin{cases}
e_{4(j+s)+1}\otimes e_{4(j+s)+1},\quad i=j;\\
0,\quad\text{otherwise.}
\end{cases}$$

If $3s\le j<5s$, then $$b_{ij}=
\begin{cases}
e_{4(j+s)+2}\otimes e_{4(j+s)+2},\quad i=j;\\
0,\quad\text{otherwise.}
\end{cases}$$

If $5s\le j<6s$, then $$b_{ij}=
\begin{cases}
\kappa^\ell(\a_{3(j+m)})e_{4j+3}\otimes e_{4j+3},\quad i=j;\\
0,\quad\text{otherwise.}
\end{cases}$$

$(2)$ If $r_0=1$, then $\Omega^{1}(Y_t^{(1)})$ is described with
$(7s\times 7s)$-matrix with the following elements $b_{ij}${\rm:}

If $0\le j<2s$, then $$b_{ij}=
\begin{cases}
e_{4(j+m)+1}\otimes e_{4j},\quad i=j;\\
0,\quad\text{otherwise.}
\end{cases}$$

If $2s\le j<4s$, then $$b_{ij}=
\begin{cases}
e_{4(j+m)+2}\otimes e_{4j+1},\quad i=j;\\
0,\quad\text{otherwise.}
\end{cases}$$

If $4s\le j<6s$, then $$b_{ij}=
\begin{cases}
\kappa^\ell(\a_{3(j+m)})e_{4(j+m)+3}\otimes e_{4j+2},\quad i=j;\\
0,\quad\text{otherwise.}
\end{cases}$$

If $6s\le j<7s$, then $$b_{ij}=
\begin{cases}
\kappa^\ell(\a_{3(j+m)})e_{4(j+m+1)}\otimes e_{4j+3},\quad i=j;\\
0,\quad\text{otherwise.}
\end{cases}$$

$(3)$ If $r_0=2$, then $\Omega^{2}(Y_t^{(1)})$ is described with
$(6s\times 6s)$-matrix with the following elements $b_{ij}${\rm:}

If $0\le j<s$, then $$b_{ij}=
\begin{cases}
\kappa^\ell(\a_{3(j+m)})e_{4(j+m)-1}\otimes e_{4j},\quad i=j;\\
0,\quad\text{otherwise.}
\end{cases}$$

If $s\le j<3s$, then $$b_{ij}=
\begin{cases}
e_{4(j+m+s)+2}\otimes e_{4(j+s)+1},\quad i=j;\\
0,\quad\text{otherwise.}
\end{cases}$$

If $3s\le j<5s$, then $$b_{ij}=
\begin{cases}
e_{4(j+m)+1}\otimes e_{4(j+s)+2},\quad i=j;\\
0,\quad\text{otherwise.}
\end{cases}$$

If $5s\le j<6s$, then $$b_{ij}=
\begin{cases}
\kappa^\ell(\a_{3(j+m)})e_{4(j+m+1)}\otimes e_{4j+3},\quad i=j;\\
0,\quad\text{otherwise.}
\end{cases}$$

$(4)$ If $r_0=3$, then $\Omega^{3}(Y_t^{(1)})$ is described with
$(8s\times 8s)$-matrix with the following elements $b_{ij}${\rm:}

If $0\le j<2s$, then $$b_{ij}=
\begin{cases}
e_{4(j+m)+2}\otimes e_{4j},\quad i=j;\\
0,\quad\text{otherwise.}
\end{cases}$$

If $2s\le j<4s$, then $$b_{ij}=
\begin{cases}
\kappa^\ell(\a_{3(j+m)})e_{4(j+m)+3}\otimes e_{4j+1},\quad i=j;\\
0,\quad\text{otherwise.}
\end{cases}$$

If $4s\le j<6s$, then $$b_{ij}=
\begin{cases}
\kappa^\ell(\a_{3(j+m)})e_{4(j+m+1)}\otimes e_{4j+2},\quad i=j;\\
0,\quad\text{otherwise.}
\end{cases}$$

If $6s\le j<8s$, then $$b_{ij}=
\begin{cases}
e_{4(j+m+1)+1}\otimes e_{4j+3},\quad i=j;\\
0,\quad\text{otherwise.}
\end{cases}$$

$(5)$ If $r_0=4$, then $\Omega^{4}(Y_t^{(1)})$ is described with
$(9s\times 9s)$-matrix with the following elements $b_{ij}${\rm:}

If $0\le j<s$, then $$b_{ij}=
\begin{cases}
\kappa^\ell(\a_{3(j+m)})e_{4(j+m)-1}\otimes e_{4j},\quad i=j;\\
0,\quad\text{otherwise.}
\end{cases}$$

If $s\le j<2s$, then $$b_{ij}=
\begin{cases}
\kappa^\ell(\a_{3(j+m)})e_{4(j+m)}\otimes e_{4j},\quad i=j;\\
0,\quad\text{otherwise.}
\end{cases}$$

If $2s\le j<4s$, then $$b_{ij}=
\begin{cases}
e_{4(j+m+s)+1}\otimes e_{4j+1},\quad i=j;\\
0,\quad\text{otherwise.}
\end{cases}$$

If $4s\le j<5s$, then $$b_{ij}=
\begin{cases}
e_{4(j+m)+1}\otimes e_{4j+2},\quad i=j;\\
0,\quad\text{otherwise.}
\end{cases}$$

If $5s\le j<6s$, then $$b_{ij}=
\begin{cases}
e_{4(j+m)+2}\otimes e_{4(j+s)+2},\quad i=j;\\
0,\quad\text{otherwise.}
\end{cases}$$

If $6s\le j<7s$, then $$b_{ij}=
\begin{cases}
e_{4(j+m+s)+1}\otimes e_{4(j+s)+2},\quad i=j;\\
0,\quad\text{otherwise.}
\end{cases}$$

If $7s\le j<8s$, then $$b_{ij}=
\begin{cases}
e_{4(j+m+s)+2}\otimes e_{4j+2},\quad i=j;\\
0,\quad\text{otherwise.}
\end{cases}$$

If $8s\le j<9s$, then $$b_{ij}=
\begin{cases}
\kappa^\ell(\a_{3(j+m)})e_{4(j+m+1)-1}\otimes e_{4j+3},\quad i=j;\\
0,\quad\text{otherwise.}
\end{cases}$$

$(6)$ If $r_0=5$, then $\Omega^{5}(Y_t^{(1)})$ is described with
$(8s\times 8s)$-matrix with the following elements $b_{ij}${\rm:}

If $0\le j<2s$, then $$b_{ij}=
\begin{cases}
e_{4(j+m)+1}\otimes e_{4j},\quad i=j;\\
0,\quad\text{otherwise.}
\end{cases}$$

If $2s\le j<4s$, then $$b_{ij}=
\begin{cases}
\kappa^\ell(\a_{3(j+m)})e_{4(j+m+1)}\otimes e_{4j+1},\quad i=j;\\
0,\quad\text{otherwise.}
\end{cases}$$

If $4s\le j<6s$, then $$b_{ij}=
\begin{cases}
\kappa^\ell(\a_{3(j+m)})e_{4(j+m+s)+3}\otimes e_{4j+2},\quad i=j;\\
0,\quad\text{otherwise.}
\end{cases}$$

If $6s\le j<8s$, then $$b_{ij}=
\begin{cases}
e_{4(j+m+1)+2}\otimes e_{4j+3},\quad i=j;\\
0,\quad\text{otherwise.}
\end{cases}$$

$(7)$ If $r_0=6$, then $\Omega^{6}(Y_t^{(1)})$ is described with
$(9s\times 9s)$-matrix with the following elements $b_{ij}${\rm:}

If $0\le j<s$, then $$b_{ij}=
\begin{cases}
\kappa^\ell(\a_{3(j+m)})e_{4(j+m)}\otimes e_{4j},\quad i=j;\\
0,\quad\text{otherwise.}
\end{cases}$$

If $s\le j<2s$, then $$b_{ij}=
\begin{cases}
e_{4(j+m+s)+1}\otimes e_{4(j+s)+1},\quad i=j;\\
0,\quad\text{otherwise.}
\end{cases}$$

If $2s\le j<3s$, then $$b_{ij}=
\begin{cases}
e_{4(j+m+s)+2}\otimes e_{4j+1},\quad i=j;\\
0,\quad\text{otherwise.}
\end{cases}$$

If $3s\le j<4s$, then $$b_{ij}=
\begin{cases}
e_{4(j+m)+1}\otimes e_{4j+1},\quad i=j;\\
0,\quad\text{otherwise.}
\end{cases}$$

If $4s\le j<5s$, then $$b_{ij}=
\begin{cases}
e_{4(j+m)+2}\otimes e_{4(j+s)+1},\quad i=j;\\
0,\quad\text{otherwise.}
\end{cases}$$

If $5s\le j<7s$, then $$b_{ij}=
\begin{cases}
e_{4(j+m+s)+2}\otimes e_{4(j+s)+2},\quad i=j;\\
0,\quad\text{otherwise.}
\end{cases}$$

If $7s\le j<8s$, then $$b_{ij}=
\begin{cases}
\kappa^\ell(\a_{3(j+m)})e_{4(j+m+1)-1}\otimes e_{4j+3},\quad i=j;\\
0,\quad\text{otherwise.}
\end{cases}$$

If $8s\le j<9s$, then $$b_{ij}=
\begin{cases}
\kappa^\ell(\a_{3(j+m)})e_{4(j+m+1)}\otimes e_{4j+3},\quad i=j;\\
0,\quad\text{otherwise.}
\end{cases}$$

$(8)$ If $r_0=7$, then $\Omega^{7}(Y_t^{(1)})$ is described with
$(8s\times 8s)$-matrix with the following elements $b_{ij}${\rm:}

If $0\le j<2s$, then $$b_{ij}=
\begin{cases}
e_{4(j+m)+2}\otimes e_{4j},\quad i=j;\\
0,\quad\text{otherwise.}
\end{cases}$$

If $2s\le j<4s$, then $$b_{ij}=
\begin{cases}
\kappa^\ell(\a_{3(j+m)})e_{4(j+m)+3}\otimes e_{4j+1},\quad i=j;\\
0,\quad\text{otherwise.}
\end{cases}$$

If $4s\le j<6s$, then $$b_{ij}=
\begin{cases}
\kappa^\ell(\a_{3(j+m)})e_{4(j+m+1)}\otimes e_{4j+2},\quad i=j;\\
0,\quad\text{otherwise.}
\end{cases}$$

If $6s\le j<8s$, then $$b_{ij}=
\begin{cases}
e_{4(j+m+1)+1}\otimes e_{4j+3},\quad i=j;\\
0,\quad\text{otherwise.}
\end{cases}$$

$(9)$ If $r_0=8$, then $\Omega^{8}(Y_t^{(1)})$ is described with
$(6s\times 6s)$-matrix with the following elements $b_{ij}${\rm:}

If $0\le j<s$, then $$b_{ij}=
\begin{cases}
\kappa^\ell(\a_{3(j+m)})e_{4(j+m)-1}\otimes e_{4j},\quad i=j;\\
0,\quad\text{otherwise.}
\end{cases}$$

If $s\le j<3s$, then $$b_{ij}=
\begin{cases}
e_{4(j+m+s)+2}\otimes e_{4(j+s)+1},\quad i=j;\\
0,\quad\text{otherwise.}
\end{cases}$$

If $3s\le j<5s$, then $$b_{ij}=
\begin{cases}
e_{4(j+m)+1}\otimes e_{4(j+s)+2},\quad i=j;\\
0,\quad\text{otherwise.}
\end{cases}$$

If $5s\le j<6s$, then $$b_{ij}=
\begin{cases}
\kappa^\ell(\a_{3(j+m)})e_{4(j+m+1)}\otimes e_{4j+3},\quad i=j;\\
0,\quad\text{otherwise.}
\end{cases}$$

$(10)$ If $r_0=9$, then $\Omega^{9}(Y_t^{(1)})$ is described with
$(7s\times 7s)$-matrix with the following elements $b_{ij}${\rm:}

If $0\le j<s$, then $$b_{ij}=
\begin{cases}
\kappa^\ell(\a_{3(j+m)})e_{4(j+m)+3}\otimes e_{4j},\quad i=j;\\
0,\quad\text{otherwise.}
\end{cases}$$

If $s\le j<3s$, then $$b_{ij}=
\begin{cases}
\kappa^\ell(\a_{3(j+m)})e_{4(j+m+1)}\otimes e_{4(j+s)+1},\quad i=j;\\
0,\quad\text{otherwise.}
\end{cases}$$

If $3s\le j<5s$, then $$b_{ij}=
\begin{cases}
e_{4(j+m+1)+1}\otimes e_{4(j+s)+2},\quad i=j;\\
0,\quad\text{otherwise.}
\end{cases}$$

If $5s\le j<7s$, then $$b_{ij}=
\begin{cases}
e_{4(j+m+1+s)+2}\otimes e_{4j+3},\quad i=j;\\
0,\quad\text{otherwise.}
\end{cases}$$

$(11)$ If $r_0=10$, then $\Omega^{10}(Y_t^{(1)})$ is described with
$(6s\times 6s)$-matrix with the following elements $b_{ij}${\rm:}

If $0\le j<s$, then $$b_{ij}=
\begin{cases}
\kappa^\ell(\a_{3(j+m)})e_{4(j+m)}\otimes e_{4j},\quad i=j;\\
0,\quad\text{otherwise.}
\end{cases}$$

If $s\le j<3s$, then $$b_{ij}=
\begin{cases}
e_{4(j+m)+1}\otimes e_{4(j+s)+1},\quad i=j;\\
0,\quad\text{otherwise.}
\end{cases}$$

If $3s\le j<5s$, then $$b_{ij}=
\begin{cases}
e_{4(j+m)+2}\otimes e_{4(j+s)+2},\quad i=j;\\
0,\quad\text{otherwise.}
\end{cases}$$

If $5s\le j<6s$, then $$b_{ij}=
\begin{cases}
\kappa^\ell(\a_{3(j+m)})e_{4(j+m+s)+3}\otimes e_{4j+3},\quad i=j;\\
0,\quad\text{otherwise.}
\end{cases}$$

\medskip
$({\rm II})$ Represent an arbitrary $t_0\in\N$ in the form
$t_0=11\ell_0+r_0$, where $0\le r_0\le 10.$ Then
$\Omega^{t_0}(Y_t^{(1)})$ is a $\Omega^{r_0}(Y_t^{(1)})$, whose left
components twisted by $\sigma^{\ell_0}$.
\end{pr}

\begin{pr}[Translates for the case 2]
$({\rm I})$ Let $r_0\in\N$, $r_0<11$. Denote by $\kappa_0=f_2(s,1)\kappa^\ell(\a_0)$. 
Then $r_0$-translates of the
elements $Y^{(2)}_t$ are described by the following way.

$(1)$ If $r_0=0$, then $\Omega^{0}(Y_t^{(2)})$ is described with
$(6s\times 6s)$-matrix with one nonzero element that is of the following form{\rm:}
$$b_{0,0}=\kappa^\ell(\a_0)w_{4(j+m)\ra 4(j+m+1)}\otimes e_{4j}.$$

$(2)$ If $r_0=1$, then $\Omega^{1}(Y_t^{(2)})$ is described with
$(7s\times 7s)$-matrix with one nonzero element that is of the following form{\rm:}
$$b_{j,6s+(5)_s}=\kappa_0w_{4(j+m+1)\ra 4(j+m+2)}\otimes e_{4j+3}.$$

$(3)$ If $r_0=2$, then $\Omega^{2}(Y_t^{(2)})$ is described with
$(6s\times 6s)$-matrix with one nonzero element that is of the following form{\rm:}
$$b_{j,5s+(4)_s}=\kappa_0w_{4(j+m+1)\ra 4(j+m+2)}\otimes e_{4j+3}.$$

$(4)$ If $r_0=3$, then $\Omega^{3}(Y_t^{(2)})$ is described with
$(8s\times 8s)$-matrix with the following two nonzero elements{\rm:}
$$b_{j,4s+(4)_s}=\kappa_0w_{4(j+m+1)\ra 4(j+m+2)}\otimes e_{4j+2};$$
$$b_{j,5s+(4)_s}=\kappa_0w_{4(j+m+1)\ra 4(j+m+2)}\otimes e_{4j+2}.$$

$(5)$ If $r_0=4$, then $\Omega^{4}(Y_t^{(2)})$ is described with
$(9s\times 9s)$-matrix with one nonzero element that is of the following form{\rm:}
$$b_{j,s+(4)_s}=\kappa_0w_{4(j+m)\ra 4(j+m+1)}\otimes e_{4j}.$$

$(6)$ If $r_0=5$, then $\Omega^{5}(Y_t^{(2)})$ is described with
$(8s\times 8s)$-matrix with the following two nonzero elements{\rm:}
$$b_{j,2s+(3)_s}=\kappa_0w_{4(j+m+1)\ra 4(j+m+2)}\otimes e_{4j+1};$$
$$b_{j,3s+(3)_s}=\kappa_0w_{4(j+m+1)\ra 4(j+m+2)}\otimes e_{4j+1}.$$

$(7)$ If $r_0=6$, then $\Omega^{6}(Y_t^{(2)})$ is described with
$(9s\times 9s)$-matrix with the following two nonzero elements{\rm:}
$$b_{j,(3)_s}=\kappa_0w_{4(j+m)\ra 4(j+m+1)}\otimes e_{4j};$$
$$b_{j,8s+(2)_s}=\kappa_0w_{4(j+m+1)\ra 4(j+m+2)}\otimes e_{4j+3}.$$

$(8)$ If $r_0=7$, then $\Omega^{7}(Y_t^{(2)})$ is described with
$(8s\times 8s)$-matrix with the following two nonzero elements{\rm:}
$$b_{j,4s+(2)_s}=\kappa_0w_{4(j+m+1)\ra 4(j+m+2)}\otimes e_{4j+2};$$
$$b_{j,5s+(2)_s}=\kappa_0w_{4(j+m+1)\ra 4(j+m+2)}\otimes e_{4j+2}.$$

$(9)$ If $r_0=8$, then $\Omega^{8}(Y_t^{(2)})$ is described with
$(6s\times 6s)$-matrix with one nonzero element that is of the following form{\rm:}
$$b_{j,5s+(1)_s}=\kappa_0w_{4(j+m+1)\ra 4(j+m+2)}\otimes e_{4j+3}.$$

$(10)$ If $r_0=9$, then $\Omega^{9}(Y_t^{(2)})$ is described with
$(7s\times 7s)$-matrix with the following two nonzero elements{\rm:}
$$b_{j,s+(1)_s}=\kappa_0w_{4(j+m+1)\ra 4(j+m+2)}\otimes e_{4(j+s)+1};$$
$$b_{j,2s+(1)_s}=\kappa_0w_{4(j+m+1)\ra 4(j+m+2)}\otimes e_{4(j+s)+1}.$$

$(11)$ If $r_0=10$, then $\Omega^{10}(Y_t^{(2)})$ is described with
$(6s\times 6s)$-matrix with one nonzero element that is of the following form{\rm:}
$$b_{j,(1)_s}=\kappa_0w_{4(j+m)\ra 4(j+m+1)}\otimes e_{4j}.$$

\medskip
$({\rm II})$ Represent an arbitrary $t_0\in\N$ in the form
$t_0=11\ell_0+r_0$, where $0\le r_0\le 10.$ Then
$\Omega^{t_0}(Y_t^{(2)})$ is a $\Omega^{r_0}(Y_t^{(2)})$, whose left
components twisted by $\sigma^{\ell_0}$.
\end{pr}

\begin{pr}[Translates for the case 3]
$({\rm I})$ Let $r_0\in\N$, $r_0<11$. $r_0$-translates of the
elements $Y^{(3)}_t$ are described by the following way.

$(1)$ If $r_0=0$, then $\Omega^{0}(Y_t^{(3)})$ is described with
$(7s\times 6s)$-matrix with the following two nonzero elements{\rm:}
$$b_{0,0}=w_{4j\ra 4j+1}\otimes e_{4j};$$
$$b_{0,s}=w_{4j\ra 4j+1}\otimes e_{4j}.$$

$(2)$ If $r_0=1$, then $\Omega^{1}(Y_t^{(3)})$ is described with
$(6s\times 7s)$-matrix with the following nonzero elements{\rm:}
$$b_{(s+j+1)_{2s},s+(1)_s}=w_{4(j+m+s+1)+1\ra 4(j+m+s+1)+2}\otimes w_{4(j+s)+1\ra 4(j+1)};$$
$$b_{2s+(s+j+1)_{2s},s+(1)_s}=e_{4(j+m+s+1)+2}\otimes w_{4(j+s)+1\ra 4(j+s+1)+1};$$
$$b_{(s+j+1)_{2s},2s+(1)_s}=w_{4(j+m+s+1)+1\ra 4(j+m+s+1)+2}\otimes w_{4(j+s)+1\ra 4(j+1)};$$
$$b_{2s+(s+j+1)_{2s},2s+(1)_s}=e_{4(j+m+s+1)+2}\otimes w_{4(j+s)+1\ra 4(j+s+1)+1};$$
$$b_{(j+1)_{2s},3s+(1)_s}=e_{4(j+m+1)+1}\otimes w_{4(j+s)+2\ra 4(j+1)};$$
$$b_{(j+1)_{2s},4s+(1)_s}=e_{4(j+m+1)+1}\otimes w_{4(j+s)+2\ra 4(j+1)};$$
$$b_{j+s,5s+(1)_s}=-\kappa^\ell(\a_{3(j+m)})w_{4(j+m+1)\ra 4(j+m+2)}\otimes e_{4j+3}.$$

$(3)$ If $r_0=2$, then $\Omega^{2}(Y_t^{(3)})$ is described with
$(8s\times 6s)$-matrix with the following nonzero elements{\rm:}
$$b_{(j+1)_s,2s+(1)_s}=f_1((j)_s,s-1)\kappa^\ell(\a_{3(j+m)})e_{4(j+m)+3}\otimes w_{4j+1\ra 4(j+1)};$$
$$b_{(j+1)_s,3s+(1)_s}=-f_1((j)_s,s-1)\kappa^\ell(\a_{3(j+m)})e_{4(j+m)+3}\otimes w_{4j+1\ra 4(j+1)};$$
$$b_{j-s,4s+(1)_s}=\kappa^\ell(\a_{3(j+m)})w_{4(j+m+s)+1\ra 4(j+m+1)}\otimes e_{4j+2};$$
$$b_{j-s,5s+(1)_s}=\kappa^\ell(\a_{3(j+m)})w_{4(j+m+s)+1\ra 4(j+m+1)}\otimes e_{4j+2}.$$

$(4)$ If $r_0=3$, then $\Omega^{3}(Y_t^{(3)})$ is described with
$(9s\times 8s)$-matrix with the following nonzero elements{\rm:}
$$b_{j-s,s+(1)_s}=-\kappa^\ell(\a_{3(j+m)})w_{4(j+m+s)+2\ra 4(j+m+1)}\otimes e_{4j};$$
$$b_{j+s,s+(1)_s}=-\kappa^\ell(\a_{3(j+m)})w_{4(j+m)+3\ra 4(j+m+1)}\otimes w_{4j\ra 4(j+s)+1};$$
$$b_{j,2s+(1)_s}=-w_{4(j+m)+3\ra 4(j+m+s+1)+1}\otimes e_{4j+1};$$
$$b_{j,3s+(1)_s}=-w_{4(j+m)+3\ra 4(j+m+s+1)+1}\otimes e_{4j+1}.$$

$(5)$ If $r_0=4$, then $\Omega^{4}(Y_t^{(3)})$ is described with
$(8s\times 9s)$-matrix with the following nonzero elements{\rm:}
$$b_{j,(1)_s}=w_{4(j+m-1)+3\ra 4(j+m)+1}\otimes e_{4j};$$
$$b_{s+(j+1)_s,2s}=f_1((j)_s,s-1)\kappa^\ell(\a_{3(j+m)})e_{4(j+m+1)}\otimes w_{4j+1\ra 4(j+1)};$$
$$b_{s+(j+1)_s,3s}=-f_1((j)_s,s-1)\kappa^\ell(\a_{3(j+m)})e_{4(j+m+1)}\otimes w_{4j+1\ra 4(j+1)};$$
$$b_{2s+(s+j+1)_{2s},6s}=-w_{4(j+m+1)+1\ra 4(j+m+1)+2}\otimes w_{4j+3\ra 4(j+s+1)+1};$$
$$b_{6s+(s+j+1)_{2s}-f(s,1)s,6s}=e_{4(j+m+1)+2}\otimes w_{4j+3\ra 4(j+s+1)+2};$$
$$b_{2s+(s+j+1)_{2s},7s}=w_{4(j+m+1)+1\ra 4(j+m+1)+2}\otimes w_{4j+3\ra 4(j+s+1)+1};$$
$$b_{5s+(s+j+1)_{2s}+f(s,1)s,7s}=-e_{4(j+m+1)+2}\otimes w_{4j+3\ra 4(j+s+1)+2}.$$

$(6)$ If $r_0=5$, then $\Omega^{5}(Y_t^{(3)})$ is described with
$(9s\times 8s)$-matrix with the following nonzero elements{\rm:}
$$b_{(s+j+1)_{2s},s}=-e_{4(j+m+s+1)+1}\otimes w_{4(j+s)+1\ra 4(j+1)};$$
$$b_{(s+j+1)_{2s},2s}=w_{4(j+m+s+1)+1\ra 4(j+m+s+1)+2}\otimes w_{4j+1\ra 4(j+1)};$$
$$b_{(j+1)_{2s},3s}=-e_{4(j+m+1)+1}\otimes w_{4j+1\ra 4(j+1)};$$
$$b_{(j+1)_{2s},4s}=w_{4(j+m+1)+1\ra 4(j+m+1)+2}\otimes w_{4(j+s)+1\ra 4(j+1)};$$
$$b_{(s+j+1)_{2s},7s}=\kappa^\ell(\a_{3(j+m)})w_{4(j+m+s+1)+1\ra 4(j+m+1)+3}\otimes w_{4j+3\ra 4(j+1)};$$
$$b_{(j+1)_{2s},7s}=\kappa^\ell(\a_{3(j+m)})w_{4(j+m+1)+1\ra 4(j+m+1)+3}\otimes w_{4j+3\ra 4(j+1)};$$
$$b_{4s+(s+j+1)_{2s},7s}=-\kappa^\ell(\a_{3(j+m)})e_{4(j+m+1)+3}\otimes w_{4j+3\ra 4(j+s+1)+2};$$
$$b_{4s+(j+1)_{2s},7s}=-\kappa^\ell(\a_{3(j+m)})e_{4(j+m+1)+3}\otimes w_{4j+3\ra 4(j+1)+2};$$
$$b_{2s+(s+j+1)_{2s},8s}=-\kappa^\ell(\a_{3(j+m)})e_{4(j+m+2)}\otimes w_{4j+3\ra 4(j+s+1)+1}.$$

$(7)$ If $r_0=6$, then $\Omega^{6}(Y_t^{(3)})$ is described with
$(8s\times 9s)$-matrix with the following nonzero elements{\rm:}
$$b_{(j+1)_s,5s}=-\kappa^\ell(\a_{3(j+m)})e_{4(j+m+1)}\otimes w_{4j+2\ra 4(j+1)};$$
$$b_{(j+1)_s,6s}=w_{4(j+m+1)\ra 4(j+m+1)+1}\otimes w_{4j+3\ra 4(j+1)};$$
$$b_{s+(j+1)_{2s}+f(s,1)s,6s}=-e_{4(j+m+1)+1}\otimes w_{4j+3\ra 4(j+1)+1};$$
$$b_{2s+(j+1)_{2s}-f(s,1)s,7s}=-e_{4(j+m+1)+1}\otimes w_{4j+3\ra 4(j+1)+1}.$$

$(8)$ If $r_0=7$, then $\Omega^{7}(Y_t^{(3)})$ is described with
$(6s\times 8s)$-matrix with the following nonzero elements{\rm:}
$$b_{(j+s+1)_{2s},s}=-e_{4(j+m+s+1)+2}\otimes w_{4(j+s)+1\ra 4(j+1)};$$
$$b_{(j+s+1)_{2s},2s}=-e_{4(j+m+s+1)+2}\otimes w_{4(j+s)+1\ra 4(j+1)};$$
$$b_{s+(j+1)_s,5s}=\kappa^\ell(\a_{3(j+m)})w_{4(j+m+1+f(s,1)s)+2\ra 4(j+m+2)}\otimes w_{4j+3\ra 4(j+1)};$$
$$b_{2s+(j+1)_s,5s}=-\kappa^\ell(\a_{3(j+m)})w_{4(j+m+1)+3\ra 4(j+m+2)}\otimes w_{4j+3\ra 4(j+s+1+f(s,1)s)+1};$$
$$b_{4s+(j+s+1)_{2s},5s}=-\kappa^\ell(\a_{3(j+m)})e_{4(j+m+2)}\otimes w_{4j+3\ra 4(j+s+1)+2};$$
$$b_{4s+(j+1)_{2s},5s}=-\kappa^\ell(\a_{3(j+m)})e_{4(j+m+2)}\otimes w_{4j+3\ra 4(j+1)+2}.$$

$(9)$ If $r_0=8$, then $\Omega^{8}(Y_t^{(3)})$ is described with
$(7s\times 6s)$-matrix with the following nonzero elements{\rm:}
$$b_{(j+1)_s,0}=f_1((j)_s,s-1)\kappa^\ell(\a_{3(j+m)})e_{4(j+m)+3}\otimes w_{4j\ra 4(j+1)};$$
$$b_{(j+1)_s,s+f(s,1)s}=-\kappa^\ell(\a_{3(j+m)})w_{4(j+m)+3\ra 4(j+m+1)}\otimes w_{4(j+s)+1\ra 4(j+1)};$$
$$b_{3s+(j+s+1)_{2s},3s}=e_{4(j+m+1)+1}\otimes w_{4(j+s)+2\ra 4(j+s+1)+2};$$
$$b_{(j+1)_s,3s+f(s,1)s}=w_{4(j+m)+3\ra 4(j+m+1)+1}\otimes w_{4(j+s)+2\ra 4(j+1)};$$
$$b_{3s+(j+s+1)_{2s},4s}=e_{4(j+m+1)+1}\otimes w_{4(j+s)+2\ra 4(j+s+1)+2};$$
$$b_{j,5s}=w_{4(j+m+1)\ra 4(j+m+s+1)+2}\otimes e_{4j+3};$$
$$b_{j-s,6s}=w_{4(j+m+1)\ra 4(j+m+s+1)+2}\otimes e_{4j+3}.$$

$(10)$ If $r_0=9$, then $\Omega^{9}(Y_t^{(3)})$ is described with
$(6s\times 7s)$-matrix with the following two nonzero elements{\rm:}
$$b_{j,3s}=-w_{4(j+m+1)+1\ra 4(j+m+1)+2}\otimes e_{4(j+s)+2};$$
$$b_{j,4s}=-w_{4(j+m+1)+1\ra 4(j+m+1)+2}\otimes e_{4(j+s)+2}.$$

$(11)$ If $r_0=10$, then $\Omega^{10}(Y_t^{(3)})$ is described with
$(6s\times 6s)$-matrix with the following nonzero elements{\rm:}
$$b_{(j+1)_s,0}=f_1((j)_s,s-1)\kappa^\ell(\a_{3(j+m)})e_{4(j+m+1)}\otimes w_{4j\ra 4(j+1)};$$
$$b_{j+3s,0}=-\kappa^\ell(\a_{3(j+m)})w_{4(j+m+s)+2\ra 4(j+m+1)}\otimes w_{4j\ra 4j+2};$$
$$b_{j+4s,0}=\kappa^\ell(\a_{3(j+m)})w_{4(j+m)+2\ra 4(j+m+1)}\otimes w_{4j\ra 4(j+s)+2};$$
$$b_{j+5s,0}=-\kappa^\ell(\a_{3(j+m)})w_{4(j+m)+3\ra 4(j+m+1)}\otimes w_{4j\ra 4j+3};$$
$$b_{j,s}=w_{4(j+m)+1\ra 4(j+m+1)+1}\otimes e_{4(j+s)+1};$$
$$b_{j,2s}=w_{4(j+m)+1\ra 4(j+m+1)+1}\otimes e_{4(j+s)+1};$$
$$b_{3s+(j+s+1)_{2s},3s+(s-1)_s}=e_{4(j+m+1)+2}\otimes w_{4(j+s)+2\ra 4(j+s+1)+2};$$
$$b_{3s+(j+s+1)_{2s},4s+(s-1)_s}=e_{4(j+m+1)+2}\otimes w_{4(j+s)+2\ra 4(j+s+1)+2};$$
$$b_{3s+(j+1)_{2s},5s+(s-1)_s}=-\kappa^\ell(\a_{3(j+m)})w_{4(j+m+s+1)+2\ra 4(j+m+1)+3}\otimes w_{4j+3\ra 4(j+1)+2};$$
$$b_{3s+(j+s+1)_{2s},5s+(s-1)_s}=\kappa^\ell(\a_{3(j+m)})w_{4(j+m+1)+2\ra 4(j+m+1)+3}\otimes w_{4j+3\ra 4(j+s+1)+2};$$
$$b_{5s+(j+1)_s,5s+(s-1)_s}=f_1((j)_s,s-1)\kappa^\ell(\a_{3(j+m)})e_{4(j+m+1)+3}\otimes w_{4j+3\ra 4(j+1)+3}.$$

\medskip
$({\rm II})$ Represent an arbitrary $t_0\in\N$ in the form
$t_0=11\ell_0+r_0$, where $0\le r_0\le 10.$ Then
$\Omega^{t_0}(Y_t^{(3)})$ is a $\Omega^{r_0}(Y_t^{(3)})$, whose left
components twisted by $\sigma^{\ell_0}$,
and coefficients multiplied by $(-1)^{\ell_0}$.
\end{pr}

\begin{pr}[Translates for the case 4]\label{type4}
$({\rm I})$ Let $r_0\in\N$, $r_0<11$. $r_0$-translates of the
elements $Y^{(4)}_t$ are described by the following way.

$(1)$ If $r_0=0$, then $\Omega^{0}(Y_t^{(4)})$ is described with
$(7s\times 6s)$-matrix with the following elements $b_{ij}${\rm:}

If $0\le j<s$, then $$b_{ij}=
\begin{cases}
w_{4(j+m)\ra 4(j+m+s)+1}\otimes e_{4j},\quad i=(j)_s;\\
0,\quad\text{otherwise.}
\end{cases}$$

If $s\le j<5s-1$, then $b_{ij}=0$.

If $5s-1\le j<6s-1$, then $$b_{ij}=
\begin{cases}
-\kappa^\ell(\a_{3(j+m)})w_{4(j+m)+2\ra 4(j+m)+3}\otimes e_{4j+2},\quad i=j-s;\\
0,\quad\text{otherwise.}
\end{cases}$$

If $6s-1\le j<7s$, then $b_{ij}=0$.

$(2)$ If $r_0=1$, then $\Omega^{1}(Y_t^{(4)})$ is described with
$(6s\times 7s)$-matrix with the following elements $b_{ij}${\rm:}

If $0\le j<s$, then $$b_{ij}=
\begin{cases}
\kappa_1w_{4(j+m)+1\ra 4(j+m)+3}\otimes e_{4j},\quad i=j;\\
\kappa_1w_{4(j+m)+2\ra 4(j+m)+3}\otimes w_{4j\ra 4j+1},\quad i=j+2s;\\
\kappa_1e_{4(j+m)+3}\otimes w_{4j\ra 4j+2},\quad i=j+4s,\text{ }j<s-1;\\
-\kappa_1e_{4(j+m)+3}\otimes w_{4j\ra 4(j+s)+2},\quad i=j+5s,\text{ }j<s-1;\\
0,\quad\text{otherwise,}
\end{cases}$$
where $\kappa_1=\kappa^\ell(\a_{3(j+m)}).$

If $s\le j<2s-1$, then $$b_{ij}=
\begin{cases}
-w_{4(j+m+1)+1\ra 4(j+m+1)+2}\otimes w_{4(j+s)+1\ra 4(j+1)},\quad i=s+(j+1)_s;\\
0,\quad\text{otherwise.}
\end{cases}$$

If $2s-1\le j<3s-1$, then $b_{ij}=0$.

If $3s-1\le j<3s$, then $$b_{ij}=
\begin{cases}
-w_{4(j+m+1)+1\ra 4(j+m+1)+2}\otimes w_{4(j+s)+1\ra 4(j+1)},\quad i=s+(j+1)_s;\\
0,\quad\text{otherwise.}
\end{cases}$$

If $3s\le j<4s-1$, then $b_{ij}=0$.

If $4s-1\le j<4s$, then $$b_{ij}=
\begin{cases}
-w_{4(j+m)+3\ra 4(j+m+s+1)+1}\otimes e_{4(j+s)+2},\quad i=j+s;\\
-w_{4(j+m+1)\ra 4(j+m+s+1)+1}\otimes w_{4(j+s)+2\ra 4j+3},\quad i=j+3s;\\
0,\quad\text{otherwise.}
\end{cases}$$

If $4s\le j<5s-1$, then $$b_{ij}=
\begin{cases}
-w_{4(j+m)+3\ra 4(j+m+s+1)+1}\otimes e_{4(j+s)+2},\quad i=j+s;\\
-w_{4(j+m+1)\ra 4(j+m+s+1)+1}\otimes w_{4(j+s)+2\ra 4j+3},\quad i=j+2s;\\
0,\quad\text{otherwise.}
\end{cases}$$

If $5s-1\le j<5s$, then $b_{ij}=0$.

If $5s\le j<6s$, then $$b_{ij}=
\begin{cases}
f_2(j,6s-1)\kappa_1w_{4(j+m+1)\ra 4(j+m+2)}\otimes e_{4j+3},\quad i=j+s;\\
f_2(s,1)\kappa_1w_{4(j+m+1)+3\ra 4(j+m+2)}\otimes w_{4j+3\ra 4(j+s+1+sf(s,1))+2},\\\quad\quad\quad i=4s+(j+1)_s,\text{ }j=5s+(s-2)_s;\\
f_2(s,1)\kappa_1e_{4(j+m+2)}\otimes w_{4j+3\ra 4(j+s+1)+3},\quad i=6s+(j+1)_s,\text{ }j=5s+(s-2)_s;\\
0,\quad\text{otherwise,}
\end{cases}$$
where $\kappa_1=\kappa^\ell(\a_{3(j+m+2)})$.

$(3)$ If $r_0=2$, then $\Omega^{2}(Y_t^{(4)})$ is described with
$(8s\times 6s)$-matrix with the following elements $b_{ij}${\rm:}

If $0\le j<s-1$, then $$b_{ij}=
\begin{cases}
w_{4(j+m+s)+1\ra 4(j+m+s)+2}\otimes w_{4j\ra 4j+2},\quad i=j+3s;\\
0,\quad\text{otherwise.}
\end{cases}$$

If $s-1\le j<s$, then $b_{ij}=0$.

If $s\le j<2s-1$, then $$b_{ij}=
\begin{cases}
-w_{4(j+m+s)-1\ra 4(j+m+s)+2}\otimes e_{4j},\quad i=j-s;\\
0,\quad\text{otherwise.}
\end{cases}$$

If $2s-1\le j<2s$, then $$b_{ij}=
\begin{cases}
-e_{4(j+m+s)+2}\otimes w_{4j\ra 4(j+s)+1},\quad i=j;\\
0,\quad\text{otherwise.}
\end{cases}$$

If $2s\le j<3s$, then $$b_{ij}=
\begin{cases}
-\kappa^\ell(\a_{3(j+m)})e_{4(j+m)+3}\otimes w_{4j+1\ra 4(j+1)},\quad i=(j+1)_s;\\
-\kappa^\ell(\a_{3(j+m)})w_{4(j+m)+2\ra 4(j+m)+3}\otimes e_{4j+1},\quad i=j-s,\text{ }j=3s-1;\\
0,\quad\text{otherwise.}
\end{cases}$$

If $3s\le j<4s$, then $$b_{ij}=
\begin{cases}
-\kappa^\ell(\a_{3(j+m)})e_{4(j+m)+3}\otimes w_{4j+1\ra 4(j+1)},\quad i=(j+1)_s;\\
-\kappa^\ell(\a_{3(j+m)})w_{4(j+m)+2\ra 4(j+m)+3}\otimes e_{4j+1},\quad i=j-s,\text{ }j<4s-1;\\
0,\quad\text{otherwise.}
\end{cases}$$

If $4s\le j<5s-1$, then $$b_{ij}=
\begin{cases}
-\kappa_1w_{4(j+m)+3\ra 4(j+m+1)}\otimes w_{4j+2\ra 4(j+1)},\quad i=(j+1)_s;\\
-\kappa_1w_{4(j+m+s)+1\ra 4(j+m+1)}\otimes e_{4j+2},\quad i=j-s;\\
0,\quad\text{otherwise,}
\end{cases}$$
where $\kappa_1=\kappa^\ell(\a_{3(j+m+1)})$.

If $5s-1\le j<6s-1$, then $b_{ij}=0$.

If $6s-1\le j<6s$, then $$b_{ij}=
\begin{cases}
\kappa_1w_{4(j+m)+3\ra 4(j+m+1)}\otimes w_{4j+2\ra 4(j+1)},\quad i=(j+1)_s;\\
\kappa_1w_{4(j+m+s)+1\ra 4(j+m+1)}\otimes e_{4j+2},\quad i=j-s;\\
0,\quad\text{otherwise,}
\end{cases}$$
where $\kappa_1=-\kappa^\ell(\a_{3(j+m+1)})$.

If $6s\le j<7s-1$, then $$b_{ij}=
\begin{cases}
e_{4(j+m+s+1)+1}\otimes w_{4j+3\ra 4(j+1)+2},\quad i=3s+(j+1)_s;\\
0,\quad\text{otherwise.}
\end{cases}$$

If $7s-1\le j<8s-1$, then $b_{ij}=0$.

If $8s-1\le j<8s$, then $$b_{ij}=
\begin{cases}
e_{4(j+m+s+1)+1}\otimes w_{4j+3\ra 4(j+1)+2},\quad i=3s+(j+1)_s;\\
0,\quad\text{otherwise.}
\end{cases}$$

$(4)$ If $r_0=3$, then $\Omega^{3}(Y_t^{(4)})$ is described with
$(9s\times 8s)$-matrix with the following elements $b_{ij}${\rm:}

If $0\le j<s-1$, then $$b_{ij}=
\begin{cases}
\kappa^\ell(\a_{3(j+m)})w_{4(j+m)+2\ra 4(j+m)+3}\otimes e_{4j},\quad i=j;\\
\kappa^\ell(\a_{3(j+m)})e_{4(j+m)+3}\otimes w_{4j\ra 4j+1},\quad i=j+2s;\\
\kappa^\ell(\a_{3(j+m)})e_{4(j+m)+3}\otimes w_{4j\ra 4(j+s)+1},\quad i=j+3s;\\
0,\quad\text{otherwise.}
\end{cases}$$

If $s-1\le j<2s-1$, then $b_{ij}=0$.

If $2s-1\le j<2s$, then $$b_{ij}=
\begin{cases}
\kappa^\ell(\a_{3(j+m+1)})w_{4(j+m+s)+2\ra 4(j+m+1)}\otimes e_{4j},\quad i=j-s;\\
0,\quad\text{otherwise.}
\end{cases}$$

If $2s\le j<3s-1$, then $b_{ij}=0$.

If $3s-1\le j<3s$, then $$b_{ij}=
\begin{cases}
-w_{4(j+m)+3\ra 4(j+m+1)+1}\otimes e_{4j+1},\quad i=j;\\
0,\quad\text{otherwise.}
\end{cases}$$

If $3s\le j<4s-1$, then $$b_{ij}=
\begin{cases}
-w_{4(j+m)+3\ra 4(j+m+1)+1}\otimes e_{4j+1},\quad i=j;\\
0,\quad\text{otherwise.}
\end{cases}$$

If $4s-1\le j<4s$, then $b_{ij}=0$.

If $4s\le j<5s-1$, then $$b_{ij}=
\begin{cases}
-w_{4(j+m+1)\ra 4(j+m+s+1)+1}\otimes e_{4j+2},\quad i=j;\\
-e_{4(j+m+s+1)+1}\otimes w_{4j+2\ra 4j+3},\quad i=j+3s;\\
0,\quad\text{otherwise.}
\end{cases}$$

If $7\le j<8$, then $$b_{ij}=
\begin{cases}
-e_{4(j+m+1)+2}\otimes w_{4j+2\ra 4(j+1)},\quad i=(j+1)_s,\text{ }s=1;\\
0,\quad\text{otherwise.}
\end{cases}$$

If $6s-2\le j<6s-1$, then $$b_{ij}=
\begin{cases}
-e_{4(j+m+s+1)+2}\otimes w_{4(j+s)+2\ra 4(j+1)},\quad i=(j+1)_s,\text{ }s>1;\\
0,\quad\text{otherwise.}
\end{cases}$$

If $7s-1\le j<7s$, then $$b_{ij}=
\begin{cases}
-w_{4(j+m+1)\ra 4(j+m+1)+1}\otimes e_{4(j+s)+2},\quad i=j-s;\\
-e_{4(j+m+1)+1}\otimes w_{4(j+s)+2\ra 4j+3},\quad i=j;\\
0,\quad\text{otherwise.}
\end{cases}$$

If $7s\le j<8s+(s-2)_s$, then $b_{ij}=0$.

If $8s+(s-2)_s\le j<8s+(s-2)_s+1$, then $$b_{ij}=
\begin{cases}
\kappa^\ell(\a_{3(j+m)})e_{4(j+m+1)+3}\otimes w_{4j+3\ra 4(j+1+sf(s,1))+1},\quad i=2s+(j+1)_s;\\
0,\quad\text{otherwise.}
\end{cases}$$

If $8s+(s-2)_s+1\le j<9s$, then $b_{ij}=0$.

$(5)$ If $r_0=4$, then $\Omega^{4}(Y_t^{(4)})$ is described with
$(8s\times 9s)$-matrix with the following elements $b_{ij}${\rm:}

If $0\le j<s-1$, then $$b_{ij}=
\begin{cases}
-e_{4(j+m+s)+1}\otimes w_{4j\ra 4j+1},\quad i=j+2s;\\
0,\quad\text{otherwise.}
\end{cases}$$

If $s-1\le j<2s-1$, then $b_{ij}=0$.

If $2s-1\le j<2s$, then $$b_{ij}=
\begin{cases}
w_{4(j+m+s)-1\ra 4(j+m+s)+1}\otimes e_{4j},\quad i=j-s;\\
-e_{4(j+m+s)+1}\otimes w_{4j\ra 4j+1},\quad i=j+2s;\\
0,\quad\text{otherwise.}
\end{cases}$$

If $2s\le j<3s-(2)_s$, then $$b_{ij}=
\begin{cases}
f_2(s,1)\kappa^\ell(\a_{3(j+m+1)})e_{4(j+m+1)}\otimes w_{4j+1\ra 4(j+1)},\quad i=s+(j+1)_s;\\
0,\quad\text{otherwise.}
\end{cases}$$

If $3s-(2)_s\le j<4s-(2)_s$, then $b_{ij}=0$.

If $4s-(2)_s\le j<4s$, then $$b_{ij}=
\begin{cases}
f_2(j,4s-2)\kappa^\ell(\a_{3(j+m+1)})e_{4(j+m+1)}\otimes w_{4j+1\ra 4(j+1)},\quad i=s+(j+1)_s;\\
0,\quad\text{otherwise.}
\end{cases}$$

If $4s\le j<5s-1$, then $$b_{ij}=
\begin{cases}
\kappa^\ell(\a_{3(j+m)})e_{4(j+m)+3}\otimes w_{4j+2\ra 4(j+1)},\quad i=(j+1)_s;\\
-\kappa^\ell(\a_{3(j+m)})w_{4(j+m+s)+2\ra 4(j+m)+3}\otimes e_{4j+2},\quad i=j+s;\\
0,\quad\text{otherwise.}
\end{cases}$$

If $5s-1\le j<6s-1$, then $b_{ij}=0$.

If $6s-1\le j<6s$, then $$b_{ij}=
\begin{cases}
\kappa^\ell(\a_{3(j+m)})e_{4(j+m)+3}\otimes w_{4j+2\ra 4(j+1)},\quad i=(j+1)_s;\\
-\kappa^\ell(\a_{3(j+m)})w_{4(j+m+s)+2\ra 4(j+m)+3}\otimes e_{4j+2},\quad i=j+2s;\\
0,\quad\text{otherwise.}
\end{cases}$$

If $6s\le j<7s-1$, then $$b_{ij}=
\begin{cases}
-w_{4(j+m+s+1)+1\ra 4(j+m+s+1)+2}\otimes w_{4j+3\ra 4(j+1)+1},\quad i=2s+(j+1)_s;\\
e_{4(j+m+s+1)+2}\otimes w_{4j+3\ra 4(j+1)+2},\quad i=5s+(j+1)_s;\\
0,\quad\text{otherwise.}
\end{cases}$$

If $7s-1\le j<8s-1$, then $b_{ij}=0$.

If $8s-1\le j<8s$, then $$b_{ij}=
\begin{cases}
w_{4(j+m+s+1)+1\ra 4(j+m+s+1)+2}\otimes w_{4j+3\ra 4(j+1)+1},\quad i=2s+(j+1)_s;\\
-e_{4(j+m+s+1)+2}\otimes w_{4j+3\ra 4(j+1)+2},\quad i=5s+(j+1)_s;\\
0,\quad\text{otherwise.}
\end{cases}$$

$(6)$ If $r_0=5$ and $s=1$, then $\Omega^{5}(Y_t^{(4)})$ is described with
$(9s\times 8s)$-matrix with the following nonzero elements{\rm:}
$$b_{j+3s,0}=-\kappa^\ell(\a_{3(j+m+1)})e_{4(j+m)}\otimes w_{4j\ra 4(j+1)+1};$$
$$b_{j-s,s}=-e_{4(j+m+1)+1}\otimes w_{4(j+1)+1\ra 4j};$$
$$b_{j-s,2s}=w_{4(j+m+1)+1\ra 4(j+m+1)+2}\otimes w_{4j+1\ra 4j};$$
$$b_{j-s,6s}=w_{4(j+m)+3\ra 4(j+m+1)+2}\otimes e_{4(j+1)+2};$$
$$b_{j,6s}=-e_{4(j+m+1)+2}\otimes w_{4(j+1)+2\ra 4j+3};$$
$$b_{j-7s,7s}=\kappa_1w_{4(j+m+1)+1\ra 4(j+m)+3}\otimes w_{4j+3\ra 4j};$$
$$b_{j-6s,7s}=\kappa_1w_{4(j+m)+1\ra 4(j+m)+3}\otimes w_{4j+3\ra 4j};$$
$$b_{j-3s,7s}=-\kappa_1e_{4(j+m)+3}\otimes w_{4j+3\ra 4(j+1)+2};$$
$$b_{j-2s,7s}=-\kappa_1e_{4(j+m)+3}\otimes w_{4j+3\ra 4j+2},$$
where $\kappa_1=-\kappa^\ell(\a_{3(7s+m+1)})$.

$(7)$ If $r_0=5$ and $s>1$, then $\Omega^{5}(Y_t^{(4)})$ is described with
$(9s\times 8s)$-matrix with the following elements $b_{ij}${\rm:}

If $0\le j<s$, then $$b_{ij}=
\begin{cases}
\kappa_1e_{4(j+m+1)}\otimes w_{4j\ra 4j+1},\quad i=j+2s;\\
\kappa_1w_{4(j+m)+1\ra 4(j+m+1)}\otimes e_{4j},\quad i=j,\text{ }j=s-1;\\
\kappa_1w_{4(j+m+s)+1\ra 4(j+m+1)}\otimes e_{4j},\quad i=j+s,\text{ }j=s-1;\\
-\kappa_1w_{4(j+m)+3\ra 4(j+m+1)}\otimes w_{4j\ra 4j+2},\quad i=j+4s,\text{ }j=s-1;\\
-\kappa_1w_{4(j+m)+3\ra 4(j+m+1)}\otimes w_{4j\ra 4(j+s)+2},\quad i=j+5s,\text{ }j=s-1;\\
0,\quad\text{otherwise,}
\end{cases}$$
where $\kappa_1=\kappa^\ell(\a_{3(j+m+1)})$.

If $s\le j<2s-2$, then $$b_{ij}=
\begin{cases}
-e_{4(j+m+1)+1}\otimes w_{4(j+s)+1\ra 4(j+1)},\quad i=s+(j+1)_s;\\
0,\quad\text{otherwise.}
\end{cases}$$

If $2s-2\le j<2s$, then $b_{ij}=0$.

If $2s\le j<3s-2$, then $$b_{ij}=
\begin{cases}
w_{4(j+m+1)+1\ra 4(j+m+1)+2}\otimes w_{4j+1\ra 4(j+1)},\quad i=(j+1)_s;\\
0,\quad\text{otherwise.}
\end{cases}$$

If $3s-2\le j<4s-2$, then $b_{ij}=0$.

If $4s-2\le j<4s-1$, then $$b_{ij}=
\begin{cases}
-e_{4(j+m+s+1)+1}\otimes w_{4j+1\ra 4(j+1)},\quad i=(j+1)_s;\\
0,\quad\text{otherwise.}
\end{cases}$$

If $4s-1\le j<4s$, then $$b_{ij}=
\begin{cases}
-e_{4(j+m+s+1)+1}\otimes w_{4j+1\ra 4(j+1)},\quad i=s+(j+1)_s;\\
0,\quad\text{otherwise.}
\end{cases}$$

If $4s\le j<5s-1$, then $b_{ij}=0$.

If $5s-1\le j<5s$, then $$b_{ij}=
\begin{cases}
w_{4(j+m+s+1)+1\ra 4(j+m+s+1)+2}\otimes w_{4(j+s)+1\ra 4(j+1)},\quad i=(j+1)_s;\\
0,\quad\text{otherwise.}
\end{cases}$$

If $5s\le j<5s-2$, then $b_{ij}=0$.

If $5s-2\le j<5s-1$, then $$b_{ij}=
\begin{cases}
w_{4(j+m+s+1)+1\ra 4(j+m+s+1)+2}\otimes w_{4(j+s)+1\ra 4(j+1)},\quad i=s+(j+1)_s;\\
0,\quad\text{otherwise.}
\end{cases}$$

If $5s-1\le j<5s$, then $b_{ij}=0$.

If $5s\le j<6s-1$, then $$b_{ij}=
\begin{cases}
w_{4(j+m)+3\ra 4(j+m+1)+2}\otimes e_{4(j+s)+2},\quad i=j-s;\\
e_{4(j+m+1)+2}\otimes w_{4(j+s)+2\ra 4j+3},\quad i=j+2s;\\
0,\quad\text{otherwise.}
\end{cases}$$

If $6s-1\le j<7s-1$, then $b_{ij}=0$.

If $7s-1\le j<7s$, then $$b_{ij}=
\begin{cases}
w_{4(j+m)+3\ra 4(j+m+1)+2}\otimes e_{4(j+s)+2},\quad i=j-s;\\
-e_{4(j+m+1)+2}\otimes w_{4(j+s)+2\ra 4j+3},\quad i=j;\\
0,\quad\text{otherwise.}
\end{cases}$$

If $7s\le j<8s-2$, then $b_{ij}=0$.

If $8s-2\le j<8s-1$, then $$b_{ij}=
\begin{cases}
\kappa_1w_{4(j+m+s+1)+1\ra 4(j+m+1)+3}\otimes w_{4j+3\ra 4(j+1)},\quad i=(j+1)_s;\\
\kappa_1w_{4(j+m+1)+1\ra 4(j+m+1)+3}\otimes w_{4j+3\ra 4(j+1)},\quad i=s+(j+1)_s;\\
-\kappa_1e_{4(j+m+1)+3}\otimes w_{4j+3\ra 4(j+s+1)+2},\quad i=4s+(j+1)_s;\\
-\kappa_1e_{4(j+m+1)+3}\otimes w_{4j+3\ra 4(j+1)+2},\quad i=5s+(j+1)_s;\\
0,\quad\text{otherwise,}
\end{cases}$$
where $\kappa_1=\kappa^\ell(\a_{3(j+m+1)})$.

If $8s-1\le j<8s$, then $b_{ij}=0$.

If $8s\le j<9s-1$, then $$b_{ij}=
\begin{cases}
-\kappa^\ell(\a_{3(j+m+2)})e_{4(j+m+2)}\otimes w_{4j+3\ra 4(j+1)+1},\quad i=2s+(j+1)_s;\\
0,\quad\text{otherwise.}
\end{cases}$$

If $9s-1\le j<9s$, then $b_{ij}=0$.

$(8)$ If $r_0=6$ and $s=1$, then $\Omega^{6}(Y_t^{(4)})$ is described with
$(8s\times 9s)$-matrix with the following nonzero elements{\rm:}
$$b_{j+3s,0}=w_{4(j+m+1)+1\ra 4(j+m+1)+2}\otimes w_{4j\ra 4(j+1)+1};$$
$$b_{j+3s,s}=-e_{4(j+m+1)+2}\otimes w_{4j\ra 4j+1};$$
$$b_{j,3s}=\kappa^\ell(\a_{3(j+m)})w_{4(j+m)+1\ra 4(j+m)+3}\otimes e_{4j+1};$$
$$b_{j+s,3s}=-\kappa^\ell(\a_{3(j+m)})w_{4(j+m+1)+2\ra 4(j+m)+3}\otimes e_{4j+1};$$
$$b_{j+s,5s}=-\kappa^\ell(\a_{3(j+m)})w_{4(j+m)+2\ra 4(j+m)}\otimes e_{4j+2};$$
$$b_{j+2s,5s}=\kappa^\ell(\a_{3(j+m)})w_{4(j+m)+3\ra 4(j+m)}\otimes w_{4j+2\ra 4j+3};$$
$$b_{j-6s,6s}=-w_{4(j+m)\ra 4(j+m)+1}\otimes w_{4j+3\ra 4j};$$
$$b_{j-5s,6s}=e_{4(j+m)+1}\otimes w_{4j+3\ra 4j+1}.$$

$(9)$ If $r_0=6$ and $s>1$, then $\Omega^{6}(Y_t^{(4)})$ is described with
$(8s\times 9s)$-matrix with the following elements $b_{ij}${\rm:}

If $0\le j<s$, then $$b_{ij}=
\begin{cases}
-w_{4(j+m)\ra 4(j+m+s)+2}\otimes e_{4j},\quad i=j;\\
w_{4(j+m+s)+1\ra 4(j+m+s)+2}\otimes w_{4j\ra 4(j+s)+1},\quad i=j+3s;\\
-e_{4(j+m+s)+2}\otimes w_{4j\ra 4(j+s)+2},\quad i=j+6s,\text{ }j<s-1;\\
0,\quad\text{otherwise.}
\end{cases}$$

If $s\le j<2s$, then $$b_{ij}=
\begin{cases}
w_{4(j+m+s)+1\ra 4(j+m+s)+2}\otimes w_{4j\ra 4(j+s)+1},\quad i=j;\\
0,\quad\text{otherwise.}
\end{cases}$$

If $2s\le j<2s-1$, then $b_{ij}=0$.

If $2s-1\le j<2s$, then $$b_{ij}=
\begin{cases}
-e_{4(j+m+s)+2}\otimes w_{4j\ra 4(j+s)+2},\quad i=j+4s;\\
0,\quad\text{otherwise.}
\end{cases}$$

If $2s\le j<3s-1$, then $$b_{ij}=
\begin{cases}
\kappa^\ell(\a_{3(j+m)})w_{4(j+m)+2\ra 4(j+m)+3}\otimes w_{4j+1\ra 4j+2},\quad i=j+3s;\\
-\kappa^\ell(\a_{3(j+m)})e_{4(j+m)+3}\otimes w_{4j+1\ra 4j+3},\quad i=j+5s;\\
0,\quad\text{otherwise.}
\end{cases}$$

If $3s-1\le j<4s-1$, then $b_{ij}=0$.

If $4s-1\le j<4s$, then $$b_{ij}=
\begin{cases}
\kappa^\ell(\a_{3(j+m)})w_{4(j+m)+2\ra 4(j+m)+3}\otimes w_{4j+1\ra 4j+2},\quad i=j+3s;\\
-\kappa^\ell(\a_{3(j+m)})e_{4(j+m)+3}\otimes w_{4j+1\ra 4j+3},\quad i=j+4s;\\
0,\quad\text{otherwise.}
\end{cases}$$

If $4s\le j<5s-1$, then $$b_{ij}=
\begin{cases}
\kappa_1e_{4(j+m+1)}\otimes w_{4j+2\ra 4(j+1)},\quad i=(j+1)_s;\\
\kappa_1w_{4(j+m)+2\ra 4(j+m+1)}\otimes e_{4j+2},\quad i=j+s;\\
-\kappa_1w_{4(j+m)+3\ra 4(j+m+1)}\otimes w_{4j+2\ra 4j+3},\quad i=j+3s;\\
0,\quad\text{otherwise,}
\end{cases}$$
where $\kappa_1=\kappa^\ell(\a_{3(j+m+1)})$.

If $5s-1\le j<6s-1$, then $b_{ij}=0$.

If $6s-1\le j<6s$, then $$b_{ij}=
\begin{cases}
-\kappa^\ell(\a_{3(j+m)})e_{4(j+m+1)}\otimes w_{4j+2\ra 4(j+1)},\quad i=(j+1)_s;\\
-\kappa^\ell(\a_{3(j+m)})w_{4(j+m)+2\ra 4(j+m+1)}\otimes e_{4j+2},\quad i=j+s;\\
\kappa^\ell(\a_{3(j+m)})w_{4(j+m)+3\ra 4(j+m+1)}\otimes w_{4j+2\ra 4j+3},\quad i=j+2s;\\
0,\quad\text{otherwise.}
\end{cases}$$

If $6s\le j<7s-1$, then $b_{ij}=0$.

If $7s-1\le j<7s$, then $$b_{ij}=
\begin{cases}
e_{4(j+m+s+1)+1}\otimes w_{4j+3\ra 4(j+s+1)+1},\quad i=s+(j+1)_s;\\
0,\quad\text{otherwise.}
\end{cases}$$

If $7s\le j<8s-1$, then $$b_{ij}=
\begin{cases}
e_{4(j+m+s+1)+1}\otimes w_{4j+3\ra 4(j+s+1)+1},\quad i=s+(j+1)_s;\\
0,\quad\text{otherwise.}
\end{cases}$$

If $8s-1\le j<8s$, then $b_{ij}=0$.

$(10)$ If $r_0=7$ and $s=1$, then $\Omega^{7}(Y_t^{(4)})$ is described with
$(6s\times 8s)$-matrix with the following nonzero elements{\rm:}
$$b_{j-s,s}=-e_{4(j+m+1)+2}\otimes w_{4(j+1)+1\ra 4j};$$
$$b_{j+s,2s}=w_{4(j+m)+3\ra 4(j+m+1)+2}\otimes e_{4(j+1)+1};$$
$$b_{j+s,4s}=-w_{4(j+m)\ra 4(j+m)+1}\otimes e_{4(j+1)+2};$$
$$b_{j+3s,4s}=e_{4(j+m)+1}\otimes w_{4(j+1)+2\ra 4j+3};$$
$$b_{j-5s,5s}=\kappa_1w_{4(j+m+1)+2\ra 4(j+m)}\otimes w_{4j+3\ra 4j};$$
$$b_{j-s,5s}=-\kappa_1e_{4(j+m)}\otimes w_{4j+3\ra 4(j+1)+2};$$
$$b_{j,5s}=-\kappa_1e_{4(j+m)}\otimes w_{4j+3\ra 4j+2},$$
where $\kappa_1=-\kappa^\ell(\a_{3(5s+m)})$.

$(11)$ If $r_0=7$ and $s>1$, then $\Omega^{7}(Y_t^{(4)})$ is described with
$(6s\times 8s)$-matrix with the following elements $b_{ij}${\rm:}

If $0\le j<s$, then $$b_{ij}=
\begin{cases}
\kappa^\ell(\a_{3(j+m)})w_{4(j+m+s)+2\ra 4(j+m)+3}\otimes e_{4j},\quad i=j+s;\\
-\kappa^\ell(\a_{3(j+m)})e_{4(j+m)+3}\otimes w_{4j\ra 4j+1},\quad i=j+2s;\\
0,\quad\text{otherwise.}
\end{cases}$$

If $s\le j<3s$, then $b_{ij}=0$.

If $3s\le j<4s-1$, then $$b_{ij}=
\begin{cases}
-w_{4(j+m+1)\ra 4(j+m+s+1)+1}\otimes e_{4(j+s)+2},\quad i=j+s;\\
e_{4(j+m+s+1)+1}\otimes w_{4(j+s)+2\ra 4j+3},\quad i=j+3s;\\
0,\quad\text{otherwise.}
\end{cases}$$

If $4s-1\le j<5s-1$, then $b_{ij}=0$.

If $5s-1\le j<5s$, then $$b_{ij}=
\begin{cases}
-w_{4(j+m+1)\ra 4(j+m+s+1)+1}\otimes e_{4(j+s)+2},\quad i=j+s;\\
e_{4(j+m+s+1)+1}\otimes w_{4(j+s)+2\ra 4j+3},\quad i=j+3s;\\
0,\quad\text{otherwise.}
\end{cases}$$

If $5s\le j<6s-2$, then $b_{ij}=0$.

If $6s-2\le j<6s-1$, then $$b_{ij}=
\begin{cases}
\kappa_1e_{4(j+m+2)}\otimes w_{4j+3\ra 4(j+s+1)+2},\quad i=4s+(j+1)_s;\\
\kappa_1e_{4(j+m+2)}\otimes w_{4j+3\ra 4(j+1)+2},\quad i=5s+(j+1)_s;\\
0,\quad\text{otherwise,}
\end{cases}$$
where $\kappa_1=\kappa^\ell(\a_{3(j+m+2)})$.

If $6s-1\le j<6s$, then $b_{ij}=0$.

$(12)$ If $r_0=8$ and $s=1$, then $\Omega^{8}(Y_t^{(4)})$ is described with
$(7s\times 6s)$-matrix with the following nonzero elements{\rm:}
$$b_{j+2s,0}=\kappa^\ell(\a_{3(j+m)})w_{4(j+m+1)+2\ra 4(j+m)+3}\otimes w_{4j\ra 4(j+1)+1};$$
$$b_{j,2s}=\kappa^\ell(\a_{3(j+m)})w_{4(j+m+1)+2\ra 4(j+m)}\otimes e_{4(j+1)+1};$$
$$b_{j-4s,4s}=-w_{4(j+m)+3\ra 4(j+m)+1}\otimes w_{4(j+1)+2\ra 4j};$$
$$b_{j+s,4s}=-w_{4(j+m)\ra 4(j+m)+1}\otimes w_{4(j+1)+2\ra 4j+3};$$
$$b_{j-5s,5s}=-w_{4(j+m)+3\ra 4(j+m+1)+2}\otimes w_{4j+3\ra 4j};$$
$$b_{j-4s,5s}=e_{4(j+m+1)+2}\otimes w_{4j+3\ra 4(j+1)+1};$$
$$b_{j-s,5s}=w_{4(j+m+1)+1\ra 4(j+m+1)+2}\otimes w_{4j+3\ra 4j+2};$$
$$b_{j-3s,6s}=-w_{4(j+m+1)+1\ra 4(j+m+1)+2}\otimes w_{4j+3\ra 4j+2}.$$

$(13)$ If $r_0=8$ and $s>1$, then $\Omega^{8}(Y_t^{(4)})$ is described with
$(7s\times 6s)$-matrix with the following elements $b_{ij}${\rm:}

If $0\le j<s$, then $$b_{ij}=
\begin{cases}
\kappa^\ell(\a_{3(j+m)})e_{4(j+m)+3}\otimes w_{4j\ra 4(j+1)},\quad i=(j+1)_s;\\
\kappa^\ell(\a_{3(j+m)})w_{4(j+m)+2\ra 4(j+m)+3}\otimes w_{4j\ra 4j+1},\quad i=j+s,\text{ }j<s-1;\\
\kappa^\ell(\a_{3(j+m)})w_{4(j+m-1)+3\ra 4(j+m)+3}\otimes e_{4j},\quad i=j,\text{ }j=s-1;\\
\kappa^\ell(\a_{3(j+m)})w_{4(j+m+s)+2\ra 4(j+m)+3}\otimes w_{4j\ra 4(j+s)+1},\quad i=j+2s,\text{ }j=s-1;\\
0,\quad\text{otherwise.}
\end{cases}$$

If $s\le j<2s-1$, then $$b_{ij}=
\begin{cases}
\kappa_1w_{4(j+m)+3\ra 4(j+m+1)}\otimes w_{4(j+s)+1\ra 4(j+1)},\quad i=(j+1)_s;\\
\kappa_1w_{4(j+m+s)+2\ra 4(j+m+1)}\otimes e_{4(j+s)+1},\quad i=j;\\
0,\quad\text{otherwise,}
\end{cases}$$
where $\kappa_1=-\kappa^\ell(\a_{3(j+m+1)})$.

If $2s-1\le j<3s-1$, then $b_{ij}=0$.

If $3s-1\le j<3s$, then $$b_{ij}=
\begin{cases}
-\kappa_1w_{4(j+m)+1\ra 4(j+m+1)}\otimes w_{4(j+s)+1\ra 4(j+s)+2},\quad i=j+2s;\\
\kappa_1e_{4(j+m+1)}\otimes w_{4(j+s)+1\ra 4j+3},\quad i=j+3s;\\
0,\quad\text{otherwise,}
\end{cases}$$
where $\kappa_1=\kappa^\ell(\a_{3(j+m+1)})$.

If $3s\le j<4s-1$, then $$b_{ij}=
\begin{cases}
-w_{4(j+m+1)\ra 4(j+m+s+1)+1}\otimes w_{4(j+s)+2\ra 4j+3},\quad i=j+2s;\\
0,\quad\text{otherwise.}
\end{cases}$$

If $4s-1\le j<5s-1$, then $b_{ij}=0$.

If $5s-1\le j<5s$, then $$b_{ij}=
\begin{cases}
-w_{4(j+m+1)\ra 4(j+m+s+1)+1}\otimes w_{4(j+s)+2\ra 4j+3},\quad i=j+s;\\
0,\quad\text{otherwise.}
\end{cases}$$

If $5s\le j<6s-1$, then $b_{ij}=0$.

If $6s-1\le j<6s$, then $$b_{ij}=
\begin{cases}
e_{4(j+m+1)+2}\otimes w_{4j+3\ra 4(j+1)+1},\quad i=s+(j+1)_s;\\
0,\quad\text{otherwise.}
\end{cases}$$

If $6s\le j<6s-2$, then $b_{ij}=0$.

If $6s-2\le j<6s-1$, then $$b_{ij}=
\begin{cases}
-w_{4(j+m+1)+1\ra 4(j+m+1)+2}\otimes w_{4j+3\ra 4(j+s+1)+2},\quad i=3s+(j+1)_s;\\
0,\quad\text{otherwise.}
\end{cases}$$

If $6s-1\le j<6s$, then $b_{ij}=0$.

If $6s\le j<7s-1$, then $$b_{ij}=
\begin{cases}
e_{4(j+m+1)+2}\otimes w_{4j+3\ra 4(j+1)+1},\quad i=s+(j+1)_s;\\
0,\quad\text{otherwise.}
\end{cases}$$

If $7s-1\le j<7s-2$, then $b_{ij}=0$.

If $7s-2\le j<7s-1$, then $$b_{ij}=
\begin{cases}
w_{4(j+m+1)+1\ra 4(j+m+1)+2}\otimes w_{4j+3\ra 4(j+s+1)+2},\quad i=4s+(j+1)_s;\\
0,\quad\text{otherwise.}
\end{cases}$$

If $7s-1\le j<7s$, then $b_{ij}=0$.

$(14)$ If $r_0=9$ and $s=1$, then $\Omega^{9}(Y_t^{(4)})$ is described with
$(6s\times 7s)$-matrix with the following nonzero elements{\rm:}
$$b_{j,2s}=-w_{4(j+m)\ra 4(j+m)+1}\otimes e_{4(j+1)+1};$$
$$b_{j+2s,4s}=e_{4(j+m)+2}\otimes w_{4(j+1)+2\ra 4j+3};$$
$$b_{j,5s}=-\kappa^\ell(\a_{3(j+m+1)})w_{4(j+m)+2\ra 4(j+m)+3}\otimes e_{4j+3};$$
$$b_{j+s,5s}=\kappa^\ell(\a_{3(j+m+1)})w_{4(j+m+1)+2\ra 4(j+m)+3}\otimes e_{4j+3}.$$

$(15)$ If $r_0=9$ and $s>1$, then $\Omega^{9}(Y_t^{(4)})$ is described with
$(6s\times 7s)$-matrix with the following elements $b_{ij}${\rm:}

If $s-1\le j<s$, then $$b_{ij}=
\begin{cases}
-\kappa_1w_{4(j+m)+3\ra 4(j+m+1)}\otimes e_{4j},\quad i=j;\\
-\kappa_1e_{4(j+m+1)}\otimes w_{4j\ra 4j+1},\quad i=j+s;\\
0,\quad\text{otherwise,}
\end{cases}$$
where $\kappa_1=\kappa^\ell(\a_{3(j+m+1)})$.

If $s\le j<2s-1$, then $$b_{ij}=
\begin{cases}
-w_{4(j+m+1)\ra 4(j+m+s+1)+1}\otimes e_{4(j+s)+1},\quad i=j;\\
0,\quad\text{otherwise.}
\end{cases}$$

If $2s-1\le j<3s$, then $b_{ij}=0$.

If $3s\le j<4s-1$, then $$b_{ij}=
\begin{cases}
e_{4(j+m+s+1)+2}\otimes w_{4(j+s)+2\ra 4j+3},\quad i=j+2s;\\
0,\quad\text{otherwise.}
\end{cases}$$

If $4s-1\le j<5s-1$, then $b_{ij}=0$.

If $5s-1\le j<5s$, then $$b_{ij}=
\begin{cases}
e_{4(j+m+s+1)+2}\otimes w_{4(j+s)+2\ra 4j+3},\quad i=j+2s;\\
0,\quad\text{otherwise.}
\end{cases}$$

If $5s\le j<6s-2$, then $b_{ij}=0$.

If $6s-2\le j<6s-1$, then $$b_{ij}=
\begin{cases}
\kappa^\ell(\a_{3(j+m)})e_{4(j+m+1)+3}\otimes w_{4j+3\ra 4(j+1)},\quad i=(j+1)_s;\\
0,\quad\text{otherwise.}
\end{cases}$$

If $6s-1\le j<6s$, then $b_{ij}=0$.

$(16)$ If $r_0=10$ and $s=1$, then $\Omega^{10}(Y_t^{(4)})$ is described with
$(6s\times 6s)$-matrix with the following nonzero elements{\rm:}
$$b_{j+2s,0}=\kappa_1w_{4(j+m)+1\ra 4(j+m)}\otimes w_{4j\ra 4(j+1)+1};$$
$$b_{j+4s,0}=-\kappa_1w_{4(j+m)+2\ra 4(j+m)}\otimes w_{4j\ra 4(j+1)+2};$$
$$b_{j-s,s}=-w_{4(j+m)\ra 4(j+m)+1}\otimes w_{4(j+1)+1\ra 4j};$$
$$b_{j+4s,s}=-w_{4(j+m)+3\ra 4(j+m)+1}\otimes w_{4(j+1)+1\ra 4j+3};$$
$$b_{j-3s,3s}=-w_{4(j+m)\ra 4(j+m)+2}\otimes w_{4(j+1)+2\ra 4j};$$
$$b_{j+2s,3s}=-w_{4(j+m)+3\ra 4(j+m)+2}\otimes w_{4(j+1)+2\ra 4j+3};$$
$$b_{j-3s,5s}=-\kappa_2w_{4(j+m+1)+1\ra 4(j+m)+3}\otimes w_{4j+3\ra 4j+1};$$
$$b_{j-s,5s}=\kappa_2w_{4(j+m+1)+2\ra 4(j+m)+3}\otimes w_{4j+3\ra 4j+2},$$
where $\kappa_1=\kappa^\ell(\a_{3(m+1)})$, $\kappa_2=\kappa^\ell(\a_{3(5s+m+1)})$.

$(17)$ If $r_0=10$ and $s>1$, then $\Omega^{10}(Y_t^{(4)})$ is described with
$(6s\times 6s)$-matrix with the following elements $b_{ij}${\rm:}

If $0\le j<s$, then $$b_{ij}=
\begin{cases}
\kappa_1w_{4(j+m)\ra 4(j+m+1)}\otimes e_{4j},\quad i=j;\\
\kappa_1w_{4(j+m)+1\ra 4(j+m+1)}\otimes w_{4j\ra 4(j+s)+1},\quad i=j+2s;\\
-\kappa_1w_{4(j+m)+2\ra 4(j+m+1)}\otimes w_{4j\ra 4(j+s)+2},\quad i=j+4s;\\
\kappa_1e_{4(j+m+1)}\otimes w_{4j\ra 4(j+1)},\quad i=(j+1)_s,\text{ }j=s-1;\\
\kappa_1w_{4(j+m)+3\ra 4(j+m+1)}\otimes w_{4j\ra 4j+3},\quad i=j+5s,\text{ }j=s-1;\\
0,\quad\text{otherwise,}
\end{cases}$$
where $\kappa_1=\kappa^\ell(\a_{3(j+m+1)})$.

If $s\le j<2s-1$, then $$b_{ij}=
\begin{cases}
w_{4(j+m+1)\ra 4(j+m+s+1)+1}\otimes w_{4(j+s)+1\ra 4(j+1)},\quad i=(j+1)_s;\\
-w_{4(j+m)+3\ra 4(j+m+s+1)+1}\otimes w_{4(j+s)+1\ra 4j+3},\quad i=j+4s;\\
0,\quad\text{otherwise.}
\end{cases}$$

If $2s-1\le j<3s-1$, then $b_{ij}=0$.

If $3s-1\le j<3s$, then $$b_{ij}=
\begin{cases}
w_{4(j+m+1)\ra 4(j+m+s+1)+1}\otimes w_{4(j+s)+1\ra 4(j+1)},\quad i=(j+1)_s;\\
w_{4(j+m)+3\ra 4(j+m+s+1)+1}\otimes w_{4(j+s)+1\ra 4j+3},\quad i=j+3s;\\
0,\quad\text{otherwise.}
\end{cases}$$

If $3s\le j<4s-1$, then $$b_{ij}=
\begin{cases}
w_{4(j+m+1)\ra 4(j+m+s+1)+2}\otimes w_{4(j+s)+2\ra 4(j+1)},\quad i=(j+1)_s;\\
-w_{4(j+m)+3\ra 4(j+m+s+1)+2}\otimes w_{4(j+s)+2\ra 4j+3},\quad i=j+2s;\\
0,\quad\text{otherwise.}
\end{cases}$$

If $4s-1\le j<5s-1$, then $b_{ij}=0$.

If $5s-1\le j<5s$, then $$b_{ij}=
\begin{cases}
w_{4(j+m+1)\ra 4(j+m+s+1)+2}\otimes w_{4(j+s)+2\ra 4(j+1)},\quad i=(j+1)_s;\\
w_{4(j+m)+3\ra 4(j+m+s+1)+2}\otimes w_{4(j+s)+2\ra 4j+3},\quad i=j+s;\\
0,\quad\text{otherwise.}
\end{cases}$$

If $5s\le j<6s$, then $$b_{ij}=
\begin{cases}
-\kappa_1w_{4(j+m+1)\ra 4(j+m+1)+3}\otimes w_{4j+3\ra 4(j+1)},\quad i=(j+1)_s;\\
-\kappa_1w_{4(j+m+s+1+j_0)+1\ra 4(j+m+1)+3}\otimes w_{4j+3\ra 4(j+1+j_0)+1},\quad i=2s+(j+1)_s;\\
\kappa_1w_{4(j+m+s+1+j_0)+2\ra 4(j+m+1)+3}\otimes w_{4j+3\ra 4(j+1+j_0)+2},\quad i=4s+(j+1)_s;\\
-f_2(j,6s-1)\kappa_1w_{4(j+m)+3\ra 4(j+m+1)+3}\otimes e_{4j+3},\quad i=j;\\
-\kappa_1e_{4(j+m+1)+3}\otimes w_{4j+3\ra 4(j+1)+3},\quad i=5s+(j+1)_s,\text{ }j=6s-2;\\
0,\quad\text{otherwise,}
\end{cases}$$
where $\kappa_1=\kappa^\ell(\a_{3(j+m+1)})$, $j_0=sf(j,6s-1)$

\medskip
$({\rm II})$ Represent an arbitrary $t_0\in\N$ in the form
$t_0=11\ell_0+r_0$, where $0\le r_0\le 10.$ Then
$\Omega^{t_0}(Y_t^{(4)})$ is a $\Omega^{r_0}(Y_t^{(4)})$, whose left
components twisted by $\sigma^{\ell_0}$.
\end{pr}

\begin{pr}[Translates for the case 5]
$({\rm I})$ Let $r_0\in\N$, $r_0<11$. Denote by $\kappa_0=\kappa^\ell(\a_{3((3)_s)})$. Then $r_0$-translates of the
elements $Y^{(5)}_t$ are described by the following way.

$(1)$ If $r_0=0$, then $\Omega^{0}(Y_t^{(5)})$ is described with
$(6s\times 6s)$-matrix with one nonzero element that is of the following form{\rm:}
$$b_{0,0}=\kappa^\ell(\a_0) w_{4j\ra 4j+3}\otimes e_{4j}.$$

$(2)$ If $r_0=1$, then $\Omega^{1}(Y_t^{(5)})$ is described with
$(8s\times 7s)$-matrix with the following nonzero elements{\rm:}
$$b_{(j+s+1)_{2s},2s+(2)_s}=f_2((j)_s,s-1)\kappa_0w_{4(j+m+s+1)+1\ra 4(j+m+1)+3}\otimes w_{4j+1\ra 4(j+1)};$$
$$b_{(j+s+1)_{2s},3s+(2)_s}=-f_2((j)_s,s-1)\kappa_0w_{4(j+m+s+1)+1\ra 4(j+m+1)+3}\otimes w_{4j+1\ra 4(j+1)};$$
$$b_{6s+(j)_s,4s+(2)_s+sf(s,1)+sf(s,3)}=-\kappa^\ell(\g_{j+1+m})\kappa_0w_{4(j+m+1)\ra 4(j+m+2)}\otimes w_{4j+2\ra 4j+3}.$$

$(3)$ If $r_0=2$, then $\Omega^{2}(Y_t^{(5)})$ is described with
$(9s\times 6s)$-matrix with the following two nonzero elements{\rm:}
$$b_{(j+1)_s,(2)_s}=\kappa_0e_{4(j+m)+3}\otimes w_{4j\ra 4(j+1)};$$
$$b_{j+s,s+(2)_s}=f_2((j)_s,s-1)\kappa^\ell(\g_{j+m})\kappa_0w_{4(j+m)+2\ra 4(j+m+1)}\otimes w_{4j\ra 4j+1}.$$

$(4)$ If $r_0=3$, then $\Omega^{3}(Y_t^{(5)})$ is described with
$(8s\times 8s)$-matrix with the following two nonzero elements{\rm:}
$$b_{j+3s,(2)_s}=-\kappa^\ell(\a_{3(j+1+m)})\kappa^\ell(\g_{j+m})f_2((j)_s,s-1)\kappa_0w_{4(j+m)+3\ra 4(j+m+s+1)+1}\otimes w_{4j\ra 4(j+s)+1};$$
$$b_{s+(j+1)_s,3s-1+(2)_s}=f_2((j)_s,(s-2)_s)\kappa^\ell(\g_{j+1+m})\kappa_0w_{4(j+m+1)+2\ra 4(j+m+2)}\otimes w_{4j+1\ra 4(j+1)}.$$

$(5)$ If $r_0=4$, then $\Omega^{4}(Y_t^{(5)})$ is described with
$(9s\times 9s)$-matrix with the following nonzero elements{\rm:}
$$b_{j+3s,(1)_s}=\kappa^\ell(\g_{j+m})\kappa_0w_{4(j+m)+1\ra 4(j+m)}\otimes w_{4j\ra 4(j+1)+1},\text{ }s=1;$$
$$b_{j+7s,(1)_s}=-\kappa^\ell(\g_{j+m})\kappa_0w_{4(j+m)+2\ra 4(j+m)}\otimes w_{4j\ra 4(j+1)+2},\text{ }s=1;$$
$$b_{(j+1)_s,(1)_s}=-f_2((j)_s,s-2)\kappa^\ell(\g_{j+m})\kappa_0w_{4(j+m)+3\ra 4(j+m+1)}\otimes w_{4j\ra 4(j+1)},\text{ }s>1;$$
$$b_{s+(j+1)_s,(1)_s}=f_2((j)_s,s-2)\kappa^\ell(\g_{j+m})\kappa_0e_{4(j+m+1)}\otimes w_{4j\ra 4(j+1)},\text{ }s>1;$$
\begin{multline*}
b_{(j+1)_s,4s+1-3f(s,1)}=f_2(s,1)f_1((j)_s,s-2)\kappa^\ell(\g_{j+m})\kappa^\ell(\a_{3(j+1+m)})\kappa_0\\
w_{4(j+m)+3\ra 4(j+m+s+1+sf(s,1))+2}\otimes w_{4(j+s+sf(s,1))+1\ra 4(j+1)};\end{multline*}
\begin{multline*}
b_{3s+(j+1)_s,7s+(1)_s}=-f_2((j)_s,(s-2)_s)\kappa^\ell(\g_{j+m})\kappa_0\\
w_{4(j+m+s+1+sf(s,1))+1\ra 4(j+m+1)+3}\otimes w_{4j+3\ra 4(j+1+sf(s,1))+1};\end{multline*}
\begin{multline*}
b_{7s+(j+1)_s,7s+(1)_s}=f_2((j)_s,(s-2)_s)\kappa^\ell(\g_{j+m})\kappa_0\\
w_{4(j+m+s+1+sf(s,1))+2\ra 4(j+m+1)+3}\otimes w_{4j+3\ra 4(j+1+sf(s,1))+2}.\end{multline*}

$(6)$ If $r_0=5$ and $s=1$, then $\Omega^{5}(Y_t^{(5)})$ is described with
$(8s\times 8s)$-matrix with the following nonzero elements{\rm:}
$$b_{j+s,s}=\kappa^\ell(\a_{3(j+m)})\kappa_0w_{4(j+m)\ra 4(j+m)+2}\otimes w_{4j\ra 4(j+1)+1};$$
$$b_{j+2s,s}=\kappa^\ell(\a_{3(j+m)})\kappa_0w_{4(j+m)\ra 4(j+m)+2}\otimes w_{4j\ra 4j+1};$$
$$b_{j-2s,2s}=\kappa_0w_{4(j+m)+1\ra 4(j+m)+3}\otimes w_{4j+1\ra 4j};$$
$$b_{j-s,2s}=\kappa_0w_{4(j+m+1)+1\ra 4(j+m)+3}\otimes w_{4j+1\ra 4j};$$
$$b_{j-3s,6s}=-2\kappa^\ell(\a_{3(j+m)})\kappa_0w_{4(j+m)\ra 4(j+m+1)+1}\otimes w_{4j+3\ra 4(j+1)+1}.$$

$(7)$ If $r_0=5$ and $s>1$, then $\Omega^{5}(Y_t^{(5)})$ is described with
$(8s\times 8s)$-matrix with the following nonzero elements{\rm:}
$$b_{s+(j+1)_s,1}=-\kappa^\ell(\a_{3(j+m)})f_1((j)_s,s-2)\kappa_0w_{4(j+m+s+1)+1\ra 4(j+m+s+1)+2}\otimes w_{4j\ra 4(j+1)};$$
$$b_{(j+1)_s,s+1}=\kappa^\ell(\a_{3(j+m)})f_1((j)_s,s-2)\kappa_0w_{4(j+m+s+1)+1\ra 4(j+m+s+1)+2}\otimes w_{4j\ra 4(j+1)};$$
$$b_{(j+1)_s,3s+1}=\kappa_1\kappa_0w_{4(j+m+s+1)+1\ra 4(j+m+1)+3}\otimes w_{4j+1\ra 4(j+1)};$$
$$b_{s+(j+1)_s,3s+1}=\kappa_1\kappa_0w_{4(j+m+1)+1\ra 4(j+m+1)+3}\otimes w_{4j+1\ra 4(j+1)};$$
$$b_{3s+(j+1)_s,5s+1}=\kappa_2\kappa_0e_{4(j+m+2)}\otimes w_{4j+2\ra 4(j+1)+1},$$
where $\kappa_1=f_2(1,s-2)\kappa^\ell(\g_{3s+1+m})$, $\kappa_2=\kappa^\ell(\g_{5s+1+m})\kappa^\ell(\g_{5s+2+m})f_2(1,s-2)$.

$(8)$ If $r_0=6$ and $s=1$, then $\Omega^{6}(Y_t^{(5)})$ is described with
$(6s\times 9s)$-matrix with the following nonzero elements{\rm:}
$$b_{j+2s,0}=-\kappa_0w_{4(j+m+1)+2\ra 4(j+m)+3}\otimes w_{4j\ra 4j+1};$$
$$b_{j+3s,0}=-\kappa_0w_{4(j+m+1)+1\ra 4(j+m)+3}\otimes w_{4j\ra 4(j+1)+1};$$
$$b_{j+5s,0}=-\kappa_0w_{4(j+m)+2\ra 4(j+m)+3}\otimes w_{4j\ra 4j+2};$$
$$b_{j+7s,0}=\kappa_0e_{4(j+m)+3}\otimes w_{4j\ra 4j+3};$$
$$b_{j-s,s}=-\kappa^\ell(\a_{3(j+m)})\kappa_0w_{4(j+m)\ra 4(j+m+1)+2}\otimes w_{4(j+1)+1\ra 4j};$$
$$b_{j-4s,5s}=\kappa_0w_{4(j+m+1)+1\ra 4(j+m)}\otimes w_{4j+3\ra 4(j+1)+1};$$
$$b_{j-2s,5s}=-\kappa_0w_{4(j+m)+1\ra 4(j+m)}\otimes w_{4j+3\ra 4j+1}.$$

$(9)$ If $r_0=6$ and $s>1$, then $\Omega^{6}(Y_t^{(5)})$ is described with
$(6s\times 9s)$-matrix with the following two nonzero elements{\rm:}
$$b_{(j+1)_s,2s+1}=f_1((j)_s,s-2)\kappa^\ell(\a_{3(j+2+m)})\kappa_0w_{4(j+m+1)\ra 4(j+m+1)+2}\otimes w_{4(j+s)+1\ra 4(j+1)};$$
$$b_{3s+(j+1)_s,4s+1}=-f_2((j)_s,s-2)\kappa^\ell(\a_{3(j+2+m)})\kappa_0e_{4(j+m+s+1)+1}\otimes w_{4(j+s)+2\ra 4(j+s+1)+1}.$$

$(10)$ If $r_0=7$ and $s=1$, then $\Omega^{7}(Y_t^{(5)})$ is described with
$(7s\times 8s)$-matrix with the following nonzero elements{\rm:}
$$b_{j+4s,0}=-\kappa_0w_{4(j+m)\ra 4(j+m)+3}\otimes w_{4j\ra 4j+2};$$
$$b_{j+5s,0}=-\kappa_0w_{4(j+m)\ra 4(j+m)+3}\otimes w_{4j\ra 4(j+1)+2};$$
$$b_{j,s}=\kappa_0w_{4(j+m)+2\ra 4(j+m)}\otimes w_{4(j+1)+1\ra 4j};$$
$$b_{j-2s,2s}=-\kappa_0w_{4(j+m)+2\ra 4(j+m)}\otimes w_{4(j+1)+1\ra 4j};$$
$$b_{j-s,2s}=\kappa_0w_{4(j+m+1)+2\ra 4(j+m)}\otimes w_{4(j+1)+1\ra 4j};$$
$$b_{j-2s,3s}=-\kappa^\ell(\a_{3(j+2+m)})\kappa_0w_{4(j+m)+2\ra 4(j+m+1)+1}\otimes w_{4(j+1)+2\ra 4j};$$
$$b_{j+2s,3s}=\kappa^\ell(\a_{3(j+2+m)})\kappa_0w_{4(j+m)\ra 4(j+m+1)+1}\otimes w_{4(j+1)+2\ra 4j+2};$$
$$b_{j-4s,4s}=\kappa^\ell(\a_{3(j+2+m)})\kappa_0w_{4(j+m)+2\ra 4(j+m+1)+1}\otimes w_{4(j+1)+2\ra 4j};$$
$$b_{j-3s,5s}=\kappa^\ell(\a_{3(j+2+m)})\kappa_0w_{4(j+m)+3\ra 4(j+m)+2}\otimes w_{4j+3\ra 4(j+1)+1}.$$

$(11)$ If $r_0=7$ and $s>1$, then $\Omega^{7}(Y_t^{(5)})$ is described with
$(7s\times 8s)$-matrix with the following nonzero elements{\rm:}
$$b_{(j+1)_s,1}=-\kappa^\ell(\g_{j+1+m})f_1((j)_s,s-2)\kappa_0w_{4(j+m+1)+2\ra 4(j+m+1)+3}\otimes w_{4j\ra 4(j+1)};$$
$$b_{s+(j+1)_s,2s+1}=f_2((j)_s,s-2)\kappa_0w_{4(j+m+s+1)+2\ra 4(j+m+2)}\otimes w_{4(j+s)+1\ra 4(j+1)};$$
$$b_{7s+(j)_s,4s+1}=-f_1((j)_s,s-2)\kappa^\ell(\a_{3(j+2+m)})\kappa_0w_{4(j+m+s+1)+1\ra 4(j+m+s+2)+1}\otimes w_{4(j+s)+2\ra 4j+3}.$$

$(12)$ If $r_0=8$ and $s=1$, then $\Omega^{8}(Y_t^{(5)})$ is described with
$(6s\times 6s)$-matrix with the following nonzero elements{\rm:}
$$b_{j+2s,0}=-\kappa_0w_{4(j+m+1)+2\ra 4(j+m)}\otimes w_{4j\ra 4(j+1)+1};$$
$$b_{j+3s,0}=-\kappa_0w_{4(j+m+1)+1\ra 4(j+m)}\otimes w_{4j\ra 4j+2};$$
$$b_{j-3s,3s}=-\kappa^\ell(\a_{3(j+1+m)})\kappa_0w_{4(j+m)+3\ra 4(j+m)+2}\otimes w_{4(j+1)+2\ra 4j};$$
$$b_{j-s,3s}=-\kappa^\ell(\a_{3(j+1+m)})\kappa_0e_{4(j+m)+2}\otimes w_{4(j+1)+2\ra 4j+1};$$
$$b_{j-4s,4s}=-\kappa^\ell(\a_{3(j+1+m)})\kappa_0w_{4(j+m)+3\ra 4(j+m)+2}\otimes w_{4(j+1)+2\ra 4j};$$
$$b_{j-4s,5s}=\kappa_0w_{4(j+m+1)+2\ra 4(j+m)+3}\otimes w_{4j+3\ra 4(j+1)+1};$$
$$b_{j-s,5s}=\kappa_0w_{4(j+m+1)+1\ra 4(j+m)+3}\otimes w_{4j+3\ra 4j+2}.$$

$(13)$ If $r_0=8$ and $s>1$, then $\Omega^{8}(Y_t^{(5)})$ is described with
$(6s\times 6s)$-matrix with the following two nonzero elements{\rm:}
$$b_{(j+1)_s,1}=-f_1(1,s-2)\kappa_0w_{4(j+m)+3\ra 4(j+m+1)}\otimes w_{4j\ra 4(j+1)};$$
$$b_{j,2s+1}=f_1((j)_s,s-2)\kappa^\ell(\a_{3(j+1+m)})\kappa_0w_{4(j+m+s)+2\ra 4(j+m+s+1)+1}\otimes e_{4(j+s)+1}.$$

$(14)$ If $r_0=9$ and $s=1$, then $\Omega^{9}(Y_t^{(5)})$ is described with
$(6s\times 7s)$-matrix with the following two nonzero elements{\rm:}
$$b_{j-s,2s}=\kappa^\ell(\a_{3(j+m)})\kappa_0w_{4(j+m)\ra 4(j+m+1)+1}\otimes w_{4(j+1)+1\ra 4j+1};$$
$$b_{j+2s,4s}=\kappa^\ell(\a_{3(j+m)})\kappa_0w_{4(j+m)+2\ra 4(j+m+1)+2}\otimes w_{4(j+1)+2\ra 4j+3}.$$

$(15)$ If $r_0=9$ and $s>1$, then $\Omega^{9}(Y_t^{(5)})$ is described with
$(6s\times 7s)$-matrix with the following two nonzero elements{\rm:}
$$b_{(j+1)_s,2s}=\kappa_1\kappa_0w_{4(j+m+1)+3\ra 4(j+m+s+2)+1}\otimes w_{4(j+s)+1\ra 4(j+1)};$$
$$b_{(j+1)_s,4s}=\kappa_1\kappa_0w_{4(j+m+1)+3\ra 4(j+m+s+2)+2}\otimes w_{4(j+s)+2\ra 4(j+1)},$$
where $\kappa_1=f_2(s,3)\kappa^\ell(\a_{3(m+2)})$.

$(16)$ If $r_0=10$ and $s=1$, then $\Omega^{10}(Y_t^{(5)})$ is described with
$(7s\times 6s)$-matrix with the following two nonzero elements{\rm:}
$$b_{j-3s,3s}=-\kappa^\ell(\a_{3(j+1+m)})\kappa_0w_{4(j+m)\ra 4(j+m+1)+2}\otimes w_{4j+1\ra 4j};$$
$$b_{j+2s,3s}=-\kappa^\ell(\a_{3(j+1+m)})\kappa_0w_{4(j+m)+3\ra 4(j+m+1)+2}\otimes w_{4j+1\ra 4j+3}.$$

$(17)$ If $r_0=10$ and $s>1$, then $\Omega^{10}(Y_t^{(5)})$ is described with
$(7s\times 6s)$-matrix with the following two nonzero elements{\rm:}
$$b_{(j+1)_s,s}=-f_2((j)_s,s-1)f_2((j)_s,s-3)\kappa^\ell(\a_{3(j+1+m)})\kappa_0w_{4(j+m+1)\ra 4(j+m+1)+1}\otimes w_{4j\ra 4(j+1)};$$
$$b_{5s+(j)_s,5s}=f_2((j)_s,s-3)\kappa_0w_{4(j+m)+3\ra 4(j+m+1)+3}\otimes w_{4j+2\ra 4j+3}.$$

\medskip
$({\rm II})$ Represent an arbitrary $t_0\in\N$ in the form
$t_0=11\ell_0+r_0$, where $0\le r_0\le 10.$ Then
$\Omega^{t_0}(Y_t^{(5)})$ is a $\Omega^{r_0}(Y_t^{(5)})$, whose left
components twisted by $\sigma^{\ell_0}$,
and coefficients multiplied by $(-1)^{\ell_0}$.
\end{pr}

\begin{pr}[Translates for the case 6]
$({\rm I})$ Let $r_0\in\N$, $r_0<11$. $r_0$-translates of the
elements $Y^{(6)}_t$ are described by the following way.

$(1)$ If $r_0=0$, then $\Omega^{0}(Y_t^{(6)})$ is described with
$(8s\times 6s)$-matrix with the following elements $b_{ij}${\rm:}

If $0\le j<s$, then $$b_{ij}=
\begin{cases}
w_{4j\ra 4j+2}\otimes e_{4j},\quad i=j;\\
0,\quad\text{otherwise.}
\end{cases}$$

If $s\le j<3s$, then $b_{ij}=0$.

If $3s\le j<4s$, then $$b_{ij}=
\begin{cases}
-\kappa^\ell(\a_{3(j+m+1)})w_{4j+1\ra 4j+3}\otimes e_{4j+1},\quad i=j-s;\\
0,\quad\text{otherwise.}
\end{cases}$$

If $4s\le j<6s$, then $$b_{ij}=
\begin{cases}
-\kappa^\ell(\a_{3(j+m+2)})w_{4j+2\ra 4(j+1)}\otimes e_{4j+2},\quad i=j-s,\text{ }j<5s-1\text{ or }j=6s-1;\\
0,\quad\text{otherwise.}
\end{cases}$$

If $6s\le j<8s$, then $$b_{ij}=
\begin{cases}
w_{4j+3\ra 4(j+1)+1}\otimes e_{4j+3},\quad i=5s+(j)_s,\text{ }j<7s-1\text{ or }j=8s-1;\\
0,\quad\text{otherwise.}
\end{cases}$$

$(2)$ If $r_0=1$, then $\Omega^{1}(Y_t^{(6)})$ is described with
$(9s\times 7s)$-matrix with the following elements $b_{ij}${\rm:}

If $0\le j<s$, then $$b_{ij}=
\begin{cases}
-\kappa^\ell(\a_{3(j+m+1)})w_{4(j+m+s)+1\ra 4(j+m)+3}\otimes e_{4j},\quad i=j+s;\\
0,\quad\text{otherwise.}
\end{cases}$$

If $s\le j<2s$, then $$b_{ij}=
\begin{cases}
-\kappa^\ell(\a_{3(j+m+2)})w_{4(j+m+s)+1\ra 4(j+m+1)}\otimes e_{4j},\quad i=j-s;\\
-\kappa^\ell(\a_{3(j+m+2)})w_{4(j+m+s)+2\ra 4(j+m+1)}\otimes w_{4j\ra 4(j+s)+1},\quad i=j+s;\\
0,\quad\text{otherwise.}
\end{cases}$$

If $2s\le j<2s-1$, then $b_{ij}=0$.

If $2s-1\le j<2s$, then $$b_{ij}=
\begin{cases}
-\kappa^\ell(\a_{3(j+m+2)})w_{4(j+m)+3\ra 4(j+m+1)}\otimes w_{4j\ra 4(j+s)+2},\quad i=j+3s;\\
\kappa^\ell(\a_{3(j+m+2)})w_{4(j+m)+3\ra 4(j+m+1)}\otimes w_{4j\ra 4j+2},\quad i=j+4s;\\
0,\quad\text{otherwise.}
\end{cases}$$

If $2s\le j<3s-1$, then $b_{ij}=0$.

If $3s-1\le j<3s$, then $$b_{ij}=
\begin{cases}
-w_{4(j+m)+3\ra 4(j+m+s+1)+1}\otimes w_{4j+1\ra 4j+2},\quad i=j+2s;\\
0,\quad\text{otherwise.}
\end{cases}$$

If $3s\le j<4s-1$, then $$b_{ij}=
\begin{cases}
-w_{4(j+m)+3\ra 4(j+m+s+1)+1}\otimes w_{4j+1\ra 4j+2},\quad i=j+2s;\\
0,\quad\text{otherwise.}
\end{cases}$$

If $4s-1\le j<4s$, then $b_{ij}=0$.

If $4s\le j<5s-1$, then $$b_{ij}=
\begin{cases}
-w_{4(j+m)+3\ra 4(j+m+1)+1}\otimes e_{4j+2},\quad i=j;\\
0,\quad\text{otherwise.}
\end{cases}$$

If $5s-1\le j<6s-1$, then $b_{ij}=0$.

If $6s-1\le j<6s$, then $$b_{ij}=
\begin{cases}
w_{4(j+m+1)+1\ra 4(j+m+1)+2}\otimes w_{4(j+s)+2\ra 4(j+1)},\quad i=(j+1)_s;\\
w_{4(j+m)+3\ra 4(j+m+1)+2}\otimes e_{4(j+s)+2},\quad i=j-s;\\
0,\quad\text{otherwise.}
\end{cases}$$

If $6s\le j<7s-1$, then $b_{ij}=0$.

If $7s-1\le j<7s$, then $$b_{ij}=
\begin{cases}
-w_{4(j+m)+3\ra 4(j+m+s+1)+1}\otimes e_{4(j+s)+2},\quad i=j-s;\\
0,\quad\text{otherwise.}
\end{cases}$$

If $7s\le j<8s-1$, then $$b_{ij}=
\begin{cases}
w_{4(j+m+s+1)+1\ra 4(j+m+s+1)+2}\otimes w_{4j+2\ra 4(j+1)},\quad i=(j+1)_s;\\
w_{4(j+m)+3\ra 4(j+m+s+1)+2}\otimes e_{4j+2},\quad i=j-2s;\\
0,\quad\text{otherwise.}
\end{cases}$$

If $8s-1\le j<8s$, then $b_{ij}=0$.

If $8s\le j<9s$, then $$b_{ij}=
\begin{cases}
\kappa^\ell(\a_{3(j+m+2)})w_{4(j+m+1)\ra 4(j+m+1)+3}\otimes e_{4j+3},\quad i=j-2s;\\
0,\quad\text{otherwise.}
\end{cases}$$

$(3)$ If $r_0=2$, then $\Omega^{2}(Y_t^{(6)})$ is described with
$(8s\times 6s)$-matrix with the following elements $b_{ij}${\rm:}

If $0\le j<s$, then $$b_{ij}=
\begin{cases}
-w_{4(j+m)-1\ra 4(j+m)+1}\otimes e_{4j},\quad i=j;\\
0,\quad\text{otherwise.}
\end{cases}$$

If $s\le j<2s$, then $b_{ij}=0$.

If $2s\le j<4s$, then $$b_{ij}=
\begin{cases}
-\kappa^\ell(\a_{3(j+m+2)})w_{4(j+m)+3\ra 4(j+m+1)}\otimes w_{4j+1\ra 4(j+1)},\\\quad\quad\quad i=(j+1)_s,\text{ }j<3s-1\text{ or }j=4s-1;\\
0,\quad\text{otherwise.}
\end{cases}$$

If $4s\le j<5s$, then $$b_{ij}=
\begin{cases}
-\kappa^\ell(\a_{3(j+m+1)})w_{4(j+m+s)+1\ra 4(j+m)+3}\otimes e_{4j+2},\quad i=j-s;\\
0,\quad\text{otherwise.}
\end{cases}$$

If $5s\le j<6s$, then $$b_{ij}=
\begin{cases}
-\kappa^\ell(\a_{3(j+m+1)})w_{4(j+m+s)+1\ra 4(j+m)+3}\otimes e_{4j+2},\quad i=j-s;\\
0,\quad\text{otherwise.}
\end{cases}$$

If $6s\le j<8s$, then $b_{ij}=0$.

$(4)$ If $r_0=3$, then $\Omega^{3}(Y_t^{(6)})$ is described with
$(9s\times 8s)$-matrix with the following elements $b_{ij}${\rm:}

If $0\le j<s$, then $$b_{ij}=
\begin{cases}
\kappa^\ell(\a_{3(j+m+2)})w_{4(j+m)+2\ra 4(j+m+1)}\otimes e_{4j},\quad i=j;\\
\kappa^\ell(\a_{3(j+m+2)})w_{4(j+m)+3\ra 4(j+m+1)}\otimes w_{4j\ra 4j+1},\quad i=j+2s;\\
0,\quad\text{otherwise.}
\end{cases}$$

If $s\le j<3s-1$, then $b_{ij}=0$.

If $3s-1\le j<3s$, then $$b_{ij}=
\begin{cases}
e_{4(j+m+s+1)+2}\otimes w_{4j+1\ra 4(j+1)},\quad i=(j+1)_s;\\
0,\quad\text{otherwise.}
\end{cases}$$

If $3s\le j<2s$, then $b_{ij}=0$.

If $2s\le j<3s-1$, then $$b_{ij}=
\begin{cases}
-e_{4(j+m+s+1)+2}\otimes w_{4j+1\ra 4(j+1)},\quad i=s+(j+1)_s;\\
0,\quad\text{otherwise.}
\end{cases}$$

If $3s-1\le j<4s$, then $b_{ij}=0$.

If $4s\le j<5s-1$, then $$b_{ij}=
\begin{cases}
e_{4(j+m+1)+2}\otimes w_{4(j+s)+1\ra 4(j+1)},\quad i=(j+1)_s;\\
0,\quad\text{otherwise.}
\end{cases}$$

If $5s-1\le j<5s$, then $$b_{ij}=
\begin{cases}
-e_{4(j+m+1)+2}\otimes w_{4(j+s)+1\ra 4(j+1)},\quad i=s+(j+1)_s;\\
0,\quad\text{otherwise.}
\end{cases}$$

If $5s\le j<7s$, then $b_{ij}=0$.

If $7s\le j<8s$, then $$b_{ij}=
\begin{cases}
\kappa^\ell(\a_{3(j+m+2)})f_1(j,8s-1)w_{4(j+m+s-f(j,8s-1)s+1)+2\ra 4(j+m+1)+3}\otimes w_{4j+3\ra 4(j+1)},\\\quad\quad\quad i=(j+1)_s;\\
\kappa^\ell(\a_{3(j+m+2)})f_1(j,8s-1)e_{4(j+m+1)+3}\otimes w_{4j+3\ra 4(j+s-f(j,8s-1)s+1)+1},\\\quad\quad\quad i=2s+(j+1)_s;\\
0,\quad\text{otherwise.}
\end{cases}$$

If $8s\le j<9s$, then $b_{ij}=0$.

$(5)$ If $r_0=4$, then $\Omega^{4}(Y_t^{(6)})$ is described with
$(8s\times 9s)$-matrix with the following elements $b_{ij}${\rm:}

If $0\le j<s$, then $$b_{ij}=
\begin{cases}
-w_{4(j+m)-1\ra 4(j+m)+2}\otimes e_{4j},\quad i=j;\\
w_{4(j+m)+1\ra 4(j+m)+2}\otimes w_{4j\ra 4(j+s)+1},\quad i=j+3s;\\
-e_{4(j+m)+2}\otimes w_{4j\ra 4(j+s)+2},\quad i=j+7s;\\
0,\quad\text{otherwise.}
\end{cases}$$

If $s\le j<2s$, then $$b_{ij}=
\begin{cases}
w_{4(j+m)+1\ra 4(j+m)+2}\otimes w_{4j\ra 4(j+s)+1},\quad i=j+s;\\
-e_{4(j+m)+2}\otimes w_{4j\ra 4(j+s)+2},\quad i=j+4s;\\
0,\quad\text{otherwise.}
\end{cases}$$

If $2s\le j<4s$, then $$b_{ij}=
\begin{cases}
-\kappa^\ell(\a_{3(j+m+1)})e_{4(j+m)+3}\otimes w_{4j+1\ra 4(j+1)},\quad i=(j+1)_s,\text{ }j<3s-1\text{ or }j=4s-1;\\
0,\quad\text{otherwise.}
\end{cases}$$

If $4s\le j<8s$, then $b_{ij}=0$.

$(6)$ If $r_0=5$, then $\Omega^{5}(Y_t^{(6)})$ is described with
$(6s\times 8s)$-matrix with the following elements $b_{ij}${\rm:}

If $0\le j<s$, then $$b_{ij}=
\begin{cases}
\kappa^\ell(\a_{3(j+m+1)})w_{4(j+m)+1\ra 4(j+m)+3}\otimes e_{4j},\quad i=j;\\
\kappa^\ell(\a_{3(j+m+1)})w_{4(j+m+s)+1\ra 4(j+m)+3}\otimes e_{4j},\quad i=j+s;\\
-\kappa^\ell(\a_{3(j+m+1)})e_{4(j+m)+3}\otimes w_{4j\ra 4j+2},\quad i=j+4s;\\
-\kappa^\ell(\a_{3(j+m+1)})e_{4(j+m)+3}\otimes w_{4j\ra 4(j+s)+2},\quad i=j+5s;\\
0,\quad\text{otherwise.}
\end{cases}$$

If $s\le j<5s$, then $b_{ij}=0$.

If $5s\le j<6s$, then $$b_{ij}=
\begin{cases}
\kappa^\ell(\a_{3(j+m+3)})f_1(j,6s-1)e_{4(j+m+2)}\otimes w_{4j+3\ra 4(j+s-f(j,6s-1)s+1)+1},\\\quad\quad\quad i=2s+(j+1)_s;\\
0,\quad\text{otherwise.}
\end{cases}$$

$(7)$ If $r_0=6$, then $\Omega^{6}(Y_t^{(6)})$ is described with
$(7s\times 9s)$-matrix with the following elements $b_{ij}${\rm:}

If $0\le j<s$, then $$b_{ij}=
\begin{cases}
-\kappa^\ell(\a_{3(j+m+1)})w_{4(j+m)\ra 4(j+m)+3}\otimes e_{4j},\quad i=j;\\
\kappa^\ell(\a_{3(j+m+1)})w_{4(j+m)+1\ra 4(j+m)+3}\otimes w_{4j\ra 4j+1},\quad i=j+s;\\
\kappa^\ell(\a_{3(j+m+1)})w_{4(j+m+s)+1\ra 4(j+m)+3}\otimes w_{4j\ra 4(j+s)+1},\quad i=j+3s;\\
0,\quad\text{otherwise.}
\end{cases}$$

If $s\le j<3s$, then $$b_{ij}=
\begin{cases}
\kappa^\ell(\a_{3(j+m+2)})e_{4(j+m+1)}\otimes w_{4(j+s)+1\ra 4(j+1)},\\\quad\quad\quad i=(j+1)_s,\text{ }j<2s-1\text{ or }j=3s-1;\\
0,\quad\text{otherwise.}
\end{cases}$$

If $3s\le j<5s$, then $$b_{ij}=
\begin{cases}
-w_{4(j+m+1)\ra 4(j+m+1)+1}\otimes w_{4(j+s)+2\ra 4(j+1)},\\\quad\quad\quad i=(j+1)_s,\text{ }j<4s-1\text{ or }j=5s-1;\\
0,\quad\text{otherwise.}
\end{cases}$$

If $5s\le j<6s$, then $$b_{ij}=
\begin{cases}
f_1(j,6s-1)w_{4(j+m+1)\ra 4(j+m+s+1)+2}\otimes w_{4j+3\ra 4(j+1)},\quad i=(j+1)_s;\\
0,\quad\text{otherwise.}
\end{cases}$$

If $6s\le j<6s-1$, then $b_{ij}=0$.

If $6s-1\le j<6s$, then $$b_{ij}=
\begin{cases}
w_{4(j+m+s+1)+1\ra 4(j+m+s+1)+2}\otimes w_{4j+3\ra 4(j+s+1)+1},\quad i=3s+(j+1)_s;\\
0,\quad\text{otherwise.}
\end{cases}$$

If $6s\le j<5s$, then $b_{ij}=0$.

If $5s\le j<6s-1$, then $$b_{ij}=
\begin{cases}
-e_{4(j+m+s+1)+2}\otimes w_{4j+3\ra 4(j+1)+1},\quad i=4s+(j+1)_s;\\
e_{4(j+m+s+1)+2}\otimes w_{4j+3\ra 4(j+s+1)+2},\quad i=5s+(j+1)_s;\\
0,\quad\text{otherwise.}
\end{cases}$$

If $6s-1\le j<6s$, then $b_{ij}=0$.

If $6s\le j<7s$, then $$b_{ij}=
\begin{cases}
-f_1(j,7s-1)w_{4(j+m+1)\ra 4(j+m+s+1)+2}\otimes w_{4j+3\ra 4(j+1)},\quad i=(j+1)_s;\\
0,\quad\text{otherwise.}
\end{cases}$$

If $7s\le j<6s$, then $b_{ij}=0$.

If $6s\le j<7s-1$, then $$b_{ij}=
\begin{cases}
w_{4(j+m+s+1)+1\ra 4(j+m+s+1)+2}\otimes w_{4j+3\ra 4(j+s+1)+1},\quad i=3s+(j+1)_s;\\
0,\quad\text{otherwise.}
\end{cases}$$

If $7s-1\le j<7s$, then $$b_{ij}=
\begin{cases}
-e_{4(j+m+s+1)+2}\otimes w_{4j+3\ra 4(j+1)+1},\quad i=4s+(j+1)_s;\\
e_{4(j+m+s+1)+2}\otimes w_{4j+3\ra 4(j+s+1)+2},\quad i=5s+(j+1)_s;\\
0,\quad\text{otherwise.}
\end{cases}$$

$(8)$ If $r_0=7$, then $\Omega^{7}(Y_t^{(6)})$ is described with
$(6s\times 8s)$-matrix with the following elements $b_{ij}${\rm:}

If $0\le j<s$, then $$b_{ij}=
\begin{cases}
\kappa^\ell(\a_{3(j+m+2)})w_{4(j+m+s)+2\ra 4(j+m+1)}\otimes e_{4j},\quad i=j+s;\\
-\kappa^\ell(\a_{3(j+m+2)})w_{4(j+m)+3\ra 4(j+m+1)}\otimes w_{4j\ra 4j+1},\quad i=j+2s;\\
-\kappa^\ell(\a_{3(j+m+2)})e_{4(j+m+1)}\otimes w_{4j\ra 4j+2},\quad i=j+4s;\\
-\kappa^\ell(\a_{3(j+m+2)})e_{4(j+m+1)}\otimes w_{4j\ra 4(j+s)+2},\quad i=j+5s;\\
0,\quad\text{otherwise.}
\end{cases}$$

If $s\le j<4s-1$, then $b_{ij}=0$.

If $4s-1\le j<4s$, then $$b_{ij}=
\begin{cases}
e_{4(j+m+1)+2}\otimes w_{4(j+s)+2\ra 4(j+1)},\quad i=(j+1)_s;\\
0,\quad\text{otherwise.}
\end{cases}$$

If $4s\le j<5s-1$, then $$b_{ij}=
\begin{cases}
e_{4(j+m+1)+2}\otimes w_{4(j+s)+2\ra 4(j+1)},\quad i=(j+1)_s;\\
0,\quad\text{otherwise.}
\end{cases}$$

If $5s-1\le j<5s$, then $b_{ij}=0$.

If $5s\le j<6s$, then $$b_{ij}=
\begin{cases}
\kappa^\ell(\a_{3(j+m+2)})w_{4(j+m+s-f(j,6s-1)s+1)+2\ra 4(j+m+1)+3}\otimes w_{4j+3\ra 4(j+1)},\quad i=(j+1)_s;\\
0,\quad\text{otherwise.}
\end{cases}$$

$(9)$ If $r_0=8$, then $\Omega^{8}(Y_t^{(6)})$ is described with
$(6s\times 6s)$-matrix with the following elements $b_{ij}${\rm:}

If $0\le j<s$, then $$b_{ij}=
\begin{cases}
\kappa^\ell(\a_{3(j+m+2)})w_{4(j+m)+3\ra 4(j+m+1)}\otimes w_{4j\ra 4(j+1)},\quad i=(j+1)_s;\\
\kappa^\ell(\a_{3(j+m+2)})w_{4(j+m)+2\ra 4(j+m+1)}\otimes w_{4j\ra 4j+1},\quad i=j+s;\\
\kappa^\ell(\a_{3(j+m+2)})w_{4(j+m+s)+2\ra 4(j+m+1)}\otimes w_{4j\ra 4(j+s)+1},\quad i=j+2s;\\
\kappa^\ell(\a_{3(j+m+2)})e_{4(j+m+1)}\otimes w_{4j\ra 4j+3},\quad i=j+5s;\\
0,\quad\text{otherwise.}
\end{cases}$$

If $s\le j<3s$, then $$b_{ij}=
\begin{cases}
w_{4(j+m)+3\ra 4(j+m+1)+1}\otimes w_{4(j+s)+1\ra 4(j+1)},\quad i=(j+1)_s,\text{ }j<2s-1\text{ or }j=3s-1;\\
0,\quad\text{otherwise.}
\end{cases}$$

If $3s\le j<4s-1$, then $$b_{ij}=
\begin{cases}
w_{4(j+m)+3\ra 4(j+m+1)+2}\otimes w_{4(j+s)+2\ra 4(j+1)},\quad i=(j+1)_s;\\
w_{4(j+m+1)+1\ra 4(j+m+1)+2}\otimes w_{4(j+s)+2\ra 4(j+s+1)+2},\quad i=3s+(j+1)_s;\\
0,\quad\text{otherwise.}
\end{cases}$$

If $4s-1\le j<4s$, then $$b_{ij}=
\begin{cases}
w_{4(j+m+1)+1\ra 4(j+m+1)+2}\otimes w_{4(j+s)+2\ra 4(j+s+1)+2},\quad i=4s+(j+1)_s;\\
0,\quad\text{otherwise.}
\end{cases}$$

If $4s\le j<5s-1$, then $b_{ij}=0$.

If $5s-1\le j<5s$, then $$b_{ij}=
\begin{cases}
w_{4(j+m)+3\ra 4(j+m+1)+2}\otimes w_{4(j+s)+2\ra 4(j+1)},\quad i=(j+1)_s;\\
w_{4(j+m+1)+1\ra 4(j+m+1)+2}\otimes w_{4(j+s)+2\ra 4(j+s+1)+2},\quad i=3s+(j+1)_s;\\
0,\quad\text{otherwise.}
\end{cases}$$

If $5s\le j<4s$, then $b_{ij}=0$.

If $4s\le j<5s-1$, then $$b_{ij}=
\begin{cases}
w_{4(j+m+1)+1\ra 4(j+m+1)+2}\otimes w_{4(j+s)+2\ra 4(j+s+1)+2},\quad i=4s+(j+1)_s;\\
0,\quad\text{otherwise.}
\end{cases}$$

If $5s-1\le j<5s$, then $b_{ij}=0$.

If $5s\le j<6s$, then $$b_{ij}=
\begin{cases}
-\kappa^\ell(\a_{3(j+m+2)})w_{4(j+m)+3\ra 4(j+m+1)+3}\otimes w_{4j+3\ra 4(j+1)},\quad i=(j+1)_s;\\
-\kappa^\ell(\a_{3(j+m+2)})w_{4(j+m+1+f(j,6s-1)s)+2\ra 4(j+m+1)+3}\otimes w_{4j+3\ra 4(j+1+f(j,6s-1)s)+1},\\\quad\quad\quad i=2s+(j+1)_s;\\
-\kappa^\ell(\a_{3(j+m+2)})w_{4(j+m+1+f(j,6s-1)s)+1\ra 4(j+m+1)+3}\otimes w_{4j+3\ra 4(j+s-f(j,6s-1)s+1)+2},\\\quad\quad\quad i=3s+(j+1)_s;\\
0,\quad\text{otherwise.}
\end{cases}$$

$(10)$ If $r_0=9$, then $\Omega^{9}(Y_t^{(6)})$ is described with
$(7s\times 7s)$-matrix with the following elements $b_{ij}${\rm:}

If $0\le j<s$, then $$b_{ij}=
\begin{cases}
w_{4(j+m)+3\ra 4(j+m+s+1)+1}\otimes e_{4j},\quad i=j;\\
w_{4(j+m+1)\ra 4(j+m+s+1)+1}\otimes w_{4j\ra 4j+1},\quad i=j+s;\\
0,\quad\text{otherwise.}
\end{cases}$$

If $s\le j<2s$, then $$b_{ij}=
\begin{cases}
w_{4(j+m+1)\ra 4(j+m+s+1)+1}\otimes w_{4j\ra 4(j+s)+1},\quad i=j;\\
0,\quad\text{otherwise.}
\end{cases}$$

If $2s\le j<5s-1$, then $b_{ij}=0$.

If $5s-1\le j<6s-1$, then $$b_{ij}=
\begin{cases}
-f_2(j,5s-1)\kappa^\ell(\a_{3(j+m+2)})e_{4(j+m+1)+3}\otimes w_{4j+2\ra 4(j+1)},\quad i=(j+1)_s;\\
0,\quad\text{otherwise.}
\end{cases}$$

If $6s-1\le j<6s$, then $b_{ij}=0$.

If $6s\le j<7s$, then $$b_{ij}=
\begin{cases}
f_2(j,7s-1)\kappa^\ell(\a_{3(j+m+3)})e_{4(j+m+2)}\otimes w_{4j+3\ra 4(j+1+f(j,7s-1)s)+1},\quad i=s+(j+1)_s;\\
0,\quad\text{otherwise.}
\end{cases}$$

$(11)$ If $r_0=10$, then $\Omega^{10}(Y_t^{(6)})$ is described with
$(6s\times 6s)$-matrix with the following elements $b_{ij}${\rm:}

If $0\le j<s$, then $$b_{ij}=
\begin{cases}
\kappa^\ell(\a_{3(j+m+1)})w_{4(j+m)\ra 4(j+m)+3}\otimes e_{4j},\quad i=j;\\
-\kappa^\ell(\a_{3(j+m+1)})w_{4(j+m+s)+1\ra 4(j+m)+3}\otimes w_{4j\ra 4j+1},\quad i=j+s;\\
\kappa^\ell(\a_{3(j+m+1)})w_{4(j+m)+1\ra 4(j+m)+3}\otimes w_{4j\ra 4(j+s)+1},\quad i=j+2s;\\
\kappa^\ell(\a_{3(j+m+1)})w_{4(j+m+s)+2\ra 4(j+m)+3}\otimes w_{4j\ra 4j+2},\quad i=j+3s;\\
-\kappa^\ell(\a_{3(j+m+1)})w_{4(j+m)+2\ra 4(j+m)+3}\otimes w_{4j\ra 4(j+s)+2},\quad i=j+4s;\\
\kappa^\ell(\a_{3(j+m+1)})e_{4(j+m)+3}\otimes w_{4j\ra 4j+3},\quad i=j+5s;\\
0,\quad\text{otherwise.}
\end{cases}$$

If $s\le j<2s$, then $$b_{ij}=
\begin{cases}
f_1(j,2s-1)w_{4(j+m+1)\ra 4(j+m+1)+2}\otimes w_{4(j+s)+1\ra 4(j+1)},\quad i=(j+1)_s;\\
0,\quad\text{otherwise.}
\end{cases}$$

If $2s\le j<s$, then $b_{ij}=0$.

If $s\le j<2s-1$, then $$b_{ij}=
\begin{cases}
-w_{4(j+m+1)+1\ra 4(j+m+1)+2}\otimes w_{4(j+s)+1\ra 4(j+s+1)+1},\quad i=s+(j+1)_s;\\
0,\quad\text{otherwise.}
\end{cases}$$

If $2s-1\le j<2s$, then $$b_{ij}=
\begin{cases}
-w_{4(j+m+1)+1\ra 4(j+m+1)+2}\otimes w_{4(j+s)+1\ra 4(j+s+1)+1},\quad i=2s+(j+1)_s;\\
0,\quad\text{otherwise.}
\end{cases}$$

If $2s\le j<3s$, then $$b_{ij}=
\begin{cases}
w_{4(j+m)+2\ra 4(j+m+1)+2}\otimes w_{4(j+s)+1\ra 4(j+s)+2},\quad i=j+2s;\\
-w_{4(j+m)+3\ra 4(j+m+1)+2}\otimes w_{4(j+s)+1\ra 4j+3},\quad i=j+3s;\\
0,\quad\text{otherwise.}
\end{cases}$$

If $3s\le j<6s$, then $b_{ij}=0$.

\medskip
$({\rm II})$ Represent an arbitrary $t_0\in\N$ in the form
$t_0=11\ell_0+r_0$, where $0\le r_0\le 10.$ Then
$\Omega^{t_0}(Y_t^{(6)})$ is a $\Omega^{r_0}(Y_t^{(6)})$, whose left
components twisted by $\sigma^{\ell_0}$.
\end{pr}

\begin{pr}[Translates for the case 7]
$({\rm I})$ Let $r_0\in\N$, $r_0<11$. $r_0$-translates of the
elements $Y^{(7)}_t$ are described by the following way.

$(1)$ If $r_0=0$, then $\Omega^{0}(Y_t^{(7)})$ is described with
$(8s\times 6s)$-matrix with the following elements $b_{ij}${\rm:}

If $0\le j<2s$, then $$b_{ij}=
\begin{cases}
w_{4j\ra 4(j+s)+2}\otimes e_{4j},\quad i=(j)_s;\\
0,\quad\text{otherwise.}
\end{cases}$$

If $2s\le j<4s$, then $$b_{ij}=
\begin{cases}
w_{4j+1\ra 4j+3}\otimes e_{4j+1},\quad i=j-s;\\
0,\quad\text{otherwise.}
\end{cases}$$

If $4s\le j<6s$, then $$b_{ij}=
\begin{cases}
w_{4j+2\ra 4(j+1)}\otimes e_{4j+2},\quad i=j-s;\\
0,\quad\text{otherwise.}
\end{cases}$$

If $6s\le j<8s$, then $b_{ij}=0$.

$(2)$ If $r_0=1$, then $\Omega^{1}(Y_t^{(7)})$ is described with
$(9s\times 7s)$-matrix with the following elements $b_{ij}${\rm:}

If $0\le j<s$, then $$b_{ij}=
\begin{cases}
w_{4(j+m)+1\ra 4(j+m)+3}\otimes e_{4j},\quad i=j;\\
w_{4(j+m+s)+1\ra 4(j+m)+3}\otimes e_{4j},\quad i=j+s;\\
0,\quad\text{otherwise.}
\end{cases}$$

If $s\le j<2s$, then $$b_{ij}=
\begin{cases}
w_{4(j+m)+1\ra 4(j+m+1)}\otimes e_{4j},\quad i=j;\\
w_{4(j+m+s)+2\ra 4(j+m+1)}\otimes w_{4j\ra 4(j+s)+1},\quad i=j+s;\\
w_{4(j+m)+2\ra 4(j+m+1)}\otimes w_{4j\ra 4j+1},\quad i=j+2s;\\
0,\quad\text{otherwise.}
\end{cases}$$

If $2s\le j<3s$, then $$b_{ij}=
\begin{cases}
w_{4(j+m)+2\ra 4(j+m+1)+1}\otimes e_{4j+1},\quad i=j;\\
0,\quad\text{otherwise.}
\end{cases}$$

If $3s\le j<4s$, then $$b_{ij}=
\begin{cases}
w_{4(j+m)+2\ra 4(j+m+1)+1}\otimes e_{4j+1},\quad i=j;\\
0,\quad\text{otherwise.}
\end{cases}$$

If $4s\le j<5s$, then $b_{ij}=0$.

If $5s\le j<6s$, then $$b_{ij}=
\begin{cases}
w_{4(j+m)+3\ra 4(j+m+s+1)+2}\otimes e_{4(j+s)+2},\quad i=j-s;\\
w_{4(j+m+1)\ra 4(j+m+s+1)+2}\otimes w_{4(j+s)+2\ra 4j+3},\quad i=j+s;\\
0,\quad\text{otherwise.}
\end{cases}$$

If $6s\le j<7s$, then $b_{ij}=0$.

If $7s\le j<8s$, then $$b_{ij}=
\begin{cases}
w_{4(j+m)+3\ra 4(j+m+1)+2}\otimes e_{4j+2},\quad i=j-2s;\\
w_{4(j+m+1)\ra 4(j+m+1)+2}\otimes w_{4j+2\ra 4j+3},\quad i=j-s;\\
0,\quad\text{otherwise.}
\end{cases}$$

If $8s\le j<9s$, then $$b_{ij}=
\begin{cases}
w_{4(j+m+1+f(j,9s-1)s)+1\ra 4(j+m+1)+3}\otimes w_{4j+3\ra 4(j+1)},\quad i=(j+1)_s;\\
0,\quad\text{otherwise.}
\end{cases}$$

$(3)$ If $r_0=2$, then $\Omega^{2}(Y_t^{(7)})$ is described with
$(8s\times 6s)$-matrix with the following elements $b_{ij}${\rm:}

If $2s\le j<3s$, then $$b_{ij}=
\begin{cases}
w_{4(j+m)+2\ra 4(j+m+1)}\otimes e_{4j+1},\quad i=j-s;\\
0,\quad\text{otherwise.}
\end{cases}$$

If $3s\le j<4s$, then $$b_{ij}=
\begin{cases}
w_{4(j+m)+2\ra 4(j+m+1)}\otimes e_{4j+1},\quad i=j-s;\\
0,\quad\text{otherwise.}
\end{cases}$$

If $4s\le j<5s$, then $$b_{ij}=
\begin{cases}
w_{4(j+m+s)+1\ra 4(j+m)+3}\otimes e_{4j+2},\quad i=j-s;\\
0,\quad\text{otherwise.}
\end{cases}$$

If $5s\le j<6s$, then $$b_{ij}=
\begin{cases}
w_{4(j+m+s)+1\ra 4(j+m)+3}\otimes e_{4j+2},\quad i=j-s;\\
0,\quad\text{otherwise.}
\end{cases}$$

If $6s\le j<7s-1$, then $b_{ij}=0$.

If $7s-1\le j<7s$, then $$b_{ij}=
\begin{cases}
e_{4(j+m+s+1)+2}\otimes w_{4j+3\ra 4(j+s+1)+1},\quad i=s+(j+1)_s;\\
0,\quad\text{otherwise.}
\end{cases}$$

If $7s\le j<6s$, then $b_{ij}=0$.

If $6s\le j<7s-1$, then $$b_{ij}=
\begin{cases}
e_{4(j+m+s+1)+2}\otimes w_{4j+3\ra 4(j+s+1)+1},\quad i=2s+(j+1)_s;\\
0,\quad\text{otherwise.}
\end{cases}$$

If $7s-1\le j<7s$, then $b_{ij}=0$.

If $7s\le j<8s$, then $$b_{ij}=
\begin{cases}
e_{4(j+m+s+1)+2}\otimes w_{4j+3\ra 4(j+s+1)+1},\quad i=s+f(j,8s-1)s+(j+1)_s;\\
0,\quad\text{otherwise.}
\end{cases}$$

$(4)$ If $r_0=3$, then $\Omega^{3}(Y_t^{(7)})$ is described with
$(9s\times 8s)$-matrix with the following elements $b_{ij}${\rm:}

If $0\le j<s$, then $$b_{ij}=
\begin{cases}
w_{4(j+m)+2\ra 4(j+m+1)}\otimes e_{4j},\quad i=j;\\
w_{4(j+m+s)+2\ra 4(j+m+1)}\otimes e_{4j},\quad i=j+s;\\
0,\quad\text{otherwise.}
\end{cases}$$

If $s\le j<2s$, then $b_{ij}=0$.

If $2s\le j<3s$, then $$b_{ij}=
\begin{cases}
e_{4(j+m+1)+2}\otimes w_{4j+1\ra 4(j+1)},\quad i=f(j,3s-1)s+(j+1)_s;\\
w_{4(j+m)+3\ra 4(j+m+1)+2}\otimes e_{4j+1},\quad i=j;\\
0,\quad\text{otherwise.}
\end{cases}$$

If $3s\le j<4s$, then $b_{ij}=0$.

If $4s\le j<5s$, then $$b_{ij}=
\begin{cases}
e_{4(j+m+s+1)+2}\otimes w_{4(j+s)+1\ra 4(j+1)},\quad i=s-f(j,5s-1)s+(j+1)_s;\\
w_{4(j+m)+3\ra 4(j+m+s+1)+2}\otimes e_{4(j+s)+1},\quad i=j-s;\\
0,\quad\text{otherwise.}
\end{cases}$$

If $5s\le j<7s$, then $b_{ij}=0$.

If $7s\le j<8s$, then $$b_{ij}=
\begin{cases}
w_{4(j+m+s-f(j,8s-1)s+1)+2\ra 4(j+m+1)+3}\otimes w_{4j+3\ra 4(j+1)},\quad i=(j+1)_s;\\
w_{4(j+m-f(j,8s-1)s+1)+2\ra 4(j+m+1)+3}\otimes w_{4j+3\ra 4(j+1)},\quad i=s+(j+1)_s;\\
0,\quad\text{otherwise.}
\end{cases}$$

If $8s\le j<9s$, then $b_{ij}=0$.

$(5)$ If $r_0=4$, then $\Omega^{4}(Y_t^{(7)})$ is described with
$(8s\times 9s)$-matrix with the following elements $b_{ij}${\rm:}

If $0\le j<s$, then $$b_{ij}=
\begin{cases}
w_{4(j+m+s)-1\ra 4(j+m+s)+2}\otimes e_{4j},\quad i=j;\\
w_{4(j+m+s)+1\ra 4(j+m+s)+2}\otimes w_{4j\ra 4j+1},\quad i=j+2s;\\
e_{4(j+m+s)+2}\otimes w_{4j\ra 4j+2},\quad i=j+5s;\\
0,\quad\text{otherwise.}
\end{cases}$$

If $s\le j<2s$, then $$b_{ij}=
\begin{cases}
w_{4(j+m+s)-1\ra 4(j+m+s)+2}\otimes e_{4j},\quad i=j-s;\\
w_{4(j+m+s)+1\ra 4(j+m+s)+2}\otimes w_{4j\ra 4j+1},\quad i=j+2s;\\
e_{4(j+m+s)+2}\otimes w_{4j\ra 4j+2},\quad i=j+6s;\\
0,\quad\text{otherwise.}
\end{cases}$$

If $2s\le j<3s$, then $$b_{ij}=
\begin{cases}
w_{4(j+m+s)+1\ra 4(j+m)+3}\otimes e_{4j+1},\quad i=j;\\
w_{4(j+m+s)+2\ra 4(j+m)+3}\otimes w_{4j+1\ra 4j+2},\quad i=j+3s;\\
0,\quad\text{otherwise.}
\end{cases}$$

If $3s\le j<4s$, then $$b_{ij}=
\begin{cases}
w_{4(j+m+s)+1\ra 4(j+m)+3}\otimes e_{4j+1},\quad i=j;\\
w_{4(j+m+s)+2\ra 4(j+m)+3}\otimes w_{4j+1\ra 4j+2},\quad i=j+4s;\\
0,\quad\text{otherwise.}
\end{cases}$$

If $4s\le j<8s$, then $b_{ij}=0$.

$(6)$ If $r_0=5$, then $\Omega^{5}(Y_t^{(7)})$ is described with
$(6s\times 8s)$-matrix with the following elements $b_{ij}${\rm:}

If $s\le j<2s$, then $$b_{ij}=
\begin{cases}
w_{4(j+m+1)+1\ra 4(j+m+1)+2}\otimes w_{4(j+s)+1\ra 4(j+1)},\quad i=s-f(j,2s-1)s+(j+1)_s;\\
w_{4(j+m+1)\ra 4(j+m+1)+2}\otimes e_{4(j+s)+1},\quad i=j+s;\\
0,\quad\text{otherwise.}
\end{cases}$$

If $2s\le j<3s$, then $$b_{ij}=
\begin{cases}
w_{4(j+m+1)+1\ra 4(j+m+1)+2}\otimes w_{4(j+s)+1\ra 4(j+1)},\quad i=f(j,3s-1)s+(j+1)_s;\\
w_{4(j+m+1)\ra 4(j+m+1)+2}\otimes e_{4(j+s)+1},\quad i=j+s;\\
0,\quad\text{otherwise.}
\end{cases}$$

If $3s\le j<5s$, then $b_{ij}=0$.

If $5s\le j<6s$, then $$b_{ij}=
\begin{cases}
e_{4(j+m+2)}\otimes w_{4j+3\ra 4(j+s-f(j,6s-1)s+1)+1},\quad i=2s+(j+1)_s;\\
e_{4(j+m+2)}\otimes w_{4j+3\ra 4(j+1+f(j,6s-1)s)+1},\quad i=3s+(j+1)_s;\\
0,\quad\text{otherwise.}
\end{cases}$$

$(7)$ If $r_0=6$, then $\Omega^{6}(Y_t^{(7)})$ is described with
$(7s\times 9s)$-matrix with the following elements $b_{ij}${\rm:}

If $0\le j<s$, then $$b_{ij}=
\begin{cases}
w_{4(j+m)+1\ra 4(j+m)+3}\otimes w_{4j\ra 4j+1},\quad i=j+s;\\
w_{4(j+m+s)+1\ra 4(j+m)+3}\otimes w_{4j\ra 4(j+s)+1},\quad i=j+3s;\\
0,\quad\text{otherwise.}
\end{cases}$$

If $s\le j<2s$, then $$b_{ij}=
\begin{cases}
e_{4(j+m+1)}\otimes w_{4(j+s)+1\ra 4(j+1)},\quad i=(j+1)_s;\\
w_{4(j+m+s)+1\ra 4(j+m+1)}\otimes e_{4(j+s)+1},\quad i=j;\\
0,\quad\text{otherwise.}
\end{cases}$$

If $2s\le j<3s$, then $$b_{ij}=
\begin{cases}
e_{4(j+m+1)}\otimes w_{4(j+s)+1\ra 4(j+1)},\quad i=(j+1)_s;\\
w_{4(j+m+s)+1\ra 4(j+m+1)}\otimes e_{4(j+s)+1},\quad i=j+s;\\
0,\quad\text{otherwise.}
\end{cases}$$

If $3s\le j<4s$, then $$b_{ij}=
\begin{cases}
e_{4(j+m+s+1)+1}\otimes w_{4(j+s)+2\ra 4(j+s+1)+1},\quad i=s+f(j,4s-1)2s+(j+1)_s;\\
0,\quad\text{otherwise.}
\end{cases}$$

If $4s\le j<5s$, then $$b_{ij}=
\begin{cases}
e_{4(j+m+s+1)+1}\otimes w_{4(j+s)+2\ra 4(j+s+1)+1},\quad i=3s-f(j,5s-1)2s+(j+1)_s;\\
0,\quad\text{otherwise.}
\end{cases}$$

If $5s\le j<6s$, then $$b_{ij}=
\begin{cases}
w_{4(j+m+1)+1\ra 4(j+m+1)+2}\otimes w_{4j+3\ra 4(j+1)+1},\quad i=3s-f(j,6s-1)2s+(j+1)_s;\\
0,\quad\text{otherwise.}
\end{cases}$$

If $6s\le j<7s$, then $$b_{ij}=
\begin{cases}
w_{4(j+m+1)+1\ra 4(j+m+1)+2}\otimes w_{4j+3\ra 4(j+1)+1},\quad i=s+f(j,7s-1)2s+(j+1)_s;\\
0,\quad\text{otherwise.}
\end{cases}$$

$(8)$ If $r_0=7$, then $\Omega^{7}(Y_t^{(7)})$ is described with
$(6s\times 8s)$-matrix with the following elements $b_{ij}${\rm:}

If $s\le j<2s$, then $$b_{ij}=
\begin{cases}
w_{4(j+m)+3\ra 4(j+m+s+1)+1}\otimes e_{4(j+s)+1},\quad i=j+s;\\
w_{4(j+m+1)\ra 4(j+m+s+1)+1}\otimes w_{4(j+s)+1\ra 4(j+s)+2},\quad i=j+3s;\\
e_{4(j+m+s+1)+1}\otimes w_{4(j+s)+1\ra 4j+3},\quad i=j+5s;\\
0,\quad\text{otherwise.}
\end{cases}$$

If $2s\le j<3s$, then $$b_{ij}=
\begin{cases}
w_{4(j+m)+3\ra 4(j+m+s+1)+1}\otimes e_{4(j+s)+1},\quad i=j+s;\\
w_{4(j+m+1)\ra 4(j+m+s+1)+1}\otimes w_{4(j+s)+1\ra 4(j+s)+2},\quad i=j+3s;\\
e_{4(j+m+s+1)+1}\otimes w_{4(j+s)+1\ra 4j+3},\quad i=j+5s;\\
0,\quad\text{otherwise.}
\end{cases}$$

If $3s\le j<6s$, then $b_{ij}=0$.

$(9)$ If $r_0=8$, then $\Omega^{8}(Y_t^{(7)})$ is described with
$(6s\times 6s)$-matrix with the following elements $b_{ij}${\rm:}

If $0\le j<s$, then $$b_{ij}=
\begin{cases}
w_{4(j+m)+3\ra 4(j+m+1)}\otimes w_{4j\ra 4(j+1)},\quad i=(j+1)_s;\\
w_{4(j+m)+2\ra 4(j+m+1)}\otimes w_{4j\ra 4j+1},\quad i=j+s;\\
w_{4(j+m+s)+2\ra 4(j+m+1)}\otimes w_{4j\ra 4(j+s)+1},\quad i=j+2s;\\
0,\quad\text{otherwise.}
\end{cases}$$

If $s\le j<2s$, then $$b_{ij}=
\begin{cases}
w_{4(j+m)+3\ra 4(j+m+s+1)+1}\otimes w_{4(j+s)+1\ra 4(j+1)},\quad i=(j+1)_s;\\
w_{4(j+m+s)+2\ra 4(j+m+s+1)+1}\otimes e_{4(j+s)+1},\quad i=j;\\
0,\quad\text{otherwise.}
\end{cases}$$

If $2s\le j<3s$, then $$b_{ij}=
\begin{cases}
w_{4(j+m)+3\ra 4(j+m+s+1)+1}\otimes w_{4(j+s)+1\ra 4(j+1)},\quad i=(j+1)_s;\\
w_{4(j+m+s)+2\ra 4(j+m+s+1)+1}\otimes e_{4(j+s)+1},\quad i=j;\\
0,\quad\text{otherwise.}
\end{cases}$$

If $3s\le j<4s$, then $$b_{ij}=
\begin{cases}
e_{4(j+m+s+1)+2}\otimes w_{4(j+s)+2\ra 4(j+s+1)+1},\quad i=s+f(j,4s-1)s+(j+1)_s;\\
0,\quad\text{otherwise.}
\end{cases}$$

If $4s\le j<5s$, then $$b_{ij}=
\begin{cases}
e_{4(j+m+s+1)+2}\otimes w_{4(j+s)+2\ra 4(j+s+1)+1},\quad i=2s-f(j,5s-1)s+(j+1)_s;\\
0,\quad\text{otherwise.}
\end{cases}$$

If $5s\le j<6s$, then $$b_{ij}=
\begin{cases}
w_{4(j+m+s-f(j,6s-1)s+1)+2\ra 4(j+m+1)+3}\otimes w_{4j+3\ra 4(j+s-f(j,6s-1)s+1)+1},\quad i=s+(j+1)_s;\\
w_{4(j+m+f(j,6s-1)s+1)+2\ra 4(j+m+1)+3}\otimes w_{4j+3\ra 4(j+f(j,6s-1)s+1)+1},\quad i=2s+(j+1)_s;\\
0,\quad\text{otherwise.}
\end{cases}$$

$(10)$ If $r_0=9$, then $\Omega^{9}(Y_t^{(7)})$ is described with
$(7s\times 7s)$-matrix with the following elements $b_{ij}${\rm:}

If $2s\le j<3s$, then $$b_{ij}=
\begin{cases}
w_{4(j+m+1)\ra 4(j+m+1)+2}\otimes e_{4j+1},\quad i=j-s;\\
e_{4(j+m+1)+2}\otimes w_{4j+1\ra 4j+3},\quad i=j+3s;\\
0,\quad\text{otherwise.}
\end{cases}$$

If $3s\le j<4s$, then $$b_{ij}=
\begin{cases}
w_{4(j+m+1)\ra 4(j+m+1)+2}\otimes e_{4j+1},\quad i=j-s;\\
e_{4(j+m+1)+2}\otimes w_{4j+1\ra 4j+3},\quad i=j+3s;\\
0,\quad\text{otherwise.}
\end{cases}$$

If $4s\le j<7s$, then $b_{ij}=0$.

$(11)$ If $r_0=10$, then $\Omega^{10}(Y_t^{(7)})$ is described with
$(6s\times 6s)$-matrix with the following elements $b_{ij}${\rm:}

If $0\le j<s$, then $$b_{ij}=
\begin{cases}
w_{4(j+m+s)+1\ra 4(j+m)+3}\otimes w_{4j\ra 4j+1},\quad i=j+s;\\
w_{4(j+m)+1\ra 4(j+m)+3}\otimes w_{4j\ra 4(j+s)+1},\quad i=j+2s;\\
w_{4(j+m+s)+2\ra 4(j+m)+3}\otimes w_{4j\ra 4j+2},\quad i=j+3s;\\
w_{4(j+m)+2\ra 4(j+m)+3}\otimes w_{4j\ra 4(j+s)+2},\quad i=j+4s;\\
0,\quad\text{otherwise.}
\end{cases}$$

If $s\le j<2s$, then $$b_{ij}=
\begin{cases}
w_{4(j+m+1)\ra 4(j+m+s+1)+2}\otimes w_{4(j+s)+1\ra 4(j+1)},\quad i=(j+1)_s;\\
w_{4(j+m)+3\ra 4(j+m+s+1)+2}\otimes w_{4(j+s)+1\ra 4j+3},\quad i=j+4s;\\
0,\quad\text{otherwise.}
\end{cases}$$

If $2s\le j<3s$, then $$b_{ij}=
\begin{cases}
w_{4(j+m+1)\ra 4(j+m+s+1)+2}\otimes w_{4(j+s)+1\ra 4(j+1)},\quad i=(j+1)_s;\\
w_{4(j+m)+3\ra 4(j+m+s+1)+2}\otimes w_{4(j+s)+1\ra 4j+3},\quad i=j+3s;\\
0,\quad\text{otherwise.}
\end{cases}$$

If $3s\le j<5s$, then $b_{ij}=0$.

If $5s\le j<6s$, then $$b_{ij}=
\begin{cases}
w_{4(j+m+1)\ra 4(j+m+2)}\otimes w_{4j+3\ra 4(j+1)},\quad i=(j+1)_s;\\
w_{4(j+m+s-f(j,6s-1)s+1)+1\ra 4(j+m+2)}\otimes w_{4j+3\ra 4(j+f(j,6s-1)s+1)+1},\quad i=2s+(j+1)_s;\\
w_{4(j+m+s-f(j,6s-1)s+1)+2\ra 4(j+m+2)}\otimes w_{4j+3\ra 4(j+f(j,6s-1)s+1)+2},\quad i=4s+(j+1)_s;\\
0,\quad\text{otherwise.}
\end{cases}$$

\medskip
$({\rm II})$ Represent an arbitrary $t_0\in\N$ in the form
$t_0=11\ell_0+r_0$, where $0\le r_0\le 10.$ Then
$\Omega^{t_0}(Y_t^{(7)})$ is a $\Omega^{r_0}(Y_t^{(7)})$, whose left
components twisted by $\sigma^{\ell_0}$.
\end{pr}

\begin{pr}[Translates for the case 8]
$({\rm I})$ Let $r_0\in\N$, $r_0<11$. $r_0$-translates of the
elements $Y^{(8)}_t$ are described by the following way.

$(1)$ If $r_0=0$, then $\Omega^{0}(Y_t^{(8)})$ is described with
$(9s\times 6s)$-matrix with one nonzero element that is of the following form{\rm:}
$$b_{0,0}=w_{4(j+m)\ra 4(j+m)+3}\otimes e_{4j}.$$

$(2)$ If $r_0=1$, then $\Omega^{1}(Y_t^{(8)})$ is described with
$(8s\times 7s)$-matrix with the following two nonzero elements{\rm:}
$$b_{s+(j+1)_s,4s-1}=w_{4(j+m+s+1)+1\ra 4(j+m+2)}\otimes w_{4j+1\ra 4(j+1)};$$
$$b_{s+(j+1)_s,6s-1}=w_{4(j+m+s+1)+1\ra 4(j+m+1)+3}\otimes w_{4j+2\ra 4(j+1)}.$$

$(3)$ If $r_0=2$, then $\Omega^{2}(Y_t^{(8)})$ is described with
$(9s\times 6s)$-matrix with one nonzero element that is of the following form{\rm:}
$$b_{(j+1)_s,4s-1}=w_{4(j+m)+3\ra 4(j+m+1+s)+1}\otimes w_{4j+1\ra 4(j+1)}.$$

$(4)$ If $r_0=3$, then $\Omega^{3}(Y_t^{(8)})$ is described with
$(8s\times 8s)$-matrix with the following two nonzero elements{\rm:}
$$b_{s+(j+1)_s,2s-1}=e_{4(j+m+s+1)+2}\otimes w_{4j\ra 4(j+1)};$$
$$b_{s+(j+1)_s,4s-1}=w_{4(j+m+s+1)+2\ra 4(j+m+1)+3}\otimes w_{4j+1\ra 4(j+1)}.$$

$(5)$ If $r_0=4$, then $\Omega^{4}(Y_t^{(8)})$ is described with
$(6s\times 9s)$-matrix with the following two nonzero elements{\rm:}
$$b_{(j+1)_s,3s-1}=w_{4(j+m)+3\ra 4(j+m+1)+2}\otimes w_{4(j+s)+1\ra 4(j+1)};$$
$$b_{s+(j+1)_s,3s-1}=w_{4(j+m+1)\ra 4(j+m+1)+2}\otimes w_{4(j+s)+1\ra 4(j+1)}.$$

$(6)$ If $r_0=5$, then $\Omega^{5}(Y_t^{(8)})$ is described with
$(7s\times 8s)$-matrix with the following two nonzero elements{\rm:}
$$b_{(j+1)_s,s-1}=w_{4(j+m+1+sf(j,s-1))+1\ra 4(j+m+1)+3}\otimes w_{4j\ra 4(j+1)};$$
$$b_{(j+1)_s,3s-1}=w_{4(j+m+s+1)+1\ra 4(j+m+2)}\otimes w_{4(j+s)+1\ra 4(j+1)}.$$

$(7)$ If $r_0=6$, then $\Omega^{6}(Y_t^{(8)})$ is described with
$(6s\times 9s)$-matrix with the following two nonzero elements{\rm:}
$$b_{(j+1)_s,3s-1}=w_{4(j+m+1)\ra 4(j+m+s+1)+1}\otimes w_{4(j+s)+1\ra 4(j+1)};$$
$$b_{s+(j+1)_s,3s-1}=e_{4(j+m+s+1)+1}\otimes w_{4(j+s)+1\ra 4(j+s+1)+1}.$$

$(8)$ If $r_0=7$, then $\Omega^{7}(Y_t^{(8)})$ is described with
$(6s\times 8s)$-matrix with the following two nonzero elements{\rm:}
$$b_{(j+1)_s,s-1}=w_{4(j+m+1+sf(j,s-1))+2\ra 4(j+m+2)}\otimes w_{4j\ra 4(j+1)};$$
$$b_{(j+1)_s,3s-1}=w_{4(j+m+s+1)+2\ra 4(j+m+s+2)+1}\otimes w_{4(j+s)+1\ra 4(j+1)}.$$

$(9)$ If $r_0=8$, then $\Omega^{8}(Y_t^{(8)})$ is described with
$(7s\times 6s)$-matrix with the following two nonzero elements{\rm:}
$$b_{(j+1)_s,4s-1}=w_{4(j+m)+3\ra 4(j+m+1)+2}\otimes w_{4j+1\ra 4(j+1)};$$
$$b_{s+(j+1)_s,4s-1}=e_{4(j+m+1)+2}\otimes w_{4j+1\ra 4(j+1)+1}.$$

$(10)$ If $r_0=9$, then $\Omega^{9}(Y_t^{(8)})$ is described with
$(6s\times 7s)$-matrix with the following nonzero elements{\rm:}
$$b_{j+s,s-1}=w_{4(j+m+1)\ra 4(j+m+1)+3}\otimes w_{4j\ra 4j+1};$$
$$b_{j+5s,s-1}=w_{4(j+m+1)+2\ra 4(j+m)+3}\otimes w_{4j\ra 4j+3},\text{ }s=1;$$
$$b_{j+2s,s-1}=w_{4(j+m+1)\ra 4(j+m+1)+3}\otimes w_{4j\ra 4(j+s)+1},\text{ }s>1;$$
$$b_{(j+1)_s,3s-1}=w_{4(j+m+1)+3\ra 4(j+m+s+2)+2}\otimes w_{4(j+s)+1\ra 4(j+1)}.$$

$(11)$ If $r_0=10$ and $s=1$, then $\Omega^{10}(Y_t^{(8)})$ is described with
$(8s\times 6s)$-matrix with the following nonzero elements{\rm:}
$$b_{j,0}=w_{4(j+m)\ra 4(j+m+1)+2}\otimes e_{4j};$$
$$b_{j+5s,0}=w_{4(j+m)+3\ra 4(j+m+1)+2}\otimes w_{4j\ra 4j+3};$$
$$b_{j-s,s}=w_{4(j+m)\ra 4(j+m+1)+2}\otimes e_{4j};$$
$$b_{j-3s,3s}=w_{4(j+m)\ra 4(j+m)+3}\otimes w_{4j+1\ra 4j};$$
$$b_{j-5s,5s}=e_{4(j+m)}\otimes w_{4j+2\ra 4j};$$
$$b_{j-2s,7s}=w_{4(j+m)+3\ra 4(j+m)+1}\otimes e_{4j+3}.$$

$(12)$ If $r_0=10$ and $s>1$, then $\Omega^{10}(Y_t^{(8)})$ is described with
$(8s\times 6s)$-matrix with the following nonzero elements{\rm:}
$$b_{j+2s,2s-1}=w_{4(j+m)+2\ra 4(j+m+1)+2}\otimes w_{4j\ra 4(j+s)+2};$$
$$b_{j+4s,2s-1}=w_{4(j+m)+3\ra 4(j+m+1)+2}\otimes w_{4j\ra 4j+3};$$
$$b_{s+(j+1)_s,3s-2}=w_{4(j+m+s+1)+1\ra 4(j+m+1)+3}\otimes w_{4j+1\ra 4(j+1)+1};$$
$$b_{(j+1)_s,4s-2}=w_{4(j+m+1)\ra 4(j+m+1)+3}\otimes w_{4j+1\ra 4(j+1)};$$
$$b_{j+2s,4s-2}=w_{4(j+m)+3\ra 4(j+m+1)+3}\otimes w_{4j+1\ra 4j+3};$$
$$b_{j+2s,4s-1}=w_{4(j+m)+3\ra 4(j+m+1)+3}\otimes w_{4j+1\ra 4j+3};$$
$$b_{s+(j+1)_s,5s-2}=w_{4(j+m+s+1)+1\ra 4(j+m+2)}\otimes w_{4j+2\ra 4(j+1)+1};$$
$$b_{3s+(j+1)_s,5s-2}=w_{4(j+m+s+1)+2\ra 4(j+m+2)}\otimes w_{4j+2\ra 4(j+1)+2}.$$

\medskip
$({\rm II})$ Represent an arbitrary $t_0\in\N$ in the form
$t_0=11\ell_0+r_0$, where $0\le r_0\le 10.$ Then
$\Omega^{t_0}(Y_t^{(8)})$ is a $\Omega^{r_0}(Y_t^{(8)})$, whose left
components twisted by $\sigma^{\ell_0}$.
\end{pr}

\begin{pr}[Translates for the case 9]
$({\rm I})$ Let $r_0\in\N$, $r_0<11$. $r_0$-translates of the
elements $Y^{(9)}_t$ are described by the following way.

$(1)$ If $r_0=0$, then $\Omega^{0}(Y_t^{(9)})$ is described with
$(9s\times 6s)$-matrix with the following elements $b_{ij}${\rm:}

If $s\le j<2s$, then $$b_{ij}=
\begin{cases}
\kappa^\ell(\a_{3(j+m+2)})e_{4j}\otimes e_{4j},\quad i=j-s;\\
0,\quad\text{otherwise.}
\end{cases}$$

If $2s\le j<4s$, then $$b_{ij}=
\begin{cases}
e_{4j+1}\otimes e_{4j+1},\quad i=j-s;\\
0,\quad\text{otherwise.}
\end{cases}$$

If $4s\le j<5s$, then $b_{ij}=0$.

If $5s\le j<6s$, then $$b_{ij}=
\begin{cases}
e_{4(j+s)+2}\otimes e_{4(j+s)+2},\quad i=j-2s;\\
0,\quad\text{otherwise.}
\end{cases}$$

If $6s\le j<7s$, then $b_{ij}=0$.

If $7s\le j<8s$, then $$b_{ij}=
\begin{cases}
e_{4j+2}\otimes e_{4j+2},\quad i=j-3s;\\
0,\quad\text{otherwise.}
\end{cases}$$

If $8s\le j<9s$, then $$b_{ij}=
\begin{cases}
-\kappa^\ell(\a_{3(j+m+2)})e_{4j+3}\otimes e_{4j+3},\quad i=j-3s;\\
0,\quad\text{otherwise.}
\end{cases}$$

$(2)$ If $r_0=1$, then $\Omega^{1}(Y_t^{(9)})$ is described with
$(8s\times 7s)$-matrix with the following elements $b_{ij}${\rm:}

If $0\le j<s$, then $$b_{ij}=
\begin{cases}
e_{4(j+m+s)+1}\otimes e_{4j},\quad i=j+s;\\
0,\quad\text{otherwise.}
\end{cases}$$

If $s\le j<2s$, then $$b_{ij}=
\begin{cases}
e_{4(j+m+s)+1}\otimes e_{4j},\quad i=j-s;\\
0,\quad\text{otherwise.}
\end{cases}$$

If $2s\le j<3s$, then $$b_{ij}=
\begin{cases}
\kappa^\ell(\a_{3(j+m+3)})w_{4(j+m)+2\ra 4(j+m+1)}\otimes e_{4j+1},\quad i=j;\\
-\kappa^\ell(\a_{3(j+m+3)})w_{4(j+m)+3\ra 4(j+m+1)}\otimes w_{4j+1\ra 4j+2},\quad i=j+2s;\\
-\kappa^\ell(\a_{3(j+m+3)})e_{4(j+m+1)}\otimes w_{4j+1\ra 4j+3},\quad i=j+4s;\\
0,\quad\text{otherwise.}
\end{cases}$$

If $3s\le j<4s$, then $$b_{ij}=
\begin{cases}
\kappa^\ell(\a_{3(j+m+3)})w_{4(j+m)+2\ra 4(j+m+1)}\otimes e_{4j+1},\quad i=j;\\
-\kappa^\ell(\a_{3(j+m+3)})w_{4(j+m)+3\ra 4(j+m+1)}\otimes w_{4j+1\ra 4j+2},\quad i=j+2s;\\
-\kappa^\ell(\a_{3(j+m+3)})e_{4(j+m+1)}\otimes w_{4j+1\ra 4j+3},\quad i=j+3s;\\
0,\quad\text{otherwise.}
\end{cases}$$

If $4s\le j<5s$, then $$b_{ij}=
\begin{cases}
-\kappa^\ell(\a_{3(j+m+2)})e_{4(j+m)+3}\otimes e_{4j+2},\quad i=j;\\
0,\quad\text{otherwise.}
\end{cases}$$

If $5s\le j<6s$, then $$b_{ij}=
\begin{cases}
-\kappa^\ell(\a_{3(j+m+2)})e_{4(j+m)+3}\otimes e_{4j+2},\quad i=j;\\
0,\quad\text{otherwise.}
\end{cases}$$

If $6s\le j<7s$, then $$b_{ij}=
\begin{cases}
-w_{4(j+m+s+1)+1\ra 4(j+m+s+1)+2}\otimes w_{4j+3\ra 4(j+1)},\quad i=(j+s+1)_{2s};\\
e_{4(j+m+s+1)+2}\otimes w_{4j+3\ra 4(j+s+1)+1},\quad i=2s+(j+s+1)_{2s};\\
-w_{4(j+m+1)\ra 4(j+m+s+1)+2}\otimes e_{4j+3},\quad i=j;\\
0,\quad\text{otherwise.}
\end{cases}$$

If $7s\le j<8s$, then $$b_{ij}=
\begin{cases}
w_{4(j+m+s+1)+1\ra 4(j+m+s+1)+2}\otimes w_{4j+3\ra 4(j+1)},\quad i=(j+s+1)_{2s};\\
-e_{4(j+m+s+1)+2}\otimes w_{4j+3\ra 4(j+s+1)+1},\quad i=2s+(j+s+1)_{2s};\\
w_{4(j+m+1)\ra 4(j+m+s+1)+2}\otimes e_{4j+3},\quad i=j-s;\\
0,\quad\text{otherwise.}
\end{cases}$$

$(3)$ If $r_0=2$, then $\Omega^{2}(Y_t^{(9)})$ is described with
$(9s\times 6s)$-matrix with the following elements $b_{ij}${\rm:}

If $0\le j<s$, then $$b_{ij}=
\begin{cases}
\kappa^\ell(\a_{3(j+m+2)})w_{4(j+m)-1\ra 4(j+m)}\otimes e_{4j},\quad i=j;\\
0,\quad\text{otherwise.}
\end{cases}$$

If $s\le j<2s$, then $$b_{ij}=
\begin{cases}
-e_{4(j+m)+1}\otimes w_{4(j+s)+1\ra 4(j+s)+2},\quad i=j+2s;\\
0,\quad\text{otherwise.}
\end{cases}$$

If $2s\le j<3s$, then $$b_{ij}=
\begin{cases}
e_{4(j+m)+2}\otimes e_{4j+1},\quad i=j-s;\\
0,\quad\text{otherwise.}
\end{cases}$$

If $3s\le j<4s$, then $$b_{ij}=
\begin{cases}
-e_{4(j+m+s)+1}\otimes w_{4j+1\ra 4j+2},\quad i=j+s;\\
0,\quad\text{otherwise.}
\end{cases}$$

If $4s\le j<5s$, then $$b_{ij}=
\begin{cases}
e_{4(j+m+s)+2}\otimes e_{4(j+s)+1},\quad i=j-2s;\\
0,\quad\text{otherwise.}
\end{cases}$$

If $5s\le j<6s$, then $$b_{ij}=
\begin{cases}
-w_{4(j+m)+1\ra 4(j+m)+2}\otimes e_{4(j+s)+2},\quad i=j-2s;\\
0,\quad\text{otherwise.}
\end{cases}$$

If $6s\le j<7s$, then $$b_{ij}=
\begin{cases}
-w_{4(j+m)+1\ra 4(j+m)+2}\otimes e_{4(j+s)+2},\quad i=j-2s;\\
0,\quad\text{otherwise.}
\end{cases}$$

If $7s\le j<8s$, then $$b_{ij}=
\begin{cases}
\kappa^\ell(\a_{3(j+m+2)})f_2(j,8s-1)e_{4(j+m)+3}\otimes w_{4j+3\ra 4(j+1)},\quad i=(j+1)_s;\\
0,\quad\text{otherwise.}
\end{cases}$$

If $8s\le j<9s-1$, then $$b_{ij}=
\begin{cases}
\kappa^\ell(\a_{3(j+m+3)})w_{4(j+m)+3\ra 4(j+m+1)}\otimes w_{4j+3\ra 4(j+1)},\quad i=(j+1)_s;\\
0,\quad\text{otherwise.}
\end{cases}$$

If $9s-1\le j<8s$, then $b_{ij}=0$.

If $8s\le j<9s$, then $$b_{ij}=
\begin{cases}
-\kappa^\ell(\a_{3(j+m+3)})e_{4(j+m+1)}\otimes e_{4j+3},\quad i=j-3s;\\
0,\quad\text{otherwise.}
\end{cases}$$

$(4)$ If $r_0=3$, then $\Omega^{3}(Y_t^{(9)})$ is described with
$(8s\times 8s)$-matrix with the following elements $b_{ij}${\rm:}

If $0\le j<s$, then $$b_{ij}=
\begin{cases}
-e_{4(j+m+s)+2}\otimes e_{4j},\quad i=j+s;\\
0,\quad\text{otherwise.}
\end{cases}$$

If $s\le j<2s$, then $$b_{ij}=
\begin{cases}
e_{4(j+m+s)+2}\otimes e_{4j},\quad i=j-s;\\
0,\quad\text{otherwise.}
\end{cases}$$

If $2s\le j<3s$, then $$b_{ij}=
\begin{cases}
-\kappa^\ell(\a_{3(j+m+2)})e_{4(j+m)+3}\otimes e_{4j+1},\quad i=j;\\
0,\quad\text{otherwise.}
\end{cases}$$

If $3s\le j<4s$, then $$b_{ij}=
\begin{cases}
-\kappa^\ell(\a_{3(j+m+2)})e_{4(j+m)+3}\otimes e_{4j+1},\quad i=j;\\
0,\quad\text{otherwise.}
\end{cases}$$

If $4s\le j<5s$, then $$b_{ij}=
\begin{cases}
-\kappa^\ell(\a_{3(j+m+3)})e_{4(j+m+1)}\otimes e_{4j+2},\quad i=j;\\
0,\quad\text{otherwise.}
\end{cases}$$

If $5s\le j<6s$, then $$b_{ij}=
\begin{cases}
-\kappa^\ell(\a_{3(j+m+3)})e_{4(j+m+1)}\otimes e_{4j+2},\quad i=j;\\
0,\quad\text{otherwise.}
\end{cases}$$

If $6s\le j<7s$, then $$b_{ij}=
\begin{cases}
e_{4(j+m+s+1)+1}\otimes e_{4j+3},\quad i=j+s;\\
0,\quad\text{otherwise.}
\end{cases}$$

If $7s\le j<8s$, then $$b_{ij}=
\begin{cases}
e_{4(j+m+s+1)+1}\otimes e_{4j+3},\quad i=j-s;\\
0,\quad\text{otherwise.}
\end{cases}$$

$(5)$ If $r_0=4$, then $\Omega^{4}(Y_t^{(9)})$ is described with
$(6s\times 9s)$-matrix with the following elements $b_{ij}${\rm:}

If $0\le j<s$, then $$b_{ij}=
\begin{cases}
-\kappa^\ell(\a_{3(j+m+1)})e_{4(j+2)-1}\otimes e_{4j},\quad i=j;\\
0,\quad\text{otherwise.}
\end{cases}$$

If $s\le j<2s$, then $$b_{ij}=
\begin{cases}
-w_{4(j+m)+1\ra 4(j+m)+2}\otimes e_{4(j+s)+1},\quad i=j+s;\\
-e_{4(j+m)+2}\otimes w_{4(j+s)+1\ra 4(j+s)+2},\quad i=j+4s;\\
0,\quad\text{otherwise.}
\end{cases}$$

If $2s\le j<3s$, then $$b_{ij}=
\begin{cases}
-w_{4(j+m)+1\ra 4(j+m)+2}\otimes e_{4(j+s)+1},\quad i=j+s;\\
-e_{4(j+m)+2}\otimes w_{4(j+s)+1\ra 4(j+s)+2},\quad i=j+5s;\\
0,\quad\text{otherwise.}
\end{cases}$$

If $3s\le j<4s$, then $$b_{ij}=
\begin{cases}
e_{4(j+m+s)+1}\otimes e_{4(j+s)+2},\quad i=j+s;\\
0,\quad\text{otherwise.}
\end{cases}$$

If $4s\le j<5s$, then $$b_{ij}=
\begin{cases}
e_{4(j+m+s)+1}\otimes e_{4(j+s)+2},\quad i=j+2s;\\
0,\quad\text{otherwise.}
\end{cases}$$

If $5s\le j<6s$, then $$b_{ij}=
\begin{cases}
-\kappa^\ell(\a_{3(j+m+3)})f_1(j,6s-1)e_{4(j+m+1)}\otimes w_{4j+3\ra 4(j+1)},\quad i=s+(j+1)_s;\\
-\kappa^\ell(\a_{3(j+m+3)})w_{4(j+m)+3\ra 4(j+m+1)}\otimes e_{4j+3},\quad i=j+3s;\\
0,\quad\text{otherwise.}
\end{cases}$$

$(6)$ If $r_0=5$, then $\Omega^{5}(Y_t^{(9)})$ is described with
$(7s\times 8s)$-matrix with the following elements $b_{ij}${\rm:}

If $0\le j<s$, then $$b_{ij}=
\begin{cases}
\kappa^\ell(\a_{3(j+m+2)})w_{4(j+m)+1\ra 4(j+m)+3}\otimes e_{4j},\quad i=j;\\
\kappa^\ell(\a_{3(j+m+2)})w_{4(j+m+s)+1\ra 4(j+m)+3}\otimes e_{4j},\quad i=j+s;\\
\kappa^\ell(\a_{3(j+m+2)})e_{4(j+m)+3}\otimes w_{4j\ra 4j+2},\quad i=j+4s;\\
\kappa^\ell(\a_{3(j+m+2)})e_{4(j+m)+3}\otimes w_{4j\ra 4(j+s)+2},\quad i=j+5s;\\
0,\quad\text{otherwise.}
\end{cases}$$

If $s\le j<2s$, then $$b_{ij}=
\begin{cases}
\kappa^\ell(\a_{3(j+m+3)})e_{4(j+m+1)}\otimes e_{4(j+s)+1},\quad i=j+s;\\
\kappa^\ell(\a_{3(j+m+3)})w_{4(j+m)+3\ra 4(j+m+1)}\otimes w_{4(j+s)+1\ra 4(j+s)+2},\quad i=j+3s;\\
0,\quad\text{otherwise.}
\end{cases}$$

If $2s\le j<3s$, then $$b_{ij}=
\begin{cases}
\kappa^\ell(\a_{3(j+m+3)})e_{4(j+m+1)}\otimes e_{4(j+s)+1},\quad i=j+s;\\
\kappa^\ell(\a_{3(j+m+3)})w_{4(j+m)+3\ra 4(j+m+1)}\otimes w_{4(j+s)+1\ra 4(j+s)+2},\quad i=j+3s;\\
0,\quad\text{otherwise.}
\end{cases}$$

If $3s\le j<4s$, then $$b_{ij}=
\begin{cases}
-e_{4(j+m+s+1)+1}\otimes w_{4(j+s)+2\ra 4(j+1)},\quad i=(j+s+1)_{2s};\\
-w_{4(j+m)+3\ra 4(j+m+s+1)+1}\otimes e_{4(j+s)+2},\quad i=j+s;\\
0,\quad\text{otherwise.}
\end{cases}$$

If $4s\le j<5s$, then $$b_{ij}=
\begin{cases}
-e_{4(j+m+s+1)+1}\otimes w_{4(j+s)+2\ra 4(j+1)},\quad i=(j+s+1)_{2s};\\
-w_{4(j+m)+3\ra 4(j+m+s+1)+1}\otimes e_{4(j+s)+2},\quad i=j+s;\\
0,\quad\text{otherwise.}
\end{cases}$$

If $5s\le j<6s$, then $$b_{ij}=
\begin{cases}
-w_{4(j+m+1)+1\ra 4(j+m+1)+2}\otimes w_{4j+3\ra 4(j+1)},\quad i=(j+1)_{2s};\\
-e_{4(j+m+1)+2}\otimes e_{4j+3},\quad i=j+2s;\\
0,\quad\text{otherwise.}
\end{cases}$$

If $6s\le j<7s$, then $$b_{ij}=
\begin{cases}
-w_{4(j+m+1)+1\ra 4(j+m+1)+2}\otimes w_{4j+3\ra 4(j+1)},\quad i=(j+1)_{2s};\\
e_{4(j+m+1)+2}\otimes e_{4j+3},\quad i=j;\\
0,\quad\text{otherwise.}
\end{cases}$$

$(7)$ If $r_0=6$, then $\Omega^{6}(Y_t^{(9)})$ is described with
$(6s\times 9s)$-matrix with the following elements $b_{ij}${\rm:}

If $0\le j<s$, then $$b_{ij}=
\begin{cases}
-\kappa^\ell(\a_{3(j+m+2)})e_{4(j+m)}\otimes e_{4j},\quad i=j;\\
0,\quad\text{otherwise.}
\end{cases}$$

If $s\le j<2s$, then $$b_{ij}=
\begin{cases}
e_{4(j+m+s)+1}\otimes e_{4(j+s)+1},\quad i=j;\\
0,\quad\text{otherwise.}
\end{cases}$$

If $2s\le j<3s$, then $$b_{ij}=
\begin{cases}
e_{4(j+m+s)+1}\otimes e_{4(j+s)+1},\quad i=j+s;\\
0,\quad\text{otherwise.}
\end{cases}$$

If $3s\le j<4s$, then $$b_{ij}=
\begin{cases}
e_{4(j+m+s)+2}\otimes e_{4(j+s)+2},\quad i=j+2s;\\
0,\quad\text{otherwise.}
\end{cases}$$

If $4s\le j<5s$, then $$b_{ij}=
\begin{cases}
e_{4(j+m+s)+2}\otimes e_{4(j+s)+2},\quad i=j+2s;\\
0,\quad\text{otherwise.}
\end{cases}$$

If $5s\le j<6s$, then $$b_{ij}=
\begin{cases}
\kappa^\ell(\a_{3(j+m+2)})e_{4(j+m)+3}\otimes e_{4j+3},\quad i=j+2s;\\
0,\quad\text{otherwise.}
\end{cases}$$

$(8)$ If $r_0=7$, then $\Omega^{7}(Y_t^{(9)})$ is described with
$(6s\times 8s)$-matrix with the following elements $b_{ij}${\rm:}

If $0\le j<s$, then $$b_{ij}=
\begin{cases}
\kappa^\ell(\a_{3(j+m+3)})w_{4(j+m+s)+2\ra 4(j+m+1)}\otimes e_{4j},\quad i=j+s;\\
-\kappa^\ell(\a_{3(j+m+3)})w_{4(j+m)+3\ra 4(j+m+1)}\otimes w_{4j\ra 4j+1},\quad i=j+2s;\\
\kappa^\ell(\a_{3(j+m+3)})e_{4(j+m+1)}\otimes w_{4j\ra 4j+2},\quad i=j+4s;\\
\kappa^\ell(\a_{3(j+m+3)})e_{4(j+m+1)}\otimes w_{4j\ra 4(j+s)+2},\quad i=j+5s;\\
0,\quad\text{otherwise.}
\end{cases}$$

If $s\le j<2s$, then $$b_{ij}=
\begin{cases}
-w_{4(j+m)+3\ra 4(j+m+s+1)+1}\otimes e_{4(j+s)+1},\quad i=j+s;\\
w_{4(j+m+1)\ra 4(j+m+s+1)+1}\otimes w_{4(j+s)+1\ra 4(j+s)+2},\quad i=j+3s;\\
-e_{4(j+m+s+1)+1}\otimes w_{4(j+s)+1\ra 4j+3},\quad i=j+5s;\\
0,\quad\text{otherwise.}
\end{cases}$$

If $2s\le j<3s$, then $$b_{ij}=
\begin{cases}
-w_{4(j+m)+3\ra 4(j+m+s+1)+1}\otimes e_{4(j+s)+1},\quad i=j+s;\\
w_{4(j+m+1)\ra 4(j+m+s+1)+1}\otimes w_{4(j+s)+1\ra 4(j+s)+2},\quad i=j+3s;\\
-e_{4(j+m+s+1)+1}\otimes w_{4(j+s)+1\ra 4j+3},\quad i=j+5s;\\
0,\quad\text{otherwise.}
\end{cases}$$

If $3s\le j<4s$, then $$b_{ij}=
\begin{cases}
-e_{4(j+m+s+1)+2}\otimes w_{4(j+s)+2\ra 4(j+1)},\quad i=(j+s+1)_{2s};\\
-w_{4(j+m+1)\ra 4(j+m+s+1)+2}\otimes e_{4(j+s)+2},\quad i=j+s;\\
w_{4(j+m+s+1)+1\ra 4(j+m+s+1)+2}\otimes w_{4(j+s)+2\ra 4j+3},\quad i=j+3s;\\
0,\quad\text{otherwise.}
\end{cases}$$

If $4s\le j<5s$, then $$b_{ij}=
\begin{cases}
-e_{4(j+m+s+1)+2}\otimes w_{4(j+s)+2\ra 4(j+1)},\quad i=(j+s+1)_{2s};\\
-w_{4(j+m+1)\ra 4(j+m+s+1)+2}\otimes e_{4(j+s)+2},\quad i=j+s;\\
w_{4(j+m+s+1)+1\ra 4(j+m+s+1)+2}\otimes w_{4(j+s)+2\ra 4j+3},\quad i=j+3s;\\
0,\quad\text{otherwise.}
\end{cases}$$

If $5s\le j<6s$, then $$b_{ij}=
\begin{cases}
\kappa^\ell(\a_{3(j+m+3)})w_{4(j+m+s-f_0(j,6s-1)s+1)+2\ra 4(j+m+1)+3}\otimes w_{4j+3\ra 4(j+1)},\quad i=s+(j+1)_s;\\
-\kappa^\ell(\a_{3(j+m+3)})e_{4(j+m+1)+3}\otimes w_{4j+3\ra 4(j+f_0(j,6s-1)s+1)+1},\quad i=2s+(j+1)_s;\\
-\kappa^\ell(\a_{3(j+m+3)})w_{4(j+m+s+1)+1\ra 4(j+m+1)+3}\otimes e_{4j+3},\quad i=j+s;\\
-\kappa^\ell(\a_{3(j+m+3)})w_{4(j+m+1)+1\ra 4(j+m+1)+3}\otimes e_{4j+3},\quad i=j+2s;\\
0,\quad\text{otherwise.}
\end{cases}$$

$(9)$ If $r_0=8$, then $\Omega^{8}(Y_t^{(9)})$ is described with
$(7s\times 6s)$-matrix with the following elements $b_{ij}${\rm:}

If $0\le j<s$, then $$b_{ij}=
\begin{cases}
e_{4(j+m)+1}\otimes w_{4j\ra 4(j+s)+2},\quad i=j+4s;\\
0,\quad\text{otherwise.}
\end{cases}$$

If $s\le j<2s$, then $$b_{ij}=
\begin{cases}
-w_{4(j+m)-1\ra 4(j+m)+1}\otimes e_{4j},\quad i=j-s;\\
e_{4(j+m)+1}\otimes w_{4j\ra 4(j+s)+2},\quad i=j+2s;\\
0,\quad\text{otherwise.}
\end{cases}$$

If $2s\le j<3s$, then $$b_{ij}=
\begin{cases}
-e_{4(j+m)+2}\otimes e_{4j+1},\quad i=j-s;\\
0,\quad\text{otherwise.}
\end{cases}$$

If $3s\le j<4s$, then $$b_{ij}=
\begin{cases}
-e_{4(j+m)+2}\otimes e_{4j+1},\quad i=j-s;\\
0,\quad\text{otherwise.}
\end{cases}$$

If $4s\le j<5s-1$, then $$b_{ij}=
\begin{cases}
-\kappa^\ell(\a_{3(j+m+2)})e_{4(j+m)+3}\otimes w_{4j+2\ra 4(j+1)},\quad i=(j+1)_s;\\
0,\quad\text{otherwise.}
\end{cases}$$

If $5s-1\le j<4s$, then $b_{ij}=0$.

If $4s\le j<5s$, then $$b_{ij}=
\begin{cases}
-\kappa^\ell(\a_{3(j+m+2)})w_{4(j+m+s)+1\ra 4(j+m)+3}\otimes e_{4j+2},\quad i=j-s;\\
0,\quad\text{otherwise.}
\end{cases}$$

If $5s\le j<6s-1$, then $b_{ij}=0$.

If $6s-1\le j<6s$, then $$b_{ij}=
\begin{cases}
\kappa^\ell(\a_{3(j+m+2)})e_{4(j+m)+3}\otimes w_{4j+2\ra 4(j+1)},\quad i=(j+1)_s;\\
0,\quad\text{otherwise.}
\end{cases}$$

If $6s\le j<5s$, then $b_{ij}=0$.

If $5s\le j<6s$, then $$b_{ij}=
\begin{cases}
\kappa^\ell(\a_{3(j+m+2)})w_{4(j+m+s)+1\ra 4(j+m)+3}\otimes e_{4j+2},\quad i=j-s;\\
0,\quad\text{otherwise.}
\end{cases}$$

If $6s\le j<7s$, then $$b_{ij}=
\begin{cases}
-\kappa^\ell(\a_{3(j+m+3)})f_2(j,7s-1)e_{4(j+m+1)}\otimes e_{4j+3},\quad i=j-s;\\
0,\quad\text{otherwise.}
\end{cases}$$

$(10)$ If $r_0=9$, then $\Omega^{9}(Y_t^{(9)})$ is described with
$(6s\times 7s)$-matrix with the following elements $b_{ij}${\rm:}

If $0\le j<s$, then $$b_{ij}=
\begin{cases}
\kappa^\ell(\a_{3(j+m+2)})e_{4(j+m)+3}\otimes e_{4j},\quad i=j;\\
0,\quad\text{otherwise.}
\end{cases}$$

If $s\le j<2s$, then $$b_{ij}=
\begin{cases}
-w_{4(j+m+1)\ra 4(j+m+s+1)+2}\otimes e_{4(j+s)+1},\quad i=j;\\
-e_{4(j+m+s+1)+2}\otimes w_{4(j+s)+1\ra 4j+3},\quad i=j+4s;\\
0,\quad\text{otherwise.}
\end{cases}$$

If $2s\le j<3s$, then $$b_{ij}=
\begin{cases}
-w_{4(j+m+1)\ra 4(j+m+s+1)+2}\otimes e_{4(j+s)+1},\quad i=j;\\
-e_{4(j+m+s+1)+2}\otimes w_{4(j+s)+1\ra 4j+3},\quad i=j+4s;\\
0,\quad\text{otherwise.}
\end{cases}$$

If $3s\le j<4s$, then $$b_{ij}=
\begin{cases}
-e_{4(j+m+1)+1}\otimes e_{4(j+s)+2},\quad i=j;\\
0,\quad\text{otherwise.}
\end{cases}$$

If $4s\le j<5s$, then $$b_{ij}=
\begin{cases}
-e_{4(j+m+1)+1}\otimes e_{4(j+s)+2},\quad i=j;\\
0,\quad\text{otherwise.}
\end{cases}$$

If $5s\le j<6s-1$, then $b_{ij}=0$.

If $6s-1\le j<6s$, then $$b_{ij}=
\begin{cases}
-\kappa^\ell(\a_{3(j+m+s-2)})w_{4(j+m+1)+3\ra 4(j+m+2)}\otimes w_{4j+3\ra 4(j+1)},\quad i=(j+1)_s;\\
0,\quad\text{otherwise.}
\end{cases}$$

If $6s\le j<5s$, then $b_{ij}=0$.

If $5s\le j<6s$, then $$b_{ij}=
\begin{cases}
\kappa^\ell(\a_{3(j+m+s-2)})e_{4(j+m+2)}\otimes w_{4j+3\ra 4(j+f_0(j,6s-1)s+1)+1},\quad i=s+(j+1)_s;\\
\kappa^\ell(\a_{3(j+m+s-2)})w_{4(j+m+s+1)+2\ra 4(j+m+2)}\otimes e_{4j+3},\quad i=j;\\
0,\quad\text{otherwise.}
\end{cases}$$

$(11)$ If $r_0=10$, then $\Omega^{10}(Y_t^{(9)})$ is described with
$(8s\times 6s)$-matrix with the following elements $b_{ij}${\rm:}

If $0\le j<s$, then $$b_{ij}=
\begin{cases}
w_{4(j+m)\ra 4(j+m)+2}\otimes e_{4j},\quad i=j;\\
w_{4(j+m)+1\ra 4(j+m)+2}\otimes w_{4j\ra 4(j+s)+1},\quad i=j+2s;\\
e_{4(j+m)+2}\otimes w_{4j\ra 4(j+s)+2},\quad i=j+4s;\\
0,\quad\text{otherwise.}
\end{cases}$$

If $s\le j<2s$, then $$b_{ij}=
\begin{cases}
w_{4(j+m)\ra 4(j+m)+2}\otimes e_{4j},\quad i=j-s;\\
-w_{4(j+m)+1\ra 4(j+m)+2}\otimes w_{4j\ra 4(j+s)+1},\quad i=j;\\
-e_{4(j+m)+2}\otimes w_{4j\ra 4(j+s)+2},\quad i=j+2s;\\
0,\quad\text{otherwise.}
\end{cases}$$

If $2s\le j<3s$, then $$b_{ij}=
\begin{cases}
\kappa^\ell(\a_{3(j+m+2)})f_2(j,3s-1)w_{4(j+m+s)+1\ra 4(j+m)+3}\otimes e_{4j+1},\quad i=j-s;\\
\kappa^\ell(\a_{3(j+m+2)})f_2(j,3s-1)w_{4(j+m+s)+2\ra 4(j+m)+3}\otimes w_{4j+1\ra 4j+2},\quad i=j+s;\\
\kappa^\ell(\a_{3(j+m+2)})f_2(j,3s-1)e_{4(j+m)+3}\otimes w_{4j+1\ra 4j+3},\quad i=j+3s;\\
0,\quad\text{otherwise.}
\end{cases}$$

If $3s\le j<4s$, then $$b_{ij}=
\begin{cases}
-\kappa^\ell(\a_{3(j+m+2)})f_2(j,4s-1)w_{4(j+m+s)+1\ra 4(j+m)+3}\otimes e_{4j+1},\quad i=j-s;\\
-\kappa^\ell(\a_{3(j+m+2)})f_2(j,4s-1)w_{4(j+m+s)+2\ra 4(j+m)+3}\otimes w_{4j+1\ra 4j+2},\quad i=j+s;\\
\kappa^\ell(\a_{3(j+m+2)})f_2(j,4s-1)e_{4(j+m)+3}\otimes w_{4j+1\ra 4j+3},\quad i=j+2s;\\
0,\quad\text{otherwise.}
\end{cases}$$

If $4s\le j<5s-1$, then $$b_{ij}=
\begin{cases}
\kappa^\ell(\a_{3(j+m+s-3)})e_{4(j+m+1)}\otimes w_{4j+2\ra 4(j+1)},\quad i=(j+1)_s;\\
0,\quad\text{otherwise.}
\end{cases}$$

If $5s-1\le j<4s$, then $b_{ij}=0$.

If $4s\le j<5s$, then $$b_{ij}=
\begin{cases}
\kappa^\ell(\a_{3(j+m+s-3)})w_{4(j+m+s)+2\ra 4(j+m+1)}\otimes e_{4j+2},\quad i=j-s;\\
\kappa^\ell(\a_{3(j+m+s-3)})w_{4(j+m)+3\ra 4(j+m+1)}\otimes w_{4j+2\ra 4j+3},\quad i=j+s;\\
0,\quad\text{otherwise.}
\end{cases}$$

If $5s\le j<6s-1$, then $b_{ij}=0$.

If $6s-1\le j<6s$, then $$b_{ij}=
\begin{cases}
-\kappa^\ell(\a_{3(j+m+s-3)})e_{4(j+m+1)}\otimes w_{4j+2\ra 4(j+1)},\quad i=(j+1)_s;\\
0,\quad\text{otherwise.}
\end{cases}$$

If $6s\le j<5s$, then $b_{ij}=0$.

If $5s\le j<6s$, then $$b_{ij}=
\begin{cases}
-\kappa^\ell(\a_{3(j+m+s-3)})w_{4(j+m+s)+2\ra 4(j+m+1)}\otimes e_{4j+2},\quad i=j-s;\\
0,\quad\text{otherwise.}
\end{cases}$$

If $6s\le j<7s$, then $$b_{ij}=
\begin{cases}
e_{4(j+m+1)+1}\otimes w_{4j+3\ra 4(j+s+1)+1},\quad i=s+(j+s+1)_{2s};\\
w_{4(j+m)+3\ra 4(j+m+1)+1}\otimes e_{4j+3},\quad i=j-s;\\
0,\quad\text{otherwise.}
\end{cases}$$

If $7s\le j<6s$, then $b_{ij}=0$.

If $6s\le j<7s-1$, then $$b_{ij}=
\begin{cases}
w_{4(j+m+1)\ra 4(j+m+1)+1}\otimes w_{4j+3\ra 4(j+s-f_0(j,7s-1)s+1)},\quad i=(j+1)_s;\\
0,\quad\text{otherwise.}
\end{cases}$$

If $7s-1\le j<8s-1$, then $b_{ij}=0$.

If $8s-1\le j<8s$, then $$b_{ij}=
\begin{cases}
w_{4(j+m+1)\ra 4(j+m+1)+1}\otimes w_{4j+3\ra 4(j+1)},\quad i=(j+1)_s;\\
0,\quad\text{otherwise.}
\end{cases}$$

If $8s\le j<7s$, then $b_{ij}=0$.

If $7s\le j<8s$, then $$b_{ij}=
\begin{cases}
e_{4(j+m+1)+1}\otimes w_{4j+3\ra 4(j+s+1)+1},\quad i=s+(j+s+1)_{2s};\\
0,\quad\text{otherwise.}
\end{cases}$$

\medskip
$({\rm II})$ Represent an arbitrary $t_0\in\N$ in the form
$t_0=11\ell_0+r_0$, where $0\le r_0\le 10.$ Then
$\Omega^{t_0}(Y_t^{(9)})$ is a $\Omega^{r_0}(Y_t^{(9)})$, whose left
components twisted by $\sigma^{\ell_0}$,
and coefficients multiplied by $(-1)^{\ell_0}$.
\end{pr}

\begin{pr}[Translates for the case 10]
$({\rm I})$ Let $r_0\in\N$, $r_0<11$. Denote by $$\kappa_0=\begin{cases}
\kappa^\ell(\a_6),\quad s=1;\\-\kappa^\ell(\g_2)\kappa^\ell(\a_6),\quad s>1.
\end{cases}$$ Then $r_0$-translates of the
elements $Y^{(10)}_t$ are described by the following way.

$(1)$ If $r_0=0$, then $\Omega^{0}(Y_t^{(10)})$ is described with
$(9s\times 6s)$-matrix with one nonzero element that is of the following form{\rm:}
$$b_{0,s}=\kappa^\ell(\a_6)w_{4j\ra 4(j+1)}\otimes e_{4j}.$$

$(2)$ If $r_0=1$, then $\Omega^{1}(Y_t^{(10)})$ is described with
$(8s\times 7s)$-matrix with the following two nonzero elements{\rm:}
$$b_{(j+s+1)_{2s},2s}=f_2(s,1)\kappa_0w_{4(j+m+s+1)+1\ra 4(j+m+s+2)}\otimes w_{4j+1\ra 4(j+1)};$$
$$b_{(j+s+1)_{2s},3s}=-f_2(s,1)\kappa_0w_{4(j+m+s+1)+1\ra 4(j+m+s+2)}\otimes w_{4j+1\ra 4(j+1)}.$$

$(3)$ If $r_0=2$, then $\Omega^{2}(Y_t^{(10)})$ is described with
$(9s\times 6s)$-matrix with one nonzero element that is of the following form{\rm:}
$$b_{(j+1)_s,0}=\kappa_0w_{4(j+m)+3\ra 4(j+m+1)}\otimes w_{4j\ra 4(j+1)}.$$

$(4)$ If $r_0=3$, then $\Omega^{3}(Y_t^{(10)})$ is described with
$(8s\times 8s)$-matrix with the following two nonzero elements{\rm:}
$$b_{(j+1)_{2s},0}=-\kappa^\ell(\a_{3(j+m+2)})\kappa_0e_{4(j+m+1)+2}\otimes w_{4j\ra 4(j+1)};$$
$$b_{(j+1)_{2s},s}=-\kappa^\ell(\a_{3(j+m+2)})\kappa_0e_{4(j+m+1)+2}\otimes w_{4j\ra 4(j+1)}.$$

$(5)$ If $r_0=4$, then $\Omega^{4}(Y_t^{(10)})$ is described with
$(6s\times 9s)$-matrix with one nonzero element that is of the following form{\rm:}
$$b_{(j+1)_s,0}=-f_2(s,1)\kappa_0e_{4(j+m)+3}\otimes w_{4j\ra 4(j+1)}.$$

$(6)$ If $r_0=5$, then $\Omega^{5}(Y_t^{(10)})$ is described with
$(7s\times 8s)$-matrix with the following two nonzero elements{\rm:}
$$b_{(j+s+1)_{2s},0}=-f_2(s,1)\kappa^\ell(\g_{j+m})\kappa_0w_{4(j+m+s+1)+1\ra 4(j+m+1)+3}\otimes w_{4j\ra 4(j+1)};$$
$$b_{(j+1)_{2s},0}=-f_2(s,1)\kappa^\ell(\g_{j+m})\kappa_0w_{4(j+m+1)+1\ra 4(j+m+1)+3}\otimes w_{4j\ra 4(j+1)}.$$

$(7)$ If $r_0=6$, then $\Omega^{6}(Y_t^{(10)})$ is described with
$(6s\times 9s)$-matrix with one nonzero element that is of the following form{\rm:}
$$b_{(j+1)_s,0}=-f_2(s,1)\kappa_0e_{4(j+m+1)}\otimes w_{4j\ra 4(j+1)}.$$

$(8)$ If $r_0=7$, then $\Omega^{7}(Y_t^{(10)})$ is described with
$(6s\times 8s)$-matrix with one nonzero element that is of the following form{\rm:}
$$b_{(j+1)_s,0}=\kappa_1\kappa_0w_{4(j+m+sf(j,s-1)+1)+2\ra 4(j+m+2)}\otimes w_{4j\ra 4(j+1)},$$
where $$\kappa_1=\begin{cases}-\kappa^\ell(\g_m),\quad s=1;\\
\kappa^\ell(\g_m)\kappa^\ell(\g_{m+1})\kappa^\ell(\g_{m+2}),\quad s>1.\end{cases}$$

$(9)$ If $r_0=8$, then $\Omega^{8}(Y_t^{(10)})$ is described with
$(7s\times 6s)$-matrix with one nonzero element that is of the following form{\rm:}
$$b_{(j+1)_s,s(1-f(s,1))}=\kappa_1\kappa_0w_{4(j+m)+3\ra 4(j+m+s+1)+1}\otimes w_{4j\ra 4(j+1)},$$
where $$\kappa_1=\begin{cases}-\kappa^\ell(\a_{3(m+4)}),\quad s=1;\\
\kappa^\ell(\g_{m-1})\kappa^\ell(\g_m)\kappa^\ell(\g_{m+1})\kappa^{\ell+1}(\a_0),\quad s>1.\end{cases}$$

$(10)$ If $r_0=9$ and $s=1$, then $\Omega^{9}(Y_t^{(10)})$ is described with
$(6s\times 7s)$-matrix with one nonzero element that is of the following form{\rm:}
$$b_{(j+1)_s,0}=-\kappa^\ell(\g_{j+m})\kappa_0e_{4(j+m+1)+3}\otimes w_{4j\ra 4(j+1)}.$$

$(11)$ If $r_0=9$ and $s>1$, then $\Omega^{9}(Y_t^{(10)})$ is described with
$(6s\times 7s)$-matrix with the following two nonzero elements{\rm:}
$$b_{5s+j,0}=-\kappa_1\kappa_0w_{4(j+m+1)+2\ra 4(j+m+1)+3}\otimes w_{4j\ra 4j+3};$$
$$b_{6s+j,0}=\kappa_1\kappa_0w_{4(j+m+s+1)+2\ra 4(j+m+1)+3}\otimes w_{4j\ra 4j+3},$$
where $\kappa_1=-\kappa^\ell(\g_{m-1})\kappa^\ell(\g_m)\kappa^\ell(\g_{m+1})$.

$(12)$ If $r_0=10$ and $s=1$, then $\Omega^{10}(Y_t^{(10)})$ is described with
$(8s\times 6s)$-matrix with the following two nonzero elements{\rm:}
$$b_{(j+1)_s,0}=-\kappa^\ell(\a_{3(j+m+3)})\kappa_0w_{4(j+m+1)\ra 4(j+m+s+1)+2}\otimes w_{4j\ra 4(j+1)};$$
$$b_{(j+1)_s,s}=\kappa^\ell(\a_{3(j+m+3)})\kappa_0w_{4(j+m+1)\ra 4(j+m+s+1)+2}\otimes w_{4j\ra 4(j+1)}.$$

$(13)$ If $r_0=10$ and $s>1$, then $\Omega^{10}(Y_t^{(10)})$ is described with
$(8s\times 6s)$-matrix with the following nonzero elements{\rm:}
$$b_{(j+1)_s,0}=\kappa_1\kappa_0w_{4(j+m+1)\ra 4(j+m+s+1)+2}\otimes w_{4j\ra 4(j+1)};$$
$$b_{5s+j,0}=-\kappa_1\kappa_0w_{4(j+m+1)-1\ra 4(j+m+s+1)+2}\otimes w_{4j\ra 4(j+1)-1};$$
$$b_{(j+1)_s,s}=\kappa_1\kappa_0w_{4(j+m+1)\ra 4(j+m+s+1)+2}\otimes w_{4j\ra 4(j+1)};$$
$$b_{4s+j,s}=-\kappa_1\kappa_0w_{4(j+m+1)-1\ra 4(j+m+s+1)+2}\otimes w_{4j\ra 4(j+1)-1};$$
$$b_{5s+(j+1)_s,7s-1+sf((j)_s,0)}=\kappa_1\kappa_0w_{4(j+m+s+1)+3\ra 4(j+m+s+1)+5}\otimes w_{4j+3\ra 4(j+1)+3},$$
where $\kappa_1=\kappa^\ell(\g_{m-2})\kappa^\ell(\g_{m-1})\kappa^\ell(\g_m)\kappa^\ell(\g_{m+1})\kappa^{\ell+1}(\a_3)$.

\medskip
$({\rm II})$ Represent an arbitrary $t_0\in\N$ in the form
$t_0=11\ell_0+r_0$, where $0\le r_0\le 10.$ Then
$\Omega^{t_0}(Y_t^{(10)})$ is a $\Omega^{r_0}(Y_t^{(10)})$, whose left
components twisted by $\sigma^{\ell_0}$,
and coefficients multiplied by $(-1)^{\ell_0}$.
\end{pr}

\begin{pr}[Translates for the case 11]
$({\rm I})$ Let $r_0\in\N$, $r_0<11$. $r_0$-translates of the
elements $Y^{(11)}_t$ are described by the following way.

$(1)$ If $r_0=0$, then $\Omega^{0}(Y_t^{(11)})$ is described with
$(8s\times 6s)$-matrix with the following elements $b_{ij}${\rm:}

If $0\le j<2s$, then $$b_{ij}=
\begin{cases}
w_{4(j+m)\ra 4(j+m)+1}\otimes e_{4j},\quad i=(j)_s;\\
0,\quad\text{otherwise.}
\end{cases}$$

If $2s\le j<4s$, then $$b_{ij}=
\begin{cases}
\kappa^\ell(\a_{3(j+m+2+1)})w_{4(j+m)+1\ra 4(j+m+1)}\otimes e_{4j+1},\quad i=j-s;\\
0,\quad\text{otherwise.}
\end{cases}$$

If $4s\le j<6s$, then $$b_{ij}=
\begin{cases}
-\kappa^\ell(\a_{3(j+m+2)})w_{4(j+m)+2\ra 4(j+m)+3}\otimes e_{4j+2},\quad i=j-s;\\
0,\quad\text{otherwise.}
\end{cases}$$

If $6s\le j<8s$, then $$b_{ij}=
\begin{cases}
f_1(j,7s)w_{4(j+m)+3\ra 4(j+m+1)+2}\otimes e_{4j+3},\quad i=5s+(j)_s;\\
0,\quad\text{otherwise.}
\end{cases}$$

$(2)$ If $r_0=1$, then $\Omega^{1}(Y_t^{(11)})$ is described with
$(9s\times 7s)$-matrix with the following elements $b_{ij}${\rm:}

If $0\le j<s$, then $$b_{ij}=
\begin{cases}
\kappa_1w_{4(j+m)+1\ra 4(j+m+1)}\otimes e_{4j},\quad i=j;\\
\kappa_1w_{4(j+m+s)+1\ra 4(j+m+1)}\otimes e_{4j},\quad i=j+s;\\
-\kappa_1w_{4(j+m)+2\ra 4(j+m+1)}\otimes w_{4j\ra 4j+1},\quad i=j+2s;\\
-\kappa_1w_{4(j+m+s)+2\ra 4(j+m+1)}\otimes w_{4j\ra 4(j+s)+1},\quad i=j+3s;\\
0,\quad\text{otherwise,}
\end{cases}$$
where $\kappa_1=\kappa^\ell(\a_{3(j+m+3)})$.

If $s\le j<2s$, then $$b_{ij}=
\begin{cases}
w_{4(j+m+s)+2\ra 4(j+m+s+1)+1}\otimes e_{4(j+s)+1},\quad i=j+s;\\
w_{4(j+m)+3\ra 4(j+m+s+1)+1}\otimes w_{4(j+s)+1\ra 4(j+s)+2},\quad i=j+3s;\\
w_{4(j+m+1)\ra 4(j+m+s+1)+1}\otimes w_{4(j+s)+1\ra 4j+3},\quad i=j+5s;\\
0,\quad\text{otherwise.}
\end{cases}$$

If $2s\le j<3s$, then $$b_{ij}=
\begin{cases}
w_{4(j+m+s+1)+1\ra 4(j+m+s+1)+2}\otimes w_{4j+1\ra 4(j+1)},\quad i=(j+s+1)_{2s};\\
-w_{4(j+m)+3\ra 4(j+m+s+1)+2}\otimes w_{4j+1\ra 4j+2},\quad i=j+2s;\\
0,\quad\text{otherwise.}
\end{cases}$$

If $3s\le j<4s$, then $$b_{ij}=
\begin{cases}
w_{4(j+m)+2\ra 4(j+m+1)+1}\otimes e_{4j+1},\quad i=j;\\
w_{4(j+m)+3\ra 4(j+m+1)+1}\otimes w_{4j+1\ra 4j+2},\quad i=j+2s;\\
w_{4(j+m+1)\ra 4(j+m+1)+1}\otimes w_{4j+1\ra 4j+3},\quad i=j+3s;\\
0,\quad\text{otherwise.}
\end{cases}$$

If $4s\le j<5s$, then $$b_{ij}=
\begin{cases}
w_{4(j+m+1)+1\ra 4(j+m+1)+2}\otimes w_{4(j+s)+1\ra 4(j+1)},\quad i=(j+1)_{2s};\\
-w_{4(j+m)+3\ra 4(j+m+1)+2}\otimes w_{4(j+s)+1\ra 4(j+s)+2},\quad i=j+s;\\
0,\quad\text{otherwise.}
\end{cases}$$

If $5s\le j<7s$, then $$b_{ij}=
\begin{cases}
-w_{4(j+m)+3\ra 4(j+m+s+1)+2}\otimes e_{4(j+s)+2},\quad i=j-s;\\
0,\quad\text{otherwise.}
\end{cases}$$

If $7s\le j<8s$, then $$b_{ij}=
\begin{cases}
\kappa_1w_{4(j+m+s+1+sf(j,8s-1))+1\ra 4(j+m+1)+3}\otimes w_{4j+3\ra 4(j+1)},\quad i=(j+1)_s;\\
\kappa_1w_{4(j+m+1+sf(j,8s-1))+1\ra 4(j+m+1)+3}\otimes w_{4j+3\ra 4(j+1)},\quad i=s+(j+1)_s;\\
\kappa_1w_{4(j+m+s+1+sf(j,8s-1))+2\ra 4(j+m+1)+3}\otimes w_{4j+3\ra 4(j+s+1+sf(j,8s-1))+1},\quad i=2s+(j+1)_s;\\
\kappa_1w_{4(j+m+1+sf(j,8s-1))+2\ra 4(j+m+1)+3}\otimes w_{4j+3\ra 4(j+1+sf(j,8s-1))+1},\quad i=3s+(j+1)_s;\\
\kappa_1w_{4(j+m+1)\ra 4(j+m+1)+3}\otimes e_{4j+3},\quad i=j-s;\\
0,\quad\text{otherwise,}
\end{cases}$$
where $\kappa_1=-\kappa^\ell(\a_{3(j+m+3)})$.

If $8s\le j<9s$, then $$b_{ij}=
\begin{cases}
\kappa_1w_{4(j+m+s+1)+1\ra 4(j+m+2)}\otimes w_{4j+3\ra 4(j+1)},\quad i=(j+s+1)_{2s};\\
\kappa_1w_{4(j+m+1)\ra 4(j+m+2)}\otimes e_{4j+3},\quad i=j-2s;\\
0,\quad\text{otherwise,}
\end{cases}$$
where $\kappa_1=\kappa^\ell(\a_{3(j+m+4)})$.

$(3)$ If $r_0=2$, then $\Omega^{2}(Y_t^{(11)})$ is described with
$(8s\times 6s)$-matrix with the following elements $b_{ij}${\rm:}

If $0\le j<2s$, then $$b_{ij}=
\begin{cases}
-f_1(j,s)w_{4(j+m)-1\ra 4(j+m)+2}\otimes e_{4j},\quad i=(j)_s;\\
-e_{4(j+m)+2}\otimes w_{4j\ra 4j+1},\quad i=j+s;\\
0,\quad\text{otherwise.}
\end{cases}$$

If $2s\le j<4s$, then $$b_{ij}=
\begin{cases}
\kappa^{\ell+1}(\a_{3(j+m+2)})e_{4(j+m)+3}\otimes w_{4j+1\ra 4(j+1)},\quad i=(j+1)_s;\\
0,\quad\text{otherwise.}
\end{cases}$$

If $4s\le j<6s$, then $$b_{ij}=
\begin{cases}
-\kappa^\ell(\a_{3(j+m+3)})w_{4(j+m+s)+1\ra 4(j+m+1)}\otimes e_{4j+2},\quad i=j-s;\\
0,\quad\text{otherwise.}
\end{cases}$$

If $6s\le j<8s$, then $$b_{ij}=
\begin{cases}
-f_2((j)_s,s-1)f_1(j,7s)w_{4(j+m)+3\ra 4(j+m+1)+1}\otimes w_{4j+3\ra 4(j+1)},\quad i=(j+1)_s;\\
e_{4(j+m+1)+1}\otimes w_{4j+3\ra 4(j+s+1)+2},\quad i=3s+(j+s+1)_{2s};\\
0,\quad\text{otherwise.}
\end{cases}$$

$(4)$ If $r_0=3$, then $\Omega^{3}(Y_t^{(11)})$ is described with
$(6s\times 8s)$-matrix with the following elements $b_{ij}${\rm:}

If $0\le j<s$, then $$b_{ij}=
\begin{cases}
-\kappa_1w_{4(j+m+s)+2\ra 4(j+m)+3}\otimes e_{4j},\quad i=j+s;\\
\kappa_1e_{4(j+m)+3}\otimes w_{4j\ra 4j+1},\quad i=j+2s;\\
0,\quad\text{otherwise,}
\end{cases}$$
where $\kappa_1=-\kappa^\ell(\a_{3(j+m+2)})$.

If $s\le j<3s$, then $$b_{ij}=
\begin{cases}
f_2((j)_s,s-1)f_1(j,2s)e_{4(j+m+s+1)+2}\otimes w_{4(j+s)+1\ra 4(j+1)},\quad i=(j+s+1)_{2s};\\
0,\quad\text{otherwise.}
\end{cases}$$

If $3s\le j<4s$, then $$b_{ij}=
\begin{cases}
-w_{4(j+m+1)\ra 4(j+m+1)+1}\otimes e_{4(j+s)+2},\quad i=j+s;\\
-e_{4(j+m+1)+1}\otimes w_{4(j+s)+2\ra 4j+3},\quad i=j+4s;\\
0,\quad\text{otherwise.}
\end{cases}$$

If $4s\le j<5s$, then $$b_{ij}=
\begin{cases}
-w_{4(j+m+1)\ra 4(j+m+1)+1}\otimes e_{4(j+s)+2},\quad i=j+s;\\
-e_{4(j+m+1)+1}\otimes w_{4(j+s)+2\ra 4j+3},\quad i=j+2s;\\
0,\quad\text{otherwise.}
\end{cases}$$

If $5s\le j<6s$, then $$b_{ij}=
\begin{cases}
\kappa_1w_{4(j+m+s+1+sf(j,6s-1))+2\ra 4(j+m+2)}\otimes w_{4j+3\ra 4(j+1)},\quad i=(j+1)_s;\\
\kappa_1w_{4(j+m+1+sf(j,6s-1))+2\ra 4(j+m+2)}\otimes w_{4j+3\ra 4(j+1)},\quad i=s+(j+1)_s;\\
\kappa_1w_{4(j+m+1)+3\ra 4(j+m+2)}\otimes w_{4j+3\ra 4(j+s+1+sf(j,6s-1))+1},\quad i=2s+(j+1)_s;\\
0,\quad\text{otherwise,}
\end{cases}$$
where $\kappa_1=-\kappa^\ell(\a_{3(j+m+4)})$.

$(5)$ If $r_0=4$, then $\Omega^{4}(Y_t^{(11)})$ is described with
$(7s\times 9s)$-matrix with the following elements $b_{ij}${\rm:}

If $0\le j<s$, then $$b_{ij}=
\begin{cases}
\kappa_1w_{4(j+m)-1\ra 4(j+m)+3}\otimes e_{4j},\quad i=j;\\
-\kappa_1e_{4(j+m)+3}\otimes w_{4j\ra 4(j+1)},\quad i=(j+1)_s,\text{ }j<s-1;\\
-\kappa_1w_{4(j+m)+1\ra 4(j+m)+3}\otimes w_{4j\ra 4(j+s)+1},\quad i=j+3s;\\
\kappa_1w_{4(j+m+s)+2\ra 4(j+m)+3}\otimes w_{4j\ra 4j+2},\quad i=j+5s;\\
0,\quad\text{otherwise,}
\end{cases}$$
where $\kappa_1=-\kappa^\ell(\a_{3(j+m+2)})$.

If $s\le j<3s$, then $$b_{ij}=
\begin{cases}
\kappa_1f_2((j)_s,s-1)f_1(j,2s)e_{4(j+m+1)}\otimes w_{4(j+s)+1\ra 4(j+1)},\quad i=s+(j+1)_s;\\
-\kappa_1w_{4(j+m)+1\ra 4(j+m+1)}\otimes e_{4(j+s)+1},\quad i=j+s;\\
0,\quad\text{otherwise,}
\end{cases}$$
where $\kappa_1=-\kappa^\ell(\a_{3(j+m+3)})$.

If $3s\le j<4s$, then $$b_{ij}=
\begin{cases}
w_{4(j+m)+3\ra 4(j+m+1)+1}\otimes w_{4(j+s)+2\ra 4(j+1)},\quad i=(j+1)_s,\text{ }j<4s-1;\\
-f_2((j)_s,s-1)w_{4(j+m+1)\ra 4(j+m+1)+1}\otimes w_{4(j+s)+2\ra 4(j+1)},\quad i=s+(j+1)_s;\\
-w_{4(j+m)+2\ra 4(j+m+1)+1}\otimes e_{4(j+s)+2},\quad i=j+2s;\\
0,\quad\text{otherwise.}
\end{cases}$$

If $4s\le j<5s$, then $$b_{ij}=
\begin{cases}
w_{4(j+m)+3\ra 4(j+m+1)+1}\otimes w_{4(j+s)+2\ra 4(j+1)},\quad i=(j+1)_s,\text{ }j=5s-1;\\
f_2(j,5s-1)w_{4(j+m+1)\ra 4(j+m+1)+1}\otimes w_{4(j+s)+2\ra 4(j+1)},\quad i=s+(j+1)_s;\\
-w_{4(j+m)+2\ra 4(j+m+1)+1}\otimes e_{4(j+s)+2},\quad i=j+3s;\\
0,\quad\text{otherwise.}
\end{cases}$$

If $5s\le j<6s$, then $$b_{ij}=
\begin{cases}
w_{4(j+m+s+1)+1\ra 4(j+m+s+1)+2}\otimes w_{4j+3\ra 4(j+1)+1},\quad i=2s+(j+1)_{2s};\\
-w_{4(j+m)+3\ra 4(j+m+s+1)+2}\otimes w_{4j+3\ra 4(j+1)},\quad i=(j+1)_s,\text{ }j<6s-1;\\
e_{4(j+m+s+1)+2}\otimes w_{4j+3\ra 4(j+1)+2},\quad i=7s+(j+1)_s,\text{ }j<6s-1;\\
e_{4(j+m+s+1)+2}\otimes w_{4j+3\ra 4(j+1)+2},\quad i=5s+(j+1)_s,\text{ }j=6s-1;\\
0,\quad\text{otherwise.}
\end{cases}$$

If $6s\le j<7s$, then $$b_{ij}=
\begin{cases}
w_{4(j+m+s+1)+1\ra 4(j+m+s+1)+2}\otimes w_{4j+3\ra 4(j+1)+1},\quad i=2s+(j+1)_{2s};\\
e_{4(j+m+s+1)+2}\otimes w_{4j+3\ra 4(j+1)+2},\quad i=5s+(j+1)_s,\text{ }j<7s-1;\\
-w_{4(j+m)+3\ra 4(j+m+s+1)+2}\otimes w_{4j+3\ra 4(j+1)},\quad i=(j+1)_s,\text{ }j=7s-1;\\
e_{4(j+m+s+1)+2}\otimes w_{4j+3\ra 4(j+1)+2},\quad i=7s+(j+1)_s,\text{ }j=7s-1;\\
0,\quad\text{otherwise.}
\end{cases}$$

$(6)$ If $r_0=5$, then $\Omega^{5}(Y_t^{(11)})$ is described with
$(6s\times 8s)$-matrix with the following elements $b_{ij}${\rm:}

If $0\le j<s$, then $$b_{ij}=
\begin{cases}
\kappa^\ell(\a_{3(j+m+3)})e_{4(j+m+1)}\otimes w_{4j\ra 4j+1},\quad i=j+2s;\\
0,\quad\text{otherwise.}
\end{cases}$$

If $s\le j<3s$, then $$b_{ij}=
\begin{cases}
w_{4(j+m+1)\ra 4(j+m+1)+1}\otimes e_{4(j+s)+1},\quad i=j+s;\\
0,\quad\text{otherwise.}
\end{cases}$$

If $3s\le j<5s$, then $$b_{ij}=
\begin{cases}
w_{4(j+m+1)+1\ra 4(j+m+1)+2}\otimes w_{4(j+s)+2\ra 4(j+1)},\quad i=(j+1)_{2s};\\
w_{4(j+m)+3\ra 4(j+m+1)+2}\otimes e_{4(j+s)+2},\quad i=j+s;\\
f_1(j,4s)e_{4(j+m+1)+2}\otimes w_{4(j+s)+2\ra 4j+3},\quad i=6s+(j)_{2s};\\
0,\quad\text{otherwise.}
\end{cases}$$

If $5s\le j<6s$, then $$b_{ij}=
\begin{cases}
-\kappa_1w_{4(j+m+s+1+sf(j,6s-1))+1\ra 4(j+m+1)+3}\otimes w_{4j+3\ra 4(j+1)},\quad i=(j+1)_s;\\
-\kappa_1w_{4(j+m+1+sf(j,6s-1))+1\ra 4(j+m+1)+3}\otimes w_{4j+3\ra 4(j+1)},\quad i=s+(j+1)_s;\\
\kappa_1e_{4(j+m+1)+3}\otimes w_{4j+3\ra 4(j+s+1+sf(j,6s-1))+2},\quad i=4s+(j+1)_s;\\
\kappa_1e_{4(j+m+1)+3}\otimes w_{4j+3\ra 4(j+1+sf(j,6s-1))+2},\quad i=5s+(j+1)_s;\\
0,\quad\text{otherwise,}
\end{cases}$$
where $\kappa_1=\kappa^\ell(\a_{3(j+m+3)})f_2((j)_s,s-1)$.

$(7)$ If $r_0=6$, then $\Omega^{6}(Y_t^{(11)})$ is described with
$(6s\times 9s)$-matrix with the following elements $b_{ij}${\rm:}

If $0\le j<s$, then $$b_{ij}=
\begin{cases}
\kappa_1w_{4(j+m)+1\ra 4(j+m+1)}\otimes w_{4j\ra 4j+1},\quad i=j+s;\\
-\kappa_1w_{4(j+m+s)+2\ra 4(j+m+1)}\otimes w_{4j\ra 4j+1},\quad i=j+2s;\\
\kappa_1w_{4(j+m)+2\ra 4(j+m+1)}\otimes w_{4j\ra 4(j+s)+1},\quad i=j+4s;\\
\kappa_1e_{4(j+m+1)}\otimes w_{4j\ra 4(j+1)},\quad i=(j+1)_s,\text{ }j=s-1;\\
0,\quad\text{otherwise,}
\end{cases}$$
where $\kappa_1=\kappa^\ell(\a_{3(j+m+3)})$.

If $s\le j<3s$, then $$b_{ij}=
\begin{cases}
w_{4(j+m)+2\ra 4(j+m+1)+1}\otimes e_{4(j+s)+1},\quad i=j+2s-sf_0(j,2s);\\
-w_{4(j+m+1)\ra 4(j+m+1)+1}\otimes w_{4(j+s)+1\ra 4(j+1)},\quad i=(j+1)_s,\text{ }2s-1\le j<3s-1;\\
0,\quad\text{otherwise.}
\end{cases}$$

If $3s\le j<5s$, then $$b_{ij}=
\begin{cases}
-w_{4(j+m)+3\ra 4(j+m+1)+2}\otimes w_{4(j+s)+2\ra 4j+3},\quad i=7s+(j)_s;\\
w_{4(j+m+1)\ra 4(j+m+1)+2}\otimes w_{4(j+s)+2\ra 4(j+1)},\quad i=(j+1)_s,\text{ }4s-1\le j<5s-1;\\
0,\quad\text{otherwise.}
\end{cases}$$

If $5s\le j<6s$, then $$b_{ij}=
\begin{cases}
\kappa_1w_{4(j+m+1)\ra 4(j+m+1)+3}\otimes w_{4j+3\ra 4(j+1)},\quad i=(j+1)_s;\\
\kappa_1w_{4(j+m+s+1+sf(j,6s-1))+1\ra 4(j+m+1)+3}\otimes w_{4j+3\ra 4(j+s+1+sf(j,6s-1))+1},\quad i=s+(j+1)_s;\\
\kappa_1w_{4(j+m+1+sf(j,6s-1))+2\ra 4(j+m+1)+3}\otimes w_{4j+3\ra 4(j+s+1+sf(j,6s-1))+1},\quad i=2s+(j+1)_s;\\
-\kappa_1w_{4(j+m+1+sf(j,6s-1))+1\ra 4(j+m+1)+3}\otimes w_{4j+3\ra 4(j+1+sf(j,6s-1))+1},\quad i=3s+(j+1)_s;\\
-\kappa_1w_{4(j+m+s+1+sf(j,6s-1))+2\ra 4(j+m+1)+3}\otimes w_{4j+3\ra 4(j+1+sf(j,6s-1))+1},\quad i=4s+(j+1)_s;\\
0,\quad\text{otherwise,}
\end{cases}$$
where $\kappa_1=-\kappa^\ell(\a_{3(j+m+3)})f_2(j,6s-1)$.

$(8)$ If $r_0=7$, then $\Omega^{7}(Y_t^{(11)})$ is described with
$(7s\times 8s)$-matrix with the following elements $b_{ij}${\rm:}

If $s\le j<2s$, then $$b_{ij}=
\begin{cases}
w_{4(j+m)+3\ra 4(j+m+s+1)+1}\otimes w_{4j\ra 4(j+s)+1},\quad i=j+s;\\
w_{4(j+m+1)\ra 4(j+m+s+1)+1}\otimes w_{4j\ra 4(j+s)+2},\quad i=j+3s;\\
w_{4(j+m+1)\ra 4(j+m+s+1)+1}\otimes w_{4j\ra 4j+2},\quad i=j+4s;\\
0,\quad\text{otherwise.}
\end{cases}$$

If $2s\le j<4s$, then $$b_{ij}=
\begin{cases}
w_{4(j+m)+3\ra 4(j+m+s+1)+2}\otimes e_{4j+1},\quad i=j;\\
0,\quad\text{otherwise.}
\end{cases}$$

If $4s\le j<6s$, then $$b_{ij}=
\begin{cases}
\kappa_1w_{4(j+m+s+1)+2\ra 4(j+m+1)+3}\otimes w_{4j+2\ra 4(j+1)},\quad i=(j+s+1)_{2s};\\
\kappa_1w_{4(j+m+1)\ra 4(j+m+1)+3}\otimes e_{4j+2},\quad i=j;\\
-\kappa_1w_{4(j+m+1)+1\ra 4(j+m+1)+3}\otimes w_{4j+2\ra 4j+3},\quad i=j+2s;\\
0,\quad\text{otherwise,}
\end{cases}$$
where $\kappa_1=\kappa^\ell(\a_{3(j+m+3)})f_1(j,5s)$.

If $6s\le j<7s$, then $$b_{ij}=
\begin{cases}
\kappa_1w_{4(j+m+s+1+sf(j,7s-1))+2\ra 4(j+m+2)}\otimes w_{4j+3\ra 4(j+1)},\quad i=s+(j+1)_s;\\
-\kappa_1w_{4(j+m+1)+3\ra 4(j+m+2)}\otimes w_{4j+3\ra 4(j+1+sf(j,7s-1))+1},\quad i=2s+(j+1)_s;\\
-\kappa_1e_{4(j+m+2)}\otimes w_{4j+3\ra 4(j+1+sf(j,7s-1))+2},\quad i=4s+(j+1)_s;\\
-\kappa_1e_{4(j+m+2)}\otimes w_{4j+3\ra 4(j+s+1+sf(j,7s-1))+2},\quad i=5s+(j+1)_s;\\
0,\quad\text{otherwise,}
\end{cases}$$
where $\kappa_1=-\kappa^\ell(\a_{3(j+m+4)})$.

$(9)$ If $r_0=8$, then $\Omega^{8}(Y_t^{(11)})$ is described with
$(6s\times 6s)$-matrix with the following elements $b_{ij}${\rm:}

If $0\le j<s$, then $$b_{ij}=
\begin{cases}
-\kappa_1w_{4(j+m)-1\ra 4(j+m)+3}\otimes e_{4j},\quad i=j;\\
\kappa_1f_2(j,s-1)e_{4(j+m)+3}\otimes w_{4j\ra 4(j+1)},\quad i=(j+1)_s;\\
\kappa_1w_{4(j+m)+2\ra 4(j+m)+3}\otimes w_{4j\ra 4j+1},\quad i=j+s;\\
-\kappa_1w_{4(j+m+s)+1\ra 4(j+m)+3}\otimes w_{4j\ra 4j+2},\quad i=j+3s;\\
0,\quad\text{otherwise,}
\end{cases}$$
where $\kappa_1=\kappa^\ell(\a_{3(j+m+2)})$.

If $s\le j<3s$, then $$b_{ij}=
\begin{cases}
-w_{4(j+m)+3\ra 4(j+m+1)+2}\otimes w_{4(j+s)+1\ra 4(j+1)},\quad i=(j+1)_s;\\
0,\quad\text{otherwise.}
\end{cases}$$

If $3s\le j<5s$, then $$b_{ij}=
\begin{cases}
w_{4(j+m+1)\ra 4(j+m+s+1)+1}\otimes w_{4(j+s)+2\ra 4j+3},\quad i=5s+(j)_s;\\
0,\quad\text{otherwise.}
\end{cases}$$

If $5s\le j<6s$, then $$b_{ij}=
\begin{cases}
\kappa_1w_{4(j+m+s+1+sf(j,6s-1))+2\ra 4(j+m+2)}\otimes w_{4j+3\ra 4(j+s+1+sf(j,6s-1))+1},\quad i=s+(j+1)_s;\\
\kappa_1w_{4(j+m+1+sf(j,6s-1))+2\ra 4(j+m+2)}\otimes w_{4j+3\ra 4(j+1+sf(j,6s-1))+1},\quad i=2s+(j+1)_s;\\
\kappa_1w_{4(j+m+1+sf(j,6s-1))+1\ra 4(j+m+2)}\otimes w_{4j+3\ra 4(j+s+1+sf(j,6s-1))+2},\quad i=3s+(j+1)_s;\\
0,\quad\text{otherwise,}
\end{cases}$$
where $\kappa_1=\kappa^\ell(\a_{3(j+m-1)})\kappa^{\ell+1}(\g_{j+m-2})$.

$(10)$ If $r_0=9$, then $\Omega^{9}(Y_t^{(11)})$ is described with
$(8s\times 7s)$-matrix with the following elements $b_{ij}${\rm:}

If $0\le j<s$, then $$b_{ij}=
\begin{cases}
w_{4(j+m)+3\ra 4(j+m+s+1)+2}\otimes e_{4j},\quad i=j;\\
-w_{4(j+m+1)\ra 4(j+m+s+1)+2}\otimes w_{4j\ra 4(j+s)+1},\quad i=j+2s;\\
0,\quad\text{otherwise.}
\end{cases}$$

If $s\le j<2s$, then $$b_{ij}=
\begin{cases}
w_{4(j+m+1)\ra 4(j+m+s+1)+2}\otimes w_{4j\ra 4(j+s-sf_0(j,s))+1},\quad i=j;\\
0,\quad\text{otherwise.}
\end{cases}$$

If $2s\le j<4s$, then $$b_{ij}=
\begin{cases}
\kappa_1w_{4(j+m+1)\ra 4(j+m+1)+3}\otimes e_{4j+1},\quad i=j-s;\\
-\kappa_1e_{4(j+m+1)+3}\otimes w_{4j+1\ra 4(j+1)},\quad i=(j+1)_s,\text{ }j<3s-1\text{ or }j=4s-1;\\
0,\quad\text{otherwise,}
\end{cases}$$
where $\kappa_1=\kappa^{\ell+1}(\a_{3(j+m-3)})$.

If $4s\le j<6s$, then $$b_{ij}=
\begin{cases}
\kappa_1w_{4(j+m+s+1)+1\ra 4(j+m+2)}\otimes e_{4j+2},\quad i=j-s;\\
-\kappa_1w_{4(j+m+1)+3\ra 4(j+m+2)}\otimes w_{4j+2\ra 4(j+1)},\quad i=(j+1)_s,\text{ }5s-1\le j<6s-1;\\
0,\quad\text{otherwise,}
\end{cases}$$
where $\kappa_1=f_1(j,5s)\kappa^\ell(\a_{3(j+m-1)})\kappa^\ell(\g_{j+m-2})$.

If $6s\le j<8s$, then $$b_{ij}=
\begin{cases}
w_{4(j+m+1)+3\ra 4(j+m+s+2)+1}\otimes w_{4j+3\ra 4(j+1)},\quad i=(j+1)_s,\text{ }j<7s-1\text{ or }j=8s-1;\\
-w_{4(j+m+2)\ra 4(j+m+s+2)+1}\otimes w_{4j+3\ra 4(j+1)+1},\\\quad\quad\quad i=s+(j+1)_s,\text{ }j<7s-1\text{ or }j=8s-1;\\
-w_{4(j+m+2)\ra 4(j+m+s+2)+1}\otimes w_{4j+3\ra 4(j+s+1)+1},\\\quad\quad\quad i=s+(j+1)_s,\text{ }7s-1\le j<8s-1;\\
0,\quad\text{otherwise.}
\end{cases}$$

$(11)$ If $r_0=10$, then $\Omega^{10}(Y_t^{(11)})$ is described with
$(9s\times 6s)$-matrix with the following elements $b_{ij}${\rm:}

If $0\le j<s$, then $$b_{ij}=
\begin{cases}
-\kappa_1w_{4(j+m)\ra 4(j+m)+3}\otimes e_{4j},\quad i=j;\\
\kappa_1w_{4(j+m+s)+1\ra 4(j+m)+3}\otimes w_{4j\ra 4j+1},\quad i=j+s;\\
-\kappa_1w_{4(j+m)+1\ra 4(j+m)+3}\otimes w_{4j\ra 4(j+s)+1},\quad i=j+2s;\\
-\kappa_1w_{4(j+m+s)+2\ra 4(j+m)+3}\otimes w_{4j\ra 4j+2},\quad i=j+3s;\\
\kappa_1w_{4(j+m)+2\ra 4(j+m)+3}\otimes w_{4j\ra 4(j+s)+2},\quad i=j+4s;\\
-\kappa_1e_{4(j+m)+3}\otimes w_{4j\ra 4j+3},\quad i=j+5s;\\
0,\quad\text{otherwise,}
\end{cases}$$
where $\kappa_1=\kappa^{\ell+1}(\a_{3(j+m-4)})$.

If $s\le j<2s$, then $$b_{ij}=
\begin{cases}
\kappa_1w_{4(j+m)\ra 4(j+m+1)}\otimes e_{4j},\quad i=j-s;\\
-\kappa_1w_{4(j+m+s)+1\ra 4(j+m+1)}\otimes w_{4j\ra 4j+1},\quad i=j+s;\\
\kappa_1w_{4(j+m)+2\ra 4(j+m+1)}\otimes w_{4j\ra 4(j+s)+2},\quad i=j+2s;\\
-\kappa_1w_{4(j+m+s)+2\ra 4(j+m+1)}\otimes w_{4j\ra 4j+2},\quad i=j+3s;\\
\kappa_1w_{4(j+m)+3\ra 4(j+m+1)}\otimes w_{4j\ra 4j+3},\quad i=j+4s;\\
0,\quad\text{otherwise,}
\end{cases}$$
where $\kappa_1=\kappa^{\ell+1}(\a_{3(j+m-3)})$.

If $2s\le j<3s-1$, then $$b_{ij}=
\begin{cases}
-w_{4(j+m+1)\ra 4(j+m+1)+1}\otimes w_{4j+1\ra 4(j+1)},\quad i=(j+1)_s;\\
0,\quad\text{otherwise.}
\end{cases}$$

If $3s-1\le j<4s-1$, then $b_{ij}=0$.

If $4s-1\le j<4s$, then $$b_{ij}=
\begin{cases}
-w_{4(j+m+1)\ra 4(j+m+1)+1}\otimes w_{4j+1\ra 4(j+1)},\quad i=(j+1)_s;\\
0,\quad\text{otherwise.}
\end{cases}$$

If $4s\le j<5s$, then $$b_{ij}=
\begin{cases}
-w_{4(j+m+s)+2\ra 4(j+m+s+1)+1}\otimes e_{4j+2},\quad i=j-s;\\
-w_{4(j+m)+3\ra 4(j+m+s+1)+1}\otimes w_{4j+2\ra 4j+3},\quad i=j+s;\\
-w_{4(j+m+1)\ra 4(j+m+s+1)+1}\otimes w_{4j+2\ra 4(j+1)},\quad i=(j+1)_s,\text{ }j<5s-1;\\
0,\quad\text{otherwise.}
\end{cases}$$

If $5s\le j<6s-1$, then $b_{ij}=0$.

If $6s-1\le j<6s$, then $$b_{ij}=
\begin{cases}
w_{4(j+m+1)\ra 4(j+m+s+1)+2}\otimes w_{4(j+s)+2\ra 4(j+1)},\quad i=(j+1)_s;\\
0,\quad\text{otherwise.}
\end{cases}$$

If $6s\le j<7s$, then $$b_{ij}=
\begin{cases}
-w_{4(j+m)+2\ra 4(j+m+1)+1}\otimes e_{4(j+s)+2},\quad i=j-2s;\\
w_{4(j+m)+3\ra 4(j+m+1)+1}\otimes w_{4(j+s)+2\ra 4j+3},\quad i=j-s;\\
-w_{4(j+m+1)\ra 4(j+m+1)+1}\otimes w_{4(j+s)+2\ra 4(j+1)},\quad i=(j+1)_s,\text{ }j=7s-1;\\
0,\quad\text{otherwise.}
\end{cases}$$

If $7s\le j<8s-1$, then $$b_{ij}=
\begin{cases}
w_{4(j+m+1)\ra 4(j+m+1)+2}\otimes w_{4j+2\ra 4(j+1)},\quad i=(j+1)_s;\\
0,\quad\text{otherwise.}
\end{cases}$$

If $8s-1\le j<8s$, then $b_{ij}=0$.

If $8s\le j<9s$, then $$b_{ij}=
\begin{cases}
\kappa_1w_{4(j+m+s+1+sf(j,9s-1))+1\ra 4(j+m+1)+3}\otimes w_{4j+3\ra 4(j+1+sf(j,9s-1))+1},\quad i=s+(j+1)_s;\\
-\kappa_1w_{4(j+m+1+sf(j,9s-1))+1\ra 4(j+m+1)+3}\otimes w_{4j+3\ra 4(j+s+1+sf(j,9s-1))+1},\quad i=2s+(j+1)_s;\\
-\kappa_1w_{4(j+m+s+1+sf(j,9s-1))+2\ra 4(j+m+1)+3}\otimes w_{4j+3\ra 4(j+1+sf(j,9s-1))+2},\quad i=3s+(j+1)_s;\\
\kappa_1w_{4(j+m+1+sf(j,9s-1))+2\ra 4(j+m+1)+3}\otimes w_{4j+3\ra 4(j+s+1+sf(j,9s-1))+2},\quad i=4s+(j+1)_s;\\
-\kappa_1e_{4(j+m+1)+3}\otimes w_{4j+3\ra 4(j+1)+3},\quad i=5s+(j+1)_s;\\
0,\quad\text{otherwise,}
\end{cases}$$
where $\kappa_1=-\kappa^{\ell+1}(\a_{3(j+m-3)})$.

\medskip
$({\rm II})$ Represent an arbitrary $t_0\in\N$ in the form
$t_0=11\ell_0+r_0$, where $0\le r_0\le 10.$ Then
$\Omega^{t_0}(Y_t^{(11)})$ is a $\Omega^{r_0}(Y_t^{(11)})$, whose left
components twisted by $\sigma^{\ell_0}$,
and coefficients multiplied by $(-1)^{\ell_0}$.
\end{pr}

\begin{pr}[Translates for the case 12]
$({\rm I})$ Let $r_0\in\N$, $r_0<11$. $r_0$-translates of the
elements $Y^{(12)}_t$ are described by the following way.

$(1)$ If $r_0=0$, then $\Omega^{0}(Y_t^{(12)})$ is described with
$(8s\times 6s)$-matrix with the following two nonzero elements{\rm:}
$$b_{3s,4s}=\kappa^\ell(\a_{3(j+m+2)})w_{4(j+m)+2\ra 4(j+m)+3}\otimes e_{4j+2};$$
$$b_{4s,5s}=\kappa^\ell(\a_{3(j+m+2)})w_{4(j+m)+2\ra 4(j+m)+3}\otimes e_{4j+2}.$$

$(2)$ If $r_0=1$, then $\Omega^{1}(Y_t^{(12)})$ is described with
$(9s\times 7s)$-matrix with the following nonzero elements{\rm:}
$$b_{j,(3)_s}=-\kappa_1w_{4(j+m)+1\ra 4(j+m+1)}\otimes e_{4j};$$
$$b_{j+s,(3)_s}=\kappa_1w_{4(j+m+s)+1\ra 4(j+m+1)}\otimes e_{4j};$$
$$b_{j+2s,(3)_s}=-\kappa_1w_{4(j+m)+2\ra 4(j+m+1)}\otimes w_{4j\ra 4j+1};$$
$$b_{j+3s,(3)_s}=\kappa_1w_{4(j+m+s)+2\ra 4(j+m+1)}\otimes w_{4j\ra 4(j+s)+1};$$
$$b_{j,2s+(3)_s}=-w_{4(j+m)+2\ra 4(j+m+1)+2}\otimes e_{4j+1};$$
$$b_{j-s,4s+(3)_s}=-w_{4(j+m+s)+2\ra 4(j+m+s+1)+2}\otimes e_{4(j+s)+1};$$
$$b_{(j+1)_s+f(s,1),7s+(3)_s-(1)_s}=-\kappa_2w_{4(j+m+s+1)+1\ra 4(j+m+1)+3}\otimes w_{4j+3\ra 4(j+1)};$$
$$b_{s+(j+1)_s-f(s,1),7s+(3)_s-(1)_s}=\kappa_2w_{4(j+m+1)+1\ra 4(j+m+1)+3}\otimes w_{4j+3\ra 4(j+1)};$$
$$b_{2s+(j+1)_s+f(s,1),7s+(3)_s-(1)_s}=-\kappa_2w_{4(j+m+s+1)+2\ra 4(j+m+1)+3}\otimes w_{4j+3\ra 4(j+s+1)+1};$$
$$b_{3s+(j+1)_s-f(s,1),7s+(3)_s-(1)_s}=\kappa_2w_{4(j+m+1)+2\ra 4(j+m+1)+3}\otimes w_{4j+3\ra 4(j+1)+1},$$
where $\kappa_1=\kappa^\ell(\a_{3((3)_s+m+3)})$, $\kappa_2=-\kappa^\ell(\a_{3(7s+(3)_s-(1)_s+m)})$.

$(3)$ If $r_0=2$, then $\Omega^{2}(Y_t^{(12)})$ is described with
$(8s\times 6s)$-matrix with the following two nonzero elements{\rm:}
$$b_{(j+1)_s,2s+(2)_s}=-f_2(s,1)\kappa^\ell(\a_{3(j+m+2)})e_{4(j+m)+3}\otimes w_{4j+1\ra 4(j+1)};$$
$$b_{(j+1)_s,3s+(2)_s}=f_2(s,1)\kappa^\ell(\a_{3(j+m+2)})e_{4(j+m)+3}\otimes w_{4j+1\ra 4(j+1)}.$$

$(4)$ If $r_0=3$, then $\Omega^{3}(Y_t^{(12)})$ is described with
$(6s\times 8s)$-matrix with the following two nonzero elements{\rm:}
$$b_{s+(j+1)_s-f(s,1),s+(2)_s}=-f_2(s,1)e_{4(j+m+1)+2}\otimes w_{4(j+s)+1\ra 4(j+1)};$$
$$b_{(j+1)_s+f(s,1),2s+(2)_s}=f_2(s,1)e_{4(j+m+1)+2}\otimes w_{4(j+s)+1\ra 4(j+1)}.$$

$(5)$ If $r_0=4$, then $\Omega^{4}(Y_t^{(12)})$ is described with
$(7s\times 9s)$-matrix with the following nonzero elements{\rm:}
$$b_{(j+1)_s,(2)_s}=-\kappa^\ell(\a_{3(j+m+2)})e_{4(j+m)+3}\otimes w_{4j\ra 4(j+1)};$$
$$b_{(j+1)_s,s+(2)_s+f(s,1)}=-f_2(s,1)\kappa^\ell(\a_{3(j+m+2)})w_{4(j+m)+3\ra 4(j+m+1)}\otimes w_{4(j+s)+1\ra 4(j+1)};$$
$$b_{s+(j+1)_s,s+(2)_s+f(s,1)}=f_2(s,1)\kappa^\ell(\a_{3(j+m+2)})e_{4(j+m+1)}\otimes w_{4(j+s)+1\ra 4(j+1)};$$
$$b_{s+(j+1)_s,2s+(2)_s-f(s,1)}=-f_2(s,1)\kappa^\ell(\a_{3(j+m+2)})e_{4(j+m+1)}\otimes w_{4(j+s)+1\ra 4(j+1)}.$$

$(6)$ If $r_0=5$, then $\Omega^{5}(Y_t^{(12)})$ is described with
$(6s\times 8s)$-matrix with the following two nonzero elements{\rm:}
$$b_{(j+1)_s+f(s,1),s+(2)_s}=e_{4(j+m+s+1)+1}\otimes w_{4(j+s)+1\ra 4(j+1)};$$
$$b_{s+(j+1)_s-f(s,1),2s+(2)_s}=e_{4(j+m+s+1)+1}\otimes w_{4(j+s)+1\ra 4(j+1)}.$$

$(7)$ If $r_0=6$, then $\Omega^{6}(Y_t^{(12)})$ is described with
$(6s\times 9s)$-matrix with the following nonzero elements{\rm:}
$$b_{(j+1)_s,(2)_s}=-\kappa^\ell(\a_{3(j+m+3)})e_{4(j+m+1)}\otimes w_{4j\ra 4(j+1)};$$
$$b_{j,s+(2)_s}=w_{4(j+m+s)+1\ra 4(j+m+s+1)+1}\otimes e_{4(j+s)+1};$$
$$b_{j+s,2s+(2)_s}=w_{4(j+m+s)+1\ra 4(j+m+s+1)+1}\otimes e_{4(j+s)+1}.$$

$(8)$ If $r_0=7$, then $\Omega^{7}(Y_t^{(12)})$ is described with
$(7s\times 8s)$-matrix with the following nonzero elements{\rm:}
$$b_{j+2s,(2)_s}=-w_{4(j+m)+3\ra 4(j+m+1)+1}\otimes w_{4j\ra 4j+1};$$
$$b_{j+4s,(2)_s}=-w_{4(j+m+1)\ra 4(j+m+1)+1}\otimes w_{4j\ra 4j+2};$$
$$b_{j+5s,(2)_s}=-w_{4(j+m+1)\ra 4(j+m+1)+1}\otimes w_{4j\ra 4(j+s)+2};$$
$$b_{j,s+(2)_s}=w_{4(j+m)+2\ra 4(j+m+1)+1}\otimes e_{4j};$$
$$b_{j+s,s+(2)_s}=-w_{4(j+m)+3\ra 4(j+m+1)+1}\otimes w_{4j\ra 4(j+s)+1};$$
$$b_{j+3s,s+(2)_s}=-w_{4(j+m+1)\ra 4(j+m+1)+1}\otimes w_{4j\ra 4(j+s)+2};$$
$$b_{j+4s,s+(2)_s}=-w_{4(j+m+1)\ra 4(j+m+1)+1}\otimes w_{4j\ra 4j+2}.$$

$(9)$ If $r_0=8$ and $s=2$, then $\Omega^{8}(Y_t^{(12)})$ is described with
$(6s\times 6s)$-matrix with the following nonzero elements{\rm:}
$$b_{j,s+1}=-w_{4(j+m+s)+2\ra 4(j+m+s+1)+2}\otimes e_{4(j+2)+1};$$
$$b_{s+(j+1)_s,2s+1}=-e_{4(j+m+s+1)+2}\otimes w_{4(j+2)+1\ra 4(j+s+1)+1};$$
$$b_{4s+(j+1)_s,3s+1}=-e_{4(j+m+1)+1}\otimes w_{4(j+s)+2\ra 4(j+s+1)+2};$$
$$b_{(j+1)_s,4s+1}=-w_{4(j+m)+3\ra 4(j+m+1)+1}\otimes w_{4(j+s)+2\ra 4(j+1)};$$
$$b_{3s+(j+1)_s,4s+1}=-e_{4(j+m+1)+1}\otimes w_{4(j+s)+2\ra 4(j+s+1)+2};$$
$$b_{2s+(j+1)_s,5s+1}=-w_{4(j+m+s+1)+2\ra 4(j+m)}\otimes w_{4j+3\ra 4(j+s+1)+1};$$
$$b_{3s+(j+1)_s,5s+1}=-w_{4(j+m+s+1)+1\ra 4(j+m)}\otimes w_{4j+3\ra 4(j+1)+2}.$$

$(10)$ If $r_0=8$ and $s\ne 2$, then $\Omega^{8}(Y_t^{(12)})$ is described with
$(6s\times 6s)$-matrix with the following nonzero elements{\rm:}
$$b_{s+(j+1)_s,s+1}=-e_{4(j+m+s+1)+2}\otimes w_{4(j+s)+1\ra 4(j+s+1)+1};$$
$$b_{j,2s-f_2(s,1)}=-w_{4(j+m+s)+2\ra 4(j+m+s+1)+2}\otimes e_{4(j+s)+1};$$
$$b_{(j+1)_s,3s+1}=-w_{4(j+m)+3\ra 4(j+m+1)+1}\otimes w_{4(j+s)+2\ra 4(j+1)};$$
$$b_{3s+(j+1)_s,3s+1}=-e_{4(j+m+1)+1}\otimes w_{4(j+s)+2\ra 4(j+s+1)+2};$$
$$b_{4s+(j+1)_s,4s-f_2(s,1)}=-e_{4(j+m+1)+1}\otimes w_{4(j+s)+2\ra 4(j+s+1)+2};$$
$$b_{2s+(j+1)_s,5s+(1)_s}=-f_2(s,1)\kappa^{\ell+1}(\a_{3(j+m-1)})
w_{4(j+m+1+sf(s,1))+2\ra 4(j+m+2)}\otimes w_{4j+3\ra 4(j+1+sf(s,1))+1};$$
$$b_{3s+(j+1)_s,5s+(1)_s}=-f_2(s,1)\kappa^{\ell+1}(\a_{3(j+m-1)})
w_{4(j+m+1+sf(s,1))+1\ra 4(j+m+2)}\otimes w_{4j+3\ra 4(j+s+1+sf(s,1))+2}.$$

$(11)$ If $r_0=9$ and $s=1$, then $\Omega^{9}(Y_t^{(12)})$ is described with
$(8s\times 7s)$-matrix with the following nonzero elements{\rm:}
$$b_{j+s,0}=w_{4(j+m)\ra 4(j+m+1)+2}\otimes w_{4j\ra 4j+1};$$
$$b_{j+2s,0}=-w_{4(j+m)\ra 4(j+m+1)+2}\otimes w_{4j\ra 4(j+1)+1};$$
$$b_{j-3s,3s}=\kappa^{\ell+1}(\a_{3(j+m)})e_{4(j+m)+3}\otimes w_{4j+1\ra 4j};$$
$$b_{j-4s,4s}=\kappa^{\ell+1}(\a_{3(j+m)})w_{4(j+m)+3\ra 4(j+m)}\otimes w_{4j+2\ra 4j};$$
$$b_{j-6s,6s}=-w_{4(j+m)+3\ra 4(j+m)+1}\otimes w_{4j+3\ra 4j}.$$

$(12)$ If $r_0=9$ and $s=2$, then $\Omega^{9}(Y_t^{(12)})$ is described with
$(8s\times 7s)$-matrix with the following nonzero elements{\rm:}
$$b_{j+s,1}=w_{4(j+m+1)\ra 4(j+m+1)+2}\otimes w_{4j\ra 4j+1};$$
$$b_{j+2s,1}=-w_{4(j+m+1)\ra 4(j+m+1)+2}\otimes w_{4j\ra 4(j+2)+1};$$
$$b_{(j+1)_s,3s+1}=e_{4(j+m+1)+3}\otimes w_{4j+1\ra 4(j+1)};$$
$$b_{(j+1)_s,4s+1}=w_{4(j+m+1)+3\ra 4(j+m)}\otimes w_{4j+2\ra 4(j+1)};$$
$$b_{(j+1)_s,6s+1}=-w_{4(j+m+1)+3\ra 4(j+m+2)+1}\otimes w_{4j+3\ra 4(j+1)}.$$

$(13)$ If $r_0=9$ and $s>2$, then $\Omega^{9}(Y_t^{(12)})$ is described with
$(8s\times 7s)$-matrix with the following nonzero elements{\rm:}
$$b_{j+f(s,1),s+1}=w_{4(j+m+1)\ra 4(j+m+1)+2}\otimes w_{4j\ra 4(j+s)+1};$$
$$b_{j+s+f(s,1),s+1}=-w_{4(j+m+1)\ra 4(j+m+1)+2}\otimes w_{4j\ra 4j+1};$$
$$b_{j-s,2s+1}=-\kappa^{\ell+1}(\a_{3(j+m)})w_{4(j+m+1)\ra 4(j+m+1)+3}\otimes e_{4j+1};$$
$$b_{j+3s-f(s,1),2s+1}=\kappa^{\ell+1}(\a_{3(j+m)})w_{4(j+m+1)+2\ra 4(j+m+1)+3}\otimes w_{4j+1\ra 4j+3};$$
$$b_{(j+1)_s,5s+1}=-\kappa^{\ell+1}(\a_{3(j+m-1)})w_{4(j+m+1)+3\ra 4(j+m+2)}\otimes w_{4j+2\ra 4(j+1)};$$
$$b_{(j+1)_s,7s+1}=-w_{4(j+m+1)+3\ra 4(j+m+2)+1}\otimes w_{4j+3\ra 4(j+1)}.$$

$(14)$ If $r_0=10$ and $s=1$, then $\Omega^{10}(Y_t^{(12)})$ is described with
$(9s\times 6s)$-matrix with the following nonzero elements{\rm:}
$$b_{j+s,0}=-\kappa_1w_{4(j+m+1)+1\ra 4(j+m)+3}\otimes w_{4j\ra 4j+1};$$
$$b_{j+2s,0}=\kappa_1w_{4(j+m)+1\ra 4(j+m)+3}\otimes w_{4j\ra 4(j+1)+1};$$
$$b_{j+3s,0}=\kappa_1w_{4(j+m+1)+2\ra 4(j+m)+3}\otimes w_{4j\ra 4j+2};$$
$$b_{j+4s,0}=-\kappa_1w_{4(j+m)+2\ra 4(j+m)+3}\otimes w_{4j\ra 4(j+1)+2};$$
$$b_{j+5s,0}=\kappa_1e_{4(j+m)+3}\otimes w_{4j\ra 4j+3};$$
$$b_{j,s}=-\kappa_2w_{4(j+m)+1\ra 4(j+m)}\otimes w_{4j\ra 4(j+1)+1};$$
$$b_{j+s,s}=\kappa_2w_{4(j+m+1)+1\ra 4(j+m)}\otimes w_{4j\ra 4j+1};$$
$$b_{j+2s,s}=\kappa_2w_{4(j+m)+2\ra 4(j+m)}\otimes w_{4j\ra 4(j+1)+2};$$
$$b_{j+3s,s}=-\kappa_2w_{4(j+m+1)+2\ra 4(j+m)}\otimes w_{4j\ra 4j+2};$$
$$b_{j+4s,s}=\kappa_2w_{4(j+m)+3\ra 4(j+m)}\otimes w_{4j\ra 4j+3};$$
$$b_{j-s,2s}=w_{4(j+m+1)+1\ra 4(j+m)+1}\otimes e_{4j+1};$$
$$b_{j+s,2s}=-w_{4(j+m+1)+2\ra 4(j+m)+1}\otimes w_{4j+1\ra 4j+2};$$
$$b_{j+3s,2s}=-w_{4(j+m)+3\ra 4(j+m)+1}\otimes w_{4j+1\ra 4j+3};$$
$$b_{j-3s,5s}=-w_{4(j+m+1)+1\ra 4(j+m+1)+2}\otimes w_{4(j+1)+2\ra 4j+1};$$
$$b_{j-2s,5s}=-w_{4(j+m)+2\ra 4(j+m+1)+2}\otimes e_{4(j+1)+2};$$
$$b_{j-s,5s}=e_{4(j+m+1)+2}\otimes w_{4(j+1)+2\ra 4j+2};$$
$$b_{j,5s}=-w_{4(j+m)+3\ra 4(j+m+1)+2}\otimes w_{4(j+1)+2\ra 4j+3};$$
$$b_{j-3s,7s}=-w_{4(j+m+1)+2\ra 4(j+m)+2}\otimes e_{4j+2};$$
$$b_{j-2s,7s}=w_{4(j+m)+3\ra 4(j+m)+2}\otimes w_{4j+2\ra 4j+3};$$
$$b_{j-8s,8s}=-\kappa^{\ell+1}(\a_{3(j+m)})w_{4(j+m)\ra 4(j+m)+3}\otimes w_{4j+3\ra 4j},$$
where $\kappa_1=-\kappa^{\ell+1}(\a_{3m})$, $\kappa_2=-\kappa^{\ell+1}(\a_{3(s+m)})$.

$(15)$ If $r_0=10$ and $s=2$, then $\Omega^{10}(Y_t^{(12)})$ is described with
$(9s\times 6s)$-matrix with the following nonzero elements{\rm:}
$$b_{j+s,1}=w_{4(j+m+2)+1\ra 4(j+m)+3}\otimes w_{4j\ra 4j+1};$$
$$b_{j+2s,1}=-w_{4(j+m)+1\ra 4(j+m)+3}\otimes w_{4j\ra 4(j+2)+1};$$
$$b_{j+3s,1}=-w_{4(j+m+s)+2\ra 4(j+m)+3}\otimes w_{4j\ra 4j+2};$$
$$b_{j+4s,1}=w_{4(j+m)+2\ra 4(j+m)+3}\otimes w_{4j\ra 4(j+s)+2};$$
$$b_{j+5s,1}=-e_{4(j+m)+3}\otimes w_{4j\ra 4j+3};$$
$$b_{j,s+1}=w_{4(j+m)+1\ra 4(j+m+1)}\otimes w_{4j\ra 4(j+2)+1};$$
$$b_{j+s,s+1}=-w_{4(j+m+2)+1\ra 4(j+m+1)}\otimes w_{4j\ra 4j+1};$$
$$b_{j+2s,s+1}=-w_{4(j+m)+2\ra 4(j+m+1)}\otimes w_{4j\ra 4(j+s)+2};$$
$$b_{j+3s,s+1}=w_{4(j+m+s)+2\ra 4(j+m+1)}\otimes w_{4j\ra 4j+2};$$
$$b_{j+4s,s+1}=-w_{4(j+m)+3\ra 4(j+m+1)}\otimes w_{4j\ra 4j+3};$$
$$b_{j-s,2s+1}=w_{4(j+m+2)+1\ra 4(j+m+s+1)+1}\otimes e_{4j+1};$$
$$b_{j+s,2s+1}=-w_{4(j+m+s)+2\ra 4(j+m+s+1)+1}\otimes w_{4j+1\ra 4j+2};$$
$$b_{j+3s,2s+1}=-w_{4(j+m)+3\ra 4(j+m+s+1)+1}\otimes w_{4j+1\ra 4j+3};$$
$$b_{j-2s,5s+1}=-w_{4(j+m)+2\ra 4(j+m+1)+2}\otimes e_{4(j+s)+2};$$
$$b_{j,5s+1}=-w_{4(j+m)+3\ra 4(j+m+1)+2}\otimes w_{4(j+s)+2\ra 4j+3};$$
$$b_{2s+(j+1)_s,7s}=-w_{4(j+m+s+1)+1\ra 4(j+m+s+1)+2}\otimes w_{4j+2\ra 4(j+1)+1};$$
$$b_{4s+(j+1)_s,7s}=e_{4(j+m+s+1)+2}\otimes w_{4j+2\ra 4(j+1)+2};$$
$$b_{j-3s,8s-1}=-w_{4(j+m+s)+2\ra 4(j+m+s+1)+2}\otimes e_{4j+2};$$
$$b_{j-2s,8s-1}=w_{4(j+m)+3\ra 4(j+m+s+1)+2}\otimes w_{4j+2\ra 4j+3};$$
$$b_{(j+1)_s,8s}=w_{4(j+m+1)\ra 4(j+m+1)+3}\otimes w_{4j+3\ra 4(j+1)};$$
$$b_{j-3s,8s}=-w_{4(j+m)+3\ra 4(j+m+1)+3}\otimes e_{4j+3};$$
$$b_{j-3s,9s-1}=w_{4(j+m)+3\ra 4(j+m+1)+3}\otimes e_{4j+3}.$$

$(16)$ If $r_0=10$ and $s>2$, then $\Omega^{10}(Y_t^{(12)})$ is described with
$(9s\times 6s)$-matrix with the following nonzero elements{\rm:}
$$b_{j+2s,1}=-\kappa_1w_{4(j+m)+1\ra 4(j+m)+3}\otimes w_{4j\ra 4(j+s)+1};$$
$$b_{j+4s,1}=\kappa_1w_{4(j+m)+2\ra 4(j+m)+3}\otimes w_{4j\ra 4(j+s)+2};$$
$$b_{j,s+1}=-\kappa_2w_{4(j+m)+1\ra 4(j+m+1)}\otimes w_{4j\ra 4(j+s)+1};$$
$$b_{j+2s,s+1}=\kappa_2w_{4(j+m)+2\ra 4(j+m+1)}\otimes w_{4j\ra 4(j+s)+2};$$
$$b_{j+4s,s+1}=\kappa_2w_{4(j+m)+3\ra 4(j+m+1)}\otimes w_{4j\ra 4j+3};$$
$$b_{j-s,2s+1}=w_{4(j+m+s)+1\ra 4(j+m+s+1)+1}\otimes e_{4j+1};$$
$$b_{j+s,2s+1}=-w_{4(j+m+s)+2\ra 4(j+m+s+1)+1}\otimes w_{4j+1\ra 4j+2};$$
$$b_{j+3s,2s+1}=-w_{4(j+m)+3\ra 4(j+m+s+1)+1}\otimes w_{4j+1\ra 4j+3};$$
$$b_{j-s,3s+1}=w_{4(j+m+s)+1\ra 4(j+m+s+1)+1}\otimes e_{4j+1};$$
$$b_{j+s,3s+1}=-w_{4(j+m+s)+2\ra 4(j+m+s+1)+1}\otimes w_{4j+1\ra 4j+2};$$
$$b_{j+2s,3s+1}=w_{4(j+m)+3\ra 4(j+m+s+1)+1}\otimes w_{4j+1\ra 4j+3};$$
$$b_{s+(j+1)_s,5s}=-w_{4(j+m+1)+1\ra 4(j+m+1)+2}\otimes w_{4(j+s)+2\ra 4(j+s+1)+1};$$
$$b_{3s+(j+1)_s,5s}=e_{4(j+m+1)+2}\otimes w_{4(j+s)+2\ra 4(j+s+1)+2};$$
$$b_{j-2s,5s+1}=-w_{4(j+m)+2\ra 4(j+m+1)+2}\otimes e_{4(j+s)+2};$$
$$b_{j,5s+1}=-w_{4(j+m)+3\ra 4(j+m+1)+2}\otimes w_{4(j+s)+2\ra 4j+3};$$
$$b_{j-3s,7s+1}=-w_{4(j+m+s)+2\ra 4(j+m+s+1)+2}\otimes e_{4j+2};$$
$$b_{j-2s,7s+1}=w_{4(j+m)+3\ra 4(j+m+s+1)+2}\otimes w_{4j+2\ra 4j+3};$$
$$b_{(j+1)_s,8s}=-\kappa_3w_{4(j+m+1)\ra 4(j+m+1)+3}\otimes w_{4j+3\ra 4(j+1)};$$
$$b_{2s+(j+1)_s,8s}=-\kappa_3w_{4(j+m+1)+1\ra 4(j+m+1)+3}\otimes w_{4j+3\ra 4(j+s+1)+1};$$
$$b_{4s+(j+1)_s,8s}=\kappa_3w_{4(j+m+1)+2\ra 4(j+m+1)+3}\otimes w_{4j+3\ra 4(j+s+1)+2};$$
$$b_{j-3s,8s}=\kappa_3w_{4(j+m)+3\ra 4(j+m+1)+3}\otimes e_{4j+3};$$
$$b_{j-3s,8s+1}=-\kappa^{\ell+1}(\a_{3(j+m-2)})w_{4(j+m)+3\ra 4(j+m+1)+3}\otimes e_{4j+3},$$
where $\kappa_1=\kappa^{\ell+1}(\a_{3m})$, $\kappa_2=\kappa^{\ell+1}(\a_{3(s+m-1)})$, $\kappa_3=\kappa^{\ell+1}(\a_{3m})$.

\medskip
$({\rm II})$ Represent an arbitrary $t_0\in\N$ in the form
$t_0=11\ell_0+r_0$, where $0\le r_0\le 10.$ Then
$\Omega^{t_0}(Y_t^{(12)})$ is a $\Omega^{r_0}(Y_t^{(12)})$, whose left
components twisted by $\sigma^{\ell_0}$.
\end{pr}

\begin{pr}[Translates for the case 13]
$({\rm I})$ Let $r_0\in\N$, $r_0<11$. $r_0$-translates of the
elements $Y^{(13)}_t$ are described by the following way.

$(1)$ If $r_0=0$, then $\Omega^{0}(Y_t^{(13)})$ is described with
$(9s\times 6s)$-matrix with the following elements $b_{ij}${\rm:}

If $s\le j<2s$, then $$b_{ij}=
\begin{cases}
e_{4(j+m+s)+1}\otimes e_{4(j+s)+1},\quad i=j;\\
0,\quad\text{otherwise.}
\end{cases}$$

If $2s\le j<3s$, then $b_{ij}=0$.

If $3s\le j<4s$, then $$b_{ij}=
\begin{cases}
e_{4(j+m)+1}\otimes e_{4j+1},\quad i=j-s;\\
0,\quad\text{otherwise.}
\end{cases}$$

If $4s\le j<5s$, then $b_{ij}=0$.

If $5s\le j<7s$, then $$b_{ij}=
\begin{cases}
e_{4(j+m+s)+2}\otimes e_{4(j+s)+2},\quad i=j-2s;\\
0,\quad\text{otherwise.}
\end{cases}$$

If $7s\le j<8s$, then $b_{ij}=0$.

If $8s\le j<9s$, then $$b_{ij}=
\begin{cases}
w_{4(j+m)+3\ra 4(j+m+1)}\otimes e_{4j+3},\quad i=j-3s;\\
0,\quad\text{otherwise.}
\end{cases}$$

$(2)$ If $r_0=1$, then $\Omega^{1}(Y_t^{(13)})$ is described with
$(8s\times 7s)$-matrix with the following elements $b_{ij}${\rm:}

If $0\le j<2s$, then $$b_{ij}=
\begin{cases}
e_{4(j+m)+2}\otimes w_{4j\ra 4j+1},\quad i=j+2s;\\
0,\quad\text{otherwise.}
\end{cases}$$

If $2s\le j<4s$, then $$b_{ij}=
\begin{cases}
w_{4(j+m)+2\ra 4(j+m)+3}\otimes e_{4j+1},\quad i=j;\\
0,\quad\text{otherwise.}
\end{cases}$$

If $4s\le j<6s$, then $$b_{ij}=
\begin{cases}
w_{4(j+m)+3\ra 4(j+m+1)}\otimes e_{4j+2},\quad i=j;\\
0,\quad\text{otherwise.}
\end{cases}$$

If $6s\le j<8s$, then $$b_{ij}=
\begin{cases}
e_{4(j+m+1)+1}\otimes w_{4j+3\ra 4(j+1)},\quad i=(j+1)_{2s};\\
w_{4(j+m+1)\ra 4(j+m+1)+1}\otimes e_{4j+3},\quad i=6s+(j)_s;\\
0,\quad\text{otherwise.}
\end{cases}$$

$(3)$ If $r_0=2$, then $\Omega^{2}(Y_t^{(13)})$ is described with
$(6s\times 6s)$-matrix with the following elements $b_{ij}${\rm:}

If $s\le j<3s$, then $$b_{ij}=
\begin{cases}
e_{4(j+m+s)+2}\otimes e_{4(j+s)+1},\quad i=j;\\
0,\quad\text{otherwise.}
\end{cases}$$

If $3s\le j<5s$, then $$b_{ij}=
\begin{cases}
e_{4(j+m)+1}\otimes e_{4(j+s)+2},\quad i=j;\\
0,\quad\text{otherwise.}
\end{cases}$$

If $5s\le j<6s$, then $$b_{ij}=
\begin{cases}
w_{4(j+m)+3\ra 4(j+m+1)}\otimes w_{4j+3\ra 4(j+1)},\quad i=(j+1)_s;\\
0,\quad\text{otherwise.}
\end{cases}$$

$(4)$ If $r_0=3$, then $\Omega^{3}(Y_t^{(13)})$ is described with
$(7s\times 8s)$-matrix with the following elements $b_{ij}${\rm:}

If $0\le j<s$, then $$b_{ij}=
\begin{cases}
w_{4(j+m)+2\ra 4(j+m)+3}\otimes e_{4j},\quad i=j;\\
w_{4(j+m+s)+2\ra 4(j+m)+3}\otimes e_{4j},\quad i=j+s;\\
0,\quad\text{otherwise.}
\end{cases}$$

If $s\le j<3s$, then $$b_{ij}=
\begin{cases}
w_{4(j+m)+3\ra 4(j+m+1)}\otimes e_{4(j+s)+1},\quad i=j+s;\\
0,\quad\text{otherwise.}
\end{cases}$$

If $3s\le j<5s$, then $$b_{ij}=
\begin{cases}
w_{4(j+m+1)\ra 4(j+m+1)+1}\otimes e_{4(j+s)+2},\quad i=j+s;\\
e_{4(j+m+1)+1}\otimes w_{4(j+s)+2\ra 4j+3},\quad i=6s+(j)_{2s};\\
0,\quad\text{otherwise.}
\end{cases}$$

If $5s\le j<7s$, then $$b_{ij}=
\begin{cases}
e_{4(j+m+s+1)+2}\otimes w_{4j+3\ra 4(j+1)},\quad i=(j+s+1)_{2s};\\
0,\quad\text{otherwise.}
\end{cases}$$

$(5)$ If $r_0=4$, then $\Omega^{4}(Y_t^{(13)})$ is described with
$(6s\times 9s)$-matrix with the following elements $b_{ij}${\rm:}

If $0\le j<s$, then $$b_{ij}=
\begin{cases}
w_{4(j+m)-1\ra 4(j+m)}\otimes e_{4j},\quad i=j;\\
0,\quad\text{otherwise.}
\end{cases}$$

If $s\le j<3s$, then $$b_{ij}=
\begin{cases}
e_{4(j+m)+1}\otimes e_{4(j+s)+1},\quad i=j+s;\\
0,\quad\text{otherwise.}
\end{cases}$$

If $3s\le j<5s$, then $$b_{ij}=
\begin{cases}
e_{4(j+m)+2}\otimes e_{4(j+s)+2},\quad i=j+3s-sf_0(j,4s);\\
0,\quad\text{otherwise.}
\end{cases}$$

If $5s\le j<6s$, then $b_{ij}=0$.

$(6)$ If $r_0=5$, then $\Omega^{5}(Y_t^{(13)})$ is described with
$(6s\times 8s)$-matrix with the following elements $b_{ij}${\rm:}

If $0\le j<s$, then $$b_{ij}=
\begin{cases}
e_{4(j+m+1)}\otimes w_{4j\ra 4j+1},\quad i=j+2s;\\
e_{4(j+m+1)}\otimes w_{4j\ra 4(j+s)+1},\quad i=j+3s;\\
0,\quad\text{otherwise.}
\end{cases}$$

If $s\le j<3s$, then $$b_{ij}=
\begin{cases}
e_{4(j+m+1)+1}\otimes w_{4(j+s)+1\ra 4(j+1)},\quad i=(j+1)_{2s};\\
w_{4(j+m+1)\ra 4(j+m+1)+1}\otimes e_{4(j+s)+1},\quad i=j+s;\\
0,\quad\text{otherwise.}
\end{cases}$$

If $3s\le j<5s$, then $$b_{ij}=
\begin{cases}
w_{4(j+m)+3\ra 4(j+m+1)+2}\otimes e_{4(j+s)+2},\quad i=j+s;\\
e_{4(j+m+1)+2}\otimes w_{4(j+s)+2\ra 4j+3},\quad i=6s+(j)_{2s};\\
0,\quad\text{otherwise.}
\end{cases}$$

If $5s\le j<6s$, then $$b_{ij}=
\begin{cases}
w_{4(j+m+s+1+sf(j,6s-1))+1\ra 4(j+m+1)+3}\otimes w_{4j+3\ra 4(j+1)},\quad i=(j+1)_s;\\
w_{4(j+m+1+sf(j,6s-1))+1\ra 4(j+m+1)+3}\otimes w_{4j+3\ra 4(j+1)},\quad i=s+(j+1)_s;\\
e_{4(j+m+1)+3}\otimes w_{4j+3\ra 4(j+s+1+sf(j,6s-1))+2},\quad i=4s+(j+1)_s;\\
e_{4(j+m+1)+3}\otimes w_{4j+3\ra 4(j+1+sf(j,6s-1))+2},\quad i=5s+(j+1)_s;\\
0,\quad\text{otherwise.}
\end{cases}$$

$(7)$ If $r_0=6$, then $\Omega^{6}(Y_t^{(13)})$ is described with
$(7s\times 9s)$-matrix with the following elements $b_{ij}${\rm:}

If $0\le j<2s$, then $$b_{ij}=
\begin{cases}
e_{4(j+m+s)+1}\otimes w_{4j\ra 4(j+s)+1},\quad i=j+3sf_0(j,s);\\
0,\quad\text{otherwise.}
\end{cases}$$

If $2s\le j<4s$, then $$b_{ij}=
\begin{cases}
e_{4(j+m+s)+2}\otimes e_{4j+1},\quad i=j+s(1-f_0(j,3s));\\
0,\quad\text{otherwise.}
\end{cases}$$

If $4s\le j<6s$, then $$b_{ij}=
\begin{cases}
w_{4(j+m)+2\ra 4(j+m)+3}\otimes e_{4j+2},\quad i=j+s;\\
e_{4(j+m)+3}\otimes w_{4j+2\ra 4j+3},\quad i=7s+(j)_s;\\
0,\quad\text{otherwise.}
\end{cases}$$

If $6s\le j<7s$, then $$b_{ij}=
\begin{cases}
e_{4(j+m+1)}\otimes w_{4j+3\ra 4(j+1)},\quad i=(j+1)_s;\\
0,\quad\text{otherwise.}
\end{cases}$$

$(8)$ If $r_0=7$, then $\Omega^{7}(Y_t^{(13)})$ is described with
$(6s\times 8s)$-matrix with the following elements $b_{ij}${\rm:}

If $0\le j<s$, then $$b_{ij}=
\begin{cases}
w_{4(j+m)+2\ra 4(j+m)+3}\otimes e_{4j},\quad i=j;\\
w_{4(j+m+s)+2\ra 4(j+m)+3}\otimes e_{4j},\quad i=j+s;\\
0,\quad\text{otherwise.}
\end{cases}$$

If $s\le j<3s$, then $$b_{ij}=
\begin{cases}
e_{4(j+m+1)+2}\otimes w_{4(j+s)+1\ra 4(j+1)},\quad i=(j+1)_{2s};\\
w_{4(j+m)+3\ra 4(j+m+1)+2}\otimes e_{4(j+s)+1},\quad i=j+s;\\
0,\quad\text{otherwise.}
\end{cases}$$

If $3s\le j<5s$, then $$b_{ij}=
\begin{cases}
w_{4(j+m+1)\ra 4(j+m+s+1)+1}\otimes e_{4(j+s)+2},\quad i=j+s;\\
e_{4(j+m+s+1)+1}\otimes w_{4(j+s)+2\ra 4j+3},\quad i=j+3s;\\
0,\quad\text{otherwise.}
\end{cases}$$

If $5s\le j<6s$, then $$b_{ij}=
\begin{cases}
w_{4(j+m+1+sf(j,6s-1))+2\ra 4(j+m+2)}\otimes w_{4j+3\ra 4(j+1)},\quad i=s+(j+1)_s;\\
e_{4(j+m+2)}\otimes w_{4j+3\ra 4(j+s+1+sf(j,6s-1))+2},\quad i=4s+(j+1)_s;\\
e_{4(j+m+2)}\otimes w_{4j+3\ra 4(j+1+sf(j,6s-1))+2},\quad i=5s+(j+1)_s;\\
0,\quad\text{otherwise.}
\end{cases}$$

$(9)$ If $r_0=8$, then $\Omega^{8}(Y_t^{(13)})$ is described with
$(8s\times 6s)$-matrix with the following elements $b_{ij}${\rm:}

If $0\le j<2s$, then $$b_{ij}=
\begin{cases}
w_{4(j+m)-1\ra 4(j+m+s)+2}\otimes e_{4j},\quad i=(j)_s;\\
e_{4(j+m+s)+2}\otimes w_{4j\ra 4(j+s)+1},\quad i=s+(j+s)_{2s};\\
0,\quad\text{otherwise.}
\end{cases}$$

If $2s\le j<4s$, then $$b_{ij}=
\begin{cases}
w_{4(j+m)+2\ra 4(j+m)+3}\otimes e_{4j+1},\quad i=j-s;\\
0,\quad\text{otherwise.}
\end{cases}$$

If $4s\le j<6s$, then $$b_{ij}=
\begin{cases}
w_{4(j+m)+3\ra 4(j+m+1)}\otimes w_{4j+2\ra 4(j+1)},\quad i=(j+1)_s,\text{ }j<5s-1\text{ or }j=6s-1;\\
w_{4(j+m+s)+1\ra 4(j+m+1)}\otimes e_{4j+2},\quad i=j-s;\\
e_{4(j+m+1)}\otimes w_{4j+2\ra 4j+3},\quad i=5s+(j)_s;\\
0,\quad\text{otherwise.}
\end{cases}$$

If $6s\le j<8s$, then $$b_{ij}=
\begin{cases}
e_{4(j+m+s+1)+1}\otimes w_{4j+3\ra 4(j+1)+2},\quad i=3s+(j+1)_{2s};\\
0,\quad\text{otherwise.}
\end{cases}$$

$(10)$ If $r_0=9$, then $\Omega^{9}(Y_t^{(13)})$ is described with
$(9s\times 7s)$-matrix with the following elements $b_{ij}${\rm:}

If $s\le j<2s$, then $$b_{ij}=
\begin{cases}
w_{4(j+m)+3\ra 4(j+m+1)}\otimes e_{4j},\quad i=j-s;\\
0,\quad\text{otherwise.}
\end{cases}$$

If $2s\le j<4s$, then $$b_{ij}=
\begin{cases}
w_{4(j+m+1)\ra 4(j+m+1)+1}\otimes e_{4j+1},\quad i=j-s;\\
0,\quad\text{otherwise.}
\end{cases}$$

If $4s\le j<5s$, then $$b_{ij}=
\begin{cases}
e_{4(j+m+s+1)+1}\otimes e_{4j+2},\quad i=j-s;\\
0,\quad\text{otherwise.}
\end{cases}$$

If $5s\le j<6s$, then $$b_{ij}=
\begin{cases}
e_{4(j+m+s+1)+2}\otimes w_{4(j+s)+2\ra 4j+3},\quad i=j;\\
0,\quad\text{otherwise.}
\end{cases}$$

If $6s\le j<7s$, then $$b_{ij}=
\begin{cases}
e_{4(j+m+1)+1}\otimes e_{4(j+s)+2},\quad i=j-2s;\\
0,\quad\text{otherwise.}
\end{cases}$$

If $7s\le j<8s$, then $$b_{ij}=
\begin{cases}
e_{4(j+m+1)+2}\otimes w_{4j+2\ra 4j+3},\quad i=j-s;\\
0,\quad\text{otherwise.}
\end{cases}$$

If $8s\le j<9s$, then $$b_{ij}=
\begin{cases}
e_{4(j+m+1)+3}\otimes w_{4j+3\ra 4(j+1)},\quad i=(j+1)_s;\\
0,\quad\text{otherwise.}
\end{cases}$$

$(11)$ If $r_0=10$, then $\Omega^{10}(Y_t^{(13)})$ is described with
$(8s\times 6s)$-matrix with the following elements $b_{ij}${\rm:}

If $0\le j<2s$, then $$b_{ij}=
\begin{cases}
w_{4(j+m)\ra 4(j+m+s)+1}\otimes e_{4j},\quad i=(j)_s;\\
e_{4(j+m+s)+1}\otimes w_{4j\ra 4j+1},\quad i=j+s;\\
0,\quad\text{otherwise.}
\end{cases}$$

If $2s\le j<4s$, then $$b_{ij}=
\begin{cases}
w_{4(j+m+s)+1\ra 4(j+m+1)}\otimes e_{4j+1},\quad i=j-s;\\
w_{4(j+m+s)+2\ra 4(j+m+1)}\otimes w_{4j+1\ra 4j+2},\quad i=j+s;\\
w_{4(j+m)+3\ra 4(j+m+1)}\otimes w_{4j+1\ra 4j+3},\quad i=5s+(j)_s;\\
0,\quad\text{otherwise.}
\end{cases}$$

If $4s\le j<6s$, then $$b_{ij}=
\begin{cases}
w_{4(j+m+s)+2\ra 4(j+m)+3}\otimes e_{4j+2},\quad i=j-s;\\
e_{4(j+m)+3}\otimes w_{4j+2\ra 4j+3},\quad i=5s+(j)_s;\\
0,\quad\text{otherwise.}
\end{cases}$$

If $6s\le j<8s$, then $$b_{ij}=
\begin{cases}
w_{4(j+m+1)\ra 4(j+m+s+1)+2}\otimes w_{4j+3\ra 4(j+1)},\quad i=(j+1)_s;\\
w_{4(j+m+s+1)+1\ra 4(j+m+s+1)+2}\otimes w_{4j+3\ra 4(j+1)+1},\quad i=s+(j+1)_{2s};\\
e_{4(j+m+s+1)+2}\otimes w_{4j+3\ra 4(j+1)+2},\quad i=3s+(j+1)_{2s},\text{ }j<7s;\\
0,\quad\text{otherwise.}
\end{cases}$$

If $8s\le j<7s$, then $b_{ij}=0$.

If $7s\le j<8s$, then $$b_{ij}=
\begin{cases}
e_{4(j+m+s+1)+2}\otimes w_{4j+3\ra 4(j+1)+2},\quad i=3s+(j+1)_{2s};\\
0,\quad\text{otherwise.}
\end{cases}$$

\medskip
$({\rm II})$ Represent an arbitrary $t_0\in\N$ in the form
$t_0=11\ell_0+r_0$, where $0\le r_0\le 10.$ Then
$\Omega^{t_0}(Y_t^{(13)})$ is a $\Omega^{r_0}(Y_t^{(13)})$, whose left
components twisted by $\sigma^{\ell_0}$.
\end{pr}

\begin{pr}[Translates for the case 14]
$({\rm I})$ Let $r_0\in\N$, $r_0<11$. $r_0$-translates of the
elements $Y^{(14)}_t$ are described by the following way.

$(1)$ If $r_0=0$, then $\Omega^{0}(Y_t^{(14)})$ is described with
$(9s\times 6s)$-matrix with one nonzero element that is of the following form{\rm:}
$$b_{5s,7s}=\kappa^\ell(\a_{3(j+m+3)})w_{4(j+m)+3\ra 4(j+m+1)+3}\otimes e_{4j+3}.$$

$(2)$ If $r_0=1$, then $\Omega^{1}(Y_t^{(14)})$ is described with
$(8s\times 7s)$-matrix with the following two nonzero elements{\rm:}
$$b_{j+2s,2s+(3)_s}=-\kappa^\ell(\a_{3(j+m+3)})w_{4(j+m)+3\ra 4(j+m+1)+3}\otimes w_{4j+1\ra 4j+2};$$
$$b_{j+2s,3s+(3)_s}=-\kappa^\ell(\a_{3(j+m+3)})w_{4(j+m)+3\ra 4(j+m+1)+3}\otimes w_{4j+1\ra 4j+2}.$$

$(3)$ If $r_0=2$, then $\Omega^{2}(Y_t^{(14)})$ is described with
$(6s\times 6s)$-matrix with one nonzero element that is of the following form{\rm:}
$$b_{j,(3)_s}=-f_2(s,1)\kappa^\ell(\a_{3(j+m+3)})w_{4(j+m)-1\ra 4(j+m)+3}\otimes e_{4j}.$$

$(4)$ If $r_0=3$, then $\Omega^{3}(Y_t^{(14)})$ is described with
$(7s\times 8s)$-matrix with one nonzero element that is of the following form{\rm:}
$$b_{(j+1)_s,(2)_s}=-f_2(s,1)\kappa^\ell(\a_{3(j+m+3)})w_{4(j+m+1+sf(s,1))+2\ra 4(j+m+1)+3}\otimes w_{4j\ra 4(j+1)}.$$

$(5)$ If $r_0=4$, then $\Omega^{4}(Y_t^{(14)})$ is described with
$(6s\times 9s)$-matrix with one nonzero element that is of the following form{\rm:}
$$b_{s+(j+1)_s,(2)_s}=-f_2(s,1)\kappa^\ell(\a_{3(j+m+3)})e_{4(j+m+1)}\otimes w_{4j\ra 4(j+1)}.$$

$(6)$ If $r_0=5$, then $\Omega^{5}(Y_t^{(14)})$ is described with
$(6s\times 8s)$-matrix with one nonzero element that is of the following form{\rm:}
$$b_{(j+1)_s,(2)_s}=\kappa^{\ell+1}(\a_{3(j+m-1)})w_{4(j+m+1+sf(s,1))+1\ra 4(j+m+2)}\otimes w_{4j\ra 4(j+1)}.$$

$(7)$ If $r_0=6$, then $\Omega^{6}(Y_t^{(14)})$ is described with
$(7s\times 9s)$-matrix with one nonzero element that is of the following form{\rm:}
$$b_{(j+1)_s,s+2-3f(s,1)}=-w_{4(j+m+1)\ra 4(j+m+s+1)+1}\otimes w_{4j\ra 4(j+1)}.$$

$(8)$ If $r_0=7$, then $\Omega^{7}(Y_t^{(14)})$ is described with
$(6s\times 8s)$-matrix with one nonzero element that is of the following form{\rm:}
$$b_{(j+1)_s,(2)_s}=-\kappa^{\ell+1}(\a_{3(j+m-1)})w_{4(j+m+1+sf(s,1))+2\ra 4(j+m+1)+3}\otimes w_{4j\ra 4(j+1)}.$$

$(9)$ If $r_0=8$, then $\Omega^{8}(Y_t^{(14)})$ is described with
$(8s\times 6s)$-matrix with one nonzero element that is of the following form{\rm:}
$$b_{(j+1)_s,s+2-3f(s,1)}=-f_2(s,1)w_{4(j+m)+3\ra 4(j+m+s+1)+2}\otimes w_{4j\ra 4(j+1)}.$$

$(10)$ If $r_0=9$, then $\Omega^{9}(Y_t^{(14)})$ is described with
$(9s\times 7s)$-matrix with one nonzero element that is of the following form{\rm:}
$$b_{j,(2)_s}=\kappa^{\ell+1}(\a_{3(j+m-1)})w_{4(j+m)+3\ra 4(j+m+1)+3}\otimes e_{4j}.$$

$(11)$ If $r_0=10$, then $\Omega^{10}(Y_t^{(14)})$ is described with
$(8s\times 6s)$-matrix with one nonzero element that is of the following form{\rm:}
$$b_{j+s-f(s,1),4s+1}=-f_2(s,1)\kappa^{\ell+1}(\a_{3(j+m-1)})w_{4(j+m)+3\ra 4(j+m+1)+3}\otimes w_{4j+2\ra 4j+3}.$$

\medskip
$({\rm II})$ Represent an arbitrary $t_0\in\N$ in the form
$t_0=11\ell_0+r_0$, where $0\le r_0\le 10.$ Then
$\Omega^{t_0}(Y_t^{(14)})$ is a $\Omega^{r_0}(Y_t^{(14)})$, whose left
components twisted by $\sigma^{\ell_0}$.
\end{pr}

\begin{pr}[Translates for the case 15]
$({\rm I})$ Let $r_0\in\N$, $r_0<11$. $r_0$-translates of the
elements $Y^{(15)}_t$ are described by the following way.

$(1)$ If $r_0=0$, then $\Omega^{0}(Y_t^{(15)})$ is described with
$(9s\times 6s)$-matrix with the following elements $b_{ij}${\rm:}

If $0\le j<s$, then $$b_{ij}=
\begin{cases}
\kappa^\ell(\a_{3(j+m+3)})e_{4(j+m)}\otimes e_{4j},\quad i=j;\\
0,\quad\text{otherwise.}
\end{cases}$$

If $s\le j<2s$, then $b_{ij}=0$.

If $2s\le j<3s$, then $$b_{ij}=
\begin{cases}
w_{4(j+m)+1\ra 4(j+m)+2}\otimes e_{4j+1},\quad i=j-s;\\
0,\quad\text{otherwise.}
\end{cases}$$

If $3s\le j<4s$, then $b_{ij}=0$.

If $4s\le j<5s$, then $$b_{ij}=
\begin{cases}
w_{4(j+m+s)+1\ra 4(j+m+s)+2}\otimes e_{4(j+s)+1},\quad i=j-2s;\\
0,\quad\text{otherwise.}
\end{cases}$$

If $5s\le j<7s$, then $b_{ij}=0$.

If $7s\le j<8s$, then $$b_{ij}=
\begin{cases}
\kappa^\ell(\a_{3(j+m+3)})e_{4(j+m)+3}\otimes e_{4j+3},\quad i=j-2s;\\
0,\quad\text{otherwise.}
\end{cases}$$

If $8s\le j<9s$, then $b_{ij}=0$.

$(2)$ If $r_0=1$, then $\Omega^{1}(Y_t^{(15)})$ is described with
$(8s\times 7s)$-matrix with the following elements $b_{ij}${\rm:}

If $0\le j<2s$, then $$b_{ij}=
\begin{cases}
w_{4(j+m+s)+1\ra 4(j+m+s)+2}\otimes e_{4j},\quad i=(j+s)_{2s};\\
0,\quad\text{otherwise.}
\end{cases}$$

If $2s\le j<4s$, then $$b_{ij}=
\begin{cases}
\kappa_1w_{4(j+m)+2\ra 4(j+m)+3}\otimes e_{4j+1},\quad i=j;\\
\kappa_1e_{4(j+m)+3}\otimes w_{4j+1\ra 4j+2},\quad i=j+2s;\\
0,\quad\text{otherwise,}
\end{cases}$$
where $\kappa_1=-\kappa^\ell(\a_{3(j+m+3)})$.

If $4s\le j<5s$, then $b_{ij}=0$.

If $5s\le j<6s$, then $$b_{ij}=
\begin{cases}
-\kappa^\ell(\a_{3(j+m+4)})e_{4(j+m+1)}\otimes w_{4j+2\ra 4j+3},\quad i=j+s;\\
0,\quad\text{otherwise.}
\end{cases}$$

If $6s\le j<7s$, then $$b_{ij}=
\begin{cases}
w_{4(j+m+1)\ra 4(j+m+s+1)+1}\otimes e_{4j+3},\quad i=j;\\
0,\quad\text{otherwise.}
\end{cases}$$

If $7s\le j<8s$, then $b_{ij}=0$.

$(3)$ If $r_0=2$, then $\Omega^{2}(Y_t^{(15)})$ is described with
$(6s\times 6s)$-matrix with the following elements $b_{ij}${\rm:}

If $0\le j<s$, then $$b_{ij}=
\begin{cases}
-\kappa^\ell(\a_{3(j+m+2)})e_{4(j+m)-1}\otimes e_{4j},\quad i=j;\\
0,\quad\text{otherwise.}
\end{cases}$$

If $s\le j<3s$, then $$b_{ij}=
\begin{cases}
-w_{4(j+m)+1\ra 4(j+m)+2}\otimes w_{4(j+s)+1\ra 4(j+s)+2},\quad i=j+2s;\\
0,\quad\text{otherwise.}
\end{cases}$$

If $3s\le j<5s$, then $b_{ij}=0$.

If $5s\le j<6s$, then $$b_{ij}=
\begin{cases}
\kappa^\ell(\a_{3(j+m+4)})e_{4(j+m+1)}\otimes e_{4j+3},\quad i=j;\\
0,\quad\text{otherwise.}
\end{cases}$$

$(4)$ If $r_0=3$, then $\Omega^{3}(Y_t^{(15)})$ is described with
$(7s\times 8s)$-matrix with the following elements $b_{ij}${\rm:}

If $0\le j<s$, then $$b_{ij}=
\begin{cases}
-\kappa^\ell(\a_{3(j+m+3)})w_{4(j+m)+2\ra 4(j+m)+3}\otimes e_{4j},\quad i=j;\\
-\kappa^\ell(\a_{3(j+m+3)})e_{4(j+m)+3}\otimes w_{4j\ra 4j+1},\quad i=j+2s;\\
0,\quad\text{otherwise.}
\end{cases}$$

If $s\le j<3s$, then $$b_{ij}=
\begin{cases}
-\kappa^\ell(\a_{3(j+m+4)})e_{4(j+m+1)}\otimes w_{4(j+s)+1\ra 4(j+s)+2},\quad i=j+3s;\\
0,\quad\text{otherwise.}
\end{cases}$$

If $3s\le j<5s$, then $$b_{ij}=
\begin{cases}
w_{4(j+m+1)\ra 4(j+m+s+1)+1}\otimes e_{4(j+s)+2},\quad i=j+s;\\
0,\quad\text{otherwise.}
\end{cases}$$

If $5s\le j<7s$, then $$b_{ij}=
\begin{cases}
w_{4(j+m+1)+1\ra 4(j+m+1)+2}\otimes e_{4j+3},\quad i=6s+(j)_{2s};\\
0,\quad\text{otherwise.}
\end{cases}$$

$(5)$ If $r_0=4$, then $\Omega^{4}(Y_t^{(15)})$ is described with
$(6s\times 9s)$-matrix with the following elements $b_{ij}${\rm:}

If $0\le j<s$, then $$b_{ij}=
\begin{cases}
-\kappa^\ell(\a_{3(j+m+3)})e_{4(j+m)}\otimes e_{4j},\quad i=j+s;\\
0,\quad\text{otherwise.}
\end{cases}$$

If $s\le j<3s$, then $b_{ij}=0$.

If $3s\le j<5s$, then $$b_{ij}=
\begin{cases}
w_{4(j+m+s)+1\ra 4(j+m+s)+2}\otimes e_{4(j+s)+2},\quad i=j+2s-sf_0(j,4s);\\
0,\quad\text{otherwise.}
\end{cases}$$

If $5s\le j<6s$, then $$b_{ij}=
\begin{cases}
-\kappa^\ell(\a_{3(j+m+3)})e_{4(j+m)+3}\otimes e_{4j+3},\quad i=j+3s;\\
0,\quad\text{otherwise.}
\end{cases}$$

$(6)$ If $r_0=5$, then $\Omega^{5}(Y_t^{(15)})$ is described with
$(6s\times 8s)$-matrix with the following elements $b_{ij}${\rm:}

If $0\le j<s$, then $$b_{ij}=
\begin{cases}
\kappa_1w_{4(j+m+s)+1\ra 4(j+m+1)}\otimes e_{4j},\quad i=j+s;\\
\kappa_1e_{4(j+m+1)}\otimes w_{4j\ra 4j+1},\quad i=j+2s;\\
-\kappa_1w_{4(j+m)+3\ra 4(j+m+1)}\otimes w_{4j\ra 4j+2},\quad i=j+4s;\\
0,\quad\text{otherwise,}
\end{cases}$$
where $\kappa_1=-\kappa^\ell(\a_{3(j+m+4)})$.

If $s\le j<3s$, then $$b_{ij}=
\begin{cases}
e_{4(j+m+s+1)+1}\otimes w_{4(j+s)+1\ra 4(j+1)},\quad i=(j+s+1)_{2s};\\
-w_{4(j+m)+3\ra 4(j+m+s+1)+1}\otimes w_{4(j+s)+1\ra 4(j+s)+2},\quad i=j+3s;\\
0,\quad\text{otherwise.}
\end{cases}$$

If $3s\le j<5s$, then $$b_{ij}=
\begin{cases}
-w_{4(j+m+s+1)+1\ra 4(j+m+s+1)+2}\otimes w_{4(j+s)+2\ra 4(j+1)},\quad i=(j+s+1)_{2s};\\
w_{4(j+m)+3\ra 4(j+m+s+1)+2}\otimes e_{4(j+s)+2},\quad i=j+s;\\
0,\quad\text{otherwise.}
\end{cases}$$

If $5s\le j<6s$, then $$b_{ij}=
\begin{cases}
-\kappa_1w_{4(j+m+s+1)+1\ra 4(j+m+1)+3}\otimes w_{4j+3\ra 4(j+1)},\quad i=(j+s+1)_{2s};\\
\kappa_1e_{4(j+m+1)+3}\otimes w_{4j+3\ra 4(j+1)+2},\quad i=4s+(j+1)_{2s};\\
-\kappa_1w_{4(j+m+s+1)+2\ra 4(j+m+1)+3}\otimes e_{4j+3},\quad i=j+s;\\
0,\quad\text{otherwise,}
\end{cases}$$
where $\kappa_1=\kappa^\ell(\a_{3(j+m+4)})$.

$(7)$ If $r_0=6$, then $\Omega^{6}(Y_t^{(15)})$ is described with
$(7s\times 9s)$-matrix with the following elements $b_{ij}${\rm:}

If $0\le j<s$, then $$b_{ij}=
\begin{cases}
e_{4(j+m)+1}\otimes w_{4j\ra 4j+1},\quad i=j+s;\\
0,\quad\text{otherwise.}
\end{cases}$$

If $s\le j<2s$, then $$b_{ij}=
\begin{cases}
-w_{4(j+m)\ra 4(j+m)+1}\otimes e_{4j},\quad i=j-s;\\
e_{4(j+m)+1}\otimes w_{4j\ra 4j+1},\quad i=j+2s;\\
0,\quad\text{otherwise.}
\end{cases}$$

If $2s\le j<4s$, then $b_{ij}=0$.

If $4s\le j<5s$, then $$b_{ij}=
\begin{cases}
\kappa^\ell(\a_{3(j+m+3)})w_{4(j+m)+2\ra 4(j+m)+3}\otimes e_{4j+2},\quad i=j+s;\\
0,\quad\text{otherwise.}
\end{cases}$$

If $5s\le j<6s$, then $$b_{ij}=
\begin{cases}
-\kappa_1w_{4(j+m)+2\ra 4(j+m)+3}\otimes e_{4j+2},\quad i=j+s;\\
\kappa_1e_{4(j+m)+3}\otimes w_{4j+2\ra 4j+3},\quad i=j+2s;\\
0,\quad\text{otherwise,}
\end{cases}$$
where $\kappa_1=\kappa^\ell(\a_{3(j+m+3)})$.

If $6s\le j<7s$, then $$b_{ij}=
\begin{cases}
-f_2(j,7s-1)\kappa_1e_{4(j+m+1)}\otimes e_{4j+3},\quad i=j+2s;\\
\kappa_1e_{4(j+m+1)}\otimes w_{4j+3\ra 4(j+1)},\quad i=(j+1)_s,\text{ }j<7s-1;\\
0,\quad\text{otherwise,}
\end{cases}$$
where $\kappa_1=\kappa^\ell(\a_{3(j+m+4)})$.

$(8)$ If $r_0=7$, then $\Omega^{7}(Y_t^{(15)})$ is described with
$(6s\times 8s)$-matrix with the following elements $b_{ij}${\rm:}

If $0\le j<s$, then $$b_{ij}=
\begin{cases}
\kappa_1w_{4(j+m+s)+2\ra 4(j+m)+3}\otimes e_{4j},\quad i=j+s;\\
-\kappa_1e_{4(j+m)+3}\otimes w_{4j\ra 4j+1},\quad i=j+2s;\\
0,\quad\text{otherwise,}
\end{cases}$$
where $\kappa_1=\kappa^\ell(\a_{3(j+m+3)})$.

If $s\le j<3s$, then $$b_{ij}=
\begin{cases}
e_{4(j+m+s+1)+2}\otimes w_{4(j+s)+1\ra 4(j+1)},\quad i=(j+s+1)_{2s};\\
-w_{4(j+m+1)\ra 4(j+m+s+1)+2}\otimes w_{4(j+s)+1\ra 4(j+s)+2},\quad i=j+3s;\\
0,\quad\text{otherwise.}
\end{cases}$$

If $3s\le j<5s$, then $$b_{ij}=
\begin{cases}
-w_{4(j+m+1)\ra 4(j+m+1)+1}\otimes e_{4(j+s)+2},\quad i=j+s;\\
0,\quad\text{otherwise.}
\end{cases}$$

If $5s\le j<6s$, then $$b_{ij}=
\begin{cases}
-\kappa_1e_{4(j+m+2)}\otimes w_{4j+3\ra 4(j+s+1)+2},\quad i=4s+(j+s+1)_{2s};\\
\kappa_1w_{4(j+m+s+1)+1\ra 4(j+m+2)}\otimes e_{4j+3},\quad i=j+s;\\
\kappa_1w_{4(j+m+s+1)+2\ra 4(j+m+2)}\otimes w_{4j+3\ra 4(j+1)},\quad i=s+(j+1)_s,\text{ }j=6s-1;\\
-\kappa_1w_{4(j+m+1)+3\ra 4(j+m+2)}\otimes w_{4j+3\ra 4(j+1)+1},\quad i=2s+(j+1)_s,\text{ }j=6s-1;\\
0,\quad\text{otherwise,}
\end{cases}$$
where $\kappa_1=\kappa^{\ell+1}(\a_{3(j+m-1)})$.

$(9)$ If $r_0=8$, then $\Omega^{8}(Y_t^{(15)})$ is described with
$(8s\times 6s)$-matrix with the following elements $b_{ij}${\rm:}

If $0\le j<s$, then $$b_{ij}=
\begin{cases}
e_{4(j+m)+2}\otimes w_{4j\ra 4j+1},\quad i=j+s;\\
w_{4(j+m)+1\ra 4(j+m)+2}\otimes w_{4j\ra 4(j+s)+2},\quad i=j+4s;\\
0,\quad\text{otherwise.}
\end{cases}$$

If $s\le j<2s$, then $$b_{ij}=
\begin{cases}
-w_{4(j+m)-1\ra 4(j+m)+2}\otimes e_{4j},\quad i=j-s;\\
-e_{4(j+m)+2}\otimes w_{4j\ra 4j+1},\quad i=j+s;\\
-w_{4(j+m)+1\ra 4(j+m)+2}\otimes w_{4j\ra 4(j+s)+2},\quad i=j+2s;\\
0,\quad\text{otherwise.}
\end{cases}$$

If $2s\le j<4s$, then $$b_{ij}=
\begin{cases}
\kappa^\ell(\a_{3(j+m+3)})e_{4(j+m)+3}\otimes w_{4j+1\ra 4(j+1)},\quad i=(j+1)_s,\text{ }j<3s-1\text{ or }j=4s-1;\\
0,\quad\text{otherwise.}
\end{cases}$$

If $4s\le j<6s$, then $$b_{ij}=
\begin{cases}
\kappa_1w_{4(j+m+s)+1\ra 4(j+m+1)}\otimes e_{4j+2},\quad i=j-s;\\
-\kappa_1e_{4(j+m+1)}\otimes w_{4j+2\ra 4j+3},\quad i=j+s,\text{ }j<5s;\\
0,\quad\text{otherwise,}
\end{cases}$$
where $\kappa_1=-\kappa^{\ell+1}(\a_{3(j+m-2)})$.

If $6s\le j<7s$, then $$b_{ij}=
\begin{cases}
-w_{4(j+m+1)\ra 4(j+m+1)+1}\otimes e_{4j+3},\quad i=j-s;\\
0,\quad\text{otherwise.}
\end{cases}$$

If $7s\le j<8s$, then $b_{ij}=0$.

$(10)$ If $r_0=9$, then $\Omega^{9}(Y_t^{(15)})$ is described with
$(9s\times 7s)$-matrix with the following elements $b_{ij}${\rm:}

If $0\le j<s$, then $$b_{ij}=
\begin{cases}
\kappa^{\ell+1}(\a_{3(j+m-3)})e_{4(j+m)+3}\otimes e_{4j},\quad i=j;\\
0,\quad\text{otherwise.}
\end{cases}$$

If $s\le j<2s$, then $$b_{ij}=
\begin{cases}
-\kappa^{\ell+1}(\a_{3(j+m-2)})e_{4(j+m+1)}\otimes w_{4j\ra 4(j+s)+1},\quad i=j;\\
0,\quad\text{otherwise.}
\end{cases}$$

If $2s\le j<4s$, then $$b_{ij}=
\begin{cases}
-w_{4(j+m+1)\ra 4(j+m+s+1)+1}\otimes e_{4j+1},\quad i=j-s;\\
0,\quad\text{otherwise.}
\end{cases}$$

If $4s\le j<5s$, then $b_{ij}=0$.

If $5s\le j<6s$, then $$b_{ij}=
\begin{cases}
-w_{4(j+m+1)+1\ra 4(j+m+1)+2}\otimes e_{4(j+s)+2},\quad i=j-2s;\\
0,\quad\text{otherwise.}
\end{cases}$$

If $6s\le j<7s$, then $b_{ij}=0$.

If $7s\le j<8s$, then $$b_{ij}=
\begin{cases}
-w_{4(j+m+s+1)+1\ra 4(j+m+s+1)+2}\otimes e_{4j+2},\quad i=j-3s;\\
0,\quad\text{otherwise.}
\end{cases}$$

If $8s\le j<9s$, then $$b_{ij}=
\begin{cases}
\kappa_1w_{4(j+m+1)+2\ra 4(j+m+1)+3}\otimes e_{4j+3},\quad i=j-3s;\\
-\kappa_1e_{4(j+m+1)+3}\otimes w_{4j+3\ra 4(j+1)},\quad i=(j+1)_s,\text{ }j=9s-1;\\
0,\quad\text{otherwise,}
\end{cases}$$
where $\kappa_1=-\kappa^{\ell+1}(\a_{3(j+m-2)})$.

$(11)$ If $r_0=10$, then $\Omega^{10}(Y_t^{(15)})$ is described with
$(8s\times 6s)$-matrix with the following elements $b_{ij}${\rm:}

If $0\le j<s$, then $$b_{ij}=
\begin{cases}
-w_{4(j+m)\ra 4(j+m)+1}\otimes e_{4j},\quad i=j;\\
0,\quad\text{otherwise.}
\end{cases}$$

If $s\le j<2s$, then $b_{ij}=0$.

If $2s\le j<4s$, then $$b_{ij}=
\begin{cases}
\kappa_1w_{4(j+m+s)+1\ra 4(j+m+1)}\otimes e_{4j+1},\quad i=j-s;\\
\kappa_1e_{4(j+m+1)}\otimes w_{4j+1\ra 4(j+1)},\quad i=(j+1)_s,\text{ }j<3s-1\text{ or }j=4s-1;\\
0,\quad\text{otherwise,}
\end{cases}$$
where $\kappa_1=-f_1(j,3s)\kappa^\ell(\a_{3(j+m-2)})$.

If $4s\le j<6s$, then $$b_{ij}=
\begin{cases}
\kappa_1w_{4(j+m+s)+2\ra 4(j+m)+3}\otimes e_{4j+2},\quad i=j-s;\\
\kappa_1e_{4(j+m)+3}\otimes w_{4j+2\ra 4j+3},\quad i=j+s,\text{ }j<5s;\\
0,\quad\text{otherwise,}
\end{cases}$$
where $\kappa_1=\kappa^{\ell+1}(\a_{3(j+m-3)})$.

If $6s\le j<7s$, then $$b_{ij}=
\begin{cases}
w_{4(j+m+1)+1\ra 4(j+m+1)+2}\otimes w_{4j+3\ra 4(j+s+1)+1},\quad i=s+(j+s+1)_{2s};\\
-w_{4(j+m)+3\ra 4(j+m+1)+2}\otimes e_{4j+3},\quad i=j-s;\\
0,\quad\text{otherwise.}
\end{cases}$$

If $7s\le j<8s$, then $$b_{ij}=
\begin{cases}
-w_{4(j+m+1)+1\ra 4(j+m+1)+2}\otimes w_{4j+3\ra 4(j+s+1)+1},\quad i=s+(j+s+1)_{2s};\\
0,\quad\text{otherwise.}
\end{cases}$$

\medskip
$({\rm II})$ Represent an arbitrary $t_0\in\N$ in the form
$t_0=11\ell_0+r_0$, where $0\le r_0\le 10.$ Then
$\Omega^{t_0}(Y_t^{(15)})$ is a $\Omega^{r_0}(Y_t^{(15)})$, whose left
components twisted by $\sigma^{\ell_0}$.
\end{pr}

\begin{pr}[Translates for the case 16]
$({\rm I})$ Let $r_0\in\N$, $r_0<11$. $r_0$-translates of the
elements $Y^{(16)}_t$ are described by the following way.

$(1)$ If $r_0=0$, then $\Omega^{0}(Y_t^{(16)})$ is described with
$(8s\times 6s)$-matrix with the following two nonzero elements{\rm:}
$$b_{0,0}=w_{4(j+m)\ra 4(j+m)+2}\otimes e_{4j};$$
$$b_{s,2s}=w_{4(j+m)+1\ra 4(j+m)+3}\otimes e_{4j+1}.$$

$(2)$ If $r_0=1$, then $\Omega^{1}(Y_t^{(16)})$ is described with
$(6s\times 7s)$-matrix with the following nonzero elements{\rm:}
$$b_{j,0}=w_{4(j+m)+1\ra 4(j+m)+3}\otimes e_{4j};$$
$$b_{(j+1)_s,3s-1}=w_{4(j+m+s+1)+1\ra 4(j+m+s+1)+2}\otimes w_{4(j+s)+1\ra 4(j+1)};$$
$$b_{2s+(j+1)_s,3s-1}=e_{4(j+m+s+1)+2}\otimes w_{4(j+s)+1\ra 4(j+s+1)+1};$$
$$b_{(j+1)_s,6s-1}=w_{4(j+m+1)+1\ra 4(j+m+2)}\otimes w_{4j+3\ra 4(j+1)};$$
$$b_{j+s,6s-1}=w_{4(j+m+1)\ra 4(j+m+2)}\otimes e_{4j+3}.$$

$(3)$ If $r_0=2$, then $\Omega^{2}(Y_t^{(16)})$ is described with
$(7s\times 6s)$-matrix with the following nonzero elements{\rm:}
$$b_{(j+1)_s,2s-1}=w_{4(j+m)+3\ra 4(j+m+1)}\otimes w_{4(j+s)+1\ra 4(j+1)};$$
$$b_{j,2s-1}=w_{4(j+m+s)+2\ra 4(j+m+1)}\otimes e_{4(j+s)+1};$$
$$b_{j,3s-1}=w_{4(j+m+s)+2\ra 4(j+m+1)}\otimes e_{4(j+s)+1};$$
$$b_{j,4s-1}=w_{4(j+m)+1\ra 4(j+m+1)+1}\otimes e_{4(j+s)+2}.$$

$(4)$ If $r_0=3$, then $\Omega^{3}(Y_t^{(16)})$ is described with
$(6s\times 8s)$-matrix with the following nonzero elements{\rm:}
$$b_{j+s,s-1}=w_{4(j+m+s)+2\ra 4(j+m+1)}\otimes e_{4j};$$
$$b_{j+2s,s-1}=w_{4(j+m)+3\ra 4(j+m+1)}\otimes w_{4j\ra 4j+1};$$
$$b_{j+s,2s-1}=w_{4(j+m)+3\ra 4(j+m+1)+1}\otimes e_{4(j+s)+1}.$$

$(5)$ If $r_0=4$, then $\Omega^{4}(Y_t^{(16)})$ is described with
$(6s\times 9s)$-matrix with the following nonzero elements{\rm:}
$$b_{(j+1)_s,(s-2)_s}=w_{4(j+m)+3\ra 4(j+m+1)}\otimes w_{4j\ra 4(j+1)};$$
$$b_{s+(j+1)_s,(s-2)_s}=e_{4(j+m+1)}\otimes w_{4j\ra 4(j+1)};$$
$$b_{(j+1)_s,2s-(2)_s}=w_{4(j+m)+3\ra 4(j+m+1)+1}\otimes w_{4(j+s)+1\ra 4(j+1)};$$
$$b_{s+(j+1)_s,2s-(2)_s}=w_{4(j+m+1)\ra 4(j+m+1)+1}\otimes w_{4(j+s)+1\ra 4(j+1)};$$
$$b_{j+2s+f(s,1),4s-(2)_s}=w_{4(j+m)+2\ra 4(j+m+1)+2}\otimes e_{4(j+s)+2}.$$

$(6)$ If $r_0=5$, then $\Omega^{5}(Y_t^{(16)})$ is described with
$(7s\times 8s)$-matrix with the following two nonzero elements{\rm:}
$$b_{(j+1)_s,2s-2}=e_{4(j+m+s+1)+1}\otimes w_{4j\ra 4(j+1)};$$
$$b_{j,3s-(2)_s}=w_{4(j+m+1)\ra 4(j+m+s+1)+2}\otimes e_{4j+1}.$$

$(7)$ If $r_0=6$, then $\Omega^{6}(Y_t^{(16)})$ is described with
$(6s\times 9s)$-matrix with the following nonzero elements{\rm:}
$$b_{j+s+2f(s,1),(s-2)_s}=w_{4(j+m+f(s,1))+1\ra 4(j+m)+3}\otimes w_{4j\ra 4(j+f(s,1))+1};$$
$$b_{(j+1)_s,2s-3+2f(s,1)}=w_{4(j+m+1)\ra 4(j+m+1)+2}\otimes w_{4(j+s)+1\ra 4(j+1)};$$
$$b_{s+(j+1)_s,6s-3}=w_{4(j+m+s+1)+1\ra 4(j+m+2)}\otimes w_{4j+3\ra 4(j+s+1)+1},\text{ }s>1.$$

$(8)$ If $r_0=7$, then $\Omega^{7}(Y_t^{(16)})$ is described with
$(8s\times 8s)$-matrix with the following nonzero elements{\rm:}
$$b_{s+(j+1)_s-f(s,1),(s-3)_s}=e_{4(j+m+s+1)+2}\otimes w_{4j\ra 4(j+1)};$$
$$b_{(j+1)_s+f(s,1),2s+(s-3)_s}=w_{4(j+m+1)+2\ra 4(j+m+1)+3}\otimes w_{4j+1\ra 4(j+1)};$$
$$b_{2s+(j+1)_s+f(s,1),7s+(s-3)_s}=w_{4(j+m+1)+3\ra 4(j+m+s+2)+1}\otimes w_{4j+3\ra 4(j+s+1)+1};$$
$$b_{4s+(j+1)_s+f(s,1),7s+(s-3)_s}=w_{4(j+m+2)\ra 4(j+m+s+2)+1}\otimes w_{4j+3\ra 4(j+s+1)+2};$$
$$b_{5s+(j+1)_s-f(s,1),7s+(s-3)_s}=w_{4(j+m+2)\ra 4(j+m+s+2)+1}\otimes w_{4j+3\ra 4(j+1)+2}.$$

$(9)$ If $r_0=8$, then $\Omega^{8}(Y_t^{(16)})$ is described with
$(9s\times 6s)$-matrix with the following nonzero elements{\rm:}
$$b_{(j+1)_s,(s-3)_s}=e_{4(j+m)+3}\otimes w_{4j\ra 4(j+1)};$$
$$b_{(j+1)_s,s+(s-3)_s}=w_{4(j+m)+3\ra 4(j+m+1)}\otimes w_{4j\ra 4(j+1)};$$
$$b_{(j+1)_s,2s+(s-3)_s}=w_{4(j+m)+3\ra 4(j+m+1)+1}\otimes w_{4j+1\ra 4(j+1)};$$
$$b_{s+(j+1)_s+f(s,1),5s+(s-3)_s}=e_{4(j+m+s+1)+2}\otimes w_{4(j+s)+2\ra 4(j+s+1)+1};$$
$$b_{4s+(j+1)_s-f(s,1),6s+(s-3)_s}=e_{4(j+m+1)+1}\otimes w_{4(j+s)+2\ra 4(j+s+1)+2};$$
$$b_{(j+1)_s,8s+(s-3)_s}=w_{4(j+m)+3\ra 4(j+m+1)+3}\otimes w_{4j+3\ra 4(j+1)};$$
$$b_{2s+(j+1)_s-f(s,1),8s+(s-3)_s}=w_{4(j+m+s+1)+2\ra 4(j+m+1)+3}\otimes w_{4j+3\ra 4(j+s+1)+1};$$
$$b_{3s+(j+1)_s+f(s,1),8s+(s-3)_s}=w_{4(j+m+s+1)+1\ra 4(j+m+1)+3}\otimes w_{4j+3\ra 4(j+1)+2}.$$

$(10)$ If $r_0=9$, then $\Omega^{9}(Y_t^{(16)})$ is described with
$(8s\times 7s)$-matrix with the following nonzero elements{\rm:}
$$b_{(j+1)_s,3s-3+2f(s,1)}=w_{4(j+m+1)+3\ra 4(j+m+2)}\otimes w_{4j+1\ra 4(j+1)};$$
$$b_{(j+1)_s,5s-3+2f(s,1)}=e_{4(j+m+1)+3}\otimes w_{4j+2\ra 4(j+1)};$$
$$b_{(j+1)_s,6s-3+2f(s,1)}=e_{4(j+m+1)+3}\otimes w_{4j+2\ra 4(j+1)};$$
$$b_{(j+1)_s,7s-3+2f(s,1)}=w_{4(j+m+1)+3\ra 4(j+m+s+2)+2}\otimes w_{4j+3\ra 4(j+1)}.$$

$(11)$ If $r_0=10$, then $\Omega^{10}(Y_t^{(16)})$ is described with
$(9s\times 6s)$-matrix with the following nonzero elements{\rm:}
$$b_{(j+1)_s,(s-3)_s}=e_{4(j+m+1)}\otimes w_{4j\ra 4(j+1)};$$
$$b_{(j+1)_s,s+(s-3)_s}=w_{4(j+m+1)\ra 4(j+m+1)+1}\otimes w_{4(j+s)+1\ra 4(j+1)};$$
$$b_{s+(j+1)_s+f(s,1),s+(s-3)_s}=e_{4(j+m+1)+1}\otimes w_{4(j+s)+1\ra 4(j+s+1)+1};$$
$$b_{j-2s,6s+(s-3)_s}=w_{4(j+m)+2\ra 4(j+m+1)+2}\otimes e_{4(j+s)+2};$$
$$b_{j-s,6s+(s-3)_s}=w_{4(j+m)+3\ra 4(j+m+1)+2}\otimes w_{4(j+s)+2\ra 4j+3};$$
$$b_{j-2s,7s+(s-3)_s}=w_{4(j+m)+3\ra 4(j+m+1)+3}\otimes e_{4j+3}.$$

\medskip
$({\rm II})$ Represent an arbitrary $t_0\in\N$ in the form
$t_0=11\ell_0+r_0$, where $0\le r_0\le 10.$ Then
$\Omega^{t_0}(Y_t^{(16)})$ is a $\Omega^{r_0}(Y_t^{(16)})$, whose left
components twisted by $\sigma^{\ell_0}$.
\end{pr}

\begin{pr}[Translates for the case 17]
$({\rm I})$ Let $r_0\in\N$, $r_0<11$. $r_0$-translates of the
elements $Y^{(17)}_t$ are described by the following way.

$(1)$ If $r_0=0$, then $\Omega^{0}(Y_t^{(17)})$ is described with
$(8s\times 6s)$-matrix with the following two nonzero elements{\rm:}
$$b_{0,0}=w_{4(j+m)\ra 4(j+m+s)+2}\otimes e_{4j};$$
$$b_{0,s}=w_{4(j+m)\ra 4(j+m+s)+2}\otimes e_{4j}.$$

$(2)$ If $r_0=1$, then $\Omega^{1}(Y_t^{(17)})$ is described with
$(6s\times 7s)$-matrix with the following nonzero elements{\rm:}
$$b_{s+(j+1)_s-f(s,1),s+(1)_s}=-w_{4(j+m+1)+1\ra 4(j+m+1)+2}\otimes w_{4(j+s)+1\ra 4(j+1)};$$
$$b_{(j+1)_s+f(s,1),2s+(1)_s}=-w_{4(j+m+1)+1\ra 4(j+m+1)+2}\otimes w_{4(j+s)+1\ra 4(j+1)};$$
$$b_{j+s,5s+(1)_s}=-\kappa^\ell(\a_{3(j+m+5)})w_{4(j+m+1)\ra 4(j+m+2)}\otimes e_{4j+3}.$$

$(3)$ If $r_0=2$, then $\Omega^{2}(Y_t^{(17)})$ is described with
$(7s\times 6s)$-matrix with the following nonzero elements{\rm:}
$$b_{(j+1)_s,(1)_s}=f_2(s,1)\kappa^\ell(\a_{3(j+m+3)})e_{4(j+m)+3}\otimes w_{4j\ra 4(j+1)};$$
$$b_{(j+1)_s,s+(1)_s}=-f_2(s,1)\kappa^\ell(\a_{3(j+m+4)})w_{4(j+m)+3\ra 4(j+m+1)}\otimes w_{4(j+s)+1\ra 4(j+1)};$$
$$b_{j,s+(1)_s}=-\kappa^\ell(\a_{3(j+m+4)})w_{4(j+m+s)+2\ra 4(j+m+1)}\otimes e_{4(j+s)+1};$$
$$b_{(j+1)_s,2s+(1)_s}=f_2(s,1)\kappa^\ell(\a_{3(j+m+4)})w_{4(j+m)+3\ra 4(j+m+1)}\otimes w_{4(j+s)+1\ra 4(j+1)};$$
$$b_{j,2s+(1)_s}=-\kappa^\ell(\a_{3(j+m+4)})w_{4(j+m+s)+2\ra 4(j+m+1)}\otimes e_{4(j+s)+1}.$$

$(4)$ If $r_0=3$, then $\Omega^{3}(Y_t^{(17)})$ is described with
$(6s\times 8s)$-matrix with the following nonzero elements{\rm:}
$$b_{j+s,(1)_s}=\kappa^\ell(\a_{3(j+m+4)})w_{4(j+m+s)+2\ra 4(j+m+1)}\otimes e_{4j};$$
$$b_{j+2s,(1)_s}=-\kappa^\ell(\a_{3(j+m+4)})w_{4(j+m)+3\ra 4(j+m+1)}\otimes w_{4j\ra 4j+1};$$
$$b_{j+s,s+(1)_s}=-w_{4(j+m)+3\ra 4(j+m+s+1)+1}\otimes e_{4(j+s)+1};$$
$$b_{j+s,2s+(1)_s}=-w_{4(j+m)+3\ra 4(j+m+s+1)+1}\otimes e_{4(j+s)+1}.$$

$(5)$ If $r_0=4$, then $\Omega^{4}(Y_t^{(17)})$ is described with
$(6s\times 9s)$-matrix with the following nonzero elements{\rm:}
$$b_{(j+1)_s,0}=-f_2(s,1)\kappa^\ell(\a_{3(j+m+4)})w_{4(j+m)+3\ra 4(j+m+1)}\otimes w_{4j\ra 4(j+1)};$$
$$b_{s+(j+1)_s,0}=f_2(s,1)\kappa^\ell(\a_{3(j+m+4)})e_{4(j+m+1)}\otimes w_{4j\ra 4(j+1)};$$
$$b_{(j+1)_s,2s-f(s,1)}=w_{4(j+m)+3\ra 4(j+m+s+1)+1}\otimes w_{4(j+s)+1\ra 4(j+1)};$$
$$b_{(j+1)_s,4s-f(s,1)}=-w_{4(j+m)+3\ra 4(j+m+s+1)+2}\otimes w_{4(j+s)+2\ra 4(j+1)};$$
$$b_{(j+1)_s,5s}=-f_2(s,1)\kappa^\ell(\a_{3(j+m+4)})w_{4(j+m)+3\ra 4(j+m+1)+3}\otimes w_{4j+3\ra 4(j+1)};$$
$$b_{2s+(j+1)_s,5s}=f_2(s,1)\kappa^\ell(\a_{3(j+m+4)})w_{4(j+m+1+f(s,1))+1\ra 4(j+m+1)+3}\otimes w_{4j+3\ra 4(j+s+1+f(s,1))+1};$$
$$b_{5s+(j+1)_s,5s}=-f_2(s,1)\kappa^\ell(\a_{3(j+m+4)})w_{4(j+m+1+f(s,1))+2\ra 4(j+m+1)+3}\otimes w_{4j+3\ra 4(j+s+1+f(s,1))+2}.$$

$(6)$ If $r_0=5$, then $\Omega^{5}(Y_t^{(17)})$ is described with
$(7s\times 8s)$-matrix with the following nonzero elements{\rm:}
$$b_{(j+1)_s+f(s,1),0}=-e_{4(j+m+1)+1}\otimes w_{4j\ra 4(j+1)};$$
$$b_{s+(j+1)_s-f(s,1),s}=-e_{4(j+m+1)+1}\otimes w_{4j\ra 4(j+1)};$$
$$b_{(j+1)_s,4s+f(s,1)}=f_2(s,1)\kappa^\ell(\a_{3(j+m+4)})w_{4(j+m+1)+1\ra 4(j+m+1)+3}\otimes w_{4j+2\ra 4(j+1)};$$
$$b_{s+(j+1)_s,4s+f(s,1)}=f_2(s,1)\kappa^\ell(\a_{3(j+m+4)})w_{4(j+m+s+1)+1\ra 4(j+m+1)+3}\otimes w_{4j+2\ra 4(j+1)};$$
$$b_{4s+(j+1)_s,4s+f(s,1)}=-f_2(s,1)\kappa^\ell(\a_{3(j+m+4)})e_{4(j+m+1)+3}\otimes w_{4j+2\ra 4(j+1)+2};$$
$$b_{2s+(j+1)_s,6s}=\kappa^\ell(\a_{3(j+m+5)})e_{4(j+m+2)}\otimes w_{4j+3\ra 4(j+1+f(s,1))+1}.$$

$(7)$ If $r_0=6$, then $\Omega^{6}(Y_t^{(17)})$ is described with
$(6s\times 9s)$-matrix with the following nonzero elements{\rm:}
$$b_{(j+1)_s,s+f(s,1)}=w_{4(j+m+1)\ra 4(j+m+s+1)+2}\otimes w_{4(j+s)+1\ra 4(j+1)};$$
$$b_{(j+1)_s,3s+f(s,1)}=w_{4(j+m+1)\ra 4(j+m+1)+1}\otimes w_{4(j+s)+2\ra 4(j+1)};$$
$$b_{s+(j+1)_s,5s}=f_2(s,1)\kappa^{\ell+1}(\a_{3(j+m-1)})w_{4(j+m+s+1+f(s,1))+1\ra 4(j+m+2)}\otimes w_{4j+3\ra 4(j+s+1+f(s,1))+1}.$$

$(8)$ If $r_0=7$, then $\Omega^{7}(Y_t^{(17)})$ is described with
$(8s\times 8s)$-matrix with the following nonzero elements{\rm:}
$$b_{(j+1)_s+f(s,1),0}=-e_{4(j+m+1)+2}\otimes w_{4j\ra 4(j+1)};$$
$$b_{s+(j+1)_s-f(s,1),s}=e_{4(j+m+1)+2}\otimes w_{4j\ra 4(j+1)};$$
$$b_{(j+1)_s,4s+f(s,1)}=\kappa^{\ell+1}(\a_{3(j+m-1)})w_{4(j+m+1)+2\ra 4(j+m+2)}\otimes w_{4j+2\ra 4(j+1)};$$
$$b_{4s+(j+1)_s,4s+f(s,1)}=-\kappa^{\ell+1}(\a_{3(j+m-1)})e_{4(j+m+2)}\otimes w_{4j+2\ra 4(j+1)+2};$$
$$b_{2s+(j+1)_s,6s+f(s,1)}=w_{4(j+m+1)+3\ra 4(j+m+2)+1}\otimes w_{4j+3\ra 4(j+1)+1};$$
$$b_{4s+(j+1)_s,6s+f(s,1)}=w_{4(j+m+2)\ra 4(j+m+2)+1}\otimes w_{4j+3\ra 4(j+1)+2};$$
$$b_{5s+(j+1)_s,6s+f(s,1)}=w_{4(j+m+2)\ra 4(j+m+2)+1}\otimes w_{4j+3\ra 4(j+s+1)+2}.$$

$(9)$ If $r_0=8$ and $s=1$, then $\Omega^{8}(Y_t^{(17)})$ is described with
$(9s\times 6s)$-matrix with the following nonzero elements{\rm:}
$$b_{j,0}=-\kappa^{\ell+1}(\a_{3(j+m-3)})e_{4(j+m)+3}\otimes e_{4j};$$
$$b_{j-s,s}=\kappa^{\ell+1}(\a_{3(j+m-3)})w_{4(j+m)+3\ra 4(j+m)}\otimes e_{4j};$$
$$b_{j-s,5s}=-w_{4(j+m+1)+1\ra 4(j+m+1)+2}\otimes w_{4(j+1)+2\ra 4j+2};$$
$$b_{j-6s,6s}=w_{4(j+m)+3\ra 4(j+m)+1}\otimes w_{4(j+1)+2\ra 4j};$$
$$b_{j-4s,7s}=-w_{4(j+m)+1\ra 4(j+m)+2}\otimes w_{4j+2\ra 4(j+1)+2};$$
$$b_{j-8s,8s}=-\kappa^{\ell+1}(\a_{3(j+m-3)})w_{4(j+m)+3\ra 4(j+m+1)+3}\otimes w_{4j+3\ra 4j};$$
$$b_{j-6s,8s}=-\kappa^{\ell+1}(\a_{3(j+m-3)})w_{4(j+m+1)+2\ra 4(j+m)+3}\otimes w_{4j+3\ra 4(j+1)+1};$$
$$b_{j-5s,8s}=-\kappa^{\ell+1}(\a_{3(j+m-3)})w_{4(j+m+1)+1\ra 4(j+m)+3}\otimes w_{4j+3\ra 4j+2}.$$

$(10)$ If $r_0=8$ and $s>1$, then $\Omega^{8}(Y_t^{(17)})$ is described with
$(9s\times 6s)$-matrix with the following nonzero elements{\rm:}
$$b_{(j+1)_s,0}=\kappa^{\ell+1}(\a_{3(j+m-3)})e_{4(j+m)+3}\otimes w_{4j\ra 4(j+1)};$$
$$b_{(j+1)_s,4s}=w_{4(j+m)+3\ra 4(j+m+1)+1}\otimes w_{4j+2\ra 4(j+1)};$$
$$b_{3s+(j+1)_s,5s}=-w_{4(j+m+1)+1\ra 4(j+m+1)+2}\otimes w_{4(j+s)+2\ra 4(j+s+1)+2};$$
$$b_{4s+(j+1)_s,7s}=-w_{4(j+m+s+1)+1\ra 4(j+m+s+1)+2}\otimes w_{4j+2\ra 4(j+1)+2};$$
$$b_{(j+1)_s,8s}=\kappa^{\ell+1}(\a_{3(j+m-2)})w_{4(j+m)+3\ra 4(j+m+1)+3}\otimes w_{4j+3\ra 4(j+1)};$$
$$b_{2s+(j+1)_s,8s}=\kappa^{\ell+1}(\a_{3(j+m-2)})w_{4(j+m+s+1)+2\ra 4(j+m+1)+3}\otimes w_{4j+3\ra 4(j+s+1)+1};$$
$$b_{3s+(j+1)_s,8s}=\kappa^{\ell+1}(\a_{3(j+m-2)})w_{4(j+m+s+1)+1\ra 4(j+m+1)+3}\otimes w_{4j+3\ra 4(j+1)+2}.$$

$(11)$ If $r_0=9$ and $s=1$, then $\Omega^{9}(Y_t^{(17)})$ is described with
$(8s\times 7s)$-matrix with the following two nonzero elements{\rm:}
$$b_{j-4s,4s}=-\kappa^{\ell+1}(\a_{3(j+m-2)})e_{4(j+m)+3}\otimes w_{4j+2\ra 4j};$$
$$b_{j-6s,6s}=w_{4(j+m)+3\ra 4(j+m)+2}\otimes w_{4j+3\ra 4j}.$$

$(12)$ If $r_0=9$ and $s>1$, then $\Omega^{9}(Y_t^{(17)})$ is described with
$(8s\times 7s)$-matrix with the following two nonzero elements{\rm:}
$$b_{(j+1)_s,5s}=-\kappa^{\ell+1}(\a_{3(j+m-2)})e_{4(j+m+1)+3}\otimes w_{4j+2\ra 4(j+1)};$$
$$b_{(j+1)_s,7s}=-w_{4(j+m+1)+3\ra 4(j+m+2)+2}\otimes w_{4j+3\ra 4(j+1)}.$$

$(13)$ If $r_0=10$ and $s=1$, then $\Omega^{10}(Y_t^{(17)})$ is described with
$(9s\times 6s)$-matrix with the following nonzero elements{\rm:}
$$b_{j,0}=-\kappa^{\ell+1}(\a_{3(j+m-2)})e_{4(j+m)}\otimes e_{4j};$$
$$b_{j+s,2s}=w_{4(j+m+1)+2\ra 4(j+m)+2}\otimes w_{4j+1\ra 4j+2};$$
$$b_{j+3s,2s}=w_{4(j+m)+3\ra 4(j+m)+2}\otimes w_{4j+1\ra 4j+3};$$
$$b_{j,4s}=w_{4(j+m)+2\ra 4(j+m+1)+2}\otimes w_{4(j+1)+1\ra 4(j+1)+2};$$
$$b_{j+s,4s}=-w_{4(j+m)+3\ra 4(j+m+1)+2}\otimes w_{4(j+1)+1\ra 4j+3};$$
$$b_{j-2s,7s}=-\kappa^{\ell+1}(\a_{3(j+m-2)})e_{4(j+m)+3}\otimes e_{4j+3}.$$

$(14)$ If $r_0=10$ and $s>1$, then $\Omega^{10}(Y_t^{(17)})$ is described with
$(9s\times 6s)$-matrix with the following nonzero elements{\rm:}
$$b_{(j+1)_s,0}=\kappa^{\ell+1}(\a_{3(j+m-2)})e_{4(j+m+1)}\otimes w_{4j\ra 4(j+1)};$$
$$b_{j+s,2s}=w_{4(j+m+s)+2\ra 4(j+m+s+1)+2}\otimes w_{4j+1\ra 4j+2};$$
$$b_{j+3s,2s}=w_{4(j+m)+3\ra 4(j+m+s+1)+2}\otimes w_{4j+1\ra 4j+3};$$
$$b_{j,4s}=w_{4(j+m)+2\ra 4(j+m+1)+2}\otimes w_{4(j+s)+1\ra 4(j+s)+2};$$
$$b_{j+s,4s}=-w_{4(j+m)+3\ra 4(j+m+1)+2}\otimes w_{4(j+s)+1\ra 4j+3};$$
$$b_{j-2s,7s}=-\kappa^{\ell+1}(\a_{3(j+m-2)})w_{4(j+m)+3\ra 4(j+m+1)+3}\otimes e_{4j+3}.$$

\medskip
$({\rm II})$ Represent an arbitrary $t_0\in\N$ in the form
$t_0=11\ell_0+r_0$, where $0\le r_0\le 10.$ Then
$\Omega^{t_0}(Y_t^{(17)})$ is a $\Omega^{r_0}(Y_t^{(17)})$, whose left
components twisted by $\sigma^{\ell_0}$,
and coefficients multiplied by $(-1)^{\ell_0}$.
\end{pr}

\begin{pr}[Translates for the case 18]
$({\rm I})$ Let $r_0\in\N$, $r_0<11$. $r_0$-translates of the
elements $Y^{(18)}_t$ are described by the following way.

$(1)$ If $r_0=0$, then $\Omega^{0}(Y_t^{(18)})$ is described with
$(6s\times 6s)$-matrix with the following elements $b_{ij}${\rm:}

If $s\le j<3s$, then $$b_{ij}=
\begin{cases}
w_{4(j+m+s)+1\ra 4(j+m+s)+2}\otimes e_{4(j+s)+1},\quad i=j;\\
0,\quad\text{otherwise.}
\end{cases}$$

If $3s\le j<5s$, then $b_{ij}=0$.

If $5s\le j<6s$, then $$b_{ij}=
\begin{cases}
-\kappa^\ell(\a_{3(j+m+4+1)})w_{4(j+m)+3\ra 4(j+m+1)}\otimes e_{4j+3},\quad i=j;\\
0,\quad\text{otherwise.}
\end{cases}$$

$(2)$ If $r_0=1$, then $\Omega^{1}(Y_t^{(18)})$ is described with
$(7s\times 7s)$-matrix with the following elements $b_{ij}${\rm:}

If $0\le j<s$, then $$b_{ij}=
\begin{cases}
\kappa_1w_{4(j+m)+1\ra 4(j+m)+3}\otimes e_{4j},\quad i=j;\\
-\kappa_1w_{4(j+m+s)+1\ra 4(j+m)+3}\otimes e_{4j},\quad i=j+s;\\
0,\quad\text{otherwise,}
\end{cases}$$
where $\kappa_1=\kappa^\ell(\a_{3(j+4+m)})$.

If $s\le j<3s$, then $$b_{ij}=
\begin{cases}
\kappa_1w_{4(j+m+s)+2\ra 4(j+m+1)}\otimes e_{4(j+s)+1},\quad i=j+s;\\
\kappa_1w_{4(j+m)+3\ra 4(j+m+1)}\otimes w_{4(j+s)+1\ra 4(j+s)+2},\quad i=j+3s;\\
0,\quad\text{otherwise,}
\end{cases}$$
where $\kappa_1=\kappa^\ell(\a_{3(j+m+4+1)})$.

If $3s\le j<5s$, then $$b_{ij}=
\begin{cases}
-w_{4(j+m)+3\ra 4(j+m+1)+1}\otimes e_{4(j+s)+2},\quad i=j+s;\\
0,\quad\text{otherwise.}
\end{cases}$$

If $5s\le j<7s$, then $$b_{ij}=
\begin{cases}
-w_{4(j+m+s+1)+1\ra 4(j+m+s+1)+2}\otimes w_{4j+3\ra 4(j+1)},\quad i=(j+s+1)_{2s};\\
-w_{4(j+m+1)\ra 4(j+m+s+1)+2}\otimes e_{4j+3},\quad i=6s+(j)_s;\\
0,\quad\text{otherwise.}
\end{cases}$$

$(3)$ If $r_0=2$, then $\Omega^{2}(Y_t^{(18)})$ is described with
$(6s\times 6s)$-matrix with the following elements $b_{ij}${\rm:}

If $0\le j<s$, then $$b_{ij}=
\begin{cases}
\kappa^\ell(\a_{3(j+m+4)})w_{4(j+m)-1\ra 4(j+m)}\otimes e_{4j},\quad i=j;\\
0,\quad\text{otherwise.}
\end{cases}$$

If $s\le j<3s$, then $b_{ij}=0$.

If $3s\le j<5s$, then $$b_{ij}=
\begin{cases}
-w_{4(j+m)+1\ra 4(j+m)+2}\otimes e_{4(j+s)+2},\quad i=j;\\
0,\quad\text{otherwise.}
\end{cases}$$

If $5s\le j<6s$, then $b_{ij}=0$.

$(4)$ If $r_0=3$, then $\Omega^{3}(Y_t^{(18)})$ is described with
$(6s\times 8s)$-matrix with the following elements $b_{ij}${\rm:}

If $0\le j<s$, then $$b_{ij}=
\begin{cases}
\kappa_1w_{4(j+m)+2\ra 4(j+m+1)}\otimes e_{4j},\quad i=j;\\
\kappa_1w_{4(j+m)+3\ra 4(j+m+1)}\otimes w_{4j\ra 4j+1},\quad i=j+2s;\\
0,\quad\text{otherwise,}
\end{cases}$$
where $\kappa_1=\kappa^\ell(\a_{3(j+m+4+1)})$.

If $s\le j<3s$, then $$b_{ij}=
\begin{cases}
-w_{4(j+m)+3\ra 4(j+m+1)+1}\otimes e_{4(j+s)+1},\quad i=j+s;\\
0,\quad\text{otherwise.}
\end{cases}$$

If $3s\le j<4s$, then $$b_{ij}=
\begin{cases}
-f_1(j,4s-1)e_{4(j+m+1)+2}\otimes w_{4(j+s)+2\ra 4(j+1)},\quad i=(j+1)_{2s};\\
0,\quad\text{otherwise.}
\end{cases}$$

If $4s\le j<5s$, then $$b_{ij}=
\begin{cases}
f_1(j,5s-1)e_{4(j+m+1)+2}\otimes w_{4(j+s)+2\ra 4(j+1)},\quad i=(j+1)_{2s};\\
0,\quad\text{otherwise.}
\end{cases}$$

If $5s\le j<6s$, then $$b_{ij}=
\begin{cases}
\kappa_1w_{4(j+m+s+1+sf(j,6s-1))+2\ra 4(j+m+1)+3}\otimes w_{4j+3\ra 4(j+1)},\quad i=(j+1)_s;\\
\kappa_1e_{4(j+m+1)+3}\otimes w_{4j+3\ra 4(j+s+1+sf(j,6s-1))+1},\quad i=2s+(j+1)_s;\\
0,\quad\text{otherwise,}
\end{cases}$$
where $\kappa_1=-\kappa^\ell(\a_{3(j+m+4+1)})f_1(j,6s-1)$.

$(5)$ If $r_0=4$, then $\Omega^{4}(Y_t^{(18)})$ is described with
$(7s\times 9s)$-matrix with the following elements $b_{ij}${\rm:}

If $s\le j<2s$, then $$b_{ij}=
\begin{cases}
w_{4(j+m)-1\ra 4(j+m+s)+1}\otimes e_{4j},\quad i=j-s;\\
0,\quad\text{otherwise.}
\end{cases}$$

If $2s\le j<3s$, then $$b_{ij}=
\begin{cases}
w_{4(j+m+s)+1\ra 4(j+m+s)+2}\otimes e_{4j+1},\quad i=j;\\
-e_{4(j+m+s)+2}\otimes w_{4j+1\ra 4j+2},\quad i=j+3s;\\
0,\quad\text{otherwise.}
\end{cases}$$

If $3s\le j<4s$, then $$b_{ij}=
\begin{cases}
w_{4(j+m+s)+1\ra 4(j+m+s)+2}\otimes e_{4j+1},\quad i=j;\\
-e_{4(j+m+s)+2}\otimes w_{4j+1\ra 4j+2},\quad i=j+4s;\\
0,\quad\text{otherwise.}
\end{cases}$$

If $4s\le j<6s$, then $$b_{ij}=
\begin{cases}
\kappa^\ell(\a_{3(j+4+m)})f_1(j,6s-1)e_{4(j+m)+3}\otimes w_{4j+2\ra 4(j+1)},\\\quad\quad\quad i=(j+1)_s,\text{ }j<5s-1\text{ or }j=6s-1;\\
0,\quad\text{otherwise.}
\end{cases}$$

If $6s\le j<7s$, then $b_{ij}=0$.

$(6)$ If $r_0=5$, then $\Omega^{5}(Y_t^{(18)})$ is described with
$(6s\times 8s)$-matrix with the following elements $b_{ij}${\rm:}

If $0\le j<s$, then $$b_{ij}=
\begin{cases}
\kappa_1w_{4(j+m)+1\ra 4(j+m)+3}\otimes e_{4j},\quad i=j;\\
\kappa_1w_{4(j+m+s)+1\ra 4(j+m)+3}\otimes e_{4j},\quad i=j+s;\\
-\kappa_1e_{4(j+m)+3}\otimes w_{4j\ra 4j+2},\quad i=j+4s;\\
-\kappa_1e_{4(j+m)+3}\otimes w_{4j\ra 4(j+s)+2},\quad i=j+5s;\\
0,\quad\text{otherwise,}
\end{cases}$$
where $\kappa_1=-\kappa^\ell(\a_{3(j+4+m)})$.

If $s\le j<3s$, then $$b_{ij}=
\begin{cases}
-w_{4(j+m+1)+1\ra 4(j+m+1)+2}\otimes w_{4(j+s)+1\ra 4(j+1)},\quad i=(j+1)_{2s};\\
-w_{4(j+m+1)\ra 4(j+m+1)+2}\otimes e_{4(j+s)+1},\quad i=j+s;\\
0,\quad\text{otherwise.}
\end{cases}$$

If $3s\le j<5s$, then $b_{ij}=0$.

If $5s\le j<6s$, then $$b_{ij}=
\begin{cases}
-\kappa^{\ell+1}(\a_{3(j+m)})f_1(j,6s-1)e_{4(j+m+2)}\otimes w_{4j+3\ra 4(j+s+1+sf(j,6s-1))+1},\quad i=2s+(j+1)_s;\\
0,\quad\text{otherwise.}
\end{cases}$$

$(7)$ If $r_0=6$, then $\Omega^{6}(Y_t^{(18)})$ is described with
$(8s\times 9s)$-matrix with the following elements $b_{ij}${\rm:}

If $0\le j<s$, then $$b_{ij}=
\begin{cases}
w_{4(j+m)\ra 4(j+m+s)+2}\otimes e_{4j},\quad i=j;\\
-w_{4(j+m+s)+1\ra 4(j+m+s)+2}\otimes w_{4j\ra 4(j+s)+1},\quad i=j+3s;\\
0,\quad\text{otherwise.}
\end{cases}$$

If $s\le j<2s$, then $$b_{ij}=
\begin{cases}
-w_{4(j+m)\ra 4(j+m+s)+2}\otimes e_{4j},\quad i=j-s;\\
w_{4(j+m+s)+1\ra 4(j+m+s)+2}\otimes w_{4j\ra 4(j+s)+1},\quad i=j;\\
0,\quad\text{otherwise.}
\end{cases}$$

If $2s\le j<3s$, then $$b_{ij}=
\begin{cases}
\kappa^\ell(\a_{3(j+m+4)})f_2((j)_s,s-1)w_{4(j+m)+1\ra 4(j+m)+3}\otimes e_{4j+1},\quad i=j-s;\\
0,\quad\text{otherwise.}
\end{cases}$$

If $3s\le j<4s$, then $$b_{ij}=
\begin{cases}
-\kappa^\ell(\a_{3(j+m+4)})f_2((j)_s,s-1)w_{4(j+m)+1\ra 4(j+m)+3}\otimes e_{4j+1},\quad i=j;\\
0,\quad\text{otherwise.}
\end{cases}$$

If $4s\le j<6s$, then $$b_{ij}=
\begin{cases}
-\kappa^{\ell+1}(\a_{3(j+m-1)})e_{4(j+m+1)}\otimes w_{4j+2\ra 4(j+1)},\quad i=(j+1)_s,\text{ }j<5s-1\text{ or }j=6s-1;\\
0,\quad\text{otherwise.}
\end{cases}$$

If $6s\le j<7s-1$, then $$b_{ij}=
\begin{cases}
w_{4(j+m+1)\ra 4(j+m+s+1)+1}\otimes w_{4j+3\ra 4(j+1)},\quad i=(j+1)_s;\\
0,\quad\text{otherwise.}
\end{cases}$$

If $7s-1\le j<6s$, then $b_{ij}=0$.

If $6s\le j<7s$, then $$b_{ij}=
\begin{cases}
-e_{4(j+m+s+1)+1}\otimes w_{4j+3\ra 4(j+s+1)+1},\quad i=3s+(j+1)_s-2sf(j,7s-1);\\
0,\quad\text{otherwise.}
\end{cases}$$

If $7s\le j<8s-1$, then $b_{ij}=0$.

If $8s-1\le j<8s$, then $$b_{ij}=
\begin{cases}
w_{4(j+m+1)\ra 4(j+m+s+1)+1}\otimes w_{4j+3\ra 4(j+1)},\quad i=(j+1)_s;\\
0,\quad\text{otherwise.}
\end{cases}$$

If $8s\le j<7s$, then $b_{ij}=0$.

If $7s\le j<8s$, then $$b_{ij}=
\begin{cases}
-e_{4(j+m+s+1)+1}\otimes w_{4j+3\ra 4(j+s+1)+1},\quad i=s+(j+1)_s+2sf(j,8s-1);\\
0,\quad\text{otherwise.}
\end{cases}$$

$(8)$ If $r_0=7$, then $\Omega^{7}(Y_t^{(18)})$ is described with
$(9s\times 8s)$-matrix with the following elements $b_{ij}${\rm:}

If $s\le j<2s$, then $$b_{ij}=
\begin{cases}
-\kappa_1w_{4(j+m)+2\ra 4(j+m+1)}\otimes e_{4j},\quad i=j;\\
\kappa_1w_{4(j+m)+3\ra 4(j+m+1)}\otimes w_{4j\ra 4(j+s)+1},\quad i=j+s;\\
\kappa_1e_{4(j+m+1)}\otimes w_{4j\ra 4(j+s)+2},\quad i=j+3s;\\
\kappa_1e_{4(j+m+1)}\otimes w_{4j\ra 4j+2},\quad i=j+4s;\\
0,\quad\text{otherwise,}
\end{cases}$$
where $\kappa_1=\kappa^{\ell+1}(\a_{3(j+m-1)})$.

If $2s\le j<4s$, then $$b_{ij}=
\begin{cases}
w_{4(j+m)+3\ra 4(j+m+1)+1}\otimes e_{4j+1},\quad i=j;\\
w_{4(j+m+1)\ra 4(j+m+1)+1}\otimes w_{4j+1\ra 4j+2},\quad i=j+2s;\\
-e_{4(j+m+1)+1}\otimes w_{4j+1\ra 4j+3},\quad i=j+4s;\\
0,\quad\text{otherwise.}
\end{cases}$$

If $4s\le j<9s$, then $b_{ij}=0$.

$(9)$ If $r_0=8$, then $\Omega^{8}(Y_t^{(18)})$ is described with
$(8s\times 6s)$-matrix with the following elements $b_{ij}${\rm:}

If $0\le j<s$, then $$b_{ij}=
\begin{cases}
w_{4(j+m)-1\ra 4(j+m+s)+1}\otimes e_{4j},\quad i=j;\\
e_{4(j+m+s)+1}\otimes w_{4j\ra 4j+2},\quad i=j+3s;\\
0,\quad\text{otherwise.}
\end{cases}$$

If $s\le j<2s$, then $$b_{ij}=
\begin{cases}
e_{4(j+m+s)+1}\otimes w_{4j\ra 4j+2},\quad i=j+3s;\\
0,\quad\text{otherwise.}
\end{cases}$$

If $2s\le j<4s$, then $$b_{ij}=
\begin{cases}
\kappa^{\ell+1}(\a_{3(j+m-1)})w_{4(j+m)+3\ra 4(j+m+1)}\otimes w_{4j+1\ra 4(j+1)},\\\quad\quad\quad i=(j+1)_s,\text{ }j<3s-1\text{ or }j=4s-1;\\
\kappa^{\ell+1}(\a_{3(j+m-1)})w_{4(j+m)+2\ra 4(j+m+1)}\otimes e_{4j+1},\quad i=j-s;\\
0,\quad\text{otherwise.}
\end{cases}$$

If $4s\le j<6s$, then $b_{ij}=0$.

If $6s\le j<8s$, then $$b_{ij}=
\begin{cases}
-f_1(j,7s)e_{4(j+m+s+1)+2}\otimes w_{4j+3\ra 4(j+s+1)+1},\quad i=s+(j+s+1)_{2s};\\
0,\quad\text{otherwise.}
\end{cases}$$

$(10)$ If $r_0=9$, then $\Omega^{9}(Y_t^{(18)})$ is described with
$(9s\times 7s)$-matrix with the following elements $b_{ij}${\rm:}

If $0\le j<s$, then $$b_{ij}=
\begin{cases}
-\kappa^{\ell+1}(\a_{3(j+m-1)})w_{4(j+m)+3\ra 4(j+m+1)}\otimes e_{4j},\quad i=j;\\
0,\quad\text{otherwise.}
\end{cases}$$

If $s\le j<2s$, then $b_{ij}=0$.

If $2s\le j<3s$, then $$b_{ij}=
\begin{cases}
w_{4(j+m+1)\ra 4(j+m+1)+2}\otimes e_{4j+1},\quad i=j-s;\\
-e_{4(j+m+1)+2}\otimes w_{4j+1\ra 4j+3},\quad i=j+3s;\\
0,\quad\text{otherwise.}
\end{cases}$$

If $3s\le j<4s$, then $b_{ij}=0$.

If $4s\le j<5s$, then $$b_{ij}=
\begin{cases}
w_{4(j+m+1)\ra 4(j+m+s+1)+2}\otimes e_{4(j+s)+1},\quad i=j-2s;\\
-e_{4(j+m+s+1)+2}\otimes w_{4(j+s)+1\ra 4j+3},\quad i=j+2s;\\
0,\quad\text{otherwise.}
\end{cases}$$

If $5s\le j<7s$, then $b_{ij}=0$.

If $7s\le j<8s$, then $$b_{ij}=
\begin{cases}
-\kappa^{\ell+1}(\a_{3(j+m-1)})f_1(j,8s-1)e_{4(j+m+1)+3}\otimes w_{4j+3\ra 4(j+1)},\quad i=(j+1)_s;\\
0,\quad\text{otherwise.}
\end{cases}$$

If $8s\le j<9s$, then $b_{ij}=0$.

$(11)$ If $r_0=10$, then $\Omega^{10}(Y_t^{(18)})$ is described with
$(8s\times 6s)$-matrix with the following elements $b_{ij}${\rm:}

If $0\le j<s$, then $$b_{ij}=
\begin{cases}
w_{4(j+m)\ra 4(j+m+s)+2}\otimes e_{4j},\quad i=j;\\
-w_{4(j+m+s)+1\ra 4(j+m+s)+2}\otimes w_{4j\ra 4j+1},\quad i=j+s;\\
e_{4(j+m+s)+2}\otimes w_{4j\ra 4j+2},\quad i=j+3s;\\
0,\quad\text{otherwise.}
\end{cases}$$

If $s\le j<2s$, then $$b_{ij}=
\begin{cases}
-w_{4(j+m)\ra 4(j+m+s)+2}\otimes e_{4j},\quad i=j-s;\\
-w_{4(j+m+s)+1\ra 4(j+m+s)+2}\otimes w_{4j\ra 4j+1},\quad i=j+s;\\
e_{4(j+m+s)+2}\otimes w_{4j\ra 4j+2},\quad i=j+3s;\\
0,\quad\text{otherwise.}
\end{cases}$$

If $2s\le j<3s$, then $$b_{ij}=
\begin{cases}
-\kappa_1w_{4(j+m+s)+1\ra 4(j+m)+3}\otimes e_{4j+1},\quad i=j-s;\\
\kappa_1w_{4(j+m+s)+2\ra 4(j+m)+3}\otimes w_{4j+1\ra 4j+2},\quad i=j+s;\\
\kappa_1e_{4(j+m)+3}\otimes w_{4j+1\ra 4j+3},\quad i=j+3s;\\
0,\quad\text{otherwise,}
\end{cases}$$
where $\kappa_1=\kappa^{\ell+1}(\a_{3(j+m-2)})$.

If $3s\le j<4s$, then $$b_{ij}=
\begin{cases}
-\kappa_1w_{4(j+m+s)+1\ra 4(j+m)+3}\otimes e_{4j+1},\quad i=j-s;\\
\kappa_1w_{4(j+m+s)+2\ra 4(j+m)+3}\otimes w_{4j+1\ra 4j+2},\quad i=j+s;\\
-\kappa_1e_{4(j+m)+3}\otimes w_{4j+1\ra 4j+3},\quad i=j+2s;\\
0,\quad\text{otherwise,}
\end{cases}$$
where $\kappa_1=\kappa^{\ell+1}(\a_{3(j+m-2)})$.

If $4s\le j<8s$, then $b_{ij}=0$.

\medskip
$({\rm II})$ Represent an arbitrary $t_0\in\N$ in the form
$t_0=11\ell_0+r_0$, where $0\le r_0\le 10.$ Then
$\Omega^{t_0}(Y_t^{(18)})$ is a $\Omega^{r_0}(Y_t^{(18)})$, whose left
components twisted by $\sigma^{\ell_0}$.
\end{pr}

\begin{pr}[Translates for the case 19]
$({\rm I})$ Let $r_0\in\N$, $r_0<11$. $r_0$-translates of the
elements $Y^{(19)}_t$ are described by the following way.

$(1)$ If $r_0=0$, then $\Omega^{0}(Y_t^{(19)})$ is described with
$(7s\times 6s)$-matrix with the following two nonzero elements{\rm:}
$$b_{s,s}=\kappa^\ell(\a_{3(j+5)})w_{4(j+m+s)+1\ra 4(j+m+1)}\otimes e_{4(j+s)+1};$$
$$b_{2s,2s}=\kappa^\ell(\a_{3(j+5)})w_{4(j+m+s)+1\ra 4(j+m+1)}\otimes e_{4(j+s)+1}.$$

$(2)$ If $r_0=1$, then $\Omega^{1}(Y_t^{(19)})$ is described with
$(6s\times 7s)$-matrix with the following two nonzero elements{\rm:}
$$b_{j,(3)_s}=-\kappa^\ell(\a_{3(j+m+5)})w_{4(j+m)+1\ra 4(j+m+1)}\otimes e_{4j};$$
$$b_{j+s,(3)_s}=\kappa^\ell(\a_{3(j+m+5)})w_{4(j+m+s)+1\ra 4(j+m+1)}\otimes e_{4j}.$$

$(3)$ If $r_0=2$, then $\Omega^{2}(Y_t^{(19)})$ is described with
$(6s\times 6s)$-matrix with the following nonzero elements{\rm:}
$$b_{(j+1)_s,(2)_s}=-f_2(s,1)\kappa^\ell(\a_{3(j+m+5)})w_{4(j+m)+3\ra 4(j+m+1)}\otimes w_{4j\ra 4(j+1)};$$
$$b_{(j+1)_s,s+(2)_s}=f_2(s,1)w_{4(j+m)+3\ra 4(j+m+1)+1}\otimes w_{4(j+s)+1\ra 4(j+1)};$$
$$b_{(j+1)_s,2s+(2)_s}=-f_2(s,1)w_{4(j+m)+3\ra 4(j+m+1)+1}\otimes w_{4(j+s)+1\ra 4(j+1)}.$$

$(4)$ If $r_0=3$, then $\Omega^{3}(Y_t^{(19)})$ is described with
$(7s\times 8s)$-matrix with the following two nonzero elements{\rm:}
$$b_{s+(j+1)_s-f(s,1),2s+(2)_s}=-f_2(s,1)e_{4(j+m+s+1)+2}\otimes w_{4j+1\ra 4(j+1)};$$
$$b_{(j+1)_s+f(s,1),3s+(2)_s}=f_2(s,1)e_{4(j+m+s+1)+2}\otimes w_{4j+1\ra 4(j+1)}.$$

$(5)$ If $r_0=4$, then $\Omega^{4}(Y_t^{(19)})$ is described with
$(6s\times 9s)$-matrix with the following nonzero elements{\rm:}
$$b_{(j+1)_s,(2)_s}=-\kappa^\ell(\a_{3(j+m+4)})e_{4(j+m)+3}\otimes w_{4j\ra 4(j+1)};$$
$$b_{(j+1)_s,s+(2)_s+f(s,1)}=w_{4(j+m)+3\ra 4(j+m+1)+2}\otimes w_{4(j+s)+1\ra 4(j+1)};$$
$$b_{s+(j+1)_s,s+(2)_s+f(s,1)}=-w_{4(j+m+1)\ra 4(j+m+1)+2}\otimes w_{4(j+s)+1\ra 4(j+1)};$$
$$b_{s+(j+1)_s,2s+(2)_s-f(s,1)}=w_{4(j+m+1)\ra 4(j+m+1)+2}\otimes w_{4(j+s)+1\ra 4(j+1)}.$$

$(6)$ If $r_0=5$, then $\Omega^{5}(Y_t^{(19)})$ is described with
$(8s\times 8s)$-matrix with the following nonzero elements{\rm:}
$$b_{s+(j+1)_s-f(s,1),(2)_s}=w_{4(j+m+s+1)+1\ra 4(j+m+s+1)+2}\otimes w_{4j\ra 4(j+1)};$$
$$b_{(j+1)_s+f(s,1),s+(2)_s}=-w_{4(j+m+s+1)+1\ra 4(j+m+s+1)+2}\otimes w_{4j\ra 4(j+1)};$$
$$b_{(j+1)_s+f(s,1),2s+(2)_s}=-\kappa^{\ell+1}(\a_{3(j+m-1)})w_{4(j+m+1)+1\ra 4(j+m+1)+3}\otimes w_{4j+1\ra 4(j+1)};$$
$$b_{s+(j+1)_s-f(s,1),3s+(2)_s}=-\kappa^{\ell+1}(\a_{3(j+m-1)})w_{4(j+m+1)+1\ra 4(j+m+1)+3}\otimes w_{4j+1\ra 4(j+1)}.$$

$(7)$ If $r_0=6$, then $\Omega^{6}(Y_t^{(19)})$ is described with
$(9s\times 9s)$-matrix with the following nonzero elements{\rm:}
$$b_{(j+1)_s,s+(2)_s}=-\kappa^{\ell+1}(\a_{3(j+m-1)})e_{4(j+m+1)}\otimes w_{4j\ra 4(j+1)};$$
$$b_{j-s,2s+(2)_s}=w_{4(j+m)+1\ra 4(j+m+1)+1}\otimes e_{4j+1};$$
$$b_{j,3s+(2)_s}=w_{4(j+m)+1\ra 4(j+m+1)+1}\otimes e_{4j+1}.$$

$(8)$ If $r_0=7$, then $\Omega^{7}(Y_t^{(19)})$ is described with
$(8s\times 8s)$-matrix with the following nonzero elements{\rm:}
$$b_{j+s,(2)_s}=w_{4(j+m+s)+2\ra 4(j+m+s+1)+1}\otimes e_{4j};$$
$$b_{j+2s,(2)_s}=-w_{4(j+m)+3\ra 4(j+m+s+1)+1}\otimes w_{4j\ra 4j+1};$$
$$b_{j+4s,(2)_s}=-w_{4(j+m+1)\ra 4(j+m+s+1)+1}\otimes w_{4j\ra 4j+2};$$
$$b_{j+5s,(2)_s}=-w_{4(j+m+1)\ra 4(j+m+s+1)+1}\otimes w_{4j\ra 4(j+s)+2};$$
$$b_{j+s,s+(2)_s}=-w_{4(j+m)+3\ra 4(j+m+s+1)+1}\otimes w_{4j\ra 4(j+s)+1};$$
$$b_{j+3s,s+(2)_s}=-w_{4(j+m+1)\ra 4(j+m+s+1)+1}\otimes w_{4j\ra 4(j+s)+2};$$
$$b_{j+4s,s+(2)_s}=-w_{4(j+m+1)\ra 4(j+m+s+1)+1}\otimes w_{4j\ra 4j+2}.$$

$(9)$ If $r_0=8$ and $s=2$, then $\Omega^{8}(Y_t^{(19)})$ is described with
$(9s\times 6s)$-matrix with the following nonzero elements{\rm:}
$$b_{j-s,2s+1}=w_{4(j+m)+2\ra 4(j+m+1)+2}\otimes e_{4j+1};$$
$$b_{(j+1)_s,3s+1}=w_{4(j+m)+3\ra 4(j+m+s+1)+1}\otimes w_{4j+1\ra 4(j+1)};$$
$$b_{s+(j+1)_s,4s+1}=e_{4(j+m+s+1)+2}\otimes w_{4(j+2)+1\ra 4(j+s+1)+1}.$$

$(10)$ If $r_0=8$ and $s\ne 2$, then $\Omega^{8}(Y_t^{(19)})$ is described with
$(9s\times 6s)$-matrix with the following nonzero elements{\rm:}
$$b_{(j+1)_s,s+1+f(s,1)}=w_{4(j+m)+3\ra 4(j+m+1+f(s,1))+1}\otimes w_{4(j+s+f(s,1))+1\ra 4(j+1)};$$
$$b_{s+(j+1)_s,2s+1+f(s,1)}=e_{4(j+m+1+f(s,1))+2}\otimes w_{4(j+f(s,1))+1\ra 4(j+1+f(s,1))+1};$$
$$b_{j-2s+f(s,1),4s+1-3f(s,1)}=w_{4(j+m+s+f(s,1))+2\ra 4(j+m+s+1+f(s,1))+2}\otimes e_{4(j+s+f(s,1))+1}.$$

$(11)$ If $r_0=9$ and $s=2$, then $\Omega^{9}(Y_t^{(19)})$ is described with
$(8s\times 7s)$-matrix with the following nonzero elements{\rm:}
$$b_{j,s+1}=-w_{4(j+m+1)\ra 4(j+m+s+1)+2}\otimes w_{4j\ra 4(j+2)+1};$$
$$b_{j+s,s+1}=w_{4(j+m+1)\ra 4(j+m+s+1)+2}\otimes w_{4j\ra 4j+1};$$
$$b_{(j+1)_s,3s+1}=e_{4(j+m+1)+3}\otimes w_{4j+1\ra 4(j+1)}.$$

$(12)$ If $r_0=9$ and $s\ne 2$, then $\Omega^{9}(Y_t^{(19)})$ is described with
$(8s\times 7s)$-matrix with the following nonzero elements{\rm:}
$$b_{j-f(s,1),1}=f_2(s,1)w_{4(j+m)+3\ra 4(j+m+s+1)+2}\otimes e_{4j};$$
$$b_{j-s,2s+1}=\kappa^{\ell+1}(\a_{3(j+m-1)})w_{4(j+m+1)\ra 4(j+m+1)+3}\otimes e_{4j+1};$$
$$b_{j+3s,2s+1}=-\kappa^{\ell+1}(\a_{3(j+m-1)})w_{4(j+m+1)+2\ra 4(j+m+1)+3}\otimes w_{4j+1\ra 4j+3}.$$

$(13)$ If $r_0=10$ and $s=1$, then $\Omega^{10}(Y_t^{(19)})$ is described with
$(6s\times 6s)$-matrix with the following nonzero elements{\rm:}
$$b_{j,0}=\kappa^{\ell+1}(\a_{3(j+m-1)})w_{4(j+m)\ra 4(j+m)+3}\otimes e_{4j};$$
$$b_{j+s,0}=-\kappa^{\ell+1}(\a_{3(j+m-1)})w_{4(j+m+1)+1\ra 4(j+m)+3}\otimes w_{4j\ra 4j+1};$$
$$b_{j+3s,0}=\kappa^{\ell+1}(\a_{3(j+m-1)})w_{4(j+m+1)+2\ra 4(j+m)+3}\otimes w_{4j\ra 4j+2};$$
$$b_{j-s,s}=-w_{4(j+m)\ra 4(j+m+1)+2}\otimes w_{4(j+1)+1\ra 4j};$$
$$b_{j+s,s}=-w_{4(j+m+1)+1\ra 4(j+m+1)+2}\otimes w_{4(j+1)+1\ra 4j+1}.$$

$(14)$ If $r_0=10$ and $s=2$, then $\Omega^{10}(Y_t^{(19)})$ is described with
$(6s\times 6s)$-matrix with the following nonzero elements{\rm:}
$$b_{j+s,1}=-w_{4(j+m+2)+1\ra 4(j+m)+3}\otimes w_{4j\ra 4j+1};$$
$$b_{j+2s,1}=w_{4(j+m)+1\ra 4(j+m)+3}\otimes w_{4j\ra 4(j+2)+1};$$
$$b_{j+3s,1}=w_{4(j+m+s)+2\ra 4(j+m)+3}\otimes w_{4j\ra 4j+2};$$
$$b_{j+4s,1}=-w_{4(j+m)+2\ra 4(j+m)+3}\otimes w_{4j\ra 4(j+s)+2};$$
$$b_{j+5s,1}=e_{4(j+m)+3}\otimes w_{4j\ra 4j+3};$$
$$b_{2s+(j+1)_s,2s}=-w_{4(j+m+1)+1\ra 4(j+m+1)+2}\otimes w_{4(j+2)+1\ra 4(j+s+1)+1};$$
$$b_{j+2s,2s+1}=w_{4(j+m)+2\ra 4(j+m+1)+2}\otimes w_{4(j+2)+1\ra 4(j+s)+2};$$
$$b_{j+3s,2s+1}=-w_{4(j+m)+3\ra 4(j+m+1)+2}\otimes w_{4(j+2)+1\ra 4j+3};$$
$$b_{(j+1)_s,5s}=-w_{4(j+m+1)\ra 4(j+m)}\otimes w_{4j+3\ra 4(j+1)}.$$

$(15)$ If $r_0=10$ and $s>2$, then $\Omega^{10}(Y_t^{(19)})$ is described with
$(6s\times 6s)$-matrix with the following nonzero elements{\rm:}
$$b_{j,1}=\kappa^{\ell+1}(\a_{3(j+m-1)})w_{4(j+m)\ra 4(j+m)+3}\otimes e_{4j};$$
$$b_{j+2s,1}=\kappa^{\ell+1}(\a_{3(j+m-1)})w_{4(j+m)+1\ra 4(j+m)+3}\otimes w_{4j\ra 4(j+s)+1};$$
$$b_{j+4s,1}=-\kappa^{\ell+1}(\a_{3(j+m-1)})w_{4(j+m)+2\ra 4(j+m)+3}\otimes w_{4j\ra 4(j+s)+2};$$
$$b_{(j+1)_s,s}=w_{4(j+m+1)\ra 4(j+m+1)+2}\otimes w_{4(j+s)+1\ra 4(j+1)};$$
$$b_{s+(j+1)_s,s}=-w_{4(j+m+1)+1\ra 4(j+m+1)+2}\otimes w_{4(j+s)+1\ra 4(j+s+1)+1};$$
$$b_{(j+1)_s,5s}=-\kappa^{\ell+1}(\a_{3(j+m-1)})w_{4(j+m+1)\ra 4(j+m+2)}\otimes w_{4j+3\ra 4(j+1)};$$
$$b_{2s+(j+1)_s,5s}=-\kappa^{\ell+1}(\a_{3(j+m-1)})w_{4(j+m+s+1)+1\ra 4(j+m+2)}\otimes w_{4j+3\ra 4(j+1)+1};$$
$$b_{4s+(j+1)_s,5s}=\kappa^{\ell+1}(\a_{3(j+m-1)})w_{4(j+m+s+1)+2\ra 4(j+m+2)}\otimes w_{4j+3\ra 4(j+1)+2}.$$

\medskip
$({\rm II})$ Represent an arbitrary $t_0\in\N$ in the form
$t_0=11\ell_0+r_0$, where $0\le r_0\le 10.$ Then
$\Omega^{t_0}(Y_t^{(19)})$ is a $\Omega^{r_0}(Y_t^{(19)})$, whose left
components twisted by $\sigma^{\ell_0}$,
and coefficients multiplied by $(-1)^{\ell_0}$.
\end{pr}

\begin{pr}[Translates for the case 20]
$({\rm I})$ Let $r_0\in\N$, $r_0<11$. $r_0$-translates of the
elements $Y^{(20)}_t$ are described by the following way.

$(1)$ If $r_0=0$, then $\Omega^{0}(Y_t^{(20)})$ is described with
$(7s\times 6s)$-matrix with the following elements $b_{ij}${\rm:}

If $0\le j<s$, then $$b_{ij}=
\begin{cases}
\kappa^\ell(\a_{3(j+m+4)})w_{4(j+m)\ra 4(j+m)+3}\otimes e_{4j},\quad i=j;\\
0,\quad\text{otherwise.}
\end{cases}$$

If $s\le j<2s$, then $$b_{ij}=
\begin{cases}
-\kappa^\ell(\a_{3(j+m+4+1)})w_{4(j+m+s)+1\ra 4(j+m+1)}\otimes e_{4(j+s)+1},\quad i=j;\\
0,\quad\text{otherwise.}
\end{cases}$$

If $2s\le j<3s$, then $b_{ij}=0$.

If $3s\le j<4s$, then $$b_{ij}=
\begin{cases}
w_{4(j+m+s)+2\ra 4(j+m+s+1)+1}\otimes e_{4(j+s)+2},\quad i=j;\\
0,\quad\text{otherwise.}
\end{cases}$$

If $4s\le j<6s$, then $b_{ij}=0$.

If $6s\le j<7s$, then $$b_{ij}=
\begin{cases}
w_{4(j+m)+3\ra 4(j+m+1)+2}\otimes e_{4j+3},\quad i=j-s;\\
0,\quad\text{otherwise.}
\end{cases}$$

$(2)$ If $r_0=1$, then $\Omega^{1}(Y_t^{(20)})$ is described with
$(6s\times 7s)$-matrix with the following elements $b_{ij}${\rm:}

If $0\le j<s$, then $$b_{ij}=
\begin{cases}
-\kappa^\ell(\a_{3(j+m+5)})w_{4(j+m)+1\ra 4(j+m+1)}\otimes e_{4j},\quad i=j;\\
0,\quad\text{otherwise.}
\end{cases}$$

If $s\le j<2s$, then $$b_{ij}=
\begin{cases}
-w_{4(j+m+s)+2\ra 4(j+m+s+1)+1}\otimes e_{4(j+s)+1},\quad i=j+s;\\
0,\quad\text{otherwise.}
\end{cases}$$

If $2s\le j<3s$, then $b_{ij}=0$.

If $3s\le j<4s$, then $$b_{ij}=
\begin{cases}
-w_{4(j+m)+3\ra 4(j+m+s+1)+2}\otimes e_{4(j+s)+2},\quad i=j+s;\\
0,\quad\text{otherwise.}
\end{cases}$$

If $4s\le j<5s$, then $b_{ij}=0$.

If $5s\le j<6s$, then $$b_{ij}=
\begin{cases}
\kappa^\ell(\a_{3(j+m+5)})w_{4(j+m+1)\ra 4(j+m+1)+3}\otimes e_{4j+3},\quad i=j+s;\\
0,\quad\text{otherwise.}
\end{cases}$$

$(3)$ If $r_0=2$, then $\Omega^{2}(Y_t^{(20)})$ is described with
$(6s\times 6s)$-matrix with the following elements $b_{ij}${\rm:}

If $0\le j<s$, then $$b_{ij}=
\begin{cases}
\kappa_1w_{4(j+m)+2\ra 4(j+m+1)}\otimes w_{4j\ra 4j+1},\quad i=j+s;\\
\kappa_1w_{4(j+m)+3\ra 4(j+m+1)}\otimes w_{4j\ra 4(j+1)},\quad i=(j+1)_s,\text{ }j=s-1;\\
0,\quad\text{otherwise,}
\end{cases}$$
where $\kappa_1=\kappa^\ell(\a_{3(j+m+5)})$.

If $s\le j<2s-1$, then $$b_{ij}=
\begin{cases}
w_{4(j+m)+3\ra 4(j+m+s+1)+1}\otimes w_{4(j+s)+1\ra 4(j+1)},\quad i=(j+1)_s;\\
0,\quad\text{otherwise.}
\end{cases}$$

If $2s-1\le j<3s-1$, then $b_{ij}=0$.

If $3s-1\le j<3s$, then $$b_{ij}=
\begin{cases}
w_{4(j+m)+3\ra 4(j+m+s+1)+1}\otimes w_{4(j+s)+1\ra 4(j+1)},\quad i=(j+1)_s;\\
0,\quad\text{otherwise.}
\end{cases}$$

If $3s\le j<4s-1$, then $$b_{ij}=
\begin{cases}
w_{4(j+m)+3\ra 4(j+m+s+1)+2}\otimes w_{4(j+s)+2\ra 4(j+1)},\quad i=(j+1)_s;\\
0,\quad\text{otherwise.}
\end{cases}$$

If $4s-1\le j<5s-1$, then $b_{ij}=0$.

If $5s-1\le j<5s$, then $$b_{ij}=
\begin{cases}
w_{4(j+m)+3\ra 4(j+m+s+1)+2}\otimes w_{4(j+s)+2\ra 4(j+1)},\quad i=(j+1)_s;\\
0,\quad\text{otherwise.}
\end{cases}$$

If $5s\le j<6s$, then $$b_{ij}=
\begin{cases}
-\kappa^\ell(\a_{3(j+m+5)})w_{4(j+m+s+1+sf(j,6s-1))+2\ra 4(j+m+1)+3}\otimes w_{4j+3\ra 4(j+s+1+sf(j,6s-1))+1},\\\quad\quad\quad i=s+(j+1)_s;\\
0,\quad\text{otherwise.}
\end{cases}$$

$(4)$ If $r_0=3$, then $\Omega^{3}(Y_t^{(20)})$ is described with
$(7s\times 8s)$-matrix with the following elements $b_{ij}${\rm:}

If $0\le j<s$, then $$b_{ij}=
\begin{cases}
w_{4(j+m)+3\ra 4(j+m+1)+1}\otimes w_{4j\ra 4j+1},\quad i=j+2s;\\
0,\quad\text{otherwise.}
\end{cases}$$

If $s\le j<4s$, then $b_{ij}=0$.

If $4s\le j<6s$, then $$b_{ij}=
\begin{cases}
\kappa^\ell(\a_{3(j+m+5)})f_1(j,5s)w_{4(j+m+1)+2\ra 4(j+m+1)+3}\otimes w_{4j+2\ra 4(j+1)},\\\quad\quad\quad i=(j+1)_s,\text{ }j<5s-1\text{ or }j=6s-1;\\
0,\quad\text{otherwise.}
\end{cases}$$

If $6s\le j<7s$, then $b_{ij}=0$.

$(5)$ If $r_0=4$, then $\Omega^{4}(Y_t^{(20)})$ is described with
$(6s\times 9s)$-matrix with the following elements $b_{ij}${\rm:}

If $0\le j<s$, then $$b_{ij}=
\begin{cases}
\kappa_1w_{4(j+m+s)+1\ra 4(j+m)+3}\otimes w_{4j\ra 4j+1},\quad i=j+2s;\\
-\kappa_1w_{4(j+m+s)+2\ra 4(j+m)+3}\otimes w_{4j\ra 4j+2},\quad i=j+5s;\\
\kappa_1e_{4(j+m)+3}\otimes w_{4j\ra 4(j+1)},\quad i=(j+1)_s,\text{ }j<s-1;\\
0,\quad\text{otherwise,}
\end{cases}$$
where $\kappa_1=\kappa^\ell(\a_{3(j+m+4)})$.

If $s\le j<3s$, then $$b_{ij}=
\begin{cases}
w_{4(j+m)+3\ra 4(j+m+s+1)+2}\otimes w_{4(j+s)+1\ra 4(j+1)},\quad i=(j+1)_s,\text{ }j<2s-1\text{ or }j=3s-1;\\
0,\quad\text{otherwise.}
\end{cases}$$

If $3s\le j<6s$, then $b_{ij}=0$.

$(6)$ If $r_0=5$, then $\Omega^{5}(Y_t^{(20)})$ is described with
$(8s\times 8s)$-matrix with the following elements $b_{ij}${\rm:}

If $s\le j<2s$, then $$b_{ij}=
\begin{cases}
w_{4(j+m+1)+1\ra 4(j+m+1)+2}\otimes w_{4j\ra 4(j+1)},\quad i=(j+1)_{2s};\\
w_{4(j+m+1)\ra 4(j+m+1)+2}\otimes w_{4j\ra 4(j+s)+1},\quad i=j+s;\\
0,\quad\text{otherwise.}
\end{cases}$$

If $2s\le j<4s$, then $$b_{ij}=
\begin{cases}
\kappa_1w_{4(j+m+1)+1\ra 4(j+m+1)+3}\otimes w_{4j+1\ra 4(j+1)},\\\quad\quad\quad i=(j+1)_s,\text{ }j<3s-1\text{ or }j=4s-1;\\
\kappa_1w_{4(j+m+s+1)+1\ra 4(j+m+1)+3}\otimes w_{4j+1\ra 4(j+1)},\\\quad\quad\quad i=s+(j+1)_s,\text{ }j<3s-1\text{ or }j=4s-1;\\
0,\quad\text{otherwise,}
\end{cases}$$
where $\kappa_1=\kappa^{\ell+1}(\a_{3(j+m-1)})$.

If $4s\le j<6s$, then $$b_{ij}=
\begin{cases}
\kappa^{\ell+1}(\a_{3(j+m)})e_{4(j+m+2)}\otimes w_{4j+2\ra 4(j+1)+1},\quad i=2s+(j+1)_s,\text{ }j<5s-1\text{ or }j=6s-1;\\
0,\quad\text{otherwise.}
\end{cases}$$

If $6s\le j<8s$, then $b_{ij}=0$.

$(7)$ If $r_0=6$, then $\Omega^{6}(Y_t^{(20)})$ is described with
$(9s\times 9s)$-matrix with the following elements $b_{ij}${\rm:}

If $0\le j<s$, then $$b_{ij}=
\begin{cases}
-\kappa^{\ell+1}(\a_{3(j+m-2)})w_{4(j+m)+1\ra 4(j+m)+3}\otimes w_{4j\ra 4j+1},\quad i=j+s;\\
0,\quad\text{otherwise.}
\end{cases}$$

If $s\le j<2s-1$, then $$b_{ij}=
\begin{cases}
\kappa^{\ell+1}(\a_{3(j+m-1)})e_{4(j+m+1)}\otimes w_{4j\ra 4(j+1)},\quad i=(j+1)_s;\\
0,\quad\text{otherwise.}
\end{cases}$$

If $2s-1\le j<2s$, then $b_{ij}=0$.

If $2s\le j<4s$, then $$b_{ij}=
\begin{cases}
w_{4(j+m+1)\ra 4(j+m+s+1)+1}\otimes w_{4j+1\ra 4(j+1)},\quad i=(j+1)_s,\text{ }j<3s-1\text{ or }j=4s-1;\\
0,\quad\text{otherwise.}
\end{cases}$$

If $4s\le j<5s-1$, then $$b_{ij}=
\begin{cases}
e_{4(j+m+1)+1}\otimes w_{4j+2\ra 4(j+1)+1},\quad i=s+(j+1)_s;\\
0,\quad\text{otherwise.}
\end{cases}$$

If $5s-1\le j<7s-1$, then $b_{ij}=0$.

If $7s-1\le j<7s$, then $$b_{ij}=
\begin{cases}
e_{4(j+m+s+1)+1}\otimes w_{4(j+s)+2\ra 4(j+s+1)+1},\quad i=s+(j+1)_s;\\
0,\quad\text{otherwise.}
\end{cases}$$

If $7s\le j<9s$, then $b_{ij}=0$.

$(8)$ If $r_0=7$, then $\Omega^{7}(Y_t^{(20)})$ is described with
$(8s\times 8s)$-matrix with the following elements $b_{ij}${\rm:}

If $0\le j<s$, then $$b_{ij}=
\begin{cases}
-w_{4(j+m)+3\ra 4(j+m+1)+1}\otimes w_{4j\ra 4j+1},\quad i=j+2s;\\
-w_{4(j+m+1)\ra 4(j+m+1)+1}\otimes w_{4j\ra 4j+2},\quad i=j+4s;\\
-w_{4(j+m+1)\ra 4(j+m+1)+1}\otimes w_{4j\ra 4(j+s)+2},\quad i=j+5s;\\
0,\quad\text{otherwise.}
\end{cases}$$

If $s\le j<2s$, then $b_{ij}=0$.

If $2s\le j<4s$, then $$b_{ij}=
\begin{cases}
\kappa^{\ell+1}(\a_{3(j+m)})w_{4(j+m+1)+2\ra 4(j+m+2)}\otimes w_{4j+1\ra 4(j+1)},\\\quad\quad\quad i=(j+1)_s,\text{ }j<3s-1\text{ or }j=4s-1;\\
0,\quad\text{otherwise.}
\end{cases}$$

If $4s\le j<8s$, then $b_{ij}=0$.

$(9)$ If $r_0=8$, then $\Omega^{8}(Y_t^{(20)})$ is described with
$(9s\times 6s)$-matrix with the following elements $b_{ij}${\rm:}

If $0\le j<s$, then $$b_{ij}=
\begin{cases}
-\kappa_1w_{4(j+m)+2\ra 4(j+m+1)}\otimes w_{4j\ra 4j+1},\quad i=j+s;\\
-\kappa_1w_{4(j+m)+1\ra 4(j+m+1)}\otimes w_{4j\ra 4(j+s)+2},\quad i=j+4s;\\
\kappa_1w_{4(j+m)+3\ra 4(j+m+1)}\otimes w_{4j\ra 4(j+1)},\quad i=(j+1)_s,\text{ }j=s-1;\\
0,\quad\text{otherwise,}
\end{cases}$$
where $\kappa_1=\kappa^{\ell+1}(\a_{3(j+m-1)})$.

If $s\le j<2s-1$, then $$b_{ij}=
\begin{cases}
w_{4(j+m)+3\ra 4(j+m+s+1)+1}\otimes w_{4(j+s)+1\ra 4(j+1)},\quad i=(j+1)_s;\\
0,\quad\text{otherwise.}
\end{cases}$$

If $2s-1\le j<4s-1$, then $b_{ij}=0$.

If $4s-1\le j<4s$, then $$b_{ij}=
\begin{cases}
w_{4(j+m)+3\ra 4(j+m+1)+1}\otimes w_{4j+1\ra 4(j+1)},\quad i=(j+1)_s;\\
0,\quad\text{otherwise.}
\end{cases}$$

If $4s\le j<7s$, then $b_{ij}=0$.

If $7s\le j<8s$, then $$b_{ij}=
\begin{cases}
\kappa_1w_{4(j+m)+3\ra 4(j+m+1)+3}\otimes w_{4j+3\ra 4(j+1)},\quad i=(j+1)_s;\\
\kappa_1w_{4(j+m+1+sf(j,8s-1))+2\ra 4(j+m+1)+3}\otimes w_{4j+3\ra 4(j+1+sf(j,8s-1))+1},\quad i=2s+(j+1)_s;\\
\kappa_1w_{4(j+m+1+sf(j,8s-1))+1\ra 4(j+m+1)+3}\otimes w_{4j+3\ra 4(j+s+1+sf(j,8s-1))+2},\quad i=3s+(j+1)_s;\\
0,\quad\text{otherwise,}
\end{cases}$$
where $\kappa_1=-\kappa^{\ell+1}(\a_{3(j+m-1)})$.

If $8s\le j<9s$, then $b_{ij}=0$.

$(10)$ If $r_0=9$, then $\Omega^{9}(Y_t^{(20)})$ is described with
$(8s\times 7s)$-matrix with the following elements $b_{ij}${\rm:}

If $0\le j<s$, then $$b_{ij}=
\begin{cases}
w_{4(j+m)+3\ra 4(j+m+1)+2}\otimes e_{4j},\quad i=j;\\
0,\quad\text{otherwise.}
\end{cases}$$

If $s\le j<3s-1$, then $b_{ij}=0$.

If $3s-1\le j<4s-1$, then $$b_{ij}=
\begin{cases}
-\kappa^{\ell+1}(\a_{3(j+m-1)})e_{4(j+m+1)+3}\otimes w_{4j+1\ra 4(j+1)},\quad i=(j+1)_s;\\
0,\quad\text{otherwise.}
\end{cases}$$

If $4s-1\le j<5s$, then $b_{ij}=0$.

If $5s\le j<6s$, then $$b_{ij}=
\begin{cases}
\kappa^{\ell+1}(\a_{3(j+m)})w_{4(j+m+1)+3\ra 4(j+m+2)}\otimes w_{4j+2\ra 4(j+1)},\quad i=(j+1)_s;\\
0,\quad\text{otherwise.}
\end{cases}$$

If $6s\le j<7s-1$, then $$b_{ij}=
\begin{cases}
-w_{4(j+m+1)+3\ra 4(j+m+2)+1}\otimes w_{4j+3\ra 4(j+1)},\quad i=(j+1)_s;\\
0,\quad\text{otherwise.}
\end{cases}$$

If $7s-1\le j<7s$, then $$b_{ij}=
\begin{cases}
-w_{4(j+m+2)\ra 4(j+m+2)+1}\otimes w_{4j+3\ra 4(j+1)+1},\quad i=2s+(j+1)_s;\\
-e_{4(j+m+2)+1}\otimes w_{4j+3\ra 4(j+s+1)+2},\quad i=3s+(j+1)_s;\\
0,\quad\text{otherwise.}
\end{cases}$$

If $7s\le j<8s$, then $b_{ij}=0$.

$(11)$ If $r_0=10$, then $\Omega^{10}(Y_t^{(20)})$ is described with
$(6s\times 6s)$-matrix with the following elements $b_{ij}${\rm:}

If $0\le j<s$, then $$b_{ij}=
\begin{cases}
\kappa_1w_{4(j+m)\ra 4(j+m)+3}\otimes e_{4j},\quad i=j;\\
-\kappa_1w_{4(j+m+s)+1\ra 4(j+m)+3}\otimes w_{4j\ra 4j+1},\quad i=j+s;\\
\kappa_1w_{4(j+m)+1\ra 4(j+m)+3}\otimes w_{4j\ra 4(j+s)+1},\quad i=j+2s;\\
\kappa_1w_{4(j+m+s)+2\ra 4(j+m)+3}\otimes w_{4j\ra 4j+2},\quad i=j+3s;\\
-\kappa_1w_{4(j+m)+2\ra 4(j+m)+3}\otimes w_{4j\ra 4(j+s)+2},\quad i=j+4s;\\
\kappa_1e_{4(j+m)+3}\otimes w_{4j\ra 4j+3},\quad i=j+5s;\\
0,\quad\text{otherwise,}
\end{cases}$$
where $\kappa_1=\kappa^{\ell+1}(\a_{3(j+m-2)})$.

If $s\le j<2s-1$, then $b_{ij}=0$.

If $2s-1\le j<2s$, then $$b_{ij}=
\begin{cases}
-w_{4(j+m+1)\ra 4(j+m+s+1)+2}\otimes w_{4(j+s)+1\ra 4(j+1)},\quad i=(j+1)_s;\\
0,\quad\text{otherwise.}
\end{cases}$$

If $2s\le j<5s-1$, then $b_{ij}=0$.

If $5s-1\le j<5s$, then $$b_{ij}=
\begin{cases}
w_{4(j+m+1)\ra 4(j+m+1)+1}\otimes w_{4(j+s)+2\ra 4(j+1)},\quad i=(j+1)_s;\\
-e_{4(j+m+1)+1}\otimes w_{4(j+s)+2\ra 4(j+s+1)+1},\quad i=s+(j+1)_s;\\
0,\quad\text{otherwise.}
\end{cases}$$

If $5s\le j<6s-1$, then $b_{ij}=0$.

If $6s-1\le j<6s$, then $$b_{ij}=
\begin{cases}
\kappa_1w_{4(j+m+1)+1\ra 4(j+m+2)}\otimes w_{4j+3\ra 4(j+s+1)+1},\quad i=2s+(j+1)_s;\\
\kappa_1w_{4(j+m+s+1)+2\ra 4(j+m+2)}\otimes w_{4j+3\ra 4(j+1)+2},\quad i=3s+(j+1)_s;\\
-\kappa_1w_{4(j+m+1)+2\ra 4(j+m+2)}\otimes w_{4j+3\ra 4(j+s+1)+2},\quad i=4s+(j+1)_s;\\
\kappa_1w_{4(j+m+1)+3\ra 4(j+m+2)}\otimes w_{4j+3\ra 4(j+1)+3},\quad i=5s+(j+1)_s;\\
0,\quad\text{otherwise,}
\end{cases}$$
where $\kappa_1=-f_2(s,1)\kappa^{\ell+1}(\a_{3(j+m+1)})$.

\medskip
$({\rm II})$ Represent an arbitrary $t_0\in\N$ in the form
$t_0=11\ell_0+r_0$, where $0\le r_0\le 10.$ Then
$\Omega^{t_0}(Y_t^{(20)})$ is a $\Omega^{r_0}(Y_t^{(20)})$, whose left
components twisted by $\sigma^{\ell_0}$.
\end{pr}

\begin{pr}[Translates for the case 21]
$({\rm I})$ Let $r_0\in\N$, $r_0<11$. Denote by $\kappa_0=\kappa^\ell(\g_4)\kappa^\ell(\g_5)\kappa^\ell(\g_6)-\kappa^\ell(\a_{15})$.
Then $r_0$-translates of the
elements $Y^{(21)}_t$ are described by the following way.

$(1)$ If $r_0=0$, then $\Omega^{0}(Y_t^{(21)})$ is described with
$(6s\times 6s)$-matrix with one nonzero element that is of the following form{\rm:}
$$b_{0,0}=-\kappa^\ell(\a_{15}) w_{4(j+m)\ra 4(j+m+1)}\otimes e_{4j}.$$

$(2)$ If $r_0=1$, then $\Omega^{1}(Y_t^{(21)})$ is described with
$(6s\times 7s)$-matrix with one nonzero element that is of the following form{\rm:}
$$b_{(j+1)_s,(1)_s}=f_2(s,1)\kappa_0w_{4(j+m+1+f(s,1))+1\ra 4(j+m+2)}\otimes w_{4j\ra 4(j+1)}.$$

$(3)$ If $r_0=2$, then $\Omega^{2}(Y_t^{(21)})$ is described with
$(7s\times 6s)$-matrix with one nonzero element that is of the following form{\rm:}
$$b_{(j+1)_s,1}=f_2(s,1)\kappa^{\ell+1}(\a_{3(j+m-1)})\kappa_0w_{4(j+m)+3\ra 4(j+m+1)+1}\otimes w_{4j\ra 4(j+1)}.$$

$(4)$ If $r_0=3$, then $\Omega^{3}(Y_t^{(21)})$ is described with
$(6s\times 8s)$-matrix with one nonzero element that is of the following form{\rm:}
$$b_{(j+1)_s,(1)_s}=-\kappa_0w_{4(j+m+1+f(s,1))+2\ra 4(j+m+1)+3}\otimes w_{4j\ra 4(j+1)}.$$

$(5)$ If $r_0=4$, then $\Omega^{4}(Y_t^{(21)})$ is described with
$(8s\times 9s)$-matrix with one nonzero element that is of the following form{\rm:}
$$b_{(j+1)_s,1}=-\kappa^{\ell+1}(\a_{3(j+m-1)})\kappa_0w_{4(j+m)+3\ra 4(j+m+1)+2}\otimes w_{4j\ra 4(j+1)}.$$

$(6)$ If $r_0=5$, then $\Omega^{5}(Y_t^{(21)})$ is described with
$(9s\times 8s)$-matrix with the following two nonzero elements{\rm:}
$$b_{(j+1)_s+f(s,1),(1)_s}=f_2(s,1)\kappa_0w_{4(j+m+1)+1\ra 4(j+m+1)+3}\otimes w_{4j\ra 4(j+1)};$$
$$b_{s+(j+1)_s-f(s,1),(1)_s}=f_2(s,1)\kappa_0w_{4(j+m+s+1)+1\ra 4(j+m+1)+3}\otimes w_{4j\ra 4(j+1)}.$$

$(7)$ If $r_0=6$, then $\Omega^{6}(Y_t^{(21)})$ is described with
$(8s\times 9s)$-matrix with one nonzero element that is of the following form{\rm:}
$$b_{(j+1)_s,(1)_s}=f_2(s,1)\kappa^{\ell+1}(\a_{3(j+m-1)})\kappa_0w_{4(j+m+1)\ra 4(j+m+1)+1}\otimes w_{4j\ra 4(j+1)}.$$

$(8)$ If $r_0=7$, then $\Omega^{7}(Y_t^{(21)})$ is described with
$(9s\times 8s)$-matrix with one nonzero element that is of the following form{\rm:}
$$b_{(j+1)_s,(1)_s}=f_2(s,1)\kappa^{\ell+1}(\g_{j+m-1})\kappa_0w_{4(j+m+1+f(s,1))+2\ra 4(j+m+2)}\otimes w_{4j\ra 4(j+1)}.$$

$(9)$ If $r_0=8$, then $\Omega^{8}(Y_t^{(21)})$ is described with
$(8s\times 6s)$-matrix with one nonzero element that is of the following form{\rm:}
$$b_{(j+1)_s,1}=f_2(s,1)\kappa^{\ell+1}(\a_{3(j+m-2)})\kappa_0w_{4(j+m)+3\ra 4(j+m+1)+2}\otimes w_{4j\ra 4(j+1)}.$$

$(10)$ If $r_0=9$, then $\Omega^{9}(Y_t^{(21)})$ is described with
$(6s\times 7s)$-matrix with one nonzero element that is of the following form{\rm:}
$$b_{j,(1)_s}=-f_2(s,1)\kappa^{\ell+1}(\g_{j+m-2})\kappa_0w_{4(j+m)+3\ra 4(j+m+1)+3}\otimes e_{4j}.$$

$(11)$ If $r_0=10$, then $\Omega^{10}(Y_t^{(21)})$ is described with
$(7s\times 6s)$-matrix with one nonzero element that is of the following form{\rm:}
$$b_{(j+1)_s,0}=-\kappa^{\ell+1}(\g_{j+m-2})\kappa_0w_{4(j+m+1)\ra 4(j+m+1)+3}\otimes w_{4j\ra 4(j+1)}.$$

\medskip
$({\rm II})$ Represent an arbitrary $t_0\in\N$ in the form
$t_0=11\ell_0+r_0$, where $0\le r_0\le 10.$ Then
$\Omega^{t_0}(Y_t^{(21)})$ is a $\Omega^{r_0}(Y_t^{(21)})$, whose left
components twisted by $\sigma^{\ell_0}$,
and coefficients multiplied by $(-1)^{\ell_0}$.
\end{pr}

\begin{pr}[Translates for the case 22]
$({\rm I})$ Let $r_0\in\N$, $r_0<11$. $r_0$-translates of the
elements $Y^{(22)}_t$ are described by the following way.

$(1)$ If $r_0=0$, then $\Omega^{0}(Y_t^{(22)})$ is described with
$(6s\times 6s)$-matrix with the following elements $b_{ij}${\rm:}

If $0\le j<s$, then $$b_{ij}=
\begin{cases}
\kappa^\ell(\a_{3(j+m+5)})e_{4(j+m)}\otimes e_{4j},\quad i=j;\\
0,\quad\text{otherwise.}
\end{cases}$$

If $s\le j<5s$, then $b_{ij}=0$.

If $5s\le j<6s$, then $$b_{ij}=
\begin{cases}
-\kappa^\ell(\a_{3(j+m+5)})e_{4(j+m)+3}\otimes e_{4j+3},\quad i=j;\\
0,\quad\text{otherwise.}
\end{cases}$$

$(2)$ If $r_0=1$, then $\Omega^{1}(Y_t^{(22)})$ is described with
$(6s\times 7s)$-matrix with the following elements $b_{ij}${\rm:}

If $0\le j<s$, then $$b_{ij}=
\begin{cases}
\kappa_1w_{4(j+m)+1\ra 4(j+m+1)}\otimes e_{4j},\quad i=j;\\
\kappa_1w_{4(j+m)+2\ra 4(j+m+1)}\otimes w_{4j\ra 4j+1},\quad i=j+2s;\\
\kappa_1w_{4(j+m)+3\ra 4(j+m+1)}\otimes w_{4j\ra 4j+2},\quad i=j+4s;\\
-\kappa_1e_{4(j+m+1)}\otimes w_{4j\ra 4j+3},\quad i=j+6s;\\
0,\quad\text{otherwise,}
\end{cases}$$
where $\kappa_1=-\kappa^\ell(\a_{3(j+m+6)})$.

If $s\le j<3s$, then $$b_{ij}=
\begin{cases}
-e_{4(j+m+1)+1}\otimes w_{4(j+s)+1\ra 4(j+1)},\quad i=(j+1)_{2s};\\
w_{4(j+m)+3\ra 4(j+m+1)+1}\otimes w_{4(j+s)+1\ra 4(j+s)+2},\quad i=j+3s;\\
0,\quad\text{otherwise.}
\end{cases}$$

If $3s\le j<5s$, then $$b_{ij}=
\begin{cases}
-w_{4(j+m+1)+1\ra 4(j+m+1)+2}\otimes w_{4(j+s)+2\ra 4(j+1)},\quad i=(j+1)_{2s};\\
w_{4(j+m)+3\ra 4(j+m+1)+2}\otimes e_{4(j+s)+2},\quad i=j+s;\\
0,\quad\text{otherwise.}
\end{cases}$$

If $5s\le j<6s$, then $$b_{ij}=
\begin{cases}
\kappa_1w_{4(j+m+s+1+sf(j,6s-1))+1\ra 4(j+m+1)+3}\otimes w_{4j+3\ra 4(j+1)},\quad i=(j+1)_s;\\
\kappa_1w_{4(j+m+s+1+sf(j,6s-1))+2\ra 4(j+m+1)+3}\otimes w_{4j+3\ra 4(j+s+1+sf(j,6s-1))+1},\quad i=2s+(j+1)_s;\\
\kappa_1e_{4(j+m+1)+3}\otimes w_{4j+3\ra 4(j+s+1+sf(j,6s-1))+2},\quad i=4s+(j+1)_s;\\
-\kappa_1w_{4(j+m+1)\ra 4(j+m+1)+3}\otimes e_{4j+3},\quad i=j+s;\\
0,\quad\text{otherwise,}
\end{cases}$$
where $\kappa_1=-\kappa^\ell(\a_{3(j+m+6)})$.

$(3)$ If $r_0=2$, then $\Omega^{2}(Y_t^{(22)})$ is described with
$(7s\times 6s)$-matrix with the following elements $b_{ij}${\rm:}

If $0\le j<s$, then $$b_{ij}=
\begin{cases}
e_{4(j+m+s)+1}\otimes w_{4j\ra 4j+2},\quad i=j+3s;\\
0,\quad\text{otherwise.}
\end{cases}$$

If $s\le j<2s$, then $$b_{ij}=
\begin{cases}
-w_{4(j+m)-1\ra 4(j+m+s)+1}\otimes e_{4j},\quad i=j-s;\\
e_{4(j+m+s)+1}\otimes w_{4j\ra 4j+2},\quad i=j+3s;\\
0,\quad\text{otherwise.}
\end{cases}$$

If $2s\le j<4s$, then $b_{ij}=0$.

If $4s\le j<6s$, then $$b_{ij}=
\begin{cases}
\kappa_1w_{4(j+m+s)+1\ra 4(j+m)+3}\otimes e_{4j+2},\quad i=j-s;\\
-\kappa_1e_{4(j+m)+3}\otimes w_{4j+2\ra 4(j+1)},\quad i=(j+1)_s,\text{ }j<5s-1\text{ or }j=6s-1;\\
0,\quad\text{otherwise,}
\end{cases}$$
where $\kappa_1=\kappa^\ell(\a_{3(j+m+5)})f_1(j,5s)$.

If $6s\le j<7s$, then $$b_{ij}=
\begin{cases}
-\kappa^{\ell+1}(\a_{3(j+m)})e_{4(j+m+1)}\otimes e_{4j+3},\quad i=j-s;\\
0,\quad\text{otherwise.}
\end{cases}$$

$(4)$ If $r_0=3$, then $\Omega^{3}(Y_t^{(22)})$ is described with
$(6s\times 8s)$-matrix with the following elements $b_{ij}${\rm:}

If $0\le j<s$, then $$b_{ij}=
\begin{cases}
\kappa_1w_{4(j+m)+2\ra 4(j+m)+3}\otimes e_{4j},\quad i=j;\\
\kappa_1e_{4(j+m)+3}\otimes w_{4j\ra 4j+1},\quad i=j+2s;\\
0,\quad\text{otherwise,}
\end{cases}$$
where $\kappa_1=\kappa^\ell(\a_{3(j+m+5)})$.

If $s\le j<3s$, then $$b_{ij}=
\begin{cases}
-f_2((j)_s,s-1)f_1(j,2s)e_{4(j+m+1)+2}\otimes w_{4(j+s)+1\ra 4(j+1)},\quad i=(j+1)_{2s};\\
w_{4(j+m+1)\ra 4(j+m+1)+2}\otimes w_{4(j+s)+1\ra 4(j+s)+2},\quad i=j+3s;\\
0,\quad\text{otherwise.}
\end{cases}$$

If $3s\le j<5s$, then $$b_{ij}=
\begin{cases}
w_{4(j+m+1)\ra 4(j+m+s+1)+1}\otimes e_{4(j+s)+2},\quad i=j+s;\\
0,\quad\text{otherwise.}
\end{cases}$$

If $5s\le j<6s$, then $$b_{ij}=
\begin{cases}
-\kappa^{\ell+1}(\a_{3(j+m+1)})w_{4(j+m+s+1)+2\ra 4(j+m+2)}\otimes w_{4j+3\ra 4(j+1)},\\\quad\quad\quad i=(j+1)_s,\text{ }j<6s-1;\\
-\kappa^{\ell+1}(\a_{3(j+m+1)})w_{4(j+m+1)+3\ra 4(j+m+2)}\otimes w_{4j+3\ra 4(j+s+1)+1},\\\quad\quad\quad i=2s+(j+1)_s,\text{ }j<6s-1;\\
\kappa^{\ell+1}(\a_{3(j+m+1)})e_{4(j+m+2)}\otimes w_{4j+3\ra 4(j+s+1+sf(j,6s-1))+2},\quad i=4s+(j+1)_s;\\
-\kappa^{\ell+1}(\a_{3(j+m+1)})w_{4(j+m+s+1)+1\ra 4(j+m+2)}\otimes e_{4j+3},\quad i=j+s;\\
0,\quad\text{otherwise.}
\end{cases}$$

$(5)$ If $r_0=4$, then $\Omega^{4}(Y_t^{(22)})$ is described with
$(8s\times 9s)$-matrix with the following elements $b_{ij}${\rm:}

If $0\le j<s$, then $$b_{ij}=
\begin{cases}
-w_{4(j+m)\ra 4(j+m+s)+2}\otimes e_{4j},\quad i=j+s;\\
-w_{4(j+m+s)+1\ra 4(j+m+s)+2}\otimes w_{4j\ra 4j+1},\quad i=j+2s;\\
e_{4(j+m+s)+2}\otimes w_{4j\ra 4j+2},\quad i=j+5s;\\
0,\quad\text{otherwise.}
\end{cases}$$

If $s\le j<2s$, then $$b_{ij}=
\begin{cases}
-w_{4(j+m)-1\ra 4(j+m+s)+2}\otimes e_{4j},\quad i=j-s;\\
-w_{4(j+m)\ra 4(j+m+s)+2}\otimes e_{4j},\quad i=j;\\
w_{4(j+m+s)+1\ra 4(j+m+s)+2}\otimes w_{4j\ra 4j+1},\quad i=j+2s;\\
-e_{4(j+m+s)+2}\otimes w_{4j\ra 4j+2},\quad i=j+6s;\\
0,\quad\text{otherwise.}
\end{cases}$$

If $2s\le j<4s$, then $$b_{ij}=
\begin{cases}
-\kappa^{\ell+1}(\a_{3(j+m-1)})e_{4(j+m)+3}\otimes w_{4j+1\ra 4(j+1)},\quad i=(j+1)_s,\text{ }j<3s-1\text{ or }j=4s-1;\\
0,\quad\text{otherwise.}
\end{cases}$$

If $4s\le j<5s$, then $$b_{ij}=
\begin{cases}
-\kappa_1w_{4(j+m)+1\ra 4(j+m+1)}\otimes e_{4j+2},\quad i=j;\\
\kappa_1e_{4(j+m+1)}\otimes w_{4j+2\ra 4(j+1)},\quad i=s+(j+1)_s,\text{ }j<5s-1;\\
0,\quad\text{otherwise,}
\end{cases}$$
where $\kappa_1=-\kappa^{\ell+1}(\a_{3(j+m)})$.

If $5s\le j<6s$, then $$b_{ij}=
\begin{cases}
-\kappa_1w_{4(j+m+s)+2\ra 4(j+m+1)}\otimes e_{4j+2},\quad i=j+2s;\\
\kappa_1w_{4(j+m)+3\ra 4(j+m+1)}\otimes w_{4j+2\ra 4(j+1)},\quad i=(j+1)_s,\text{ }j=6s-1;\\
\kappa_1e_{4(j+m+1)}\otimes w_{4j+2\ra 4(j+1)},\quad i=s+(j+1)_s,\text{ }j=6s-1;\\
0,\quad\text{otherwise,}
\end{cases}$$
where $\kappa_1=-\kappa^{\ell+1}(\a_{3(j+m)})$.

If $6s\le j<7s-1$, then $$b_{ij}=
\begin{cases}
w_{4(j+m+1)\ra 4(j+m+s+1)+1}\otimes w_{4j+3\ra 4(j+1)},\quad i=s+(j+1)_s;\\
0,\quad\text{otherwise.}
\end{cases}$$

If $7s-1\le j<7s$, then $b_{ij}=0$.

If $7s\le j<8s$, then $$b_{ij}=
\begin{cases}
-w_{4(j+m)+3\ra 4(j+m+s+1)+1}\otimes e_{4j+3},\quad i=j+s;\\
w_{4(j+m+1)\ra 4(j+m+s+1)+1}\otimes w_{4j+3\ra 4(j+1)},\quad i=s+(j+1)_s,\text{ }j=8s-1;\\
0,\quad\text{otherwise.}
\end{cases}$$

$(6)$ If $r_0=5$, then $\Omega^{5}(Y_t^{(22)})$ is described with
$(9s\times 8s)$-matrix with the following elements $b_{ij}${\rm:}

If $0\le j<s$, then $$b_{ij}=
\begin{cases}
\kappa_1w_{4(j+m)+1\ra 4(j+m)+3}\otimes e_{4j},\quad i=j;\\
\kappa_1w_{4(j+m+s)+1\ra 4(j+m)+3}\otimes e_{4j},\quad i=j+s;\\
-\kappa_1e_{4(j+m)+3}\otimes w_{4j\ra 4j+2},\quad i=j+4s;\\
-\kappa_1e_{4(j+m)+3}\otimes w_{4j\ra 4(j+s)+2},\quad i=j+5s;\\
0,\quad\text{otherwise,}
\end{cases}$$
where $\kappa_1=\kappa^{\ell+1}(\a_{3(j+m-1)})$.

If $s\le j<2s$, then $$b_{ij}=
\begin{cases}
-\kappa_1w_{4(j+m+s)+1\ra 4(j+m+1)}\otimes e_{4j},\quad i=j-s;\\
\kappa_1e_{4(j+m+1)}\otimes w_{4j\ra 4(j+s)+1},\quad i=j+s;\\
-\kappa_1e_{4(j+m+1)}\otimes w_{4j\ra 4j+1},\quad i=j+2s;\\
0,\quad\text{otherwise,}
\end{cases}$$
where $\kappa_1=-\kappa^{\ell+1}(\a_{3(j+m)})$.

If $2s\le j<3s$, then $$b_{ij}=
\begin{cases}
e_{4(j+m+1)+1}\otimes w_{4j+1\ra 4(j+1)},\quad i=(j+1)_{2s};\\
w_{4(j+m)+3\ra 4(j+m+1)+1}\otimes w_{4j+1\ra 4j+2},\quad i=j+2s;\\
0,\quad\text{otherwise.}
\end{cases}$$

If $3s\le j<4s$, then $$b_{ij}=
\begin{cases}
e_{4(j+m+1)+1}\otimes w_{4j+1\ra 4(j+1)},\quad i=(j+1)_{2s};\\
0,\quad\text{otherwise.}
\end{cases}$$

If $4s\le j<5s$, then $b_{ij}=0$.

If $5s\le j<6s$, then $$b_{ij}=
\begin{cases}
w_{4(j+m)+3\ra 4(j+m+s+1)+2}\otimes e_{4(j+s)+2},\quad i=j-s;\\
0,\quad\text{otherwise.}
\end{cases}$$

If $6s\le j<7s$, then $$b_{ij}=
\begin{cases}
-w_{4(j+m)+3\ra 4(j+m+1)+1}\otimes e_{4(j+s)+2},\quad i=j-s;\\
0,\quad\text{otherwise.}
\end{cases}$$

If $7s\le j<8s$, then $b_{ij}=0$.

If $8s\le j<9s$, then $$b_{ij}=
\begin{cases}
\kappa_1w_{4(j+m+1)+2\ra 4(j+m+1)+3}\otimes e_{4j+3},\quad i=j-2s;\\
\kappa_1w_{4(j+m+1)+1\ra 4(j+m+1)+3}\otimes w_{4j+3\ra 4(j+1)},\quad i=(j+1)_s,\text{ }j<9s-1;\\
\kappa_1w_{4(j+m+s+1)+1\ra 4(j+m+1)+3}\otimes w_{4j+3\ra 4(j+1)},\quad i=s+(j+1)_s,\text{ }j<9s-1;\\
-\kappa_1e_{4(j+m+1)+3}\otimes w_{4j+3\ra 4(j+1)+2},\quad i=4s+(j+1)_s,\text{ }j<9s-1;\\
-\kappa_1e_{4(j+m+1)+3}\otimes w_{4j+3\ra 4(j+s+1)+2},\quad i=5s+(j+1)_s,\text{ }j<9s-1;\\
0,\quad\text{otherwise,}
\end{cases}$$
where $\kappa_1=\kappa^{\ell+1}(\a_{3(j+m)})$.

$(7)$ If $r_0=6$, then $\Omega^{6}(Y_t^{(22)})$ is described with
$(8s\times 9s)$-matrix with the following elements $b_{ij}${\rm:}

If $0\le j<2s$, then $$b_{ij}=
\begin{cases}
w_{4(j+m)\ra 4(j+m+s)+1}\otimes e_{4j},\quad i=(j)_s;\\
e_{4(j+m+s)+1}\otimes w_{4j\ra 4(j+s)+1},\quad i=j+3sf_0(j,s);\\
0,\quad\text{otherwise.}
\end{cases}$$

If $2s\le j<4s$, then $$b_{ij}=
\begin{cases}
\kappa^{\ell+1}(\a_{3(j+m)})e_{4(j+m+1)}\otimes w_{4j+1\ra 4(j+1)},\quad i=(j+1)_s;\\
0,\quad\text{otherwise.}
\end{cases}$$

If $4s\le j<6s$, then $$b_{ij}=
\begin{cases}
\kappa_1w_{4(j+m)+2\ra 4(j+m)+3}\otimes e_{4j+2},\quad i=j+s;\\
-\kappa_1e_{4(j+m)+3}\otimes w_{4j+2\ra 4j+3},\quad i=j+2s,\text{ }j\ge 5s;\\
0,\quad\text{otherwise,}
\end{cases}$$
where $\kappa_1=-\kappa^{\ell+1}(\a_{3(j+m-1)})$.

If $6s\le j<7s$, then $b_{ij}=0$.

If $7s\le j<8s$, then $$b_{ij}=
\begin{cases}
w_{4(j+m)+3\ra 4(j+m+s+1)+2}\otimes e_{4j+3},\quad i=j;\\
0,\quad\text{otherwise.}
\end{cases}$$

$(8)$ If $r_0=7$, then $\Omega^{7}(Y_t^{(22)})$ is described with
$(9s\times 8s)$-matrix with the following elements $b_{ij}${\rm:}

If $0\le j<s$, then $$b_{ij}=
\begin{cases}
\kappa_1w_{4(j+m+s)+2\ra 4(j+m+1)}\otimes e_{4j},\quad i=j+s;\\
-\kappa_1w_{4(j+m)+3\ra 4(j+m+1)}\otimes w_{4j\ra 4j+1},\quad i=j+2s;\\
\kappa_1e_{4(j+m+1)}\otimes w_{4j\ra 4j+2},\quad i=j+4s;\\
\kappa_1e_{4(j+m+1)}\otimes w_{4j\ra 4(j+s)+2},\quad i=j+5s;\\
0,\quad\text{otherwise,}
\end{cases}$$
where $\kappa_1=-\kappa^{\ell+1}(\a_{3(j+m)})$.

If $s\le j<2s$, then $b_{ij}=0$.

If $2s\le j<3s$, then $$b_{ij}=
\begin{cases}
-e_{4(j+m+1)+2}\otimes w_{4j+1\ra 4(j+1)},\quad i=(j+1)_{2s};\\
-w_{4(j+m)+3\ra 4(j+m+1)+2}\otimes e_{4j+1},\quad i=j;\\
0,\quad\text{otherwise.}
\end{cases}$$

If $3s\le j<4s$, then $b_{ij}=0$.

If $4s\le j<5s$, then $$b_{ij}=
\begin{cases}
-e_{4(j+m+s+1)+2}\otimes w_{4(j+s)+1\ra 4(j+1)},\quad i=(j+s+1)_{2s};\\
-w_{4(j+m)+3\ra 4(j+m+s+1)+2}\otimes e_{4(j+s)+1},\quad i=j-s;\\
0,\quad\text{otherwise.}
\end{cases}$$

If $5s\le j<7s$, then $b_{ij}=0$.

If $7s\le j<8s$, then $$b_{ij}=
\begin{cases}
-\kappa_1w_{4(j+m+1+sf(j,8s-1))+2\ra 4(j+m+1)+3}\otimes w_{4j+3\ra 4(j+1)},\quad i=s+(j+1)_s;\\
\kappa_1e_{4(j+m+1)+3}\otimes w_{4j+3\ra 4(j+s+1+sf(j,8s-1))+1},\quad i=2s+(j+1)_s;\\
\kappa_1w_{4(j+m+s+1)+1\ra 4(j+m+1)+3}\otimes e_{4j+3},\quad i=j-s;\\
\kappa_1w_{4(j+m+1)+1\ra 4(j+m+1)+3}\otimes e_{4j+3},\quad i=j;\\
0,\quad\text{otherwise,}
\end{cases}$$
where $\kappa_1=-\kappa^{\ell+1}(\a_{3(j+m)})$.

If $8s\le j<9s$, then $b_{ij}=0$.

$(9)$ If $r_0=8$, then $\Omega^{8}(Y_t^{(22)})$ is described with
$(8s\times 6s)$-matrix with the following elements $b_{ij}${\rm:}

If $0\le j<2s$, then $$b_{ij}=
\begin{cases}
w_{4(j+m)-1\ra 4(j+m+s)+2}\otimes e_{4j},\quad i=j,\text{ }j<s;\\
-e_{4(j+m+s)+2}\otimes w_{4j\ra 4(j+s)+1},\quad i=s+(j+s)_{2s};\\
-w_{4(j+m+s)+1\ra 4(j+m+s)+2}\otimes w_{4j\ra 4j+2},\quad i=j+3s;\\
0,\quad\text{otherwise.}
\end{cases}$$

If $2s\le j<4s$, then $$b_{ij}=
\begin{cases}
-\kappa_1w_{4(j+m)+2\ra 4(j+m)+3}\otimes e_{4j+1},\quad i=j-s;\\
\kappa_1w_{4(j+m+s)+1\ra 4(j+m)+3}\otimes w_{4j+1\ra 4j+2},\quad i=j+s;\\
\kappa_1e_{4(j+m)+3}\otimes w_{4j+1\ra 4(j+1)},\quad i=(j+1)_s,\text{ }j<3s-1\text{ or }j=4s-1;\\
0,\quad\text{otherwise,}
\end{cases}$$
where $\kappa_1=-\kappa^{\ell+1}(\a_{3(j+m-1)})$.

If $4s\le j<5s$, then $b_{ij}=0$.

If $5s\le j<6s$, then $$b_{ij}=
\begin{cases}
\kappa^{\ell+1}(\a_{3(j+m)})e_{4(j+m+1)}\otimes w_{4j+2\ra 4j+3},\quad i=j;\\
0,\quad\text{otherwise.}
\end{cases}$$

If $6s\le j<7s$, then $$b_{ij}=
\begin{cases}
-w_{4(j+m+1)\ra 4(j+m+s+1)+1}\otimes e_{4j+3},\quad i=j-s;\\
0,\quad\text{otherwise.}
\end{cases}$$

If $7s\le j<8s$, then $b_{ij}=0$.

$(10)$ If $r_0=9$, then $\Omega^{9}(Y_t^{(22)})$ is described with
$(6s\times 7s)$-matrix with the following elements $b_{ij}${\rm:}

If $0\le j<s$, then $$b_{ij}=
\begin{cases}
\kappa^{\ell+1}(\a_{3(j+m-1)})e_{4(j+m)+3}\otimes e_{4j},\quad i=j;\\
0,\quad\text{otherwise.}
\end{cases}$$

If $s\le j<3s$, then $$b_{ij}=
\begin{cases}
w_{4(j+m+1)\ra 4(j+m+1)+2}\otimes e_{4(j+s)+1},\quad i=j;\\
0,\quad\text{otherwise.}
\end{cases}$$

If $3s\le j<5s$, then $b_{ij}=0$.

If $5s\le j<6s$, then $$b_{ij}=
\begin{cases}
\kappa_1e_{4(j+m+2)}\otimes w_{4j+3\ra 4(j+s+1+sf(j,6s-1))+1},\quad i=s+(j+1)_s;\\
\kappa_1w_{4(j+m+s+1)+2\ra 4(j+m+2)}\otimes e_{4j+3},\quad i=j;\\
-\kappa_1w_{4(j+m+1)+3\ra 4(j+m+2)}\otimes w_{4j+3\ra 4(j+1)},\quad i=(j+1)_s,\text{ }j=6s-1;\\
0,\quad\text{otherwise,}
\end{cases}$$
where $\kappa_1=-\kappa^{\ell+1}(\a_{3(j+m+1)})$.

$(11)$ If $r_0=10$, then $\Omega^{10}(Y_t^{(22)})$ is described with
$(7s\times 6s)$-matrix with the following elements $b_{ij}${\rm:}

If $0\le j<s$, then $$b_{ij}=
\begin{cases}
-\kappa_1w_{4(j+m)\ra 4(j+m)+3}\otimes e_{4j},\quad i=j;\\
-\kappa_1w_{4(j+m+s)+1\ra 4(j+m)+3}\otimes w_{4j\ra 4j+1},\quad i=j+s;\\
\kappa_1w_{4(j+m)+1\ra 4(j+m)+3}\otimes w_{4j\ra 4(j+s)+1},\quad i=j+2s;\\
\kappa_1w_{4(j+m+s)+2\ra 4(j+m)+3}\otimes w_{4j\ra 4j+2},\quad i=j+3s;\\
-\kappa_1w_{4(j+m)+2\ra 4(j+m)+3}\otimes w_{4j\ra 4(j+s)+2},\quad i=j+4s;\\
\kappa_1e_{4(j+m)+3}\otimes w_{4j\ra 4j+3},\quad i=j+5s;\\
0,\quad\text{otherwise,}
\end{cases}$$
where $\kappa_1=-\kappa^{\ell+1}(\a_{3(j+m-1)})$.

If $s\le j<3s$, then $$b_{ij}=
\begin{cases}
-\kappa_1w_{4(j+m)+1\ra 4(j+m+1)}\otimes e_{4(j+s)+1},\quad i=j;\\
\kappa_1w_{4(j+m)+2\ra 4(j+m+1)}\otimes w_{4(j+s)+1\ra 4(j+s)+2},\quad i=j+2s;\\
\kappa_1w_{4(j+m)+3\ra 4(j+m+1)}\otimes w_{4(j+s)+1\ra 4j+3},\quad i=j+4s,\text{ }j<2s;\\
\kappa_1e_{4(j+m+1)}\otimes w_{4(j+s)+1\ra 4(j+1)},\quad i=(j+1)_s,\text{ }j<2s-1\text{ or }j=3s-1;\\
0,\quad\text{otherwise,}
\end{cases}$$
where $\kappa_1=-\kappa^{\ell+1}(\a_{3(j+m)})$.

If $3s\le j<4s$, then $$b_{ij}=
\begin{cases}
w_{4(j+m+1)\ra 4(j+m+s+1)+1}\otimes w_{4(j+s)+2\ra 4(j+1)},\quad i=(j+1)_s,\text{ }j<4s-1;\\
w_{4(j+m)+3\ra 4(j+m+s+1)+1}\otimes w_{4(j+s)+2\ra 4j+3},\quad i=j+2s;\\
0,\quad\text{otherwise.}
\end{cases}$$

If $4s\le j<5s-1$, then $b_{ij}=0$.

If $5s-1\le j<5s$, then $$b_{ij}=
\begin{cases}
w_{4(j+m+1)\ra 4(j+m+s+1)+1}\otimes w_{4(j+s)+2\ra 4(j+1)},\quad i=(j+1)_s;\\
0,\quad\text{otherwise.}
\end{cases}$$

If $5s\le j<6s-1$, then $b_{ij}=0$.

If $6s-1\le j<6s$, then $$b_{ij}=
\begin{cases}
-w_{4(j+m+1)\ra 4(j+m+1)+2}\otimes w_{4j+3\ra 4(j+1)},\quad i=(j+1)_s;\\
0,\quad\text{otherwise.}
\end{cases}$$

If $6s\le j<7s$, then $$b_{ij}=
\begin{cases}
-w_{4(j+m)+3\ra 4(j+m+1)+2}\otimes e_{4j+3},\quad i=j-s;\\
-w_{4(j+m+1)\ra 4(j+m+1)+2}\otimes w_{4j+3\ra 4(j+1)},\quad i=(j+1)_s,\text{ }j<7s-1;\\
0,\quad\text{otherwise.}
\end{cases}$$

\medskip
$({\rm II})$ Represent an arbitrary $t_0\in\N$ in the form
$t_0=11\ell_0+r_0$, where $0\le r_0\le 10.$ Then
$\Omega^{t_0}(Y_t^{(22)})$ is a $\Omega^{r_0}(Y_t^{(22)})$, whose left
components twisted by $\sigma^{\ell_0}$,
and coefficients multiplied by $(-1)^{\ell_0}$.
\end{pr}

\begin{pr}[Translates for the case 23]
$({\rm I})$ Let $r_0\in\N$, $r_0<11$. $r_0$-translates of the
elements $Y^{(23)}_t$ are described by the following way.

$(1)$ If $r_0=0$, then $\Omega^{0}(Y_t^{(23)})$ is described with
$(6s\times 6s)$-matrix with one nonzero element that is of the following form{\rm:}
$$b_{j,0}=w_{4(j+m)\ra 4(j+m+1)}\otimes e_{4j}.$$

$(2)$ If $r_0=1$, then $\Omega^{1}(Y_t^{(23)})$ is described with
$(7s\times 7s)$-matrix with one nonzero element that is of the following form{\rm:}
$$b_{j,6s}=w_{4(j+m+1)\ra 4(j+m+2)}\otimes e_{4j+3}.$$

$(3)$ If $r_0=2$, then $\Omega^{2}(Y_t^{(23)})$ is described with
$(6s\times 6s)$-matrix with one nonzero element that is of the following form{\rm:}
$$b_{j,5s}=w_{4(j+m+1)\ra 4(j+m+2)}\otimes e_{4j+3}.$$

$(4)$ If $r_0=3$, then $\Omega^{3}(Y_t^{(23)})$ is described with
$(8s\times 8s)$-matrix with the following two nonzero elements{\rm:}
$$b_{j,4s}=w_{4(j+m+1)\ra 4(j+m+2)}\otimes e_{4j+2};$$
$$b_{j,5s}=w_{4(j+m+1)\ra 4(j+m+2)}\otimes e_{4j+2}.$$

$(5)$ If $r_0=4$, then $\Omega^{4}(Y_t^{(23)})$ is described with
$(9s\times 9s)$-matrix with one nonzero element that is of the following form{\rm:}
$$b_{j,s}=w_{4(j+m)\ra 4(j+m+1)}\otimes e_{4j}.$$

$(6)$ If $r_0=5$, then $\Omega^{5}(Y_t^{(23)})$ is described with
$(8s\times 8s)$-matrix with the following two nonzero elements{\rm:}
$$b_{j,2s}=w_{4(j+m+1)\ra 4(j+m+2)}\otimes e_{4j+1};$$
$$b_{j,3s}=w_{4(j+m+1)\ra 4(j+m+2)}\otimes e_{4j+1}.$$

$(7)$ If $r_0=6$, then $\Omega^{6}(Y_t^{(23)})$ is described with
$(9s\times 9s)$-matrix with the following two nonzero elements{\rm:}
$$b_{j,0}=w_{4(j+m)\ra 4(j+m+1)}\otimes e_{4j};$$
$$b_{j,8s}=w_{4(j+m+1)\ra 4(j+m+2)}\otimes e_{4j+3}.$$

$(8)$ If $r_0=7$, then $\Omega^{7}(Y_t^{(23)})$ is described with
$(8s\times 8s)$-matrix with the following two nonzero elements{\rm:}
$$b_{j,4s}=w_{4(j+m+1)\ra 4(j+m+2)}\otimes e_{4j+2};$$
$$b_{j,5s}=w_{4(j+m+1)\ra 4(j+m+2)}\otimes e_{4j+2}.$$

$(9)$ If $r_0=8$, then $\Omega^{8}(Y_t^{(23)})$ is described with
$(6s\times 6s)$-matrix with one nonzero element that is of the following form{\rm:}
$$b_{j,5s}=w_{4(j+m+1)\ra 4(j+m+2)}\otimes e_{4j+3}.$$

$(10)$ If $r_0=9$, then $\Omega^{9}(Y_t^{(23)})$ is described with
$(7s\times 7s)$-matrix with the following two nonzero elements{\rm:}
$$b_{j,s}=w_{4(j+m+1)\ra 4(j+m+2)}\otimes e_{4(j+s)+1};$$
$$b_{j,2s}=w_{4(j+m+1)\ra 4(j+m+2)}\otimes e_{4(j+s)+1}.$$

$(11)$ If $r_0=10$, then $\Omega^{10}(Y_t^{(23)})$ is described with
$(6s\times 6s)$-matrix with one nonzero element that is of the following form{\rm:}
$$b_{j,0}=w_{4(j+m)\ra 4(j+m+1)}\otimes e_{4j}.$$

\medskip
$({\rm II})$ Represent an arbitrary $t_0\in\N$ in the form
$t_0=11\ell_0+r_0$, where $0\le r_0\le 10.$ Then
$\Omega^{t_0}(Y_t^{(23)})$ is a $\Omega^{r_0}(Y_t^{(23)})$, whose left
components twisted by $\sigma^{\ell_0}$.
\end{pr}

\begin{pr}[Translates for the case 24]
$({\rm I})$ Let $r_0\in\N$, $r_0<11$. $r_0$-translates of the
elements $Y^{(24)}_t$ are described by the following way.

$(1)$ If $r_0=0$, then $\Omega^{0}(Y_t^{(24)})$ is described with
$(6s\times 6s)$-matrix with one nonzero element that is of the following form{\rm:}
$$b_{j,5}=w_{4j+3\ra 4(j+1)+3}\otimes e_{4j+3}.$$

$(2)$ If $r_0=1$, then $\Omega^{1}(Y_t^{(24)})$ is described with
$(7s\times 7s)$-matrix with the following two nonzero elements{\rm:}
$$b_{j,4s}=w_{4(j+m)+3\ra 4(j+m+1)+3}\otimes e_{4j+2};$$
$$b_{j,5s}=w_{4(j+m)+3\ra 4(j+m+1)+3}\otimes e_{4j+2}.$$

$(3)$ If $r_0=2$, then $\Omega^{2}(Y_t^{(24)})$ is described with
$(6s\times 6s)$-matrix with one nonzero element that is of the following form{\rm:}
$$b_{j,0}=w_{4(j+m)+3\ra 4(j+m+1)+3}\otimes e_{4j}.$$

$(4)$ If $r_0=3$, then $\Omega^{3}(Y_t^{(24)})$ is described with
$(8s\times 8s)$-matrix with the following two nonzero elements{\rm:}
$$b_{j-2s,2s}=w_{4(j+m)+2\ra 4(j+m)+3}\otimes w_{4j+1\ra 4j};$$
$$b_{j-2s,3s}=-w_{4(j+m)+2\ra 4(j+m)+3}\otimes w_{4j+1\ra 4j}.$$

$(5)$ If $r_0=4$, then $\Omega^{4}(Y_t^{(24)})$ is described with
$(9s\times 9s)$-matrix with the following two nonzero elements{\rm:}
$$b_{j,0}=e_{4(j+m)+3}\otimes w_{4j\ra 4(j+1)};$$
$$b_{j,s}=e_{4(j+m)}\otimes w_{4j\ra 4(j+1)}.$$

$(6)$ If $r_0=5$, then $\Omega^{5}(Y_t^{(24)})$ is described with
$(8s\times 8s)$-matrix with the following two nonzero elements{\rm:}
$$b_{j,0}=e_{4(j+m)+1}\otimes w_{4j\ra 4(j+1)};$$
$$b_{j,s}=e_{4(j+m)+1}\otimes w_{4j\ra 4(j+1)}.$$

$(7)$ If $r_0=6$, then $\Omega^{6}(Y_t^{(24)})$ is described with
$(9s\times 9s)$-matrix with one nonzero element that is of the following form{\rm:}
$$b_{j,0}=e_{4(j+m)}\otimes w_{4j\ra 4(j+1)}.$$

$(8)$ If $r_0=7$, then $\Omega^{7}(Y_t^{(24)})$ is described with
$(8s\times 8s)$-matrix with the following two nonzero elements{\rm:}
$$b_{j,0}=e_{4(j+m)+2}\otimes w_{4j\ra 4(j+1)};$$
$$b_{j,s}=e_{4(j+m)+2}\otimes w_{4j\ra 4(j+1)}.$$

$(9)$ If $r_0=8$, then $\Omega^{8}(Y_t^{(24)})$ is described with
$(6s\times 6s)$-matrix with one nonzero element that is of the following form{\rm:}
$$b_{j,0}=e_{4(j+m)+3}\otimes w_{4j\ra 4(j+1)}.$$

$(10)$ If $r_0=9$, then $\Omega^{9}(Y_t^{(24)})$ is described with
$(7s\times 7s)$-matrix with the following two nonzero elements{\rm:}
$$b_{j+s,0}=-w_{4(j+m)\ra 4(j+m)+3}\otimes w_{4j\ra 4j+1};$$
$$b_{j+2s,0}=w_{4(j+m)\ra 4(j+m)+3}\otimes w_{4j\ra 4(j+1)+1}.$$

$(11)$ If $r_0=10$, then $\Omega^{10}(Y_t^{(24)})$ is described with
$(6s\times 6s)$-matrix with the following nonzero elements{\rm:}
$$b_{j,0}=-w_{4(j+m)\ra 4(j+m+1)}\otimes e_{4j};$$
$$b_{j-5s,5s}=w_{4(j+m)\ra 4(j+m)+3}\otimes w_{4j+3\ra 4j};$$
$$b_{j,5s}=w_{4(j+m)+3\ra 4(j+m+1)+3}\otimes e_{4j+3}.$$

\medskip
$({\rm II})$ Represent an arbitrary $t_0\in\N$ in the form
$t_0=11\ell_0+r_0$, where $0\le r_0\le 10.$ Then
$\Omega^{t_0}(Y_t^{(24)})$ is a $\Omega^{r_0}(Y_t^{(24)})$, whose left
components twisted by $\sigma^{\ell_0}$.
\end{pr}

\section{Multiplications in $\HH^*(R)$}

From the descriptions of elements $Y^{(i)}_t$ and
its $\Omega$-translates we can find multiplications of the elements
using the formula \eqref{mult_formula}.

We will find a multiplication of elements of the types 4 and 3 for
$s>1$.

Consider two arbitrary elements $Y_{t_4}^{(4)}$ and $Y_{t_3}^{(3)}$.
For its degrees $t_4$ and $t_3$ we have:
\begin{align*}
t_4&=11\ell_4+1,\text{ }\ell_4(n+s)\equiv s(2s),\text{ }\ell_4\ndiv 2\text{ or }\myChar=2;\\
t_3&=11\ell_3+1,\text{ }\ell_3(n+s)\equiv 0(2s),\text{ }\ell_3\div 2\text{ or }\myChar=2.
\end{align*}
Let $t=t_4+t_3$; this is the degree of an element
$Y_{t_4}^{(4)}Y_{t_3}^{(3)}$. Then $t=11(\ell_4+\ell_3)+2$. Group of
the degree $t$ has type (5). $Y^{(3)}_t$ is an $(7s\times 6s)$-matrix with two nonzero elements
$y_{0,0}=w_{0\ra 1}\otimes e_0$ and $y_{0,s}=w_{0\ra 4s+1}\otimes e_0$. 
$\Omega^{t_3}(Y_{t_4}^{(4)})$ is an $(6s\times 7s)$-matrix that was described in proposition \ref{type4}.
Multiplication of $\Omega^{t_3}(Y_{t_4}^{(4)})$ and $Y^{(3)}_t$ an $(6s\times 6s)$-matrix 
with the following two nonzero elements:
$$b_{0,0}=\kappa^{\ell_4}(\a_0)\kappa^{\ell_3}(\a_1)\kappa^{\ell_3}(\a_2)w_{0\ra 3}\otimes e_0;$$
$$b_{0,3s-1}=-\kappa^{\ell_3}(\a_{3s-1})w_{0\ra 4s+2}\otimes w_{8s-3\ra 0}.$$
We have $\ell_3\div 2$, hence $\kappa^{\ell_3}(\a_1)=\kappa^{\ell_3}(\a_{3s-1})=1$, $\kappa^{\ell_3}(\a_2)=\kappa^{\ell_3}(\a_0)$,
$\kappa^{\ell_4}(\a_0)\kappa^{\ell_3}(\a_1)\kappa^{\ell_3}(\a_2)=\kappa^{\ell_4}(\a_0)\kappa^{\ell_3}(\a_0)=\kappa^{\ell_4+\ell_3}(\a_0)$.
$Y^{(5)}_t$ is an $(6s\times 6s)$-matrix with a single nonzero element $y_{0,0}=\kappa^\ell(\a_0)w_{0\ra 3}\otimes e_{0}$, 
hence $Y^{(4)}Y^{(3)}$ coincide with $Y^{(5)}$ for degree of type (5).


Multiplications of other elements, except $Y^{(5)}$, $Y^{(10)}$, $Y^{(17)}$, $Y^{(19)}$ and
$Y^{(21)}$, are similarly considered. To get the whole picture we should prove the following lemma.

\begin{lem}$\text{ }$

$($a$)$ Let $Y^{(5)}$ be an arbitrary element from generators of the
corresponding type. Then there are elements $Y^{(3)}$ and $Y^{(4)}$
such as $Y^{(5)}=Y^{(3)}Y^{(4)}$.

$($b$)$ Let $Y^{(10)}$ be an arbitrary element from generators of
the corresponding type. Then there are elements $Y^{(3)}$ and
$Y^{(6)}$ such as $Y^{(10)}=Y^{(3)}Y^{(6)}$.

$($c$)$ Let $Y^{(17)}$ be an arbitrary element from generators of
the corresponding type. Then there are elements $Y^{(3)}$ and
$Y^{(15)}$ such as $Y^{(17)}=Y^{(3)}Y^{(15)}$.

$($d$)$ Let $Y^{(19)}$ be an arbitrary element from generators of
the corresponding type. Then there are elements $Y^{(3)}$ and
$Y^{(18)}$ such as $Y^{(19)}=Y^{(3)}Y^{(18)}$.

$($e$)$ Let $Y^{(21)}$ be an arbitrary element from generators of
the corresponding type. Then there are elements $Y^{(3)}$ and
$Y^{(20)}$ such as $Y^{(21)}=Y^{(3)}Y^{(20)}$.

\end{lem}

\begin{proof}
The degree 1 has type 3, for all $s$. It only remains to use the
relations for type (3).
\end{proof}


\end{document}